\documentclass[reqno]{amsart}
\usepackage{fullpage,graphicx}
\usepackage[margin=1in]{geometry}
\usepackage{amsmath}
\usepackage{amsfonts,amssymb}
\usepackage{newlfont,mathrsfs,mathtools}
\usepackage{txfonts}
\usepackage{etoolbox}

\makeatletter
\let\old@tocline\@tocline
\let\section@tocline\@tocline
\newcommand{\subsection@dotsep}{4.5}
\newcommand{\subsubsection@dotsep}{4.5}
\patchcmd{\@tocline}
{\hfil}
{\nobreak
	\leaders\hbox{$\m@th
		\mkern \subsection@dotsep mu\hbox{.}\mkern \subsection@dotsep mu$}\hfill
	\nobreak}{}{}
\let\subsection@tocline\@tocline
\let\@tocline\old@tocline

\patchcmd{\@tocline}
{\hfil}
{\nobreak
	\leaders\hbox{$\m@th
		\mkern \subsubsection@dotsep mu\hbox{.}\mkern \subsubsection@dotsep mu$}\hfill
	\nobreak}{}{}
\let\subsubsection@tocline\@tocline
\let\@tocline\old@tocline

\let\old@l@subsection\l@subsection
\let\old@l@subsubsection\l@subsubsection

\def\@tocwriteb#1#2#3{%
	\begingroup
	\@xp\def\csname #2@tocline\endcsname##1##2##3##4##5##6{%
		\ifnum##1>\c@tocdepth
		\else \sbox\z@{##5\let\indentlabel\@tochangmeasure##6}\fi}%
	\csname l@#2\endcsname{#1{\csname#2name\endcsname}{\@secnumber}{}}%
	\endgroup
	\addcontentsline{toc}{#2}%
	{\protect#1{\csname#2name\endcsname}{\@secnumber}{#3}}}%

\newlength{\@tocsectionindent}
\newlength{\@tocsubsectionindent}
\newlength{\@tocsubsubsectionindent}
\newlength{\@tocsectionnumwidth}
\newlength{\@tocsubsectionnumwidth}
\newlength{\@tocsubsubsectionnumwidth}
\newcommand{\settocsectionnumwidth}[1]{\setlength{\@tocsectionnumwidth}{#1}}
\newcommand{\settocsubsectionnumwidth}[1]{\setlength{\@tocsubsectionnumwidth}{#1}}
\newcommand{\settocsubsubsectionnumwidth}[1]{\setlength{\@tocsubsubsectionnumwidth}{#1}}
\newcommand{\settocsectionindent}[1]{\setlength{\@tocsectionindent}{#1}}
\newcommand{\settocsubsectionindent}[1]{\setlength{\@tocsubsectionindent}{#1}}
\newcommand{\settocsubsubsectionindent}[1]{\setlength{\@tocsubsubsectionindent}{#1}}

\renewcommand{\l@section}{\section@tocline{1}{\@tocsectionvskip}{\@tocsectionindent}{}{\@tocsectionformat}}%
\renewcommand{\l@subsection}{\subsection@tocline{2}{\@tocsubsectionvskip}{\@tocsubsectionindent}{}{\@tocsubsectionformat}}%
\renewcommand{\l@subsubsection}{\subsubsection@tocline{3}{\@tocsubsubsectionvskip}{\@tocsubsubsectionindent}{}{\@tocsubsubsectionformat}}%
\newcommand{\@tocsectionformat}{}
\newcommand{\@tocsubsectionformat}{}
\newcommand{\@tocsubsubsectionformat}{}
\expandafter\def\csname toc@1format\endcsname{\@tocsectionformat}
\expandafter\def\csname toc@2format\endcsname{\@tocsubsectionformat}
\expandafter\def\csname toc@3format\endcsname{\@tocsubsubsectionformat}
\newcommand{\settocsectionformat}[1]{\renewcommand{\@tocsectionformat}{#1}}
\newcommand{\settocsubsectionformat}[1]{\renewcommand{\@tocsubsectionformat}{#1}}
\newcommand{\settocsubsubsectionformat}[1]{\renewcommand{\@tocsubsubsectionformat}{#1}}
\newlength{\@tocsectionvskip}
\newcommand{\settocsectionvskip}[1]{\setlength{\@tocsectionvskip}{#1}}
\newlength{\@tocsubsectionvskip}
\newcommand{\settocsubsectionvskip}[1]{\setlength{\@tocsubsectionvskip}{#1}}
\newlength{\@tocsubsubsectionvskip}
\newcommand{\settocsubsubsectionvskip}[1]{\setlength{\@tocsubsubsectionvskip}{#1}}

\patchcmd{\tocsection}{\indentlabel}{\makebox[\@tocsectionnumwidth][l]}{}{}
\patchcmd{\tocsubsection}{\indentlabel}{\makebox[\@tocsubsectionnumwidth][l]}{}{}
\patchcmd{\tocsubsubsection}{\indentlabel}{\makebox[\@tocsubsubsectionnumwidth][l]}{}{}

\newcommand{\@sectypepnumformat}{}
\renewcommand{\contentsline}[1]{%
	\expandafter\let\expandafter\@sectypepnumformat\csname @toc#1pnumformat\endcsname%
	\csname l@#1\endcsname}
\newcommand{\@tocsectionpnumformat}{}
\newcommand{\@tocsubsectionpnumformat}{}
\newcommand{\@tocsubsubsectionpnumformat}{}
\newcommand{\setsectionpnumformat}[1]{\renewcommand{\@tocsectionpnumformat}{#1}}
\newcommand{\setsubsectionpnumformat}[1]{\renewcommand{\@tocsubsectionpnumformat}{#1}}
\newcommand{\setsubsubsectionpnumformat}[1]{\renewcommand{\@tocsubsubsectionpnumformat}{#1}}
\renewcommand{\@tocpagenum}[1]{%
	\hfill {\mdseries\@sectypepnumformat #1}}

\let\oldappendix\appendix
\renewcommand{\appendix}{%
	\leavevmode\oldappendix%
	\addtocontents{toc}{%
		\protect\settowidth{\protect\@tocsectionnumwidth}{\protect\@tocsectionformat\sectionname\space}%
		\protect\addtolength{\protect\@tocsectionnumwidth}{2em}}%
}
\makeatother

\makeatletter
\settocsectionnumwidth{2em}
\settocsubsectionnumwidth{2.5em}
\settocsubsubsectionnumwidth{3em}
\settocsectionindent{1pc}%
\settocsubsectionindent{\dimexpr\@tocsectionindent+\@tocsectionnumwidth}%
\settocsubsubsectionindent{\dimexpr\@tocsubsectionindent+\@tocsubsectionnumwidth}%
\makeatother

\settocsectionvskip{0pt}
\settocsubsectionvskip{0pt}
\settocsubsubsectionvskip{0pt}

\settocsectionformat{\bfseries}
\settocsubsectionformat{\mdseries}
\settocsubsubsectionformat{\mdseries}
\setsectionpnumformat{\bfseries}
\setsubsectionpnumformat{\mdseries}
\setsubsubsectionpnumformat{\mdseries}

\let\oldtableofcontents\tableofcontents
\renewcommand{\tableofcontents}{%
	\vspace*{-\linespacing}
	\oldtableofcontents}

\allowdisplaybreaks

\usepackage[colorlinks,
linkcolor=blue,
anchorcolor=blue,
citecolor=blue
]{hyperref}

\numberwithin{equation}{section}
\newtheorem{thm}{Theorem}[section]
\newtheorem{pp}[thm]{Proposition}
\newtheorem{lm}[thm]{Lemma}
\newtheorem{df}[thm]{Definition}
\newtheorem{co}[thm]{Corollary}
\newcommand{\R}{\mathbb{R}}
\newcommand{\rd}{\,\mathrm{d}}
\newcommand{\cQ}{\mathcal{Q}}
\newcommand{\cR}{\mathcal{R}}
\newcommand{\cI}{\mathcal{I}}
\newcommand{\cJ}{\mathcal{J}}
\newcommand{\cF}{\mathcal{F}}
\newcommand{\mS}{\mathbb{S}}
\newcommand{\T}{\mathbb{T}}
\newcommand{\N}{\mathbb{N}}
\newcommand{\Z}{\mathbb{Z}}
\newcommand{\cP}{\mathcal{P}}
\newcommand{\Id}{\operatorname{Id}}
\newcommand{\Div}{\operatorname{div}}
\newcommand{\tr}{\operatorname{tr}}
\newcommand{\supp}{\operatorname{supp}}
\newcommand{\Mod}{\operatorname{mod}}
\newcommand{\tm}{\tilde{m}}
\newcommand{\tE}{\tilde{E}}
\newcommand{\tB}{\tilde{B}}
\newcommand{\tcA}{\tilde{\mathcal{A}}}

\newcommand{\ta}{\tilde{a}}
\newcommand{\cm}{\overset{\circ}{m}}
\newcommand{\cE}{\overset{\circ}{E}}
\newcommand{\cB}{\overset{\circ}{B}}
\newcommand{\ccA}{\overset{\circ}{\mathcal{A}}}
\newcommand{\mcA}{\mathcal{A}}
\newcommand{\cA}{\overset{\circ}{A}}
\newcommand{\ca}{\overset{\circ}{a}}
\newcommand{\cta}{\overset{\circ}{\ta}}
\newcommand{\cb}{\overset{\circ}{b}}
\newcommand{\ce}{\overset{\circ}{e}}
\newcommand{\cg}{\overset{\circ}{g}}
\newcommand{\cte}{\overset{\circ}{\tilde{e}}}
\newcommand{\ctg}{\overset{\circ}{\tilde{g}}}
\newcommand{\tlm}{\tilde{\lambda}}
\newcommand{\tp}{\prime\prime}
\newcommand{\CIO}{{C^0(\cI_u;C^0(\R^3))}}
\newcommand{\CIN}{C^0(\mathcal{I}_u;C^N(\R^3))}
\newcommand{\sumu}{\sum_{u\in \Z}}
\newcommand{\sumkO}{\sum_{k\in \Z^3\setminus\left\lbrace 0\right\rbrace}}

\newcommand{\PL}{P_{\leqslant\ell^{-1}}}
\newcommand{\PG}{P_{>\ell^{-1}}}
\newcommand{\UL}{U_{\leqslant\ell^{-1}}}
\newcommand{\ULO}{U_{\leqslant\ell_1^{-1}}}
\newcommand{\PLT}{P_{\leqslant\ell_2^{-1}}}
\newcommand{\UG}{U_{>\ell^{-1}}}
\newcommand{\DTL}{D_{t,\ell}}
\newcommand{\tri}{\triangle}
\newcommand{\LL}{\lesssim \lambda_{q+1}}
\newcommand{\GL}{\gtrsim \lambda_{q+1}}


\begin{document}
	\title[] {Non-uniqueness for the compressible Euler-Maxwell equations}
	
	\author{Shunkai Mao}
	\address{School of Mathematical Sciences, Fudan University, China.}
	\email[Shunkai Mao]{21110180056@m.fudan.edu.cn}
	
	\author{Peng Qu}
	\address{School of Mathematical Sciences $\&$ Shanghai Key Laboratory for Contemporary Applied Mathematics, Fudan University, China.}
	\email[Peng Qu]{pqu@fudan.edu.cn}
	\thanks{}

\begin{abstract}
We consider the Cauchy problem for the isentropic compressible Euler-Maxwell equations under general pressure laws in a three-dimensional periodic domain. For any smooth initial electron density away from the vacuum and smooth equilibrium-charged ion density, we could construct infinitely many $\alpha$--H{\"older} continuous entropy solutions emanating from the same initial data for $\alpha<\frac{1}{7}$. Especially, the electromagnetic field belongs to the H{\"o}lder class $C^{1,\alpha}$. Furthermore, we provide a continuous entropy solution satisfying the entropy inequality strictly. The proof relies on the convex integration scheme. Due to the constrain of the Maxwell equations, we propose a method of Mikado potential and construct new building blocks. 
\end{abstract}

\maketitle

{
\tableofcontents
}

\section{Introduction}\label{Introduction}
In this paper, we consider the isentropic compressible Euler-Maxwell system on the periodic domain $[0,T]\times\T^3$ with $\T^3=[-\pi,\pi]^3$ and $T\in(0,\infty)$.
The Euler-Maxwell system (see \cite{BDDCGT04,Chen84,RG69}) is a hydrodynamic model used in plasma physics to describe the motion of electrons under the influence of the corresponding electromagnetic fields. The Cauchy problem with initial condition can be expressed as follows:
\begin{equation}
\left\{
\begin{aligned}
    &\partial_tn+\Div(nu)=0,\\
	&\partial_t(nu)+\Div(nu\otimes u)+\nabla p(n)=-n(E+u\times B),    \\
	&\partial_tE-\nabla\times B=nu,\quad\Div\ E=h(x)-n,\\
	&\partial_tB+\nabla\times E=0,\ \ \quad\Div\ B=0,\\
	&(n,u,E,B)|_{t=0}=(n_0,u_0,E_0,B_0),
\end{aligned}
\right.\label{Euler-Maxwell System}
\end{equation} 
where $n=n(t,x)$ represents the density of the electrons, $u=u(t,x)$ is a vector field representing the macroscopic velocity of the electrons, and $p=p(n)$ is the pressure which is a function of the density $n$. We denote the electric and magnetic fields of the plasma as $E=(E_1,E_2,E_3)^{\top}$ and $B=(B_1,B_2,B_3)^{\top}$. $n(E+u\times B)$ represents the Coulomb force and the Lorentz force. The equilibrium-charged density of ions $h$ is a  stationary and positively smooth function,  $h=h(x)>0$.\par
In this paper, we consider weak solutions $(n,u,E,B) $ which are  H\"{o}lder continuous in space, for instance,
\begin{equation}
|u(x,t)-u(y,t)|\leqslant C|x-y|^\alpha,\qquad \forall x,y\in\T^3,\forall t\in[0,T],
\end{equation}
for some constant $C$ which is independent of $t$. Here $\alpha\in(0,1)$ is the H\"{o}lder index. Moreover, we consider the entropy inequality as
\begin{equation}\label{the entropy inequality}
	\partial_t\left(\frac{n|u|^2}{2}+\frac{|E|^2+|B|^2}{2}+ne(n)\right)+\Div\left(\left(\frac{n|u|^2}{2}+ne(n)+p(n)\right)u+E\times B\right)\leqslant0, 
\end{equation}
which describes the behavior of the combination of the kinetic energy density, the internal energy density, and the electromagnetic energy density. $e:\R\rightarrow\R$ denotes the specific internal energy that is related to the pressure $p(n)$ through $n^2e^\prime(n)=p(n)$. We would also introduce the pressure potential, $P(n):=ne(n)$, which satisfies $nP^\prime(n)=P(n)+p(n)$. We call weak solutions $(n,u,E,B)$ that solve \eqref{Euler-Maxwell System} and satisfy \eqref{the entropy inequality} in the sense of distribution as entropy solutions.\par
The study of entropy solutions to the compressible Euler equations is of great interest in the field of mathematical physics. Compressible Euler equations with $p^{\prime}(n)>0$ possess a weak-strong uniqueness principle, which asserts that if a classical solution to the Euler equations exists over a short time period, any entropy solution sharing the same initial conditions must coincide with it \cite{Daf05,GJK21,Wie18}. In a way, the entropy inequality gives us a way to select the desired solution, but it doesn't always work. In recent years, researchers have made significant progresses in understanding the non-uniqueness of weak solutions to the Euler equations by using the convex integration method.\par 
For the incompressible Euler equations, based on the commutator estimate, Constantin, E, and Titi gave proof for the Onsager conjecture regarding energy conservation for weak solutions of the 3D incompressible Euler equations in \cite{CET94}. Later, researchers focused on  the non-uniqueness of weak solutions and on finding weak solutions that do not conserve energy. De Lellis and Sz{\'e}kelyhidi provided clear proofs for the non-uniqueness of bounded weak solutions in $L^\infty(\R_x^n\times\R_t;\R^n)$ in \cite{DS09}, see also earlier results by Scheffer \cite{Sch93} and Shnirelman \cite{Shn97,Shn00}. In their subsequent work \cite{DS13}, they showed the existence of continuous periodic weak solutions to the 3D incompressible Euler equations that dissipate the total kinetic energy.  After that, they improved this result to the H{\"o}lder class $C^{\frac{1}{10}-}$ in \cite{DS14} and   further to the H{\"o}lder class $C^{\frac{1}{5}-}$ with Buckmaster and Isett as shown in \cite{BDIS15}. Isett \cite{Isett18} reached the Onsager exponent of $\frac{1}{3}-$. And then, Buckmaster, De Lellis, Sz{\'e}kelyhidi, and Vicol  provided a proof for dissipative case in \cite{BDSV19}. In addition, there were also many other important works, as seen in \cite{Buc15,BDS16,DS17,IO16}.\par
Meanwhile, many scholars have also made contributions to the research on the non-uniqueness of entropy solutions to compressible and incompressible Euler equations. The uniqueness, continuous dependence, and global stability of weak, entropy-admissible solutions to the Cauchy problem of the compressible Euler equations have been established for 1D cases with small BV initial data and mild assumptions, see for instance \cite{BCP00,LY99}. However, De Lellis-Sz{\'e}kelyhidi \cite{DS10} demonstrated that there exists bounded and compactly supported initial data such that neither the strong nor the weak energy inequalities can uniquely identify a weak solution to the incompressible or compressible Euler equations in $L^\infty(\R_x^n\times\R_t;\R^n)$ for $n\geqslant2$. Later, Chiodaroli-De Lellis-Kreml \cite{CDK15} showed the non-uniqueness of bounded admissible solutions to the 2D isentropic compressible Euler equations for the corresponding Lipschitz data that can form the shock. In subsequent work \cite{CVY21,CF22,CKMS21,KKMM20,MK18}, further discussion and supplementation have been made on the types of initial data which can generate infinitely many solutions. Especially, in \cite{KKMM20,MK18}, it has been proved that the Riemann problem for the isentropic Euler system  with a power law pressure in multiple space dimensions is ill-posed if the one-dimensional self-similar solution contains a shock. Moreover, Chen-Vasseur-Yu \cite{CVY21} and Chiodaroli-Feireisl \cite{CF22} provided the dense initial data that will generate infinitely many solutions. In recent years, important progresses have been made for the non-uniqueness of entropy solutions in the H\"older space. Isett \cite{Isett22} proved the case of the incompressible  Euler equations and provided examples that strictly dissipate kinetic energy. In \cite{DK22}, De Lellis-Kwon
found continuous entropy solutions belonging to the H{\"o}lder class $C^{\frac{1}{7}-}$, that satisfy the entropy inequality and strictly dissipate the total kinetic energy. The case of the compressible Euler equations was approached by Giri-Kwon in \cite{GK22}. Although some efforts have been made to find conditions that would make the solution unique, Luo-Xie-Xin \cite{LXX16} demonstrated that the non-uniqueness persists even in the presence of damping or rotation. \par
At the same time, there were many important progresses for the incompressible Navier-Stokes equations. Buckmaster and Vicol proved that weak solutions of the 3D incompressible Navier-Stokes equations are not unique within the class of
weak solutions with bounded kinetic energy by using the intermittent convex integration schemes in \cite{BV19}.  Cheskidov-Luo \cite{CL22} showed the non-uniqueness of weak solutions in the class $L^p_tL^\infty$ to the incompressible Navier-Stokes equations across any dimension $d\geqslant 2$ and for any given $p<2$. The utilization of intermittent convex integration schemes in investigating the non-uniqueness of the  Navier-Stokes equations with fractional viscosity includes \cite{BCV22,LQZZ22,LQ20,LT20}. Luo-Titi \cite{LT20} and Buckmaster-Colombo-Vicol \cite{BCV22} demonstrated the non-uniqueness of weak solutions with bounded kinetic energy for the 3D  hyper-viscous Navier-Stokes equations with the viscosity less than the Lions exponent $\frac{5}{4}$. Additionally, a recent study \cite{LQZZ22} revealed that the 3D hyper-viscous Navier-Stokes equations with the viscosity beyond the Lions exponent $\frac{5}{4}$ display sharp non-uniqueness at two endpoint spaces. In the study of the MHD equations \cite{BBV20,LZZ22,LZZ21,MY22}, the intermittent convex integration schemes have also had a profound impact.
 \par
In this paper, we would mainly focus on the model of the compressible Euler-Maxwell system. Peng conducted a study on the compressible Euler-Maxwell system on periodic domain with small initial data and proved that smooth solutions exist globally in time, and converge towards non-constant equilibrium states as time goes to infinity in \cite{Peng15}. In addition, the result of the non-isentropic Euler-Maxwell system without a temperature diffusion term was proved by Liu-Peng in \cite{LP17}.  Considering the Euler-Maxwell two-fluid system on $\mathbb{R}^3$, Guo-Ionescu-Pausader \cite{GIP16} established the global stability of a constant neutral background, wherein smooth and localized perturbations of a constant background with small irrotational amplitude result in global smooth solutions.  Later in \cite{DIP17}, global stability of
a constant neutral background for the one-fluid Euler-Maxwell model was proved by Deng-Ionescu-Pausader.  Moreover, Liu-Guo-Peng \cite{LGP19} studied the global existence and stability of smooth solutions near large
steady-states for an isentropic Euler-Maxwell system in $\R^3$. There were also many other important works, such as \cite{IL18} on the long term regularity of the one-fluid Euler-Maxwell system in $\R^3$. \par
In this paper, we will construct  infinitely many solutions in H{\"o}lder class $C^{\frac{1}{7}-}$ to the compressible Euler-Maxwell system with the same initial data, which satisfy the entropy inequality \eqref{the entropy inequality}. We define $m:=nu$ to represent the electron momentum. Then, the Cauchy problem for the compressible Euler-Maxwell system can be rewritten as:
\begin{equation}
	\left\{
	\begin{aligned}
		&\partial_tn+\Div\ m=0,\\
		&\partial_tm+\Div\left(\frac{m\otimes m}{n}\right)+\nabla p(n)=-nE-m\times B,    \\
		&\partial_tE-\nabla\times B=m,\quad\Div E=h(x)-n,\\
		&\partial_tB+\nabla\times E=0,\quad\ \Div B=0,\\
		&(n, M,E,B)|_{t=0}=(n_0,m_0,E_0,B_0).
	\end{aligned}
	\right.\label{Euler-Maxwell System 1}
\end{equation} 
The corresponding entropy inequality can be written as:
\begin{align}\label{entropy inequality 1}
	\partial_t\left(\frac{|m|^2}{2n}+\frac{|E|^2+|B|^2}{2}+ne(n)\right)+\Div\left(\frac{m}{n}\left(\frac{|m|^2}{2n}+ne(n)+p(n)\right)+E\times B\right)\leqslant0.
\end{align}
\subsection{Main results}
In this paper, we present two main theorems that imply the non-uniqueness of entropy solutions in the H{\"o}lder class $C^{\frac{1}{7}-}$ to the compressible  Euler-Maxwell equations.
\begin{thm}\label{thm 1}
	For any $0\leqslant \beta< 1/7$, initial density $n_0=n_0(x)\in C^\infty(\T^3)$, $h=h(x)\in C^\infty(\T^3)$, and pressure $p=p(n)\in C^\infty([\varepsilon_0,\infty))$, where $\int_{\mathbb{T}^3} n_0(x) \rd x=\int_{\mathbb{T}^3} h(x) \rd x$, and $\varepsilon_0$ is a positive constant such that $n_0(x)\geqslant \varepsilon_0$, we can find
	infinitely many distinct entropy solutions, $n \in C^\infty([0,T]\times \T^3)$, $m\in C^\beta([0, T]\times \T^3)$, $E\in C^{1,\beta}([0, T]\times \T^3)$ and $B\in C^{1,\beta}([0, T]\times \T^3)$, to the isentropic compressible Euler-Maxwell equations \eqref{Euler-Maxwell System 1} emanating from the same initial data and satisfying the energy equation 
	\begin{align}\label{energy equation 2}
		\partial_t \left( \frac{|m|^2}{2 n}+ \frac{|E|^2+|B|^2}{2} +  n e( n)\right) +
		\Div\left(\frac{m}{ n} \left(
		\frac{|m|^2}{2 n} +  n e( n) + p( n)\right)+E\times B\right)= 0\, 
	\end{align}
	in the distributional sense.
\end{thm}
\begin{thm}\label{thm 2}
For any $0\leqslant \beta< 1/7$, initial density $n_0=n_0(x)\in C^\infty(\T^3)$, $h=h(x)\in C^\infty(\T^3)$, and pressure $p=p(n)\in C^\infty([\varepsilon_0,\infty))$,   where $\int_{\mathbb{T}^3} n_0(x)\rd x=\int_{\mathbb{T}^3} h(x) \rd x$, and $\varepsilon_0$ is a positive constant such that $n_0(x)\geqslant \varepsilon_0$, there is an entropy solution  $n \in C^\infty([0,T]\times \T^3)$, $m\in C^\beta([0, T]\times \T^3)$, $E\in C^{1,\beta}([0, T]\times \T^3)$ and $B\in C^{1,\beta}([0, T]\times \T^3)$, to the isentropic compressible Euler-Maxwell equations \eqref{Euler-Maxwell System 1} satisfying the entropy inequality \eqref{entropy inequality 1} strictly in the distributional sense.
\end{thm}	
The proof of our results relies on the convex integration scheme starting from De Lellis-Sz{\'e}kelyhidi. We adapt the convex integration scheme proposed by De Lellis-Kwon \cite{DK22} and Giri-Kwon \cite{GK22} to the compressible Euler-Maxwell system. In the convex integration scheme, due to the presence of a nonlinear term $\frac{(m\otimes m)}{n}$ in the Euler equations, we commonly use it to eliminate Reynolds errors.  However, since the Maxwell equations are linear, we could only try to find perturbations that strictly satisfy them in order to keep the linear Maxwell equation hold during our iteration. The first challenge we encountered when using the scheme was to solve the Maxwell equations. If we want to use Mikado flows to construct the perturbation term $\tm_q$, we need to determine the electromagnetic field caused by Mikado flows. In fact, in the case of three dimensions, the electromagnetic field satisfies the wave equations, and we can obtain the solutions to the corresponding Cauchy problems. However, the estimates for $\tE_q$ and $\tB_q$ chosen by solving wave equations, as well as for $\tm_q$, are of the same order of magnitude. $R_{q+1}$ would be so large such that the iteration cannot continue. To overcome these difficulties, we propose a new method of Mikado potential. We would use the specially chosen electromagnetic potentials to construct new building blocks $(\cm_k,\cE_k,\cB_k)$ (see Lemma \ref{New building blocks}) which satisfy the Maxwell equations and can be used to construct the perturbation. In this way, we can not only express the solutions of the Maxwell equations explicitly, but also show that we could use a special linear combination of the main part of $\cm_k$, denoted by $\cm_{p,k}$, to construct the Mikado flows as defined in \eqref{def of psi^*}. Moreover, the estimation on $E_q$ and $B_q$ is much smaller than the one on $m_q$. However, due to the constrain of the Maxwell equations and strong resonance between the electromagnetic fields may occur, we will find that for some directions, a strong electromagnetic field can only lead to a weak fluid flow. Then, the special type of Mikado flows will lose certain frequencies, that is, the terms corresponding to certain frequency will be close to zero. If we use the special Mikado potentials to construct $\tm_p$, the low-frequency components of $\frac{\tm_p\otimes\tm_p}{n}$ and $\frac{|\tm_p|\tm_p}{2n^2}$ may vanish, which will make it difficult for us to choose weights function. To solve this, we would specially choose the strength function $\psi^*$ as defined in Proposition \ref{prop of Mikado flow scalar function}, which allows us to use the low-frequency components of $\frac{\tm_p\otimes\tm_p}{n}$ and $\frac{|\tm_p|\tm_p}{2n^2}$ to reduce the Reynolds error $R_q$ and current $\varphi_q$ separately.\par 
In the main part of the article, we would first introduce the dissipative Euler-Maxwell-Reynolds flow (see Definition \ref{Euler-Maxwell-Reynolds System}) which is an approximate solution to the Euler-Maxwell system with an entropy inequality. The iterative scheme proposed in this paper aims to construct a series of dissipative Euler-Maxwell-Reynolds flows that converge to a solution of the Euler-Maxwell system. We need to construct a starting tuple $(m_0,E_0,B_0,c_0,R_0,\varphi_{0})$ and a series of perturbations $(\tm_q,\tE_q,\tB_q)$, which can be found in Section \ref{Construction of a Starting Tuple} and Section \ref{Construction of the perturbation} respectively, such that tuples $(m_q, E_q, B_q, c_q, R_q, \varphi_q)$ satisfy Proposition \ref{Inductive proposition} and Proposition \ref{Bifurcating inductive proposition}. Then, we could give proofs for our main results.\par 
Before we construct the perturbation, we would first present a specially chosen mollification process for tuples $(m_q, E_q, B_q, c_q, R_q, \varphi_q)$. $m_q$, $E_q$ and $B_q$ would be mollified with respect to time and space, and $R_q$ and $\varphi_q$ would be mollified along the flow trajectory.  Next, we will introduce the Mikado flow, the cutoff functions and the backward flow map. Subsequently, we will introduce the new building blocks $(\cm_k,\cE_k,\cB_k)$ and strength function $\psi^*$ and  use them to construct the perturbation. \par
To construct the starting tuples, we may also utilize the Mikado potential method. First, we would create the starting tuple with a stationary density. Our main idea is to use the low-frequency components of $\frac{m_0\otimes m_0}{n}$ to eliminate the low frequency components of $\nabla p+nE_0+m_0\times B_0$.   Next, by perturbing the density slightly over time, we can provide a starting tuple for the case with time-dependent density. Finally, we present an overview of the paper's structure.
\subsection{Organization of the paper}\label{Outline of the proof}
	Section \ref{Induction scheme} contains the induction scheme for constructing Euler-Maxwell-Reynolds flows and two main propositions, Proposition \ref{Inductive proposition} and Proposition \ref{Bifurcating inductive proposition}, which will be utilized in the proof of Theorems \ref{thm 1} and \ref{thm 2}. We introduce the mollification process for $(m_q, E_q, B_q, c_q, R_q, \varphi_q)$ in Section \ref{Mollification}. The construction of the perturbation is described in Section \ref{Construction of the perturbation}. The new error $R_{q+1}$ and the updated current $\varphi_{q+1}$ would be given in Section \ref{Definition of the new errors}, and the corresponding estimates are carried out over Section \ref{Estimates on the Reynolds stress} and \ref{Estimates on the new current error}. The proofs of Proposition \ref{Inductive proposition} and \ref{Bifurcating inductive proposition} occupy Section \ref{Proof of the Inductive Proposition}. In Section \ref{Construction of a Starting Tuple}, we construct starting tuples with both cases of stationary and time-dependent density. The proofs for Theorems \ref{thm 1} and \ref{thm 2} are given in Section \ref{Proof of the Theorems}. The appendix provides proofs or statements of analytical facts that were used in the proofs of the propositions in the paper.
\section{Induction scheme}\label{Induction scheme}
\begin{df}
	For a given $n=n(t,x)\in C^\infty([T_1,T_2]\times \T^3)$, $h=h(x)\in C^\infty(\T^3)$ with $n(t,x)\geqslant \varepsilon_0$ for some positive constant $\varepsilon_0$, and $\int_{\mathbb{T}^3} n(t,x) \rd x=\int_{\mathbb{T}^3} h(x) \rd x$ for all $t$, a tuple of smooth tensors $(m,E,B,c,R,\varphi)$ is a dissipative Euler-Maxwell-Reynolds flow as long as it solves the following system
	\begin{equation}
		\left\{
		\begin{aligned}
			&\partial_tn+\Div m=0,\\
			&\partial_tm+\Div\left(\frac{m\otimes m}{n}\right)+\nabla p(n)+nE+m\times B=\Div(n(R-c\Id)),    \\		    
			&\partial_tE-\nabla\times B=m,\quad \Div E=h(x)-n,\\
			&\partial_tB+\nabla\times E=0,\quad\ \Div B=0,\\
			&\partial_t\left(\frac{|m|^2}{2n}+\frac{|E|^2+|B|^2}{2}+ne(n)\right)+\Div\left(\frac{m}{n}\left(\frac{|m|^2}{2n}+ne(n)+p(n)\right)+E\times B\right)\\
			&=n\left(\partial_t+\frac{m}{n}\cdot\nabla\right)\frac{1}{2}\tr(R)+\Div((R-c\Id)m)+\Div(n\varphi)+\partial_tH,\\
		\end{aligned}
		\right.\label{Euler-Maxwell-Reynolds System}
	\end{equation} 
in the sense of distribution. Here $H=H(t,x)$ is the global energy loss, and we assume that $H(0)=0,\partial_tH\leqslant 0$	for the term dissipative.
\end{df}
Assuming that we have constructed a dissipative Euler-Maxwell-Reynolds flow $(m_q,E_q,B_q,c_q,R_q,\varphi_q)$ at step $q\in\N$, which is a distributional sense approximate solution to the Euler-Maxwell system, we aim to build a corrected Euler-Maxwell-Reynolds flow $(m_{q+1},E_{q+1},B_{q+1},c_{q+1},R_{q+1},\varphi_{q+1})$ at the $(q+1)_{th}$ step. This will allow us to obtain a sequence of approximate solutions that can converge to a solution of the Euler-Maxwell system. To achieve this, we introduce some parameters to measure the size of our approximate solutions,
\begin{equation}
    \begin{aligned}
    \lambda_q=\lceil\lambda_0^{b^q}\rceil,
    \quad\delta_q=\lambda_q^{-2\alpha},\quad c_q=\sum_{j=q+1}^{\infty}\delta_j,
    \end{aligned}\label{def of parameter}
\end{equation}
where $\lceil x\rceil$ denotes the smallest integer $\tilde{n}\geqslant x,\lambda_0>1$ is a large parameter, and $b$ will be chosen close to 1.\par
At each step, we give a correction $(\tm_{q},\tE_{q},\tB_{q})=(m_{q+1}-m_q,E_{q+1}-E_q,B_{q+1}-B_q)$ to make the error $(R_q,\varphi_q)$ get smaller which would converge to zero (in\ H\"{o}lder space) as $q$ goes to infinity. Due to truncation and smoothing, the domains of definition for the choice of the approximate solution change at each step, here we choose it at step  $q$ as $\cI^{q-1}\times\T^3=[-\tau_{q-1},T+\tau_{q-1}]\times \T^3$, where $\tau_{-1}=(\lambda_0\delta_{0}^{\frac{1}{2}})^{-1}$ and
$$
\tau_{q}=\left(C_{n}M\lambda_{q}^{\frac{1}{2}}\lambda_{q+1}^{\frac{1}{2}}\delta_{q}^{\frac{1}{4}}\delta_{q+1}^{\frac{1}{4}}\right)^{-1},\quad q\geqslant0,
$$
for some constants $C_{n}$ depending on $n$ and $M = M(n, p, h)$ depending
on $n, p, h$, see \eqref{def of mu tau} and Section \ref{Proof of the Theorems} for the detailed choice of $C_{n}$ and $M$. For convenience, we introduce the following notation in \cite{GK22}:
\begin{itemize}
	\item $\cI + \sigma$ is the concentric enlarged interval $(a-\sigma, b+\sigma)$ when $\cI = [a,b]$.
	\item Throughout the rest of the paper, we will use $\nabla F$ to denote the Jacobian matrix of the partial derivatives of the components of the vector map $F$.
	\item Furthermore, for the sake of convenience, in what follows, when we use the notation $A\lesssim_{n, p, h, N} B$ without pointing out the dependence of the implicit constant $C$, we mean $A \lesssim_{n, p, h, N} B$, where $n$ represents the density, $p$ is the pressure, and $h$ is the equilibrium-charged density of ions, all of which are fixed in the whole iterative process, and $N$ can be chosen to have $N\leqslant N_0$ for some constants $N_0$. Moreover, we use the notation $A \lesssim_{\kappa} B$ to mean $A \leqslant CB$, where $C > 0$ may depend on some fixed constants or functions $\kappa$.
	\item  For $N=0$, if we use the notation $\|D_{t,q}F\|_{N-1}$, we are not claiming any negative Sobolev estimate on $D_{t,q}F$: the reader should just consider the advective derivative estimate to be an empty statement when $N=0$.  
\end{itemize}
We assume the following inductive estimates on $(m_q,E_q,B_q,R_q,\varphi_q)$ satisfying \eqref{Euler-Maxwell-Reynolds System} with $H(0) = 0$ and $H^\prime \leqslant 0$.
\begin{align}
		&\left\|m_q\right\|_{0}\leqslant \underline{M}-\delta_{q}^{\frac{1}{2}},&&\left\|\partial_t^rm_q\right\|_{N}\leqslant M\lambda_{q}^{N+r}\delta_{q}^{\frac{1}{2}}, && 1\leqslant N+r\leqslant 2,\label{est on m_q}\\
		&\left\|E_q\right\|_{1},\left\|\partial_{t}E_q\right\|_{0 }\leqslant\underline{M}-\delta_{q}^{\frac{1}{2}},&&\left\|\partial_{t}^rE_q\right\|_{N}\leqslant M\lambda_{q}^{N+r-1}\delta_{q}^{\frac{1}{2}},&& 2\leqslant N+r\leqslant 3,\label{est on E_q}\\
		&\left\|B_q\right\|_{1},\left\|\partial_{t}B_q\right\|_{0}\leqslant\underline{M}-\delta_{q}^{\frac{1}{2}},&&\left\|\partial_{t}^rB_q\right\|_{N}\leqslant M\lambda_{q}^{N+r-1}\delta_{q}^{\frac{1}{2}},&& 2\leqslant N+r\leqslant 3,\label{est on B_q}
\end{align}
and
\begin{align}
	&\left\|R_q\right\|_{N}\leqslant \lambda_{q}^{N-3\gamma}\delta_{q+1},	&&\left\|D_{t,q}R_q\right\|_{N-1}\leqslant \lambda_{q}^{N-3\gamma}\delta_{q}^{\frac{1}{2}}\delta_{q+1},&& N=0,1,2,\label{est on error}\\
	&\left\|\varphi_q\right\|_{N}\leqslant \lambda_{q}^{N-3\gamma}\delta_{q+1}^{\frac{3}{2}}, &&\left\|D_{t,q}\varphi_q\right\|_{N-1}\leqslant \lambda_{q}^{N-3\gamma}\delta_{q}^{\frac{1}{2}}\delta_{q+1}^{\frac{3}{2}}, &&N=0,1,2,\label{est on current}
\end{align}
where $\left\|\cdot\right\|_N=\left\|\cdot\right\|_{C^0(\cI^{q-1};C^N(\T^3))},D_{t,q}=\partial_t+\frac{m_q}{n}\cdot\nabla$, $\gamma=(b-1)^2$, and $\underline{M}(n, p, h)\geqslant 1$ which will be determined in Section \ref{Construction of a Starting Tuple}. Moreover, $M(n, p, h)$ can be chosen to satisfy $M(n, p, h)>\underline{M}(n, p, h)$. We keep this assumption throughout the iteration. Under this, we give the core inductive propositions as follows: 
\begin{pp}[Inductive proposition]\label{Inductive proposition}
Let  $n=n(t,x)\in C^{\infty}([-\tau_{-1}, T+\tau_{-1}] \times \mathbb{T}^3)$, $h=h(x)\in C^{\infty}(\mathbb{T}^3)$ satisfying $n(t,x) \geqslant \varepsilon_0$ for some positive constant $\varepsilon_0$, and $ \int_{\mathbb{T}^3} n(t, x) \rd x=\int_{\mathbb{T}^3} h(x) \rd x$ for all $t$. Let $p\in C^\infty([\varepsilon_0,,\infty))$ be a function of $n$. For any $\alpha \in(0,\frac{1}{7})$, there exists constants $M = M(n, p, h) > 1$, $\bar{b}_0(\alpha) > 1$, and $\Lambda_0=\Lambda_0(\alpha, b, M, n, p, h) > 0$ such that the following property holds. Let $c_q=\sum_{j=q+1}^\infty\delta_j$, for any $b \in (1, \bar{b}_0(\alpha))$ and $\lambda_0 \geqslant \Lambda_0$, assume that $\left(m_q, E_q, B_q, c_q, R_q, \varphi_q\right)$  is a dissipative Euler-Maxwell-Reynolds flow defined on the time interval $[-\tau_{q-1}, T+{\tau_{q-1}}]$ satisfying \eqref{est on m_q}--\eqref{est on current} for an energy loss $H$ satisfying $H(0)=0$ and $H^\prime \leqslant 0$. Then, we can find a corrected dissipative Euler-Maxwell-Reynolds flow $\left(m_{q+1}, E_{q+1}, B_{q+1}, c_{q+1}, R_{q+1}, \varphi_{q+1}\right)$ which is defined on the time interval $[-\tau_q, T+{\tau_q}]$ for the same energy loss $H$, satisfies \eqref{est on m_q}--\eqref{est on current} for $q+1$ and additionally
\begin{equation}\label{Proposition of induction}
\begin{aligned}
&\sum_{0\leqslant N+r\leqslant1}\lambda_{q+1}^{-N-r}\left\|\partial_{t}^r(m_{q+1}-m_q)\right\|_{C^0\left([0, T] ; C^N\left(\mathbb{T}^3\right)\right)}\leqslant M \delta_{q+1}^{\frac{1}{2}},\\
&\sum_{0\leqslant N+r\leqslant1}\lambda_{q+1}^{-N-r}\left\|\partial_{t}^r(E_{q+1}-E_q)\right\|_{C^0\left([0, T] ; C^N\left(\mathbb{T}^3\right)\right)}\leqslant M \lambda_{q+1}^{-1}\delta_{q+1}^{\frac{1}{2}},\\
&\sum_{0\leqslant N+r\leqslant1}\lambda_{q+1}^{-N-r}\left\|\partial_{t}^r(B_{q+1}-B_q)\right\|_{C^0\left([0, T] ; C^N\left(\mathbb{T}^3\right)\right)}\leqslant M \lambda_{q+1}^{-1}\delta_{q+1}^{\frac{1}{2}}.
\end{aligned}
\end{equation}
\end{pp}
\begin{pp}[Bifurcating inductive proposition]\label{Bifurcating inductive proposition}
Let  $n=n(t,x)\in C^{\infty}([-\tau_{-1}, T+\tau_{-1}] \times \mathbb{T}^3)$, $h=h(x)\in C^{\infty}(\mathbb{T}^3)$ satisfying $n(t,x) \geqslant \varepsilon_0$ for some positive constant $\varepsilon_0$, and $ \int_{\mathbb{T}^3} n(t, x) \rd x=\int_{\mathbb{T}^3} h(x) \rd x$ for all $t$. Let $p\in C^\infty([\varepsilon_0,,\infty))$ be a function of $n$. Let the parameters $\alpha, b, \lambda_0$, constants $M$, $\bar{b}_0$ and $\Lambda_0$,  and the tuple $\left(m_q, E_q, B_q,c_q , R_q, \varphi_q\right)$ be as in the statement of Proposition \ref{Inductive proposition}. For any time interval $\cI \subset(0, T)$ which satisfies $|\cI| \geqslant 3 \tau_q$, we can produce two different tuples $\left(m_{q+1}, E_{q+1}, B_{q+1},c_{q+1}, R_{q+1}, \varphi_{q+1}\right)$ and $\left(\overline{m}_{q+1}, \overline{E}_{q+1}, \overline{B}_{q+1}, c_{q+1}, \overline{R}_{q+1}, \overline{\varphi}_{q+1}\right)$ which share the same initial data, satisfy the same conclusions of Proposition \ref{Inductive proposition}  and additionally
\begin{equation}\label{Bifurcating proposition}
\begin{aligned}
	&\left\|m_{q+1}-\overline{m}_{q+1}\right\|_{C^0\left([0, T] ; L^2\left(\mathbb{T}^3\right)\right)} \geqslant \varepsilon_0 \delta_{q+1}^{\frac{1}{2}}, && \supp_t\left(m_{q+1}-\overline{m}_{q+1}\right) \subset \cI ,\\
	&\left\|E_{q+1}-\overline{E}_{q+1}\right\|_{C^0\left([0, T] ; L^2\left(\mathbb{T}^3\right)\right)} \geqslant \varepsilon_0 \lambda_{q+1}^{-1}\delta_{q+1}^{\frac{1}{2}}, && \supp_t\left(E_{q+1}-\overline{E}_{q+1}\right) \subset \cI ,\\
	&\left\|B_{q+1}-\overline{B}_{q+1}\right\|_{C^0\left([0, T] ; L^2\left(\mathbb{T}^3\right)\right)} \geqslant \varepsilon_0 \lambda_{q+1}^{-1}\delta_{q+1}^{\frac{1}{2}}, && \supp_t\left(B_{q+1}-\overline{B}_{q+1}\right) \subset \cI .
\end{aligned}
\end{equation}
Furthermore, if we are given two tuples $\left(m_q, E_q, B_q, c_q, R_q, \varphi_q\right)$ and $\left(\overline{m}_q, \overline{E}_q,\overline{B}_q, c_q, \overline{R}_q, \overline{\varphi}_q\right)$ satisfying \eqref{est on m_q}--\eqref{est on current}, there exists some interval $\cJ \subset(0, T)$ satisfies
\begin{align}
\supp_t\left(m_q-\overline{m}_q, E_q-\overline{E}_q, B_q-\overline{B}_q, R_q-\overline{R}_q, \varphi_q-\overline{\varphi}_q\right) \subset \mathcal{J},\label{supp of diff at q step}
\end{align}
and we can exhibit two different tuples $\left(m_{q+1}, E_{q+1},B_{q+1}, c_{q+1}, R_{q+1}, \varphi_{q+1}\right)$ and $(\overline{m}_{q+1}, \overline{E}_{q+1}, \overline{B}_{q+1}$, $ c_{q+1},  \overline{R}_{q+1}, \overline{\varphi}_{q+1})$ satisfying the same conclusions of Proposition \ref{Inductive proposition}. Moreover, the support of their difference satisfies
\begin{equation}
\supp_t\left(m_{q+1}-\overline{m}_{q+1}, E_{q+1}-\overline{E}_{q+1}, B_{q+1}-\overline{B}_{q+1}, R_{q+1}-\overline{R}_{q+1}, \varphi_{q+1}-\overline{\varphi}_{q+1}\right) \subset \mathcal{J}+\left(\lambda_q \delta_q^{\frac{1}{2}}\right)^{-1}.\label{supp of diff at q+1 step}
\end{equation}
\end{pp}
The proofs for Theorem \ref{thm 1} and \ref{thm 2} rely on the above two key propositions. Proposition \ref{Inductive proposition} presents an iterative hypothesis that allows us to construct a sequence of approximate solutions to the Euler-Maxwell system. It states that given an initial condition $(m_0,E_0,B_0,c_0,R_0,\varphi_{0})$, there exists a sequence of tuples $(m_q,E_q,B_q,c_q,R_q,\varphi_q)$ for $q \geqslant 0$ satisfying \eqref{est on m_q}--\eqref{est on current} and \eqref{Proposition of induction}. We can thus prove that $(m_q,E_q,B_q)$ converges to $(m,E,B)$ in $C^{\beta}([0,T]\times\T^3)$, where $\beta<\frac{1}{7}$, and $(n, M,E,B)$ is an entropy solution to the isentropic compressible Euler-Maxwell equations. Moreover, Proposition \ref{Bifurcating inductive proposition} is used to construct two different sequence of tuples that converge to two different entropy solutions. This gives us a method to construct infinitely many entropy solutions.\par
In the proofs of Propositions \ref{Inductive proposition} and \ref{Bifurcating inductive proposition}, the key is to construct the perturbation term $(\tm, \tE, \tB)$ from the tuple $(m_q, E_q, B_q, R_q, \varphi_q)$. We will construct general tuples $(\cm_k, \cE_k, \cB_k)$ (see Lemma \ref{New building blocks}) that serve as the building block for the perturbation term. To address the issue of resonance, we provide Lemma \ref{prop of Mikado flow scalar function} which allows us to select specific strength functions $\psi^*$ and the weights. With this, we can explicitly obtain the perturbation term with desired estimates. Finally, we present the new Reynolds error, current, and associated estimates, and proceed with the proofs of Propositions \ref{Inductive proposition} and \ref{Bifurcating inductive proposition}.

\section{Mollification}\label{Mollification}
\subsection{Mollification process for the Euler-Maxwell-Reynolds flows}
In this part, we will show the mollification process for $ (m_q,E_q,B_q,c_q,R_q,\varphi_q) $ at the $q_{th}$ step to solve the loss of temporal and spatial derivatives. First, we introduce the parameters $\ell$ and $\ell_t$, defined by
\begin{equation}\label{def of l and l_t}
	\ell=\frac{1}{\lambda_q^{\frac{3}{4}} \lambda_{q+1}^{\frac{1}{4}}}\left(\frac{\delta_{q+1}}{\delta_q}\right)^{\frac{3}{8}}, \quad \ell_t=\frac{1}{\lambda_q^{\frac{1}{2}-3 \gamma} \lambda_{q+1}^{\frac{1}{2}} \delta_q^{\frac{1}{4}} \delta_{q+1}^{\frac{1}{4}}}.
\end{equation} 
Then, we introduce some notation in Fourier analysis described in \cite{GK22} and \cite{Gra08}. We can define the Fourier transform and its inverse of a function $f$ in Schwartz space $\mathcal{S}(\R^3)$ as
$$
\hat{f}(\xi)=\frac{1}{(2\pi)^3}\int_{\R^3}f(x)e^{-ix\cdot\xi}\rd x,\qquad\check{f}(x)=\int_{\R^3}f(\xi)e^{ix\cdot\xi}\rd\xi.
$$
Moreover, the Fourier transform can be extended to linear functionals in $\mathcal{S}^\prime(\R^3)$ which is the dual space of $\mathcal{S}(\R^3)$. Following \cite{DK22,GK22}, we can multiply the Fourier transform of $f$ by a smooth cut-off function, apply the inverse Fourier transform, and obtain a smooth function which is the standard convention for Littlewood-Paley operators. Let $\phi(\xi)$ be a radial smooth function such that $\text{supp}\phi(\xi)\subset B(0,2)$ and $\phi\equiv1$ on $\overline{B(0,1)}$. Then, for any $j\in\Z$ and distribution $f$ on $\R^3$, we can define
$$
\widehat{P_{\leqslant2^j}f}(\xi):=\phi\left(\frac{\xi}{2^j}\right)\hat{f}(\xi),\qquad\widehat{P_{>2^j}f}(\xi):=\left(1-\phi\left(\frac{\xi}{2^j}\right)\right)\hat{f}(\xi),
$$
and
$$
\widehat{P_{2^j}f}(\xi):=\left(\phi\left(\frac{\xi}{2^j}\right)-\phi\left(\frac{\xi}{2^{j-1}}\right)\right)\hat{f}(\xi).
$$
For a given number $a$, we define $P_{\leqslant a}=P_{\leqslant 2^J}$ and $P_{> a}f=f-P_{\leqslant a}f,$ where $J=\lfloor\log_2a\rfloor$ is the largest integer which satisfies $2^J\leqslant a$. If $f$ is a spatially periodic function on $[c,d]\times\T^3$, $\PL f$ can be written as the space convolution of $f$ with kernel $\check{\phi}_\ell(\cdot):=2^{3J}\check{\phi}(2^J\cdot)$, and it is also a spatially periodic function on $[c,d]\times\T^3$. More details can be found in \cite{DK22,GK22,Gra08}. \par 
Similarly, we can define the mollification in time. With a little abuse of notations, we will use the same notations as before. We define the Fourier transform and its inverse of $f$ in the Schwartz space $\mathcal{S}(\R)$ as
$$
\hat{f}(\zeta)=\frac{1}{2\pi}\int_{\R}f(t)e^{-it\cdot\zeta}\rd t,\qquad\check{f}(t)=\int_{\R}f(\zeta)e^{it\cdot\zeta}\rd\zeta.
$$
Let $\phi^t(\zeta)$ be an even smooth function such that $\text{supp}\phi^t(\zeta)\subset (-2,2),\ \phi^t\equiv1\ {\mathrm on}\ [-1,1]$, and we can define for any $j\in\Z$ and distribution $f$ in $\R$,
$$
\widehat{U_{\leqslant2^j}f}(\zeta):=\phi^t\left(\frac{\zeta}{2^j}\right)\hat{f}(\zeta),\qquad\widehat{U_{>2^j}f}(\zeta):=\left(1-\phi^t\left(\frac{\zeta}{2^j}\right)\right)\hat{f}(\zeta),
$$
and
$$
\widehat{U_{2^j}f}(\zeta):=\left(\phi^t\left(\frac{\zeta}{2^j}\right)-\phi^t\left(\frac{\zeta}{2^{j-1}}\right)\right)\hat{f}(\zeta).
$$
For a given number $a$, we define $U_{\leqslant a}=U_{\leqslant 2^J},U_{> a}f=f-U_{\leqslant a}f$ where $J=\lfloor\log_2a\rfloor$ is the largest integer which satisfies $2^J\leqslant a$. Similarly, $\UL f$ can be written as the space convolution of $f$ with kernel $\check{\phi}^t_\ell(\cdot):=2^{J}\check{\phi}^t(2^J\cdot)$. Finally, we present two inequalities that will be utilized repeatedly,
\begin{align}
	\int_{\R^3}|y|^k|\check{\phi}_\ell(y)| \rd  y=\int_{\R^3}|y|^k2^{3J}|\check{\phi}(2^Jy)| \rd  y=2^{-kJ}\int_{\R^3}|x|^k|\check{\phi}(x)| \rd  x&\lesssim_k\ell^k,\label{est on space mollification function}\\
	\int_{\R}|\tau|^k|\check{\phi}^t_\ell(\tau)| \rd  \tau=\int_{\R}|\tau|^k2^{J}|\check{\phi}^t(2^J\tau)| \rd  \tau=2^{-kJ}\int_{\R}|\tau|^k|\check{\phi}^t(\tau)| \rd  \tau&\lesssim_k\ell^k,\label{est on time mollification function}
\end{align}
where $J=\lfloor-\log_2\ell\rfloor$ is the largest integer which satisfies $2^J\leqslant \ell^{-1}$.\par
For errors $R_q$ and $\varphi_q$, we need another process of mollification similar as the ones given in \cite{DK22} and \cite{GK22}. For a function $F$, we sometimes mollify it  along the flow trajectory to get good estimates on their advective derivatives along $m_\ell/n$.
We introduce the forward flow map $\Phi(\tau,x;t)$ with drift velocity $m_\ell/n$. The flow is defined on some interval $[a,b]$ starting at the initial time $t\in[
a,b)$ and satisfies
\begin{equation}\label{eq of forward flow}
	\left\lbrace 
	\begin{aligned}
		&\partial_{\tau}\Phi(\tau,x;t)=\frac{m_\ell}{n}(\tau,\Phi(\tau,x;t)),\\
		&\Phi(t,x;t)=x.
	\end{aligned}
	\right.
\end{equation}
Next, we introduce the mollification along the trajectory as
$$(\rho_\delta\ast_\Phi F)(t,x)=\int_{\R}F(t+s,\Phi(t+s,x;t))\rho_\delta(s) \rd  s,$$
where $\rho$ is a conventional mollifier on $\R$ satisfying $\left\|\rho\right\|_{L^1(\R)}=1$, $\text{supp}\rho\subset(-1,1)$, and $\rho_\delta(s)=\delta^{-1}\rho(\delta^{-1}s) $ for any $\delta>0$. Moreover, we have 
\begin{equation}\label{Dtl molification}
	\begin{aligned}		
		\DTL (\rho_\delta\ast_\Phi F)(t,x)&=\int_{\R}(\DTL F)(t+s,\Phi(t+s,x;t))\rho_\delta(s) \rd  s\\
		&=-\int_{\R} F(t+s,\Phi(t+s,x;t))\rho^{\prime}_\delta(s) \rd  s.
	\end{aligned}
\end{equation}
Finally, we give the regularized terms as
\begin{equation}
	\begin{aligned}
		&m_\ell=\UL\PL m_q,&&B_\ell=\UL\PL B_q,&& E_\ell=\UL\PL E_q,\\
		&R_\ell=\rho_{\ell_t}\ast_\Phi \PL R_q,&& \varphi_\ell=\rho_{\ell_t}\ast_\Phi \PL\varphi_q,
	\end{aligned}
\end{equation}
which can be defined on $\cI^q+(2\ell+\ell_{t})\subset\cI^{q-1}$ by the selection of sufficiently large $\lambda_0$. Next, we provide some estimates on these regularized terms. We use the notation $\cI^q=[0,T]+\tau_{q}$ and $\cI_\ell^q=[0,T]+(\tau_{q}+\ell)$. In this section, we denote $\left\|\cdot\right\|_N=\left\|\cdot\right\|_{C(\cI_{\ell}^{q};C^N(\T^3))}$.
\begin{pp}\label{est on mollification 1}
	For any $0<\alpha<\frac{1}{3}$, there exists $1<\overline{b}_1(\alpha)<3$ such that for any $b\in(1,\overline{b}_1(\alpha))$, we can find $\Lambda_1=\Lambda_1(\alpha, b, M, n)>0$ satisfying that if $\lambda_0\geqslant\Lambda_1$, the following properties hold for $s=0,1,2$:
	\begin{align}
		&\left\|\partial_{t}^rm_{\ell}\right\|_{C(\cI_{2\ell+\ell_t}^{q};C^N(\T^3))} \lesssim_{N,r} M\ell^{1-N-r} \lambda_{q} \delta_{q}^{\frac{1}{2}} \lesssim_{N, r} \ell^{-N-r} \delta_{q}^{\frac{1}{2}},&&1\leqslant N+r,\label{est on m_l}\\
		&\left\|\partial_t^rE_{\ell}\right\|_{C(\cI_{2\ell+\ell_t}^{q};C^N(\T^3))} \lesssim_{N, r}M \ell^{2-N-r}\lambda_{q} \delta_{q}^{\frac{1}{2}}\lesssim_{N, r}\ell^{1-N-r} \delta_{q}^{\frac{1}{2}},&&2\leqslant N+r,\label{est on E_l}\\
		&\left\|\partial_t^rB_{\ell}\right\|_{C(\cI_{2\ell+\ell_t}^{q};C^N(\T^3))} \lesssim_{N, r} M\ell^{2-N-r}\lambda_{q} \delta_{q}^{\frac{1}{2}}\lesssim_{N, r}\ell^{1-N-r} \delta_{q}^{\frac{1}{2}},&&2\leqslant N+r,\label{est on B_l}\\
		&\ell_t^s\left\|\DTL ^sR_{\ell}\right\|_{N}+\delta_{q+1}^{-\frac{1}{2}}\ell_t^s\left\|\DTL ^s\varphi_{\ell}\right\|_{N} \lesssim_{n, s}\lambda_{q}^{N-3\gamma}\delta_{q+1},&& N\leqslant 2,\label{est on Dtl error 0}\\
		&\ell_t^s\left\|\DTL ^sR_{\ell}\right\|_{N}+\delta_{q+1}^{-\frac{1}{2}}\ell_t^s\left\|\DTL ^s\varphi_{\ell}\right\|_{N} \lesssim_{n, N, s}\ell^{2-N}\lambda_{q}^{2-3\gamma}\delta_{q+1},&& N> 2.\label{est on Dtl error N}
	\end{align}
\end{pp}
\begin{proof}
	First, we could find $1<\bar{b}_1(\alpha)<3$  such that for any $b\in(1,\bar{b}_1(\alpha))$, there exists $\Lambda_1=\Lambda_1(\alpha, b, M, n)$ with the following property: if $\lambda_0\geqslant\Lambda_1$,\begin{align}\label{pp of parameter 1}
	M\lambda_{q}\delta_{q}^{\frac{1}{2}}\leqslant\min\left\{\ell^{-1}\delta_{q}^{\frac{1}{2}},\frac{1}{10}\ell_{t}^{-1}\right\},\quad \tau_{q}+2\ell+\ell_{t}\leqslant\tau_{q-1}.
	\end{align}  
	Then, \eqref{est on m_l}--\eqref{est on B_l} are easy to obtain from the definition of $\PL$ and $\UL$. Moreover, $m_\ell/n$ satisfies
	\begin{equation}\label{est on pt m_l / n N}
	\begin{aligned}
	&\|m_\ell/n\|_{C(\cI_{2\ell+\ell_t}^{q};C^0(\T^3))}\lesssim_n \underline{M}(n,p,h)-\delta_{q}^{\frac{1}{2}}\lesssim_n M,\\
	&\|m_\ell/n\|_{C(\cI_{2\ell+\ell_t}^{q};C^1(\T^3))}+\|\partial_{t}(m_\ell/n)\|_{C(\cI_{2\ell+\ell_t}^{q};C^0(\T^3))}\lesssim_{n} M\lambda_{q}\delta_{q}^{\frac{1}{2}}\lesssim_n\ell^{-1}\delta_{q}^{\frac{1}{2}},\\
	&\|\partial_t^r(m_\ell/n)\|_{C(\cI_{2\ell+\ell_t}^{q};C^N(\T^3))}\lesssim_{n, N, r}M\ell^{2-N-r}\lambda_{q}^2\delta_{q}^{\frac{1}{2}}\lesssim_{n, N, r}\ell^{-N-r}\delta_{q}^{\frac{1}{2}},&& N+r\geqslant2.
	\end{aligned}
	\end{equation} 
	To get \eqref{est on Dtl error 0} and \eqref{est on Dtl error N}, we calculate for $\forall F\in  C^\infty(\cI^{q-1}\times\T^3),$
	\begin{align*}
		\nabla(\rho_{\ell_t}*_\Phi F)&=\int_{\R}\nabla F(t+s,\Phi(t+s,x;t))\nabla\Phi(t+s,x;t)\rho_{\ell_{t}}(s) \rd  s,\\
		\nabla^2(\rho_{\ell_t}*_\Phi F)&=\int_{\R} \partial_i F(t+s,\Phi(t+s,x;t))\nabla^2\Phi_i(t+s,x;t)\rho_{\ell_{t}}(s) \rd  s\\
		&\quad+\int_{\R}(\nabla\Phi(t+s,x;t))^{\top}\nabla^2F(t+s,\Phi(t+s,x;t))\nabla\Phi(t+s,x;t)\rho_{\ell_{t}}(s) \rd  s.
	\end{align*}
	By \eqref{pp of parameter 1}, \eqref{DPhi-Id}, and \eqref{PhiN}, we have
	for $N\geqslant1$, $t\in\cI_\ell^q$, and $\tau\in[-\ell_{t},\ell_{t}]$, $$
	\begin{aligned}
		&|\nabla\Phi(t+\tau,x;t)-\Id|\lesssim \ell_{t}\left\|\nabla(m_\ell/n)\right\|_{C(\cI_{\ell+\ell_t}^{q};C^0(\T^3))}\lesssim_n \frac{1}{10},\\
		&|\nabla^{N+1}\Phi(t+\tau,x;t)|\lesssim\ell_{t}\left\|\nabla(m_\ell/n)\right\|_{C(\cI_{\ell+\ell_t}^{q};C^N(\T^3))}\lesssim_{n, N} \ell_t\ell^{1-N}M\lambda_{q}^2\delta_{q}^{\frac{1}{2}}\lesssim_{n, N}\ell^{1-N}\lambda_{q}.
	\end{aligned}
	$$
	Replacing $F$ with $\PL R_q$ and $\PL \varphi_q$, and using
	\begin{align*}
		&\left\|\PL R_q\right\|_{C(\cI_{\ell+\ell_t}^{q};C^{N}(\T^3))}+\delta_{q+1}^{-\frac{1}{2}}\left\|\PL\varphi_q\right\|_{C(\cI_{\ell+\ell_t}^{q};C^{N}(\T^3))}\lesssim\lambda_q^{N-3\gamma}\delta_{q+1},&N\leqslant 2,\\
		&\left\|\PL R_q\right\|_{C(\cI_{\ell+\ell_t}^{q};C^{N}(\T^3))}+\delta_{q+1}^{-\frac{1}{2}}\left\|\PL\varphi_q\right\|_{C(\cI_{\ell+\ell_t}^{q};C^{N}(\T^3))}\lesssim_{N}\ell^{2-N}\lambda_q^{2-3\gamma}\delta_{q+1},&N> 2,
	\end{align*}
	we could immediately obtain \eqref{est on Dtl error 0} and \eqref{est on Dtl error N}.
\end{proof}
Next, we  give estimates on the error caused by mollification.
\begin{pp}
	 For any $0<\alpha<\frac{1}{3}$, let the parameters $\bar{b}_1(\alpha)$ and $\Lambda_1$ be as in the statement of Proposition \ref{est on mollification 1}. Then, for any $b\in(1,\overline{b}_1(\alpha))$ and $\lambda_0\geqslant\Lambda_1$, we have the following properties for $N+r\leqslant 2$ with $N,r\in\N$,
	\begin{align}
		&\left\|\partial_{t}^r(m_{q}-m_{\ell})\right\|_{C^0(\cI_{2\ell+\ell_t}^q;C^N(\T^3))}\lesssim M\ell^{2-N-r} \lambda_{q}^{2} \delta_{q}^{\frac{1}{2}}\lesssim \ell^{1-N-r} \lambda_{q} \delta_{q}^{\frac{1}{2}},\label{est on m_q-m_l}\\
		&\left\|\partial_{t}^r(E_{q}-E_{\ell})\right\|_{C^0(\cI_{2\ell+\ell_t}^q;C^N(\T^3))}\lesssim M\ell^{2-N-r} \lambda_{q}\delta_{q}^{\frac{1}{2}}\lesssim \ell^{1-N-r}\delta_{q}^{\frac{1}{2}},\label{est on E_q-E_l}\\			&\left\|\partial_{t}^r(B_{q}-B_{\ell})\right\|_{C^0(\cI_{2\ell+\ell_t}^q;C^N(\T^3))}\lesssim M\ell^{2-N-r} \lambda_{q}\delta_{q}^{\frac{1}{2}}\lesssim \ell^{1-N-r}\delta_{q}^{\frac{1}{2}},\label{est on B_q-B_l}\\
		&\left\|\DTL(m_{q}-m_{\ell})\right\|_{0}+\ell\left\|\partial_{t}\DTL(m_{q}-m_{\ell})\right\|_{0}+\ell\left\|\DTL(m_{q}-m_{\ell})\right\|_{1}\lesssim_{n, p, M} \ell (\lambda_{q}\delta_{q}^{\frac{1}{2}})^2,\label{est on Dtl m_q-m_l}\\
		&\left\|\DTL(E_{q}-E_{\ell})\right\|_{0}+\ell\left\|\partial_{t}\DTL(E_{q}-E_{\ell})\right\|_{0}+\ell\left\|\DTL(E_{q}-E_{\ell})\right\|_{1}\lesssim_{n, M} \ell \lambda_{q}\delta_{q}^{\frac{1}{2}},\label{est on Dtl E_q-E_l}\\		
		&\left\|\DTL(B_{q}-B_{\ell})\right\|_{0}+\ell\left\|\partial_{t}\DTL(B_{q}-B_{\ell})\right\|_{0}+\ell\left\|\DTL(B_{q}-B_{\ell})\right\|_{1}\lesssim_{n, M}\ell \lambda_{q}\delta_{q}^{\frac{1}{2}},\label{est on Dtl B_q-B_l}\\
		&\left\|R_{q}-R_{\ell}\right\|_{N}+\delta_{q+1}^{-\frac{1}{2}}\left\|\DTL\left(R_{q}-R_{\ell}\right)\right\|_{N-1} \lesssim_{n} \lambda_{q+1}^{N-\frac{1}{2}} \lambda_{q}^{\frac{1}{2}} \delta_{q}^{\frac{1}{4}} \delta_{q+1}^{\frac{3}{4}},\label{est on R_q-R_l} \\
		&\left\|\varphi_{q}-\varphi_{\ell}\right\|_{N}+\delta_{q+1}^{-\frac{1}{2}}\left\|\DTL\left(\varphi_{q}-\varphi_{\ell}\right)\right\|_{N-1}\lesssim_{n} \lambda_{q+1}^{N-\frac{1}{2}} \lambda_{q}^{\frac{1}{2}} \delta_{q}^{\frac{1}{4}} \delta_{q+1}^{\frac{5}{4}}.\label{est on phi_q-phi_l}
	\end{align}
\end{pp}
\begin{proof}
	We first calculate
	$$
	\begin{aligned}
		F-\UL\PL F&=F-\PL F+\PL F-\UL\PL F=\PG F+\UG\PL F,\\
		F-\UL\PL F&=F-\UL F+\UL F-\UL\PL F=\UG F+\UL\PG F,	
	\end{aligned}
	$$ 
	and use Bernstein's inequality to get
	\begin{align*}
		&\left\|F-\PL F\right\|_{C^0(\cI_{2\ell+\ell_t}^q\times\T^3)}=\left\|\PG F\right\|_{C^0(\cI_{2\ell+\ell_t}^q\times\T^3)}\lesssim \ell^j\left\|\nabla^{j}F\right\|_{C^0(\cI_{2\ell+\ell_t}^q\times\T^3)},\\ 
		&\left\|F-\UL F\right\|_{C^0(\cI_{2\ell+\ell_t}^q\times\T^3)}=\left\|\UG F\right\|_{C^0(\cI_{2\ell+\ell_t}^q\times\T^3)}\lesssim \ell^j\left\|\partial_{t}^jF\right\|_{C^0(\cI^{q-1}\times\T^3)},
	\end{align*}
	for $\forall F\in  C^j(\cI^{q-1}\times\T^3),j\in\N$.
	Then, we have for $N,r\leqslant 2$,
	\begin{equation}
		\begin{aligned}
			\left\|F-\UL\PL F\right\|_{C^0(\cI_{2\ell+\ell_t}^q;C^N(\T^3))}&\lesssim\left\|\nabla^{N}\PG F\right\|_{C^0(\cI_{2\ell+\ell_t}^q\times\T^3)}+\ell^{-N}\left\|\UG\PL F\right\|_{C^0(\cI_{2\ell+\ell_t}^q\times\T^3)}\\
			&\lesssim\ell^{2-N}\left\|\nabla^{2}F\right\|_{C^0(\cI_{2\ell+\ell_t}^q\times\T^3)}+\ell^{2-N}\left\|\partial_{t}^2F\right\|_{C^0(\cI^{q-1}\times\T^3)},\\
			\left\|\partial_{t}^r(F-\UL\PL F)\right\|_{C^0(\cI_{2\ell+\ell_t}^q;C^0(\T^3))}&\lesssim\left\|\partial_{t}^r\UG F\right\|_{C^0(\cI_{2\ell+\ell_t}^q\times\T^3)}+\ell^{-r}\left\|\UL\PG F\right\|_{C^0(\cI_{2\ell+\ell_t}^q\times\T^3)}\\
			&\lesssim\ell^{2-r}\left\|\partial_{t}^2F\right\|_{C^0(\cI^{q-1}\times\T^3)}+\ell^{2-r}\left\|\nabla^2 F\right\|_{C^0(\cI^{q-1}\times\T^3)},\\
			\left\|\partial_{t}(F-\UL\PL F)\right\|_{C^0(\cI_{2\ell+\ell_t}^q;C^1(\T^3))}
			&\lesssim\left\|\partial_{t}\nabla F\right\|_{C^0(\cI^{q-1}\times\T^3)}.
		\end{aligned}\notag
	\end{equation}
	We can replace $F$ with $m_q,E_q,B_q$ and use \eqref{est on m_q}--\eqref{est on B_q} to get \eqref{est on m_q-m_l}--\eqref{est on B_q-B_l}. As for $R_q-R_\ell$ and $\varphi_q-\varphi_\ell$, after a simple calculation, we have the following decomposition
	\begin{equation}\label{def of F-Fl}
		\begin{aligned}
			\qquad F-\rho_{\ell_{t}} *_{\Phi} \PL F&=\left(F-\PL F\right)+\left(\PL F-\rho_{\ell_{t}} *_\Phi \PL F\right).
		\end{aligned}
	\end{equation}
	Moreover, we could calculate
	$$
	\begin{aligned}
		\left(\rho_{\ell_{t}} *_\Phi F-F\right)(t, x) 
		&=\int_{\R}(F(t+s, \Phi(t+s, x ; t))-F(t, x)) \rho_{\ell_{t}}(s) \rd s =\int_{\R} \int_{0}^{s} \DTL F(t+\tau, \Phi(t+\tau, x ; t)) \rd \tau \rho_{\ell_{t}}(s) \rd  s.
	\end{aligned}
	$$
	Immediately, we could obtain 
	\begin{align}
		\left\|F-(\rho_{\ell_{t}} *_\Phi F)\right\|_{0} &\lesssim \ell_{t}\left\|\DTL F\right\|_{C^0(\cI_{\ell+\ell_t}^{q}\times\T^3)}.\label{est on F-Fl 1}
	\end{align}
	Notice that
	\begin{align*}
		\DTL \PL F=\PL D_{t, q} F+\PL \left((m_\ell-m_q)/n\cdot\nabla\right) F+\left[m_\ell/n \cdot \nabla, \PL\right] F,
	\end{align*}
	where $[A,B]=AB-BA$ denotes the commutator of two operators. By \eqref{est on pt m_l / n N} and \eqref{est on two mollification commutator 2.1} with $\ell_1=0$, we have 
	\begin{align*}
		\left\|\DTL \PL F\right\|_{C^0(\cI_{\ell+\ell_t}^{q}\times\T^3)} 
		& \lesssim\left\|D_{t, q} F\right\|_{C^0(\cI_{\ell+\ell_t}^{q}\times\T^3)}+\left\|(m_\ell-m_q)/n\right\|_{C^0(\cI_{\ell+\ell_t}^{q}\times\T^3)}\left\|\nabla F\right\|_{C^0(\cI_{\ell+\ell_t}^{q}\times\T^3)}\\
		&\quad+\ell\left\|\nabla (m_{\ell}/n)\right\|_{C^0(\cI_{\ell+\ell_t}^{q}\times\T^3)}\left\|\nabla F\right\|_{C^0(\cI_{\ell+\ell_t}^{q}\times\T^3)}\\
		&\lesssim_{n}\left\|D_{t, q} F\right\|_{C^0(\cI_{\ell+\ell_t}^{q}\times\T^3)}+\delta_{q}^{\frac{1}{2}}\left\|\nabla F\right\|_{C^0(\cI_{\ell+\ell_t}^{q}\times\T^3)}.
	\end{align*}
	Combining it with \eqref{def of F-Fl} and
	\eqref{est on F-Fl 1}, one could have
	\begin{equation}
		\begin{aligned}
			\left\|F-(\rho_{\ell_{t}} *_{\Phi} \PL F)\right\|_{0}&\lesssim\left\|\PG F\right\|_{0}+\ell_t\left\|\DTL \PL F\right\|_{0}\lesssim_{n}\ell^2\left\|F\right\|_2+\ell_{t}\left(\left\|D_{t, q} F\right\|_{C^0(\cI_{\ell+\ell_t}^{q}\times\T^3)}+\delta_{q}^{\frac{1}{2}}\left\|\nabla F\right\|_{C^0(\cI_{\ell+\ell_t}^{q}\times\T^3)}\right).
		\end{aligned}\notag
	\end{equation}
	Then, we could apply it to $R_q$ and $\varphi_{q}$ get
	\begin{equation}
		\begin{aligned}
			\left\|R_q-R_{\ell}\right\|_{0} 
			& \lesssim_n \ell^{2}\left\|R_q\right\|_{2}+\ell_{t}\left(\left\|D_{t,q} R_q\right\|_{C^0(\cI_{\ell+\ell_t}^{q}\times\T^3)}+\delta_{q}^{\frac{1}{2}}\left\|\nabla R_q\right\|_{C^0(\cI_{\ell+\ell_t}^{q}\times\T^3)}\right) \\
			&\lesssim_{n}\left(\left(\ell \lambda_{q}\right)^{2}+\ell_{t} \lambda_{q} \delta_{q}^{\frac{1}{2}}\right) \lambda_{q}^{-3 \gamma} \delta_{q+1} \lesssim_{n} \lambda_{q}^{\frac{1}{2}} \lambda_{q+1}^{-\frac{1}{2}} \delta_{q}^{\frac{1}{4}} \delta_{q+1}^{\frac{3}{4}}, \\
			\left\|\varphi_q-\varphi_{\ell}\right\|_{0}& \lesssim_n \ell^{2}\left\|\varphi_q\right\|_{2}+\ell_{t}\left(\left\|D_{t,q} \varphi_q\right\|_{C^0(\cI_{\ell+\ell_t}^{q}\times\T^3)}+\delta_{q}^{\frac{1}{2}}\left\|\nabla \varphi_q\right\|_{C^0(\cI_{\ell+\ell_t}^{q}\times\T^3)}\right) \\
			& \lesssim_{n}\left(\left(\ell \lambda_{q}\right)^{2}+\ell_{t} \lambda_{q} \delta_{q}^{\frac{1}{2}}\right) \lambda_{q}^{-3 \gamma} \delta_{q+1}^{\frac{3}{2}} \lesssim_{n} \lambda_{q}^{\frac{1}{2}} \lambda_{q+1}^{-\frac{1}{2}} \delta_{q}^{\frac{1}{4}} \delta_{q+1}^{\frac{5}{4}}.
		\end{aligned}\label{est on R_q-R_l 1}
	\end{equation}
	Moreover, we have for $N=1,2$,
	\begin{equation}
		\begin{aligned}
			&\left\|R_{q}-R_{\ell}\right\|_{N} \leqslant\left\|R_{q}\right\|_{N}+\left\|R_{\ell}\right\|_{N} \lesssim \lambda_{q}^{N} \lambda_{q}^{-3 \gamma} \delta_{q+1} \lesssim \lambda_{q+1}^{N} \lambda_{q}^{\frac{1}{2}} \lambda_{q+1}^{-\frac{1}{2}} \delta_{q}^{\frac{1}{4}} \delta_{q+1}^{\frac{3}{4}}, \\
			&\left\|\varphi_{q}-\varphi_{\ell}\right\|_{N} \leqslant\left\|\varphi_{q}\right\|_{N}+\left\|\varphi_{\ell}\right\|_{N} \lesssim \lambda_{q}^{N} \lambda_{q}^{-3 \gamma} \delta_{q+1}^{\frac{3}{2}} \lesssim \lambda_{q+1}^{N} \lambda_{q}^{\frac{1}{2}} \lambda_{q+1}^{-\frac{1}{2}} \delta_{q}^{\frac{1}{4}} \delta_{q+1}^{\frac{5}{4}}.
		\end{aligned}\label{est on R_q-R_l 2}
	\end{equation}
	Next, we calculate the advective derivatives of $F-F_\ell$ and give two kinds of decomposition as follows:
	\begin{align*}
		\DTL\left(F-F_\ell\right)&=\DTL \PG F+\DTL \UG\PL F\\
		&=\DTL \PG F+\UG\PL D_{t,q}F+\UG\PL((m_\ell-m_q)/n\cdot\nabla)F\\
		&\quad+[m_\ell/n\cdot\nabla,\PL]\UG F+\PL[m_\ell/n\cdot\nabla,\UG] F,\\
		\DTL\left(F-F_\ell\right)&=\DTL \UG F+\DTL \UL\PG F\\
		&=\DTL \UG F+\UL\PG D_{t,q}F+\UL\PG((m_\ell-m_q)/n\cdot\nabla)F\\
		&\quad+[m_\ell/n\cdot\nabla,\UL]\PG F+\UL[m_\ell/n\cdot\nabla,\PG] F.
	\end{align*}
	By using \eqref{est on two mollification commutator 2.2} with $\ell_1=0,\ell_2=\ell$ and $\ell_1=\ell,\ell_2=0$ respectively, we derive the following estimates for $N=0,1,$
	\begin{align*}
		\left\|\DTL \PG F\right\|_{N}&\leqslant\left\|\PG D_{t, q} F\right\|_{N}+\left\|\PG((m_\ell-m_q)/n \cdot \nabla) F\right\|_{N}+\left\|\left[m_\ell/n \cdot \nabla, \PG\right] F\right\|_{N} \\
		&\lesssim\left\|\PG D_{t, q} F\right\|_{N}+\left\|((m_\ell-m_q)/n \cdot \nabla) F\right\|_{N}+\ell^{1-N}\left\|\nabla F\right\|_0\left\|\nabla(m_{\ell}/n)\right\|_0 \\
		&\lesssim_{n, M}\left\|\PG D_{t, q} F\right\|_{N}+\ell^{1-N}\lambda_{q}\delta_{q}^{\frac{1}{2}}\left\|\nabla F\right\|_0,\\
		\left\|\partial_{t}\DTL \UG F\right\|_{0}&\leqslant\left\|\partial_{t}\UG D_{t, q} F\right\|_{0}+\left\|\partial_{t}\UG((m_\ell-m_q)/n \cdot \nabla) F\right\|_{0}+\left\|\partial_{t}\left[m_\ell/n \cdot \nabla, \UG\right] F\right\|_{0} \\
		&\lesssim\left\|\partial_{t}\UG D_{t, q} F\right\|_{0}+\left\|\partial_{t}(((m_\ell-m_q)/n \cdot \nabla) F)\right\|_{C^0(\cI_{2\ell}^{q}\times\T^3)}+\left\|\nabla F\right\|_{C^0(\cI_{2\ell}^{q}\times\T^3)}\left\|\partial_{t}(m_{\ell}/n)\right\|_{C^0(\cI_{2\ell}^{q}\times\T^3)} \\
		&\lesssim_{n, M}\left\|\partial_{t}\UG D_{t, q} F\right\|_{0}+\lambda_{q}\delta_{q}^{\frac{1}{2}}\left\|\nabla F\right\|_{C^0(\cI_{2\ell}^{q}\times\T^3)}.
		\end{align*}
	Similarly, we have
	\begin{equation}
	\begin{aligned}
		\left\|\partial_{t}[m_\ell/n\cdot\nabla,\UL]\PG F\right\|_{0}&\lesssim\left\|\nabla F\right\|_{C^0(\cI_{2\ell}^{q}\times\T^3)}\left\|\partial_{t}(m_\ell/n)\right\|_{C^0(\cI_{2\ell}^{q}\times\T^3)}\lesssim_{n, M}\lambda_{q}\delta_{q}^{\frac{1}{2}}\left\|\nabla F\right\|_{C^0(\cI_{2\ell}^{q}\times\T^3)},\\
		\left\|\partial_{t}\UL[m_\ell/n\cdot\nabla,\PG] F\right\|_{0}&\lesssim\left\|\nabla F\right\|_{C^0(\cI_{2\ell}^{q}\times\T^3)}\left\|\nabla(m_\ell/n)\right\|_{C^0(\cI_{2\ell}^{q}\times\T^3)}\lesssim_{n, M}\lambda_{q}\delta_{q}^{\frac{1}{2}}\left\|\nabla F\right\|_{C^0(\cI_{2\ell}^{q}\times\T^3)},\\
		\left\|[m_\ell/n\cdot\nabla,\PL]\UG F\right\|_{N}&\lesssim\ell^{1-N}\left\|\nabla\UG F\right\|_0\left\|\nabla(m_\ell/n)\right\|_0\lesssim_{n, M}\ell^{1-N}\lambda_{q}\delta_{q}^{\frac{1}{2}}\left\|\nabla F\right\|_{C^0(\cI_{2\ell}^q\times\T^3)},\\
		\left\|\PL[m_\ell/n\cdot\nabla,\UG] F\right\|_{N}&\lesssim\ell^{1-N}\left\|\nabla F\right\|_{C^0(\cI_{2\ell}^q\times\T^3)}\left\|\partial_{t}(m_\ell/n)\right\|_{C^0(\cI_{2\ell}^q\times\T^3)}\lesssim_{n, M}\ell^{1-N}\lambda_{q}\delta_{q}^{\frac{1}{2}}\left\|\nabla F\right\|_{C^0(\cI_{2\ell}^q\times\T^3)}.
	\end{aligned}\notag
	\end{equation}
	Moreover, we could calculate
	\begin{equation}
		\begin{aligned}
			\left\|\UG\PL D_{t,q} F\right\|_{N}&\lesssim\ell^{-N}\left\|\UG D_{t,q} F\right\|_{0},\\
			\left\|\partial_{t}\UL\PG D_{t,q}F\right\|_{0}&\lesssim\ell^{-1}\left\|\PG D_{t,q} F\right\|_{C^0(\cI_{2\ell}^{q}\times\T^3)},\\
			\left\|\UG\PL((m_\ell-m_q)/n\cdot\nabla) F\right\|_{N}&\lesssim_{n}\ell^{1-N}\lambda_{q}\delta_{q}^{\frac{1}{2}}\left\|\nabla F\right\|_{C^0(\cI_{2\ell}^q\times\T^3)},\\
			\left\|\partial_{t}\UL\PG((m_\ell-m_q)/n\cdot\nabla) F\right\|_{0}&\lesssim_{n}\lambda_{q}\delta_{q}^{\frac{1}{2}}\left\|\nabla F\right\|_{C^0(\cI_{2\ell}^q\times\T^3)}.
		\end{aligned}\notag
	\end{equation}
	Combining the estimates, we have
	\begin{equation}\label{est on Dtl F-F_l}
		\begin{aligned}
			\left\|\DTL\left(F-F_\ell\right)\right\|_{N}
			&\lesssim_{n, M}\left\|\PG D_{t, q} F\right\|_{N}+\ell^{-N}\left\|\UG D_{t,q} F\right\|_{0}+\ell^{1-N}\lambda_{q}\delta_{q}^{\frac{1}{2}}\left\|\nabla F\right\|_{C^0(\cI_{2\ell}^{q}\times\T^3)},\\
			\left\|\partial_{t}\DTL\left(F-F_\ell\right)\right\|_{0}
			&\lesssim_{n, M}\left\|\partial_{t}\UG D_{t, q} F\right\|_{0}+\ell^{-1}\left\|\PG D_{t,q} F\right\|_{C^0(\cI_{2\ell}^{q}\times\T^3)}+\lambda_{q}\delta_{q}^{\frac{1}{2}}\left\|\nabla F\right\|_{C^0(\cI_{2\ell}^{q}\times\T^3)}.
		\end{aligned}
	\end{equation}
	Before we apply it to $m_q$, we need to get estimates on $\left\|\PG D_{t,q}m_q\right\|_{C^0(\cI_{2\ell}^{q};C^N(\T^3))}$ and $\left\|\partial_{t}^N\UG D_{t,q} m_q\right\|_{0}$. Recall that 
	$$
	D_{t,q}m_q+\nabla p(n)+\Div(m_q/n)m_q+nE_q+m_q\times B_q=\Div(n(R_q-c_q\Id)).
	$$
	Immediately, we have for $N=0,1,$ 
	\begin{align}
		\left\|D_{t,q}m_q\right\|_{C^0(\cI^{q-1};C^{N}(\T^3))}&\leqslant\left\|p(n)\right\|_{C^0(\cI^{q-1};C^{N+1}(\T^3))}+\left\|nR_q\right\|_{C^0(\cI^{q-1};C^{N+1}(\T^3))}+\left\|n\right\|_{C^0(\cI^{q-1};C^{N+1}(\T^3))}\notag\\
		&\quad+\left\|\Div(m_q/n)m_q\right\|_{C^0(\cI^{q-1};C^N(\T^3))}+\left\|nE_q+m_q\times B_q\right\|_{C^0(\cI^{q-1};C^N(\T^3))}\notag\\
		&\leqslant\left\|p(n)\right\|_{C^0(\cI^{q-1};C^{N+1}(\T^3))}+\left\|nR_q\right\|_{C^0(\cI^{q-1};C^{N+1}(\T^3))}+\left\|(-\partial_tn/n+m_q\cdot\nabla(n^{-1}))m_q\right\|_{C^0(\cI^{q-1};C^N(\T^3))}\notag\\
		&\quad+\left\|n\right\|_{C^0(\cI^{q-1};C^{N+1}(\T^3))}+\left\|nE_q\right\|_{C^0(\cI^{q-1};C^N(\T^3))}+\left\|m_q\times B_q\right\|_{C^0(\cI^{q-1};C^N(\T^3))}\notag\\
		&\lesssim_{n, p, M} \lambda_{q}^{N+1}\delta_{q},\label{est on Dtq m_q}\\
		\left\|\partial_{t}D_{t,q}m_q\right\|_{C^0(\cI^{q-1};C^0(\T^3))}&\leqslant\left\|\partial_{t}p(n)\right\|_{C^0(\cI^{q-1};C^1(\T^3))}+\left\|\partial_{t}\Div(nR_q)\right\|_{C^0(\cI^{q-1};C^0(\T^3))}+\left\|\partial_{t}n\right\|_{C^0(\cI^{q-1};C^1(\T^3))}\notag\\
		&\quad+\left\|\partial_{t}(\Div(m_q/n)m_q)\right\|_{C^0(\cI^{q-1};C^0(\T^3))}+\left\|\partial_{t}(nE_q)\right\|_{C^0(\cI^{q-1};C^0(\T^3))}+\left\|\partial_{t}(m_q\times B_q)\right\|_{C^0(\cI^{q-1};C^0(\T^3))}\notag\\
		&\lesssim_{n, p, M}\lambda_{q}^2\delta_{q},\notag
	\end{align}
	where we have used
	\begin{align}
		\left\|\partial_{t}\Div(nR_q)\right\|_{C^0(\cI^{q-1};C^0(\T^3))}&\leqslant\left\|\Div(D_{t,q}(nR_q))\right\|_{C^0(\cI^{q-1};C^0(\T^3))}+\left\|\Div((m_q/n\cdot\nabla)(nR_q))\right\|_{C^0(\cI^{q-1};C^0(\T^3))}\nonumber\lesssim_{n, M}(\lambda_{q}\delta_{q}^{\frac{1}{2}})^2.\nonumber
	\end{align}
	Then, it follows that
	\begin{align*}
		&\left\|\PG D_{t,q}m_q\right\|_{C^0(\cI_{2\ell}^{q};C^N(\T^3))}\lesssim\ell^{1-N}\left\|D_{t,q}m_q\right\|_{C^0(\cI_{2\ell}^{q};C^1(\T^3))}\lesssim_{n, p, M} \ell^{1-N}(\lambda_{q}\delta_{q}^{\frac{1}{2}})^2,\\
		&\left\|\partial_{t}^N\UG D_{t,q}m_q\right\|_{0}\lesssim\ell^{1-N}\left\|\partial_{t}D_{t,q}m_q\right\|_{C^0(\cI_{2\ell}^{q};C^0(\T^3))}
		\lesssim_{n, p, M}\ell^{1-N}(\lambda_{q}\delta_{q}^{\frac{1}{2}})^2.
	\end{align*}
	Combining it with \eqref{est on m_q-m_l} and \eqref{est on Dtl F-F_l}, \eqref{est on Dtl m_q-m_l} follows.
	Recall that $E_q$ and $B_q$ satisfy
	$$
	\left\lbrace
	\begin{aligned}
		&\partial_{t}E_q-\nabla\times B_q=m_q,\\
		&\partial_{t}B_q+\nabla\times E_q=0.
	\end{aligned}
	\right.
	$$ 
	We could similarly obtain 
	\begin{align*}
		\left\|D_{t,q}E_q\right\|_{C^0(\cI^{q-1};C^{1}(\T^3))}&\leqslant\left\|\nabla\times B_q+m_q+(m_q/n\cdot\nabla) E_q\right\|_{C^0(\cI^{q-1};C^{1}(\T^3))}\lesssim_{n, M}\lambda_{q}\delta_{q}^{\frac{1}{2}},\\\
		\left\|D_{t,q}B_q\right\|_{C^0(\cI^{q-1};C^{1}(\T^3))}&\leqslant\left\|\nabla\times E_q+(m_q/n\cdot\nabla) B_q\right\|_{C^0(\cI^{q-1};C^{1}(\T^3))}\lesssim_{n, M}\lambda_{q}\delta_{q}^{\frac{1}{2}},\\
		\left\|\partial_{t}D_{t,q}E_q\right\|_{C^0(\cI^{q-1};C^{0}(\T^3))}&\leqslant\left\|\partial_{t}\left(\nabla\times B_q+m_q+(m_q/n\cdot\nabla) E_q\right)\right\|_{C^0(\cI^{q-1};C^{0}(\T^3))}\lesssim_{n, M}\lambda_{q}\delta_{q}^{\frac{1}{2}},\\
		\left\|\partial_{t}D_{t,q}B_q\right\|_{C^0(\cI^{q-1};C^{0}(\T^3))}&\leqslant\left\|\partial_{t}\left(\nabla\times E_q+(m_q/n\cdot\nabla) B_q\right)\right\|_{C^0(\cI^{q-1};C^{0}(\T^3))}\lesssim_{n, M}\lambda_{q}\delta_{q}^{\frac{1}{2}}, 
	\end{align*}
	and then
	\begin{align*}
		\left\|\PG D_{t,q}E_q\right\|_{C^0(\cI_{2\ell}^{q};C^{N}(\T^3))}&\lesssim\ell^{1-N}\left\| D_{t,q}E_q\right\|_{C^0(\cI_{2\ell}^{q};C^{1}(\T^3))}\lesssim_{n, M}\ell^{1-N}\lambda_{q}\delta_{q}^{\frac{1}{2}},\\
		\left\|\PG D_{t,q}B_q\right\|_{C^0(\cI_{2\ell}^{q};C^{N}(\T^3))}&\lesssim\ell^{1-N}\left\| D_{t,q}B_q\right\|_{C^0(\cI_{2\ell}^{q};C^{1}(\T^3))}\lesssim_{n, M}\ell^{1-N}\lambda_{q}\delta_{q}^{\frac{1}{2}},\\
		\left\|\partial_{t}^N\UG D_{t,q}E_q\right\|_{0}&\lesssim\ell^{1-N}\left\| \partial_{t}D_{t,q}E_q\right\|_{C^0(\cI_{2\ell}^{q};C^{0}(\T^3))}\lesssim_{n, M}\ell^{1-N}\lambda_{q}\delta_{q}^{\frac{1}{2}},\\
		\left\|\partial_{t}^N\UG D_{t,q}B_q\right\|_{0}&\lesssim\ell^{1-N}\left\| \partial_{t}D_{t,q}B_q\right\|_{C^0(\cI_{2\ell}^{q};C^{0}(\T^3))}\lesssim_{n, M}\ell^{1-N}\lambda_{q}\delta_{q}^{\frac{1}{2}},
	\end{align*} 
	for $N=0,1$. \eqref{est on Dtl E_q-E_l} and \eqref{est on Dtl B_q-B_l} follow from \eqref{est on m_q-m_l} and \eqref{est on Dtl F-F_l}. Now, let us consider the advective derivatives of $R_q-R_\ell$ and $\varphi_q-\varphi_\ell$ when $N=1,2$. We first calculate
	\begin{equation}
		\begin{aligned}
			\left\|\DTL (R_q-R_\ell)\right\|_{N-1}&\leqslant\left\|\DTL \PG R_q\right\|_{N-1}+\left\|\DTL (\PL R_q-\rho_{\ell_{t}}*_{\Phi}\PL R_q)\right\|_{N-1},\\
			\left\|\DTL (\varphi_q-\varphi_\ell)\right\|_{N-1}&\leqslant\left\|\DTL \PG\varphi_q\right\|_{N-1}+\left\|\DTL (\PL\varphi_q-\rho_{\ell_{t}}*_{\Phi}\PL\varphi_q)\right\|_{N-1}.\\
		\end{aligned}\label{est on Dtl R_q-R_l 1}
	\end{equation}
	As for the first term, we have 
	\begin{align}
		\left\|\DTL \PG R_q\right\|_{N-1}&\leqslant\left\|D_{t,q}R_q\right\|_{N-1}+\left\|((m_\ell-m_q)/n\cdot\nabla)R_q\right\|_{N-1}+\left\|[(m_\ell/n)\cdot\nabla,\PG]R_q\right\|_{N-1}\nonumber\\
		&\lesssim\left\|D_{t,q}R_q\right\|_{N-1}+\left\|((m_\ell-m_q)/n\cdot\nabla)R_q\right\|_{N-1}+\ell^{2-N}\left\|\nabla R_q\right\|_0\left\|\nabla(m_{\ell}/n)\right\|_0\nonumber\\
		&\lesssim_{n}\lambda_{q}^{N-3\gamma}\delta_{q}^{\frac{1}{2}}\delta_{q+1}+\ell^{2-N}\lambda_{q}^{2-3\gamma}\delta_{q}^{\frac{1}{2}}\delta_{q+1}+\ell^{1-N}\lambda_{q}^{1-3\gamma}\delta_{q}^{\frac{1}{2}}\delta_{q+1}\nonumber\\
		&\lesssim_{n}\lambda_{q+1}^{N-1}\lambda_{q}^{1-3\gamma}\delta_{q}^{\frac{1}{2}}\delta_{q+1},\label{est on Dtl PG R_q}\\
		\left\|\DTL \PG\varphi_q\right\|_{N-1}&\leqslant\left\|D_{t,q}\varphi_q\right\|_{N-1}+\left\|((m_\ell-m_q)/n\cdot\nabla)\varphi_q\right\|_{N-1}+\left\|[(m_\ell/n)\cdot\nabla,\PG]\varphi_q\right\|_{N-1}\nonumber\\
		&\lesssim\left\|D_{t,q}\varphi_q\right\|_{N-1}+\left\|((m_\ell-m_q)/n\cdot\nabla)\varphi_q\right\|_{N-1}+\ell^{2-N}\left\|\nabla \varphi_q\right\|_0\left\|\nabla(m_{\ell}/n)\right\|_0\nonumber\\
		&\lesssim_{n}\lambda_{q}^{N-3\gamma}\delta_{q}^{\frac{1}{2}}\delta_{q+1}^{\frac{3}{2}}+\ell^{2-N}\lambda_{q}^{2-3\gamma}\delta_{q}^{\frac{1}{2}}\delta_{q+1}^{\frac{3}{2}}+\ell^{1-N}\lambda_{q}^{1-3\gamma}\delta_{q}^{\frac{1}{2}}\delta_{q+1}^{\frac{3}{2}}\nonumber\\
		&\lesssim_{n}\lambda_{q+1}^{N-1}\lambda_{q}^{1-3\gamma}\delta_{q}^{\frac{1}{2}}\delta_{q+1}^{\frac{3}{2}}.\label{est on Dtl PG varphi_q}
	\end{align}
	Recalling that $\left\|[(m_\ell/n)\cdot\nabla,\PG]F\right\|_{N-1}=\left\|[(m_\ell/n)\cdot\nabla,\PL]F\right\|_{N-1}$, we conclude
	\begin{align*}
		&\quad\left\|\DTL (\PL F-\rho_{\ell_{t}}*_{\Phi}\PL F)\right\|_{N-1}\\
		&\leqslant2\left\|\DTL \PL F\right\|_{C^0(\cI_{\ell+\ell_t}^{q};C^{N-1}(\T^3))}\\
		&\leqslant2\left(\left\|\PL D_{t,q}F\right\|_{C^0(\cI_{\ell+\ell_t}^{q};C^{N-1}(\T^3))}+\left\|\PL((m_q-m_\ell)/n\cdot\nabla)F\right\|_{C^0(\cI_{\ell+\ell_t}^{q};C^{N-1}(\T^3))}+\left\|[m_\ell/n\cdot\nabla,\PL]F\right\|_{C^0(\cI_{\ell+\ell_t}^{q};C^{N-1}(\T^3))}\right)\\
		&\lesssim\ell^{1-N}\left(\left\|D_{t,q}F\right\|_{C^0(\cI_{\ell+\ell_t}^{q};C^{0}(\T^3))}+\left\|((m_q-m_\ell)/n\cdot\nabla)F\right\|_{C^0(\cI_{\ell+\ell_t}^{q};C^{0}(\T^3))}+\ell\left\|\nabla F\right\|_{C^0(\cI_{\ell+\ell_t}^{q};C^{0}(\T^3))}\left\|\nabla(m_{\ell}/n)\right\|_{C^0(\cI_{\ell+\ell_t}^{q};C^{0}(\T^3))}\right)\\
		&\lesssim_{n}\ell^{1-N}\left\|D_{t,q}F\right\|_{C^0(\cI_{\ell+\ell_t}^{q};C^{0}(\T^3))}+\ell^{1-N}\delta_{q}^{\frac{1}{2}}\left\|\nabla F\right\|_{C^0(\cI_{\ell+\ell_t}^{q};C^{0}(\T^3))}.
	\end{align*}
	Applying it to $R_q$ and $\varphi_q$, and combining it with \eqref{est on Dtl R_q-R_l 1}--\eqref{est on Dtl PG varphi_q}, we could obtain
	\begin{align}
		\left\|\DTL(R_q-R_\ell)\right\|_{N-1}
		&\lesssim_{n}\ell^{1-N}\left\|D_{t,q}R_q\right\|_{C^0(\cI_{\ell+\ell_t}^{q};C^{0}(\T^3))}+\ell^{1-N}\delta_{q}^{\frac{1}{2}}\left\|\nabla R_q\right\|_{C^0(\cI_{\ell+\ell_t}^{q};C^{0}(\T^3))}+\lambda_{q+1}^{N-1}\lambda_{q}^{1-3\gamma}\delta_{q}^{\frac{1}{2}}\delta_{q+1}\notag\\
		&\lesssim_{n}\ell^{1-N}\lambda_{q}^{1-3\gamma}\delta_{q}^{\frac{1}{2}}\delta_{q+1}+\lambda_{q+1}^{N-1}\lambda_{q}^{1-3\gamma}\delta_{q}^{\frac{1}{2}}\delta_{q+1}\lesssim_{n}\lambda_{q+1}^{N} \lambda_{q}^{\frac{1}{2}} \lambda_{q+1}^{-\frac{1}{2}} \delta_{q}^{\frac{1}{4}} \delta_{q+1}^{\frac{5}{4}},\label{est on Dtl R_q-R_l}\\
		\left\|\DTL(\varphi_q- \varphi_\ell)\right\|_{N-1}
		&\lesssim_{n}\ell^{1-N}\left\|D_{t,q}\varphi_q\right\|_{C^0(\cI_{\ell+\ell_t}^{q};C^{0}(\T^3))}+\ell^{1-N}\delta_{q}^{\frac{1}{2}}\left\|\nabla \varphi_q\right\|_{C^0(\cI_{\ell+\ell_t}^{q};C^{0}(\T^3))}+\lambda_{q+1}^{N-1}\lambda_{q}^{1-3\gamma}\delta_{q}^{\frac{1}{2}}\delta_{q+1}^{\frac{3}{2}}\notag\\
		&\lesssim_{n}\ell^{1-N}\lambda_{q}^{1-3\gamma}\delta_{q}^{\frac{1}{2}}\delta_{q+1}^{\frac{3}{2}}+\lambda_{q+1}^{N-1}\lambda_{q}^{1-3\gamma}\delta_{q}^{\frac{1}{2}}\delta_{q+1}^{\frac{3}{2}}\lesssim_{n}\lambda_{q+1}^{N} \lambda_{q}^{\frac{1}{2}} \lambda_{q+1}^{-\frac{1}{2}} \delta_{q}^{\frac{1}{4}} \delta_{q+1}^{\frac{7}{4}}.\label{est on Dtl varphi_q-varphi_l}
	\end{align}
	\eqref{est on R_q-R_l} and \eqref{est on phi_q-phi_l} follow from \eqref{est on R_q-R_l 1}, \eqref{est on R_q-R_l 2}, \eqref{est on Dtl R_q-R_l}, and \eqref{est on Dtl varphi_q-varphi_l}.
\end{proof}
\subsection{Quadratic commutator}
Here we provide a quadratic commutator estimate. We apply $\UL\PL$ to the momentum equation and  obtain
$$
\partial_tm_\ell+\Div\left(\frac{m_\ell\otimes m_\ell}{n}\right)+\nabla p_\ell(n)+nE_{\ell}+m_{\ell}\times B_{\ell}=\Div( \UL\PL(nR_q-nc_q\Id))+Q(m_q,m_q),
$$
where $p_\ell=\UL\PL p,$ and $Q(m_q,m_q)$ is defined as 
\begin{equation}\label{def of Qmm}
	\begin{aligned}
		Q(m_q,m_q)&:=\Div\left(\frac{m_\ell\otimes m_\ell}{n}-\UL\PL\left(\frac{m_q\otimes m_q}{n}\right)\right)+nE_{\ell}-\UL\PL(nE_{q})\\
		&\quad\ +m_{\ell}\times B_{\ell}-\UL\PL(m_{q}\times B_{q}).	
	\end{aligned}
\end{equation}
\begin{lm}
	For any $0<\alpha<\frac{1}{3}$, let the parameters $\bar{b}_1(\alpha)$ and $\Lambda_1$ be as in the statement of Proposition \ref{est on mollification 1}. Then, for any $b\in(1,\overline{b}_1(\alpha))$, $\lambda\geqslant\Lambda_1$, and integers $N,r\in \N$ with $N+r\in [0, \overline {N}]$ for some constant $\overline N\in \N$ which is independent of $q$, $Q(m_q,m_q)$ satisfies 
	\begin{align}
		&\left\|\partial_{t}^rQ(m_q,m_q)\right\|_{N}\lesssim_{n, \overline{N}, M} \ell^{1-N-r} (\lambda_q\delta_q^\frac{1}{2})^2,\label{est on Qmm}\\
		&\left\|\partial_{t}^{r}\DTL  Q(m_q,m_q)\right\|_{C^0(\cI^q;C^N(\T^3))} \lesssim_{n, p, \overline{N}, M} \ell^{-N-r} \delta_q^{\frac{1}{2}}(\lambda_q\delta_q^\frac{1}{2})^2. \label{est on Dtl Qmm}
	\end{align}
\end{lm}
\begin{proof} 
	First, we rewrite $Q(m_q,m_q)$ as 
	\begin{align*}
		Q(m_q,m_q) &= \nabla \cdot \left(\frac{m_\ell \otimes m_\ell}{n} - \frac{\UL\PL (m_q\otimes m_q)}{n}\right)+ \nabla \cdot \left(\frac{\UL\PL (m_q\otimes m_q)}{n} - \UL\PL\left(\frac{m_q\otimes m_q}{n}\right)\right)\\
		&\quad+nE_{\ell}-\UL\PL(nE_{q})+m_{\ell}\times B_{\ell}-\UL\PL(m_{q}\times B_{q})\\
		&=: Q_1 + Q_2 + Q_{E} + Q_{B},
	\end{align*}
	where 
	\begin{align}
		Q_1 &:= \underbrace{n^{-1} \nabla \cdot \left(m_\ell\otimes m_\ell - \UL\PL (m_q\otimes m_q)\right)}_{=:Q_{11}} + \underbrace{\left(m_\ell\otimes m_\ell - \UL\PL (m_q\otimes m_q)\right) :\nabla (n^{-1})}_{=:Q_{12}},\label{def of Q_1}\\
		Q_2 &:= n^{-1}\UL\PL\Div (m_q\otimes m_q) - \UL\PL\left(n^{-1}\Div (m_q\otimes m_q)\right)\nonumber \\
		&\quad+ \UL\PL (m_q \otimes m_q) : \nabla (n^{-1}) - \UL\PL( m_q \otimes m_q :\nabla (n^{-1})),\label{def of Q_2}\\
		Q_{E}&:=nE_{\ell}-\UL\PL(nE_{q}),\label{def of Q_E}\\
		Q_{B}&:=m_{\ell}\times B_{\ell}-\UL\PL(m_{q}\times B_{q})\label{def of Q_B}.
	\end{align}
	Using \eqref{est on mollification commutator 0}, we could obtain
	\begin{align*}
		\left\|\partial_{t}^rQ_1\right\|_{N} &\leqslant\left\|\partial_{t}^rQ_{11}\right\|_{N} +\left\|\partial_{t}^rQ_{12}\right\|_{N}\lesssim_{n}
		\left\|\partial_{t}^r(m_\ell\otimes m_\ell - \UL\PL (m_q\otimes m_q))\right\|_{N+1}\nonumber\\
		&\lesssim_{n,\overline{N}} \ell^{1-N-r}\left(\left\|m_q\right\|_{C^0(\cI_{2\ell}^{q};C^{1}(\T^3))}^2+\left\|\partial_{t}m_q\right\|_{C^0(\cI_{2\ell}^{q};C^{0}(\T^3))}^2\right)\lesssim_{n,\overline {N},M} \ell^{1-N-r}(\lambda_q\delta_q^\frac{1}{2})^2,\\ 
		\left\|\partial_{t}^rQ_{B}\right\|_{N} &\lesssim_{n,\overline {N}} \ell^{2-N-r}\left(\left\|m_q\right\|_{C^0(\cI_{2\ell}^{q};C^{1}(\T^3))}\left\|B_q\right\|_{C^0(\cI_{2\ell}^{q};C^{1}(\T^3))}+\left\|\partial_{t}m_q\right\|_{C^0(\cI_{2\ell}^{q};C^{0}(\T^3))}\left\|\partial_{t}B_q\right\|_{C^0(\cI_{2\ell}^{q};C^{0}(\T^3))}\right)\nonumber\\
		&\lesssim_{n,\overline {N},M}\ell^{2-N-r}\lambda_{q}\delta_q^{\frac{1}{2}}.
	\end{align*}
	As for $Q_2$ and $Q_E$, we can use \eqref{est on two mollification commutator 1.1} and \eqref{est on two mollification commutator 1.2} with $\ell_1=\ell_2=\ell$ to obtain
	\begin{align*}
		\left\|\partial_{t}^rQ_2\right\|_N &\lesssim_{\overline {N}}\ell^{1-N-r}\left\|\Div(m_q\otimes m_q)\right\|_{C^0(\cI_{2\ell}^{q};C^{0}(\T^3))}\left\|\partial_{t}^{\max\left\lbrace1,r\right\rbrace }n^{-1}\right\|_{C^0(\cI_{2\ell}^{q};C^{\max\left\lbrace1,N\right\rbrace}(\T^3))}\\
		&\quad+\ell^{1-N-r}\left\|m_q\otimes m_q\right\|_{C^0(\cI_{2\ell}^{q};C^{0}(\T^3))}\left\|\partial_{t}^{\max\left\lbrace1,r\right\rbrace }n^{-1}\right\|_{C^0(\cI_{2\ell}^{q};C^{N+1}(\T^3))}\\ 
		&\lesssim_{n,\overline {N},M}
		\ell^{1-N-r} \lambda_q\delta_q^{\frac{1}{2}},\\
		\left\|\partial_{t}^rQ_{E}\right\|_N & \lesssim_{\overline {N}}\ell^{1-N-r}\left\|E_q\right\|_{C^0(\cI_{2\ell}^{q};C^{0}(\T^3))}\left\|\partial_{t}^{\max\left\lbrace1,r\right\rbrace }n^{-1}\right\|_{C^0(\cI_{2\ell}^{q};C^{\max\left\lbrace1,N\right\rbrace}(\T^3))}\\
		&\lesssim_{n,\overline {N},M}\ell^{1-N-r}.
	\end{align*}
	Combining the estimates, we conclude \eqref{est on Qmm}.\par 
	As for the advective derivative, we only need to get $\left\|\partial_{t}^r\DTL (n Q_{11})\right\|_{C^0(\cI^q;C^N(\T^3))} \lesssim \ell^{-N-r} \delta_q^\frac{1}{2}(\lambda_q\delta_q^\frac{1}{2})^2,$ which easily follows from
	\begin{align*}
		\left\|\partial_{t}^r\DTL Q_{11}\right\|_{C^0(\cI^q;C^N(\T^3))} &\lesssim_{\overline{N}}\left\|\partial_{t}^r\DTL  (n Q_{11})\right\|_{C^0(\cI^q;C^N(\T^3))}\\ &\quad+ \sum_{r_1+r_2=r}\sum_{N_1+N_2=N}\left\|\partial_{t}^{r_1}\DTL  n\right\|_{C^0(\cI^q;C^{N_1}(\T^3))}\left\|\partial_{t}^{r_2}Q_{11}\right\|_{C^0(\cI^q;C^{N_2}(\T^3))}.
	\end{align*}
	We first apply $\UL\PL$ to the relaxed momentum equation to obtain
	\begin{align*}
		\DTL  m_\ell 
		&=- \Div (m_\ell/n) m_\ell - \nabla p_\ell(n)+ \UL\PL(nE_q+m_q\times B_q)\\
		&\quad + \UL\PL
		\Div(n R_q - c_qn\Id) + Q(m_q, m_q) \\
		&= (\partial_t n_\ell/n-(m_\ell\cdot\nabla)n^{-1})m_\ell - \nabla p_\ell(n)+ \UL\PL(nE_q+m_q\times B_q) \\ &\quad+ \UL\PL
		\Div(n R_q - c_qn\Id)+ Q(m_q, m_q),
	\end{align*}
	where $n_\ell=\UL\PL n$. We conclude
	\begin{equation}\label{est on Dtl m_l N}
		\begin{aligned}
			\left\|\DTL m_\ell\right\|_{0}
			&\lesssim_{n, p}\left\|m_\ell\right\|_{0} +\left\|m_\ell\right\|_{0}^2 + 1+\left\|E_q\right\|_{C^0(\cI_{2\ell}^{q};C^{0}(\T^3))}+\left\|m_q\times B_q\right\|_{C^0(\cI_{2\ell}^{q};C^{0}(\T^3))}\\ 
			&\quad+\left\| R_q\right\|_{C^0(\cI_{2\ell}^{q};C^{1}(\T^3))} +\left\|Q(m_q, m_q)\right\|_0\\
			&\lesssim_{n, p, M} \lambda_q\delta_q,\\
			\left\|\DTL m_\ell\right\|_N
			&\lesssim_{n, p, N}\left\|m_\ell\right\|_N +\sum_{N_0+N_1=N}\left\|m_\ell\right\|_{N_0}\left\|m_\ell\right\|_{N_1}+ 1 +\ell^{1-N}\left\|nE_q\right\|_{C^0(\cI_{2\ell}^{q};C^{1}(\T^3))}\\
			&\quad+\ell^{1-N}\left\|m_q\times B_q\right\|_{C^0(\cI_{2\ell}^{q};C^{1}(\T^3))}+ \ell^{1-N}\left\|R_q\right\|_{C^0(\cI_{2\ell}^{q};C^{2}(\T^3))} +\left\|Q(m_q, m_q)\right\|_{N}\\
			&\lesssim_{n, p, N, M} \ell^{1-N}(\lambda_q\delta_q^\frac{1}{2})^2,
		\end{aligned}
	\end{equation}
	for $N\geqslant 1$, and 
	\begin{equation}\label{est on pt Dtl m_l N}
		\begin{aligned}
			\left\|\partial_{t}^r\DTL m_\ell\right\|_N
			&\lesssim_{n, p ,N,r} \ell^{1-N-r}\left\|\partial_{t}m_\ell\right\|_0 + \ell^{1-N-r}\left\|\partial_{t}m_\ell\right\|_0\left\|m_\ell\right\|_0 + 1 +\ell^{1-N-r}\left\|\partial_{t}(nE_q)\right\|_{C^0(\cI_{2\ell}^{q};C^{0}(\T^3))} \\
			&\quad+\ell^{1-N-r}\left\|\partial_{t}(m_q\times B_q)\right\|_{C^0(\cI_{2\ell}^{q};C^{0}(\T^3))}+ \ell^{1-N-r}\left\|\partial_{t}R_q\right\|_{C^0(\cI_{2\ell}^{q};C^{1}(\T^3))} +\left\|\partial_{t}^rQ(m_q, m_q)\right\|_N\\
			&\lesssim_{n, p, N, r, M} \ell^{1-N-r}(\lambda_q\delta_q^\frac{1}{2})^2,
		\end{aligned}
	\end{equation}
	for $N\geqslant0$ and $r\geqslant1$, where we have used 
	\begin{align*}
		\left\|\partial_{t}R_q\right\|_{C^0(\cI_{2\ell}^{q};C^{1}(\T^3))}\leqslant\left\|D_{t,q}R_q\right\|_{C^0(\cI_{2\ell}^{q};C^{1}(\T^3))}+\left\|(m_q/n\cdot\nabla)R_q\right\|_{C^0(\cI_{2\ell}^{q};C^{1}(\T^3))}\lesssim_{n, M}(\lambda_q\delta_q^\frac{1}{2})^2.
	\end{align*}
	Moreover, we could obtain
	\begin{align}
		\left\|\DTL  \nabla m _\ell \right\|_N \leqslant
		\left\|\nabla \DTL m_\ell\right\|_N +\left\|(\nabla (m_\ell/n)\cdot \nabla) m_\ell\right\|_N
		\lesssim_{n, p, N, M}\ell^{-N} (\lambda_q\delta_q^\frac{1}{2})^2,\label{est on Dtl nabla m_l}\\
		\left\|\DTL \partial_{t} m _\ell \right\|_N \leqslant
		\left\|\partial_{t} \DTL m_\ell\right\|_N +\left\|(\partial_{t} (m_\ell/n)\cdot \nabla) m_\ell\right\|_N
		\lesssim_{n, p, N, M}\ell^{-N} (\lambda_q\delta_q^\frac{1}{2})^2.\label{est on Dtl pt m_l}
	\end{align} 
	Recalling \eqref{est on m_q-m_l}, one can obtain for $N=0,1,$
	\begin{equation}\label{est on Dtl m_q}
		\begin{aligned}
			\left\|\DTL  m_q\right\|_{N}&\leqslant\left\|\DTL  m_\ell\right\|_{N}+\left\|\DTL (m_q-m_\ell)\right\|_{N} \lesssim_{n, p, M} \lambda_q^{N+1}\delta_q,\\
			\left\|\partial_{t}\DTL  m_q\right\|_{0}&\leqslant\left\|\partial_{t}\DTL  m_\ell\right\|_{0}+\left\|\partial_{t}\DTL (m_q-m_\ell)\right\|_{0} \lesssim_{n, p, M} \lambda_q^2\delta_q,\\
			\left\|\DTL \nabla m_q\right\|_{0}&\leqslant\left\|\nabla\DTL  m_q\right\|_{0}+\left\|(\nabla (m_\ell/n)\cdot \nabla) m_q\right\|_{0} \lesssim_{n, p, M} \lambda_q^{2}\delta_q,\\
			\left\|\DTL \partial_{t} m_q\right\|_{0}&\leqslant\left\|\partial_{t}\DTL  m_q\right\|_{0}+\left\|(\partial_{t} (m_\ell/n)\cdot \nabla) m_q\right\|_{0} \lesssim_{n, p, M} \lambda_q^{2}\delta_q.
		\end{aligned}
	\end{equation}
	Similarly, we have 
	\begin{align}
		\left\|\DTL E_q\right\|_{N}&\leqslant\left\|\nabla\times B_q\right\|_{N}+\left\|m_q\right\|_{N}+\left\|((m_\ell/n)\cdot\nabla) E_q\right\|_{N}\lesssim_{n, M}\lambda_{q}^N\delta_{q}^{\frac{1}{2}}+1,\label{est on DtlEq}\\
		\left\|\DTL B_q\right\|_{N}&\leqslant\left\|\nabla\times E_q\right\|_{N}+\left\|((m_\ell/n)\cdot\nabla) B_q\right\|_{N}\lesssim_{n, M}\lambda_{q}^N\delta_{q}^{\frac{1}{2}}+1, \label{est on DtlBq}
	\end{align}
	for $N=0,1,2$. Since $n Q_{11}$ can be decomposed into
	\begin{equation}\label{def of nQ11}
		n Q_{11}= (m_\ell-m_q)\cdot \nabla m_\ell + [m_q\cdot \nabla, \UL\PL] m_q +  (\Div m_\ell) m_\ell - \UL\PL( (\Div m_q) m_q),
	\end{equation}
	we have the following estimate
	\begin{align*}
		&\quad\left\|\partial_{t}^r\DTL  ( (\Div m_\ell) m_\ell - \UL\PL((\Div m_q)m_q))\right\|_{C^0(\cI^q;C^N(\T^3))}\\
		&=\left\|\partial_{t}^r(\DTL  (\partial_t n_\ell m_\ell) - \DTL \UL\PL ((\partial_t n) m_q ))\right\|_{C^0(\cI^q;C^N(\T^3))}\\
		&\leqslant\left\|\partial_{t}^r(\partial_t n_\ell(\DTL  m_\ell)+ (\DTL \partial_t n_\ell)m_\ell  - \UL\PL \DTL  ((\partial_t n)m_q))\right\|_{C^0(\cI^q;C^N(\T^3))} \\
		&\quad+\left\|\partial_{t}^r[m_\ell/n\cdot\nabla, \UL\PL] ((\partial_t n)m_q)\right\|_{C^0(\cI^q;C^N(\T^3))}\\
		&\lesssim_{n,\overline{N}}\left\|\partial_{t}^r\DTL  m_\ell\right\|_{C^0(\cI^q;C^N(\T^3))} + \sum_{r_0+r_1=r}\sum_{N_0+N_1 = N}\left\|\partial_{t}^{r_0}m_\ell\right\|_{C^0(\cI^q;C^{N_0}(\T^3))}\left\|\partial_{t}^{r_1}\DTL \partial_tn_\ell\right\|_{C^0(\cI^q;C^{N_1}(\T^3))}\\
		&\quad+ \ell^{-N-r}\left\|\DTL ((\partial_t n)m_q)\right\|_{C^0(\cI_\ell^q;C^0(\T^3))}+\left\|\partial_{t}^r[m_\ell/n\cdot\nabla, \UL\PL] ((\partial_t n)m_q)\right\|_{C^0(\cI^q;C^N(\T^3))}\\
		&\lesssim_{n, p, \overline{N}, M} \ell^{1-N-r}(\lambda_q\delta_q^\frac{1}{2})^2+\ell^{-N-r}\delta_{q}^{\frac{1}{2}}+\ell^{-N-r}\lambda_q\delta_q+\left\|\partial_{t}^r[m_\ell/n\cdot\nabla, \UL\PL] ((\partial_t n)m_q)\right\|_{C^0(\cI^q;C^N(\T^3))}\\
		&\lesssim_{n, p, \overline{N}, M} \ell^{-N-r}\delta_q^\frac{1}{2}(\lambda_q\delta_q^\frac{1}{2})^2+\left\|\partial_{t}^r[m_\ell/n\cdot\nabla, \UL\PL] ((\partial_t n)m_q)\right\|_{C^0(\cI^q;C^N(\T^3))}.
	\end{align*}
	Here we used \eqref{est on m_l}, \eqref{est on Dtl m_l N}, \eqref{est on pt Dtl m_l N}, and \eqref{est on Dtl m_q}. Next, we would use \eqref{est on two mollification commutator 1.1}--\eqref{est on two mollification commutator 2.3} with $\ell_1=\ell_2=\ell$ to obtain the estimates on  $\left\|\partial_{t}^r[m_\ell/n,\UL\PL]F\right\|_{C^0([c,d];C^{N}(\T^3))}$ and $\left\|\partial_{t}^r[m_\ell/n\cdot\nabla,\UL\PL]F\right\|_{C^0([c,d];C^{N}(\T^3))}$. We have for $N,r\in\N$,
	\begin{align}
		&\quad\left\|\partial_{t}^r[m_\ell/n,\UL\PL]F\right\|_{C^0([c,d];C^{N}(\T^3))}\nonumber\\
		&\leqslant\left\|\partial_{t}^r[m_\ell\UL\PL n^{-1},\UL\PL]F\right\|_{C^0([c,d];C^{N}(\T^3))}+\left\|\partial_{t}^r[m_\ell\PG n^{-1},\UL\PL]F\right\|_{C^0([c,d];C^{N}(\T^3))}\nonumber\\
		&\quad+\left\|\partial_{t}^r[m_\ell\UG\PL n^{-1},\UL\PL]F\right\|_{C^0([c,d];C^{N}(\T^3))}\nonumber\\
		&\lesssim_{N,r}\ell^{1-N-r}\left(\left\|\nabla(m_\ell/n)\right\|_{C^0([c,d]+\ell;C^{0}(\T^3))}+\left\|\partial_{t}(m_\ell/n)\right\|_{C^0([c,d]+\ell;C^{0}(\T^3))}\right)\left\| F\right\|_{C^0([c,d]+\ell;C^{0}(\T^3))}\nonumber\\
		&\quad+\ell^{1-N-r}\left\|\partial_{t}^{\max\left\lbrace1,r\right\rbrace }(m_\ell\PG n^{-1})\right\|_{C^0([c,d]+\ell;C^{\max\left\lbrace1,N \right\rbrace }(\T^3))}\left\| F\right\|_{C^0([c,d]+\ell;C^{0}(\T^3))}\nonumber\\
		&\quad+\ell^{1-N-r}\left\|\partial_{t}^{\max\left\lbrace1,r\right\rbrace }(m_\ell\UG\PL n^{-1})\right\|_{C^0([c,d]+\ell;C^{\max\left\lbrace1,N \right\rbrace }(\T^3))}\left\| F\right\|_{C^0([c,d]+\ell;C^{0}(\T^3))}\nonumber\\
		&\lesssim_{n, N, r, M}\ell^{1-N-r}\lambda_{q}\delta_{q}^{\frac{1}{2}}\left\| F\right\|_{C^0([c,d]+\ell;C^{0}(\T^3))},\label{est on commutator m_l / n 1}
	\end{align}	
	where we have used Bernstein inequality to get 
	\begin{align}
		\left\|\partial_{t}^r(m_\ell\PG n^{-1})\right\|_{C^0([c,d]+\ell;C^{N}(\T^3))}+\left\|\partial_{t}^r(m_\ell\UG\PL n^{-1})\right\|_{C^0([c,d]+\ell;C^{N}(\T^3))}\lesssim_{n,N,r} \ell^2.\label{est on m_l P n}
	\end{align}
	Replacing $F$ with $\partial_iF$, one could have
	\begin{equation}\label{est on commutator m_l / n 2}
		\begin{aligned}
			\left\|\partial_{t}^r[m_\ell/n\cdot\nabla,\UL\PL]F\right\|_{C^0([c,d];C^{N}(\T^3))}
			\lesssim_{n, N, r, M}\ell^{1-N-r}\lambda_{q}\delta_{q}^{\frac{1}{2}}\left\| \nabla F\right\|_{C^0([c,d]+\ell;C^{0}(\T^3))}.
		\end{aligned}
	\end{equation}
	Then, we could apply \eqref{est on commutator m_l / n 1} to $(\partial_{t}n)m_q$ to obtain
	\begin{align*}
		\left\|\partial_{t}^r\DTL  ((\Div m_\ell) m_\ell  - \UL\PL( (\Div m_q)m_q)\right\|_{C^0(\cI^q;C^N(\T^3))}\lesssim_{n, p, \overline{N}, M} \ell^{-N-r}\delta_q^\frac{1}{2}(\lambda_q\delta_q^\frac{1}{2})^2.
	\end{align*}
	Now, let us consider the first two terms in \eqref{def of nQ11}. We first use Poison summation formula, \begin{align*}
		\UL\PL f(t,x) = \int_{\R}\int_{\R^3} f(t-\tau,x-y) \check{\phi} _\ell(y)\check{\phi}^t_\ell(\tau)\rd y\rd \tau,
	\end{align*} 
	to rewrite the advective derivative of the commutator term as
	\begin{align*}
		&\quad \DTL [m\cdot \nabla, \UL\PL] m\\
		&=  (\partial_t + (m_\ell/n)(t,x)\cdot \nabla) \int ((m(t,x)-m(t-\tau,x-y))\cdot \nabla)m(t-\tau,x-y) \check{\phi}_\ell (y)\check{\phi}^t_\ell (\tau) \rd y\rd \tau\\
		&=  \int  ((\DTL m(t,x)-\DTL m(t-\tau,x-y)) \cdot  \nabla )m(t-\tau,x-y) \check{\phi}_\ell (y)\check{\phi}^t_\ell (\tau) \rd y\rd \tau\\
		&\quad-\int \left(\frac{m _\ell}{n}(t,x)-\frac{m _\ell}{n}(t-\tau,x-y)\right)_i\nabla_i  m_j(t-\tau,x-y)  \nabla_j m(t-\tau,x-y) \check{\phi}_\ell (y)\check{\phi}^t_\ell (\tau) \rd y\rd \tau\\
		&\quad+ \int ((m(t,x)-m(t-\tau,x-y))\cdot \DTL \nabla) m(t-\tau,x-y) \check{\phi}_\ell (y)\check{\phi}^t_\ell (\tau) \rd y\rd \tau\\
		&\quad+ \int (m(t,x)-m(t-\tau,x-y))_i\left(\frac{m _\ell}{n}(t,x)-\frac{m _\ell}{n}(t-\tau,x-y)\right)_j (\partial_{ij}m)(t-\tau,x-y) \check{\phi}_\ell (y)\check{\phi}^t_\ell (\tau) \rd y\rd \tau.
	\end{align*} 
	Based on the decomposition, we use \eqref{est on m_q}, \eqref{est on space mollification function}, \eqref{est on time mollification function}, \eqref{est on m_l}, \eqref{est on m_q-m_l}, \eqref{est on Dtl m_q-m_l}, \eqref{est on Dtl nabla m_l}--\eqref{est on Dtl m_q},  and
	\begin{align*}
		&|f(t,x)-f(t-\tau,x-y)|\leqslant\left\|\partial_{t}f\right\|_{0}|\tau|+\left\|\nabla f\right\|_{0}|y|,
	\end{align*}
	 to obtain
	\begin{align*}
		&\quad\left\|\DTL (n Q_{11})\right\|_{C^0(\cI^q;C^{0}(\T^3))}\\
		&\lesssim_{n}\left\|\DTL (m_q-m _\ell)\right\|_{C^0(\cI^q;C^{0}(\T^3))}\left\|\nabla m_\ell\right\|_{C^0(\cI^q;C^{0}(\T^3))}+\left\|m_q-m _\ell\right\|_{C^0(\cI^q;C^{0}(\T^3))}\left\|\DTL \nabla m _\ell\right\|_{C^0(\cI^q;C^{0}(\T^3))}\\ 
		&\quad+ \ell\left(\left\|\nabla \DTL  m_q\right\|_{0}+\left\|\partial_{t} \DTL  m_q\right\|_{0}\right)\left\|\nabla m_q\right\|_{0}+\ell\left(\left\|\nabla m_q\right\|_{0}+\left\|\partial_{t} m_q\right\|_{0}\right)\left\|\nabla m_q\right\|_{0}^2\\
		&\quad+ \ell \left(\left\|\nabla m_q\right\|_{0}+\left\|\partial_{t} m_q\right\|_{0}\right)\left\|\DTL  \nabla m_q\right\|_{0} + \ell^2 \left(\left\|\nabla m_q\right\|_{0}+\left\|\partial_{t} m_q\right\|_{0}\right)^2\left\|\nabla^2 m_q\right\|_{0}+ \delta_q^\frac{1}{2}(\lambda_q\delta_q^\frac{1}{2})^2\\
		&\lesssim_{n, p, M}\ell(\lambda_{q}\delta_{q}^{\frac{1}{2}})^3+\ell^2\lambda_{q}^4\delta_{q}^{\frac{3}{2}}+\ell^2(\lambda_{q}\delta_{q}^{\frac{1}{2}})^4+\delta_q^\frac{1}{2}(\lambda_q\delta_q^\frac{1}{2})^2\lesssim_{n, p, M}\delta_q^\frac{1}{2}(\lambda_q\delta_q^\frac{1}{2})^2.
	\end{align*}
	Moreover, we could use \eqref{est on m_l P n} to get for $N+r\geqslant 1,$
		\begin{align*}
			&\quad\left\|\partial_{t}^r\DTL (n Q_{11})\right\|_{C^0(\cI^q;C^{N}(\T^3))}\\ &\leqslant\left\|\partial_{t}^r(\partial_t + m_\ell \UL\PL n^{-1} \cdot \nabla)(n Q_{11})\right\|_{C^0(\cI^q;C^{N}(\T^3))} +\left\|\partial_{t}^r((m_\ell \PG n^{-1}) \cdot \nabla)(n Q_{11})\right\|_{C^0(\cI^q;C^{N}(\T^3))}\\
			&\quad+\left\|\partial_{t}^r((m_\ell \UG \PL n^{-1}) \cdot \nabla)(n Q_{11})\right\|_{C^0(\cI^q;C^{N}(\T^3))}\\
			&\lesssim_{n,\overline{N}} \ell^{-N-r}\left\|(\partial_t + m_\ell \UL\PL n^{-1} \cdot \nabla) (n Q_{11}) \right\|_{C^0(\cI^q;C^{0}(\T^3))} \\
			&\quad+ \sum_{r_0+r_1=r}\sum_{N_0+N_1=N}\left\|\partial_{t}^{r_0}(m_\ell\PG n^{-1})\right\|_{C^0(\cI^q;C^{N_0}(\T^3))}\left\| \partial_{t}^{r_1}\nabla (n Q_{11}) \right\|_{C^0(\cI^q;C^{N_1}(\T^3))}\\
			&\quad+ \sum_{r_0+r_1=r}\sum_{N_0+N_1=N}\left\|\partial_{t}^{r_0}(m_\ell\UG\PL n^{-1})\right\|_{C^0(\cI^q;C^{N_0}(\T^3))}\left\| \partial_{t}^{r_1}\nabla (n Q_{11}) \right\|_{C^0(\cI^q;C^{N_1}(\T^3))}\\
			&\lesssim_{n, p, \overline{N}, M} \ell^{-N-r} \left(\left\|\DTL  (n Q_{11}) \right\|_{C^0(\cI^q;C^{0}(\T^3))} +\left\|(m_\ell \PG n^{-1} \cdot \nabla) (n Q_{11}) \right\|_{C^0(\cI^q;C^{0}(\T^3))}\right)\\
			&\quad+\ell^{-N-r}\left\|(m_\ell \UG\PL n^{-1} \cdot \nabla) (n Q_{11}) \right\|_{C^0(\cI^q;C^{0}(\T^3))}+ \ell^{2-N-r} (\lambda_q \delta_q^\frac{1}{2})^2\\
			&\lesssim_{n, p, \overline{N}, M} \ell^{-N-r} \left(\left\|\DTL  (n Q_{11} )\right\|_0 + \delta_q^\frac{1}{2} (\lambda_q \delta_q^\frac{1}{2})^2\right) + \ell^{2-N-r} (\lambda_q \delta_q^\frac{1}{2})^2 \\
			&\lesssim_{n, p, \overline{N}, M} \ell^{-N-r} \delta_q^\frac{1}{2} (\lambda_q \delta_q^\frac{1}{2})^2,
		\end{align*}
	where we have used Bernstein's inequality to get the second inequality. \par
	In order to estimate the advective derivative of $Q_{12}$, we calculate
	\begin{align*}
		&\quad\left\|\partial_{t}^{r}\DTL  \left(m_\ell\otimes m_\ell - \UL\PL (m_q\otimes m_q)\right)\right\|_{C^0(\cI^q;C^{N}(\T^3))} \\
		&\leqslant\left\|\partial_{t}^{r}( (\DTL m_q)_\ell\otimes m_\ell +
		m_\ell \otimes(\DTL m_q)_\ell - \UL\PL \DTL  (m_q\otimes m_q))\right\|_{C^0(\cI^q;C^{N}(\T^3))}\\
		&\quad+ 2\left\|\partial_{t}^{r}(([m_\ell/n\cdot \nabla, \UL\PL] m_q) \otimes m_\ell)\right\|_{C^0(\cI^q;C^{N}(\T^3))}+\left\| \partial_{t}^{r}[m_\ell /n\cdot \nabla, \UL\PL] (m_q \otimes m_q)\right\|_{C^0(\cI^q;C^{N}(\T^3))}\\
		&\lesssim_{n, \overline{N}} \ell^{2-N-r} \left(\left\|\DTL m_q\right\|_{1}\left\|m_q\right\|_{1}+\left\|\partial_t\DTL m_q\right\|_{0}\left\|\partial_{t}m_q\right\|_{0}\right) \\
		&\quad+ \sum_{r_1+r_2=r}\sum_{N_1+N_2=N} \ell^{1-N_1-r_1} \lambda_{q}\delta_{q}^{\frac{1}{2}}\left\|m_q\right\|_1 \left\|\partial_{t}^{r_2}m_\ell\right\|_{N_2}+\ell^{1-N-r}\lambda_{q}\delta_{q}^{\frac{1}{2}}\left\| m_q \otimes m_q\right\|_{1}\\
		&\lesssim_{n, p, \overline{N}, M}\ell^{1-N-r} (\lambda_q \delta_q^\frac{1}{2})^2,
	\end{align*}
	where we have used \eqref{est on mollification commutator 0} for the first term, \eqref{est on commutator m_l / n 2} for the second and third term. Then, we have
	\begin{align*}
		\left\|\partial_{t}^r\DTL Q_{12}\right\|_{C^0(\cI^q;C^{N}(\T^3))}
		&\lesssim\left\|\partial_{t}^r(\nabla n^{-1}:  \DTL \left(m_\ell\otimes m_\ell - \PL (m\otimes m)\right))\right\|_{C^0(\cI^q;C^{N}(\T^3))}\\
		&\quad+\left\|\partial_{t}^r((m_\ell\otimes m_\ell - \PL (m\otimes m)): \DTL \nabla n^{-1})\right\|_{C^0(\cI^q;C^{N}(\T^3))}\\
		&\lesssim_{n, p, \overline{N}, M} \ell^{-N-r} \delta_q^\frac{1}{2}(\lambda_q\delta_q^\frac{1}{2})^2.
	\end{align*}	
	The method of estimating  $\DTL Q_B$ is similar,  
	\begin{align*}
		&\quad\left\|\partial_{t}^r\DTL  \left(m_\ell\otimes B_\ell - \PL (m_q\otimes B_q)\right)\right\|_{C^0(\cI^q;C^{N}(\T^3))} \\
		&\leqslant\left\|\partial_{t}^r ((\DTL m_q)_\ell\otimes B_\ell+m_\ell \otimes(\DTL B_q)_\ell - \UL\PL \DTL  (m_q\otimes B_q))\right\|_{C^0(\cI^q;C^{N}(\T^3))}\\
		&\quad+\left\|\partial_{t}^r[m_\ell/n\cdot \nabla, \UL\PL] (m_q \otimes B_q)\right\|_{C^0(\cI^q;C^{N}(\T^3))}+\left\|\partial_{t}^r(([m_\ell /n\cdot \nabla, \UL\PL] m_q) \otimes B_\ell)\right\|_{C^0(\cI^q;C^{N}(\T^3))}\\
		&\quad+\left\|\partial_{t}^r(m_\ell \otimes ([m_\ell/n\cdot \nabla, \UL\PL] B_q))\right\|_{C^0(\cI^q;C^{N}(\T^3))} \\
		&\lesssim_{n,  \overline{N}} \ell^{2-N-r} \left(\left\|\DTL m_q\right\|_{1}\left\|B_q\right\|_{1}+\left\|\partial_t\DTL m_q\right\|_{0}\left\|\partial_{t}B_q\right\|_{0}+\left\|\DTL B_q\right\|_{1}\left\|m_q\right\|_{1}+\left\|\partial_t\DTL B_q\right\|_{0}\left\|\partial_{t}m_q\right\|_{0}\right) \\
		&\quad+ \sum_{r_1+r_2=r}\sum_{N_1+N_2=N} \ell^{1-N_1-r_1} \lambda_{q}\delta_{q}^{\frac{1}{2}}\left(\left\|m_q\right\|_1 \left\|\partial_{t}^{r_2}B_\ell\right\|_{N_2}+\left\|B_q\right\|_1 \left\|\partial_{t}^{r_2}m_\ell\right\|_{N_2}\right) +\ell^{1-N-r} \lambda_{q}\delta_{q}^{\frac{1}{2}}\left\| m_q \otimes B_q\right\|_{1}\\
		&\lesssim_{n, p, \overline{N}, M}\ell^{1-N-r} (\lambda_q \delta_q^\frac{1}{2})^2.
	\end{align*}
	Next, we consider $\DTL Q_2$. It is necessary to give the estimates on $(\UL\PL f)g-\UL\PL(fg)$,
	\begin{align*}
		&\quad\left\|\partial_{t}^r\DTL  (f_\ell g - (fg)_\ell)\right\|_{C^0(\cI^q;C^{N}(\T^3))}\\ 
		&=\left\|\partial_{t}^r((\DTL f_\ell)g + f_\ell \DTL g - (\DTL (fg))_\ell - [{m_\ell}/{n} \cdot \nabla, \UL\PL](fg))\right\|_{C^0(\cI^q;C^{N}(\T^3))}\\
		&\leqslant\left\|\partial_{t}^r((\DTL f)_\ell g - ((\DTL f)g)_\ell)\right\|_{C^0(\cI^q;C^{N}(\T^3))} +\left\|\partial_{t}^r(f_\ell \DTL g  - (f\DTL g)_\ell)\right\|_{C^0(\cI^q;C^{N}(\T^3))}\\
		&\quad+\left\|\partial_{t}^r(\left[\left[{m_\ell}/{n} \cdot \nabla, \UL\PL\right],g\right]f)\right\|_{C^0(\cI^q;C^{N}(\T^3))}\\
		&\leqslant\left\|\partial_{t}^r((\DTL f)_\ell g - ((\DTL f)g)_\ell)\right\|_{C^0(\cI^q;C^{N}(\T^3))}+ \sum_{r_1+r_2=r}\sum_{N_1+N_2=N}\left\|\partial_{t}^{r_1}f_\ell\right\|_{C^0(\cI^q;C^{N_1}(\T^3))}\left\|\partial_{t}^{r_2}(\DTL g - (\DTL g)_\ell)\right\|_{C^0(\cI^q;C^{N_2}(\T^3))}   \\
		&\quad+\left\|\partial_{t}^r(f_\ell (\DTL g)_\ell-(f\DTL g)_\ell)\right\|_{C^0(\cI^q;C^{N}(\T^3))}+\left\|\partial_{t}^r\left[\left[({m_\ell}/{n}) \cdot \nabla, \UL\PL\right],g\right]f\right\|_{C^0(\cI^q;C^{N}(\T^3))}\\
		&\lesssim_{n,\overline{N}}\ell^{1-N-r}\left\|\DTL f\right\|_0\left\|\partial_{t}^{\max\left\lbrace1,r \right\rbrace }g\right\|_{\max\left\lbrace1,N \right\rbrace }+\sum_{r_1+r_2=r}\sum_{N_1+N_2=N}\ell^{-N_1-r_1}\left\|f\right\|_0\left\|\partial_{t}^{r_2}\DTL g\right\|_{N_2}\\
		&\quad+\ell^{2-N-r}\left(\left\|f\right\|_1\left\|\DTL g\right\|_{1}+\left\|\partial_{t}f\right\|_0\left\|\partial_{t}\DTL g\right\|_{0}\right)+\left\|\partial_{t}^r\left[\left[m_\ell/n \cdot \nabla, \UL\PL\right],g\right]f\right\|_{N},
	\end{align*}
	where we apply \eqref{est on two mollification commutator 1.2} to the first norm and \eqref{est on mollification commutator 0} to the third term. Recalling \eqref{est on pt m_l / n N}, we could apply \eqref{est on mollification commutator 3} to the last term due to the finite range of $N$. 
	Then, we could get the estimate for advective derivative of $Q_2$,
	\begin{align*}
		&\quad\left\|\partial_{t}^r\DTL  Q_2\right\|_{C^0(\cI^q;C^{N}(\T^3))}\\
		&\lesssim_{n,\overline{N}} \ell^{1-N-r}\left\|\DTL  \Div (m_q \otimes m_q)\right\|_0\left\|\partial_{t}^{\max\left\lbrace1,r \right\rbrace }n^{-1}\right\|_{\max\left\lbrace1,N \right\rbrace }+ \sum_{r_1+r_2=r}\sum_{N_1+N_2=N} \ell^{-N_1-r_1} \left\|\Div (m_q \otimes m_q)\right\|_0\left\|\partial_{t}^{r_2}\DTL  n^{-1}\right\|_{N_2}
		\\
		&\quad+ \ell^{2-N-r} \left(\left\|\Div (m_q \otimes m_q)\right\|_1\left\|\DTL  n^{-1}\right\|_{1}+\left\|\partial_{t}\Div (m_q \otimes m_q)\right\|_0\left\|\partial_{t}\DTL  n^{-1}\right\|_{0}\right) \\
		&\quad+\ell^{1-N-r}\delta_{q}^{\frac{1}{2}}\left\|\Div(m_q\otimes m_q)\right\|_1+\ell^{-N-r}\delta_{q}^{\frac{1}{2}}\left\|\Div(m_q\otimes m_q)\right\|_0\\
		&\quad+\ell^{1-N-r}\left\|\DTL(m_q \otimes m_q)\right\|_0\left\|\partial_{t}^{\max\left\lbrace1,r \right\rbrace }n^{-1}\right\|_{N+1}+\sum_{r_1+r_2=r}\sum_{N_1+N_2=N} \ell^{-N_1-r_1} \left\|m_q \otimes m_q\right\|_0\left\|\partial_{t}^{r_2}\DTL  \nabla n^{-1}\right\|_{N_2}
		\\
		&\quad + \ell^{2-N-r} \left(\left\|m_q \otimes m_q\right\|_1\left\|\DTL\nabla  n^{-1}\right\|_{1}+\left\|\partial_{t}(m_q \otimes m_q)\right\|_0\left\|\partial_{t}\DTL\nabla n^{-1}\right\|_{0}\right) \\
		&\quad+\ell^{1-N-r}\delta_{q}^{\frac{1}{2}}\left\|m_q\otimes m_q\right\|_1+\ell^{-N-r}\delta_{q}^{\frac{1}{2}}\left\|m_q\otimes m_q\right\|_0\\
		&\lesssim_{n, p, \overline{N}, M} \ell^{-N-r} (\ell \lambda_q^2 \delta_q + \lambda_q\delta_q) + \ell^{-N-r} \lambda_q \delta_q^\frac{1}{2}+ \ell^{2-N-r} \lambda_q^2 \delta_q^\frac{1}{2} \ell^{-1} \delta_q^\frac{1}{2}+ \ell^{1-N-r} \lambda_q^2 \delta_q\\
		&\quad+ \ell^{-N-r} \lambda_q \delta_q+ \ell^{1-N-r}\lambda_{q}\delta_{q}+ \ell^{2-N-r} \lambda_q \delta_q^\frac{1}{2} \ell^{-1} \delta_q^\frac{1}{2}+ \ell^{1-N-r} \lambda_q \delta_q+ \ell^{1-N-r} \delta_q^{\frac{1}{2}}  \\
		&\lesssim_{n, p, \overline{N}, M} \ell^{-N-r} \delta_q^\frac{1}{2} (\lambda_q \delta_q^\frac{1}{2})^2,
	\end{align*}
	where we have used \eqref{est on pt m_l / n N}, \eqref{est on Dtl m_q}, and
	$$
	\left\|\DTL  \Div (m_q \otimes m_q)\right\|_0 \leqslant 2\left\|(\DTL m_q) \otimes m_q\right\|_1 +\left\|\Div (m_\ell/n)\right\|_0\left\|m_q \otimes m_q\right\|_1 \lesssim_{n, M} \lambda_q^2 \delta_q + \ell^{-1} \delta_q^\frac{1}{2} \lambda_q \delta_q^\frac{1}{2}.
	$$
	Similarly, we could get the estimate for $\DTL Q_E$,
	\begin{align*}
		\left\|\partial_{t}^r\DTL  Q_E\right\|_{C^0(\cI^q;C^{N}(\T^3))}
		&\lesssim_{n,\overline{N}} \ell^{1-N-r}\left\|\DTL  E_q\right\|_0\left\|\partial_{t}^{\max\left\lbrace1,r \right\rbrace }n^{-1}\right\|_{\max\left\lbrace1,N \right\rbrace }+\sum_{r_1+r_2=r}\sum_{N_1+N_2=N} \ell^{-N_1-r_1} \left\|E_q\right\|_0
		\left\|\partial_{t}^{r_2}\DTL  n^{-1}\right\|_{N_2}\\
		&\quad+ \ell^{2-N-r} \left(\left\|E_q\right\|_1\left\|\DTL  n^{-1}\right\|_{1}+\left\|\partial_{t}E_q\right\|_0\left\|\partial_{t}\DTL  n^{-1}\right\|_{0}\right)+\ell^{1-N-r} \delta_q^\frac{1}{2}\left\|\nabla E_q\right\|_0  
		+ \ell^{-N-r} \delta_q^\frac{1}{2}\left\|E_q\right\|_0 \\
		&\lesssim_{n, \overline{N}, M}\ell^{1-N-r}+\ell^{-N-r}+\ell^{2-N-r}\lambda_{q}\delta_{q}^{\frac{1}{2}}+\ell^{-N-r}\delta_{q}^{\frac{1}{2}}\lesssim_{n, \overline{N}, M} \ell^{-N-r} \delta_q^\frac{1}{2} (\lambda_q \delta_q^\frac{1}{2})^2.
	\end{align*}
	 Up to now, we could obtain \eqref{est on Dtl Qmm}.
	Finally, we give the estimates on $\left\|\partial_{t}^r\DTL^sm_\ell\right\|_N$ for $s\geqslant 2$, which will be used repeatly in the estimate on mixed derivatives. Observe that
	\begin{equation}
		\begin{aligned}
			\DTL^2  m_\ell 
			&= \DTL((\partial_t n_\ell/n-(m_\ell\cdot\nabla)n^{-1})m_\ell-\nabla p_\ell(n)+ \UL\PL(nE_q+m_q\times B_q))+\DTL Q(m_q, m_q)\\
			&\quad+ \UL\PL
			\Div(D_{t,q}(n R_q - c_qn\Id))+\UL\PL
			\Div(((m_\ell-m_q)/n\cdot\nabla)(n R_q - c_qn\Id))\\
			&\quad+[m_\ell/n\cdot\nabla,\UL\PL]\Div(n R_q - c_qn\Id)+\UL\PL\left( \partial_k(m_\ell/n)_i\partial_i(n R_q - c_qn\Id)_{jk}\right).\notag
		\end{aligned}
	\end{equation}
	Then, it's easy to obtain
		\begin{align*}
		&\quad\left\|\partial_{t}^r\DTL^2 m_\ell\right\|_N\\
		&\lesssim_{n,N,r}\left\|\partial_{t}^r\DTL m_\ell\right\|_N+\sum_{r_0+r_1=r}\sum_{N_0+N_1=N}\left\|\partial_{t}^{r_0}\DTL m_\ell\right\|_{N_0}\left\|\partial_{t}^{r_1}m_\ell\right\|_{N_1}+\left\|\partial_{t}^{r}\DTL p_\ell(n)\right\|_{N}\\
		&\quad+\left\|\partial_{t}^{r+1}\PL\UL (nE_q+m_q\times B_q)\right\|_{N}+\sum_{r_0+r_1=r}\sum_{N_0+N_1=N}\left\|\partial_{t}^{r_0}m_\ell\right\|_{N_0}\left\|\partial_{t}^{r_1}\PL\UL (nE_q+m_q\times B_q)\right\|_{N_1+1}\\
		&\quad+\left\|\partial_{t}^{r}\DTL Q(m_q, m_q)\right\|_N+ \ell^{-N-r}\left\|D_{t,q}R_q\right\|_{C^0(\cI^q;C^{1}(\T^3))}+\ell^{-N-r}\left\|
		((m_\ell-m_q)/n\cdot\nabla)(n R_q - c_qn\Id)\right\|_{C^0(\cI^{q-1};C^{1}(\T^3))}\\
		&\quad+\left\|\partial_{t}^r[m_\ell/n\cdot\nabla,\UL\PL]\Div(n R_q - c_qn\Id)\right\|_N+\left\|\partial_{t}^{r}\UL\PL\left(\partial_k(m_\ell/n)_i\partial_i(n R_q - c_qn\Id)_{jk}\right)\right\|_{N}\\
		&\lesssim_{n, p, N, r,  M}\ell^{-N-r}\lambda_{q}^2\delta_{q}^{\frac{3}{2}}.\notag
		\end{align*}
	Moreover, we could calculate
	\begin{align*}
		\partial_t^r\DTL^s m_\ell=\partial_{t}^{r+1}\DTL^{s-1} m_\ell+\partial_{t}^{r}((m_\ell/n\cdot\nabla)\DTL^{s-1} m_\ell),
	\end{align*}
	and it is easy to obtain for $s\geqslant 3; N,r\geqslant 0$,
	\begin{equation}
		\begin{aligned}
			\left\|\partial_t^r\DTL^s m_\ell\right\|_N&\lesssim_{n}\left\|\partial_{t}^{r+1}\DTL^{s-1} m_\ell\right\|_N+\sum_{N_0+N_1=N}\sum_{r_0+r_1=r}\left\|\partial_{t}^{r_0}(m_\ell/n)\right\|_{N_0}\left\|\partial_{t}^{r_1}\DTL^{s-1} m_\ell\right\|_{N_1+1}\lesssim_{n, p, N, r, s, M}\ell^{2-N-r-s}\lambda_{q}^2\delta_{q}^{\frac{3}{2}}.\notag
		\end{aligned}
	\end{equation}
	Finally, we could give 
	\begin{equation}\label{est on Dtl s m_l}
		\left\|\partial_{t}^r\DTL^sm_\ell\right\|_N\lesssim_{n,p, N, r, s, M}
		\left\lbrace
		\begin{split}
			&\lambda_{q}\delta_{q},&&s=1,N=r=0,\\
			&\ell^{1-N-r}(\lambda_{q}\delta_{q}^\frac{1}{2})^2,	&&s=1,N+r\geqslant1,	\\
			&\ell^{2-N-r-s}\lambda_{q}^2\delta_{q}^\frac{3}{2},	&&s\geqslant2,N+r\geqslant0.	
		\end{split}
		\right.\qedhere
	\end{equation}
\end{proof}
\section{Construction of the perturbation}\label{Construction of the perturbation}
\subsection{Mikado flow and geometry lemmas}\label{Mikado flow and Geometry Lemmas}
We first introduce two important geometric lemmas and the Mikado flow which we will use as the bases of our building blocks. The concept of Mikado flow was first proposed by Daneri and Sz{\'e}kelyhidi and is described in detail in \cite[Lemma 2.3]{DS17}. For our purposes, we will use a version of the Mikado flow in \cite[Section 2.2]{GK22}.
\begin{df}[Mikado flow]\label{def of Mikado flow}
	Given a vector $f\in \Z^3$, we could define $U_{f,\tilde{x}}:\T^3\rightarrow\R^3$ which	satisfy
	\begin{equation}
		U_{f,\tilde{x}}=\psi_f(x-\tilde{x})f,
	\end{equation}
	where $\psi_f(x)=\Delta\Psi_f(x)\in C^{\infty}_c(\T^3)$, $\int_{\T^3}\psi_f\rd x=0$, and $f\cdot\nabla\psi_{f}=0$. Moreover, we have
	\begin{equation}\label{prop of Mikado flow }
		\begin{aligned}
			&\Div U_{f,\tilde{x}}=f\cdot\nabla \psi_{f,\tilde{x}}=0,\\
			&\Div (U_{f,\tilde{x}}\otimes U_{f,\tilde{x}})=(U_{f,\tilde{x}}\cdot\nabla) U_{f,\tilde{x}}=f\psi_{f,\tilde{x}}(f\cdot\nabla)\psi_{f,\tilde{x}}=0.
		\end{aligned}
	\end{equation}
	The Mikado flow is a stationary solution of the incompressible Euler equations on $\T^3$ with zero pressure. 
\end{df}
We can choose a finite set of vectors in $\Z^3$ and use the corresponding Mikado flows to construct a perturbation that reduces the error $R_q,\varphi_q$, which is usually achieved with the help of certain geometric lemmas. Here we recall the ones given by \cite{GK22}: Lemma 3.1 and Lemma 3.2, where $\mS$ is the subset of $\R^{3\times 3}$ consisting of all symmetric matrices, and  $|K|_\infty:=\underset{l,m}{\max}|k_{lm}|$ for $K=(k_{lm})_{l,m=1}^3\in\R^{3\times3}$.
\begin{lm}[Geometric Lemma I]\cite[Lemma 3.1]{GK22}.\label{Geometric Lemma I}
	Let $\cF=\left\lbrace f_i\right\rbrace_{i=1}^6 $ be a set of vectors in $\Z^3$ and there is a constant $C$ such that 
	\begin{align*}
		\sum_{i=1}^{6}f_i\otimes f_i=C\Id,\quad and\quad \left\lbrace f_i\otimes f_i\right\rbrace _{i=1}^6\ forms\ a\ basis\ of\ \mS.
	\end{align*}
	Then, we could find a positive constant $\mathcal{N}_0=\mathcal{N}_0(\cF)$ such that for any $N\leqslant \mathcal{N}_0$, we can find smooth functions $\left\lbrace \Gamma_{f_i}\right\rbrace_{i=1}^6\subset C^\infty(S_N;(0,\infty)), $ with domain $S_N:=\left\lbrace \Id-K\in \mS,|K|_\infty\leqslant N\right\rbrace $ satisfying  
	$$\Id-K=\sum_{i=1}^{6}\Gamma_{f_i}^2(\Id-K)(f_i\otimes f_i),\qquad\forall(\Id-K)\in S_N.$$
\end{lm}    
\begin{lm}[Geometric Lemma II]\cite[Lemma 3.2]{GK22}.\label{Geometric Lemma II}
	Suppose that 
	\begin{align*}
		\left\lbrace f_1,f_2,f_3\right\rbrace\ \subset\ \Z^3\setminus\left\lbrace 0\right\rbrace\ is\ an\ orthogonal\ frame\ and\ f_4=-(f_1+f_2+f_3) .  
	\end{align*}	
	Then, for any $N_0>0$, there are affine functions $\left\lbrace \Gamma_{f_k}\right\rbrace_{1\leqslant k\leqslant 4}\subset C^\infty(\mathcal{B}_{N_0};[N_0,\infty)), $ with $\mathcal{B}_{N_0}:=\left\lbrace m\in\R^3:|m|\leqslant N_0\right\rbrace $ such that 
	\begin{align*}
		m=\sum_{k=1}^{4}\Gamma_{f_k}(m)f_k,\qquad\forall m\in\mathcal{B}_{N_0}.
	\end{align*}
\end{lm}  
\subsection{Cutoffs}  
In this part, we will give partitions of unity in space $\R^3$ and in time $\R$ similar as in \cite{DK22} and \cite{GK22}. We introduce some nonnegative smooth functions $\left\lbrace\chi_\upsilon\right\rbrace_{\upsilon\in\Z^3} $ and $\left\lbrace\theta_u\right\rbrace_{u\in\Z} $ such that
$$\sum_{\upsilon\in\Z^3}\chi_\upsilon^6(x)=1,\quad \forall x\in\R^3,\quad\sum_{u\in\Z}\theta_u^6(t)=1,\quad\forall t\in\R.$$
Here, $\chi_\upsilon(x)=\chi_0(x-2\pi\upsilon),$ where $\chi_0$ is a nonnegative smooth function supported in $Q(0,9/8\pi)$ satisfying $\chi_0=1$ on $\overline{Q(0,7/8\pi)}$, and $Q(x,r)$ denotes the cube $\left\lbrace y\in\R^3:|y-x|_\infty<r\right\rbrace $. Similarly, $\theta_u(t)=\theta_0(t-u)$ where $\theta_0\in C_c^\infty(\R)$ satisfies $\theta_0=1$ on $[1/8,7/8]$ and $\theta_0=0$ on $(-1/8,9/8)^c$. Then, we divide $\Z^3$ into 27 equivalent families $[j]$ with $j\in\Z_3^3$ via the usual equivalence relation 
$$\upsilon=(\upsilon_1,\upsilon_2,\upsilon_3)\sim\tilde{\upsilon}=(\tilde{\upsilon}_1,\tilde{\upsilon}_2,\tilde{\upsilon}_3)\Longleftrightarrow\quad \upsilon_i\equiv\tilde{\upsilon}_i\quad \Mod\ 3\quad  \mathrm{for\ all}\ i=1,2,3.$$
We will apply Lemmas \ref{Geometric Lemma I} and \ref{Geometric Lemma II} to construct 27 pairwise disjoint families $\mathcal{F}^j$ of vectors in $\Z^3$, each indexed by $j$. Each family $\mathcal{F}^j$ consists of two subfamilies $\mathcal{F}^{j,R}\bigcup\mathcal{F}^{j,\varphi}$ with cardinalities $|\mathcal{F}^{j,R}|=6$ and $|\mathcal{F}^{j,\varphi}|=4$. 
Next, we introduce a notation
$$\mathscr{I}:=\left\lbrace(u,\upsilon,f):(u,\upsilon)\in\Z\times\Z^3\ and\ f\in \mathcal{F}^{[\upsilon]}\right\rbrace .$$
For each $I=(u,\upsilon,f)\in\mathscr{I}$, we use $f_I$ and $U_{f_I}$ to represent the third component of the index and the corresponding Mikado flow. Then, we can divide $\mathscr{I}$ into $\mathscr{I}_R\bigcup\mathscr{I}_\varphi$ depending on whether $f_I\in\mathcal{F}^{[\upsilon],R}$ or $f_I\in\mathcal{F}^{[\upsilon],\varphi}$. Next, we will define the cut-off parameters $\tau_q$ and $\mu_q$ with $\tau_q^{-1}>0$ and $\mu_q^{-1}\in \Z_+$, which are given by 
\begin{equation}\label{def of mu tau}
	\mu_q^{-1}=3\left\lceil\lambda_q^{\frac{1}{2}} \lambda_{q+1}^{\frac{1}{2}} \delta_q^{\frac{1}{4}} \delta_{q+1}^{-\frac{1}{4}} / 3\right\rceil, \quad \tau_q^{-1}=40 \pi \overset{\circ}{C}_{n} M \eta^{-1} \cdot \lambda_q^{\frac{1}{2}} \lambda_{q+1}^{\frac{1}{2}} \delta_q^{\frac{1}{4}} \delta_{q+1}^{\frac{1}{4}},
\end{equation}
where $\eta$ is a constant defined in Proposition \ref{def of eta}, and $\overset{\circ}{C}_{n}$ is chosen as a constant depending on $n$ such that 
$$
\left\|\partial_{t}(m_q/n)\right\|_0+\left\|\nabla(m_q/n)\right\|_0\leqslant\overset{\circ}{C}_{n}M\lambda_{q}\delta_{q}^{\frac{1}{2}}.
$$ 
Then, we have 
\begin{align}
	\tau_q\left(\left\|\partial_{t}(m_q/n)\right\|_0+\left\|\nabla(m_q/n)\right\|_0\right)\leqslant\frac{\eta}{10\pi\lambda_{q+1}\mu_q}\label{geomoetry condition 1}
\end{align}
For the two different conditions, we will define the cut-off functions:
\begin{eqnarray}
	\theta_I(t)=\left\lbrace 
	\begin{aligned}
		&\theta_u^3(\tau_q^{-1}t),\qquad I\in\mathscr{I}_R,\\
		&\theta_u^2(\tau_q^{-1}t),\qquad I\in\mathscr{I}_\varphi,
	\end{aligned}\right.
	\qquad	
	\chi(t)=\left\lbrace 
	\begin{aligned}
		&\chi_\upsilon^3(\mu_q^{-1}x),\qquad I\in\mathscr{I}_R,\\
		&\chi_\upsilon^2(\mu_q^{-1}x),\qquad I\in\mathscr{I}_\varphi.
	\end{aligned}	\right.
\end{eqnarray}
For convenience, we denote $\mathscr{I}_{u,\upsilon,R}:=\{I=(u,\upsilon,f)\in\mathscr{I}:f\in\cF^{[\upsilon],R}\}$ and $\mathscr{I}_{u,\upsilon,\varphi}:=\{I=(u,\upsilon,f)\in\mathscr{I}:f\in\cF^{[\upsilon],\varphi}\}$. 
Moreover, for any $0<\alpha<\frac{1}{3}$, there exist $1<\bar{b}_2(\alpha)\leqslant\frac{1+\alpha}{1-\alpha}$ such that for any $b\in(1,\bar{b}_2(\alpha))$, we could find $\Lambda_2(\alpha,b,M,n)$ with the following property:if $\lambda\geqslant\Lambda_2,$
\begin{equation}
	\begin{aligned}
	 \frac{1}{\tau_{q}\lambda_{q+1}}\leqslant\delta_{q+1}^{\frac{1}{2}}\leqslant\delta_{q}^{\frac{1}{2}}\leqslant\frac{1}{\lambda_{q+1}\mu_q}.	\label{relation}
	\end{aligned}	
\end{equation}
Note that the above relationship is not essential for our proof. In order to make the estimation appear simpler and clearer, we provided the above relationship.
\subsection{Backward flow}
 Next, we will introduce the backward flow map. For a given $I=(u,\upsilon,f)\in\mathscr{I}$, we denote $\xi_I=\xi_u$ as the solution of the following transport equation:
\begin{equation}
	\left\lbrace 
	\begin{aligned}
	&\partial_t\xi_u+(m_\ell/n\cdot\nabla)\xi_u=0,\\
	&\xi_u(t_u,x)=x.	
	\end{aligned}\right.\label{backward flow}
\end{equation}
$\xi_u$ is the inverse of the periodic flow $m_\ell/n$ starting at $t_u=u\tau_q$ and satisfies
\begin{equation}
\Phi_u(t,\xi_u(t,x);t_u)=x,\label{for and back}
\end{equation}
where $\Phi_u(t,x;t_u)$ is the forward flow in \eqref{eq of forward flow} 
and we have the following estimates 
\begin{pp}\label{est on backward flow}
	For any $0<\alpha<\frac{1}{3}$, let the parameters $\bar{b}_1(\alpha)$ and $\Lambda_1$ be as in the statement of Proposition \ref{est on mollification 1} and let the parameters $\bar{b}_2(\alpha)$ and $\Lambda_2$ be as in  \eqref{relation}. We set $\bar{b}_3(\alpha)=\min\{\bar{b}_1(\alpha),\bar{b}_2(\alpha)\}$. Then, for any $b\in(1,\bar{b}_3(\alpha))$, there exists $\Lambda_3(\alpha,b,M,n)\geqslant\max\{\Lambda_1,\Lambda_2\}$ such that for any $\lambda_0\geqslant\Lambda_3$, the backward flow $\xi_I$ satisfies the following estimates for any $N\geqslant 0,$ $0\leqslant s\leqslant 2$, and $0\leqslant r\leqslant 4$ on $\cI_u:=\big[t_u-\frac{1}{2}\tau_q,t_u+\frac{3}{2}\tau_q\big]\bigcap\big[-2\tau_q,T+2\tau_q\big]$:
	\begin{align}
	&\left\|\nabla\xi_I\right\|_{\CIN}+\left\|(\nabla\xi_I)^{-1}\right\|_{\CIN}\lesssim_{n,N} \ell^{-N},\label{est on nabla xi and xi^-1} \\
	&\left\|\nabla\xi_I-\Id\right\|_\CIO\lesssim_{n} (\lambda_{q+1}\mu_q)^{-1}\leqslant\frac{1}{5}\label{est on nabla xi-Id},\\
	&\left\|\Id-(\nabla\xi_I)^{-1}\right\|_\CIO\lesssim_{n}(\lambda_{q+1}\mu_q)^{-1},\label{est on Id-nabla xi^-1}\\
	&\left\|(\nabla\xi_I)^{\top}-(\nabla\xi_I)^{-1}\right\|_\CIO\lesssim_{n}(\lambda_{q+1}\mu_q)^{-1},
	\label{est on nabla xi^T - nabla xi^-1} \\
	&\left\|\partial_{t}\xi_{I}+m_\ell/n\right\|_\CIO\lesssim_{n}(\lambda_{q+1}\mu_q)^{-1},\label{est on pt xi + m_l/n}\\
	&\left\|\partial_{t}^{r+1}\xi_{I}\right\|_{\CIN}\lesssim_{n, p, h, N}\ell^{-N-r},\label{est on pt xi}\\
	&\left\|\partial_{t}^r\nabla\xi_I\right\|_{\CIN}+\left\|\partial_{t}^r(\nabla\xi_I)^{-1}\right\|_{\CIN}\lesssim_{n, p, h, N}\ell^{-N-r},\label{est on pt nabla xi}\\	
	&\left\|\DTL ^s\nabla\xi_I\right\|_{\CIN}+\left\|\DTL ^s(\nabla\xi_I)^{-1}\right\|_{\CIN}\lesssim_{n, p,  N, M}\ell^{-N}(\lambda_{q}\delta_{q}^{\frac{1}{2}})^s,\label{est on Dtl nabla xi}\\		&\left\|\partial_{t}^{r}(\det(\nabla{\xi_I}))^{-m}\right\|_{\CIN}\lesssim_{n, p, h, N}\ell^{-N-r},\qquad \label{est on pt det nabla xi}\\
	&\left\|\DTL ^{s}(\det(\nabla{\xi_I}))^{-m}\right\|_{\CIN}\lesssim_{n, p, N, M}\ell^{-N}(\lambda_{q}\delta_{q}^{\frac{1}{2}})^s,\qquad \label{est on Dtl det nabla xi}
	\end{align}
	where $m=1,2,3.$ The implicit constants in the previous inequalities are independent of the index $I$.
\end{pp}
\begin{proof}
	First, we could find $\Lambda_3(\alpha,b,M,n)\geqslant\max\{\Lambda_1,\Lambda_2\}$ such that for any $\lambda_0\geqslant\Lambda_3$,
	\begin{align}
	\frac{1}{\lambda_{q+1}\mu_q}\leqslant \min\left\{\frac{\pi}{2\eta}, \frac{1}{10M^5}\right\}.\label{pp of Lambda_3}
	\end{align}
	Then, we have  $\tau_{q}\left\|\nabla  (m_\ell/n)\right\|_0\leqslant\frac{\eta}{10\pi\lambda_{q+1}\mu_q}\leqslant \frac{1}{20}$ and 
	\begin{equation}\label{est on nabla xi}
		\begin{aligned}
		\left\|\nabla\xi_I\right\|_{\CIN}\leqslant 1+ 2\tau_{q}C_N\left\|\nabla(m_\ell/n)\right\|_{N}\exp(2\tau_qC_N\left\|\nabla(m_\ell/n)\right\|_0)\lesssim_{n,N}1+\ell^{-N}M\tau_q\lambda_{q}\delta_{q}^{\frac{1}{2}}\lesssim_{n,N}\ell^{-N}.
		\end{aligned}
	\end{equation}
	Moreover, $\xi_I-x$ satisfies
	$$
		\left\lbrace 
		\begin{aligned}
			&\partial_t(\xi_I-x)+(m_\ell/n\cdot\nabla)(\xi_I-x)=-m_\ell/n,\\
			&\xi_I(t_u,x)=x.	
		\end{aligned}\right.
	$$
	\eqref{est on nabla xi-Id} follows from \eqref{geomoetry condition 1} and \eqref{transport equation 1},
	\begin{equation}\label{est on nabla xi-Id 2}
	\begin{aligned}
	\left\|\nabla\xi_I-\Id\right\|_\CIO\lesssim2\tau_q\left\|\nabla(m_\ell/n)\right\|_0\exp(2\tau_q\left\|\nabla(m_\ell/n)\right\|_0)\lesssim_{n}(\lambda_{q+1}\mu_q)^{-1}\leqslant \frac{1}{5}.
	\end{aligned}
\end{equation}
 	From \eqref{est on nabla xi} and \eqref{est on nabla xi-Id 2}, we can obtain \eqref{est on nabla xi and xi^-1}. Observe that
	\begin{align*}
	\partial_{t}\xi_{I}=-((m_\ell/n)\cdot\nabla)\xi_{I},\quad \partial_{t}^{r+1}\xi_{I}=-\sum_{r_0+r_1=r}(\partial_{t}^{r_0}(m_\ell/n)\cdot\nabla)\partial_{t}^{r_1}\xi_{I}.
	\end{align*}
	So we could use \eqref{est on pt m_l / n N} and \eqref{est on nabla xi-Id 2} to obtain 
	\begin{align*}
	\left\|\partial_{t}\xi_{I}+m_\ell/n\right\|_\CIO&\lesssim\left\|m_\ell/n\right\|_\CIO\left\|\Id-\nabla\xi_{I}\right\|_\CIO\lesssim_{n}(\lambda_{q+1}\mu_q)^{-1},\end{align*}
	and then
	\begin{align*}
	\left\|\partial_{t}\xi_{I}\right\|_{\CIN}&\lesssim_N\sum_{N_1+N_2=N}\left\|m_\ell/n\right\|_{N_1}\left\|\nabla\xi_{I}\right\|_{C^{0}\left(\cI_{u} ; C^{N_2}\left(\R^{3}\right)\right)}\lesssim_{n, p, h, N}\ell^{-N},&& N\geqslant 1,\\
	\left\|\partial_{t}^{r+1}\xi_{I}\right\|_{\CIN}&\lesssim_N\sum_{r_0+r_1=r}\sum_{N_0+N_1 = N}\left\|\partial_{t}^{r_0}(m_\ell/n)\right\|_{N_0}\left\|\partial_{t}^{r_1}\xi_{I}\right\|_{C^{0}\left(\cI_{u} ; C^{N_1+1}\left(\R^{3}\right)\right)}\lesssim_{n, p, h,  N}\ell^{-N-r},&& N\geqslant 0,
	\end{align*}
	for $0\leqslant r\leqslant 4$. \eqref{est on pt xi} follows from the above estimates.	Noting that
	\begin{align*}
	\partial_{t}(\nabla\xi_{I})^{-1}=-(\nabla\xi_{I})^{-1}\left(\partial_{t}\nabla\xi_{I}\right)(\nabla\xi_{I})^{-1},\quad\partial_{t}^{r+1}(\nabla\xi_{I})^{-1}=-\sum_{r_0+r_1+r_2=r}\partial_{t}^{r_0}(\nabla\xi_{I})^{-1}\partial_{t}^{r_1}\left(\partial_{t}\nabla\xi_{I}\right)\partial_{t}^{r_2}(\nabla\xi_{I})^{-1},
	\end{align*}
	and using \eqref{est on m_l}, \eqref{est on nabla xi and xi^-1}, and \eqref{est on pt xi},   we could obtain
	\begin{equation}\notag
		\begin{aligned}
		\left\|\partial_{t}(\nabla\xi_{I})^{-1}\right\|_{\CIN}&\leqslant\left\|(\nabla\xi_{I})^{-1}\left(\partial_{t}\nabla\xi_{I}\right)(\nabla\xi_{I})^{-1}\right\|_{\CIN}\lesssim_{n, p, h, N} \ell^{-N-1},
		\end{aligned}
	\end{equation}
and 
	\begin{equation}\notag
	\begin{aligned}
		&\quad\left\|\partial_{t}^{r+1}(\nabla\xi_{I})^{-1}\right\|_{\CIN}\lesssim\sum_{r_0+r_1+r_2=r}\left\|\partial_{t}^{r_0}(\nabla\xi_{I})^{-1}\partial_{t}^{r_1+1}\nabla\xi_{I}\partial_{t}^{r_2}(\nabla\xi_{I})^{-1}\right\|_{C^{0}\left(\cI_{u} ; C^{N}\left(\R^{3}\right)\right)}\lesssim_{n, p, h, N} \ell^{-N-r-1},
	\end{aligned}
\end{equation}
	for $0\leqslant r\leqslant 4$. Recalling \eqref{for and back}, we could know $(\nabla\xi_I)^{-1}(t,x)=\nabla\Phi_u(t,\xi_I(t,x);t_u),$
	and then
	$$
	\begin{aligned}
		\left\|\Id-(\nabla\xi_I)^{-1}\right\|_\CIO&=\left\|\Id-\nabla\Phi_u\right\|_\CIO\lesssim\tau_q\left\|\nabla(m_\ell/n)\right\|_0\lesssim_{n}(\lambda_{q+1}\mu_q)^{-1},\\
		\left\|\nabla\xi_I-(\nabla\xi_I)^{-1}\right\|_\CIO&\leqslant\left\|\Id-(\nabla\xi_I)^{-1}\right\|_\CIO+\left\|\Id-\nabla\xi_I\right\|_\CIO\lesssim_{n}(\lambda_{q+1}\mu_q)^{-1}.
	\end{aligned}
	$$
    \eqref{est on nabla xi^T - nabla xi^-1} follows. To get \eqref{est on Dtl nabla xi}, we use the following equalities
	\begin{equation}\label{Dtl nabla xi}
		\begin{aligned}
		&\DTL \nabla \xi_{I}=-\left(\nabla \xi_{I}\right)\nabla(m_\ell/n),&& \DTL^{2} \nabla \xi_{I}=\left(\nabla \xi_{I}\right)\left(\nabla(m_\ell/n)\right)^{2}-\left(\nabla \xi_{I}\right) \DTL \nabla(m_\ell/n),\\
		&\DTL(\nabla \xi_{I})^{-1}=\nabla(m_\ell/n)(\nabla \xi_{I})^{-1}, && \DTL^{2}(\nabla \xi_{I})^{-1}=\DTL \nabla(m_\ell/n)(\nabla \xi_{I})^{-1}+\left(\nabla(m_\ell/n)\right)^{2}(\nabla \xi_{I})^{-1}.
		\end{aligned}
	\end{equation}
	Combining it with \eqref{est on pt m_l / n N} and \eqref{est on nabla xi and xi^-1}, we have for $s=0,1,2$,
	$$
	\left\|\DTL^{s} \nabla \xi_{I}\right\|_{C^{0}\left(\cI_{u} ; C^{N}\left(\R^{3}\right)\right)} \lesssim_{n, p, N, M} \ell^{-N}\left(\lambda_{q} \delta_{q}^{\frac{1}{2}}\right)^{s},\qquad
	\left\|\DTL^{s}(\nabla \xi_{I})^{-1}\right\|_{C^{0}\left(\cI_{u} ; C^{N}\left(\R^{3}\right)\right)} \lesssim_{n, p, N, M} \ell^{-N}\left(\lambda_{q} \delta_{q}^{\frac{1}{2}}\right)^{s}.
	$$
	Finally, we give the following formulas,
	$$
	\begin{aligned}
		\partial_i(\det(\nabla{\xi_I}))^{-m}
		&=-m(\det(\nabla{\xi_I}))^{-m}\tr((\nabla\xi_I)^{-1}\partial_i(\nabla{\xi_I})),\\
		\partial_t(\det(\nabla{\xi_I}))^{-m}
		&=-m(\det(\nabla{\xi_I}))^{-m}\tr((\nabla\xi_I)^{-1}\partial_t(\nabla{\xi_I})),\\
		\DTL (\det(\nabla{\xi_I}))^{-m}&=-m(\det(\nabla{\xi_I}))^{-m}\tr((\nabla\xi_I)^{-1}\DTL (\nabla{\xi_I})),
	\end{aligned}
	$$
	from which \eqref{est on pt det nabla xi} and \eqref{est on Dtl det nabla xi} follow.
\end{proof}
\subsection{Estimates on mixed derivatives}\label{Estimates on mixed derivatives}
Here, we give estimates on mixed derivatives which will be used in estimates on the perturbation. We denote $\left\|\cdot\right\|_N=\left\|\cdot\right\|_{C^0(\cI_\ell^{q};C^N(\T^3))}$.
\begin{pp}
  For any $0<\alpha<\frac{1}{3}$, let the parameters $\bar{b}_3(\alpha)$ and $\Lambda_3$ be as in the statement of Proposition \ref{est on backward flow}. For any $b\in(1,\overline{b}_3(\alpha))$ and $\lambda_0\geqslant\Lambda_3$, we have the following properties:
	\begin{align}
	&\left\|\DTL^s\nabla^r\DTL^km_{\ell}\right\|_{N} \lesssim_{n, p, N, r, M} \ell^{2-s-N-r-k} (\lambda_{q}\delta_{q}^{\frac{1}{2}})^2, \label{est on mixed m_l 1}\\
	&\left\|\DTL^{s}\nabla^r\DTL^kR_{\ell}\right\|_{N}+\delta_{q+1}^{-\frac{1}{2}}\left\|\DTL^{s} \nabla^r\DTL^k\varphi_{\ell}\right\|_{N} \lesssim_{n, p, N, r, M}\ell_{t}^{-s-k}\ell^{-N-r} \lambda_{q}^{-3 \gamma} \delta_{q+1},\label{est on mixed error 1} \\
	&\left\|\DTL^{s}\nabla^r\DTL^{k}\nabla\xi_{I}\right\|_{\CIN}+\left\|\DTL^{s}\nabla^{r}\DTL^{k}(\nabla\xi_{I})^{-1}\right\|_{\CIN} \lesssim_{n, p, N, r, M} \ell^{-N-r-k}(\lambda_{q}\delta_{q}^{\frac{1}{2}})^s,\label{est on mixed nabla xi 1}\\
	&\left\|\DTL^{s}\nabla^r\xi_{I}\right\|_{\CIN}\lesssim_{n, p, N, r, M} \ell^{1-N-r}(\lambda_{q}\delta_{q}^{\frac{1}{2}})^s, \label{est on mixed xi 1}
	\end{align}
	for $s=1,2;k=0,1,2;r\geqslant1;N\geqslant 0$, and
	\begin{align}
	&\left\|\DTL^s\partial_{t}^k\nabla^im_{\ell}\right\|_{N} \lesssim_{n, p, N, M} \ell^{2-s-N-i-k} (\lambda_{q}\delta_{q}^{\frac{1}{2}})^2, \label{est on mixed m_l 2}\\
	&\left\|\DTL^{s}\partial_{t}^k\nabla^{i} R_{\ell}\right\|_{N}+\delta_{q+1}^{-\frac{1}{2}}\left\|\DTL^{s} \partial_{t}^k\nabla^{i}\varphi_{\ell}\right\|_{N} \lesssim_{n, p, N, M} \ell_{t}^{-s}\ell^{-N-i}(\ell_t^{-1}+\ell^{-1})^{k} \lambda_{q}^{-3 \gamma} \delta_{q+1}, \label{est on mixed error 2}\\
	&\left\|\DTL^{s}\partial_{t}^k\nabla^i\nabla\xi_{I}\right\|_{\CIN}+\left\|\DTL^{s}\partial_{t}^k\nabla^{i}(\nabla\xi_{I})^{-1}\right\|_{\CIN} \lesssim_{n, p, N, M} \ell^{-N-k-i}(\lambda_{q}\delta_{q}^{\frac{1}{2}})^s, \label{est on mixed nabla xi 2}\\
	&\left\|\DTL^{s}\partial_{t}^k\nabla^i\xi_{I}\right\|_{\CIN}\lesssim_{n, p, N, M} \ell^{1-N-k-i}(\lambda_{q}\delta_{q}^{\frac{1}{2}})^s, \label{est on mixed xi 2}
\end{align}
	for $s=1,2;k=1,2;i=0,1;N\geqslant 0$. Moreover, we have for $N\geqslant0$ and $0\leqslant r\leqslant 4$,
	\begin{align}
		\left\|\partial_{t}^rR_{\ell}\right\|_{N}+\delta_{q+1}^{-\frac{1}{2}}\left\| \partial_{t}^r\varphi_{\ell}\right\|_{N} \lesssim_{n, p, h, N}\ell^{-N}(\ell_t^{-1}+\ell^{-1})^{r} \lambda_{q}^{-3 \gamma} \delta_{q+1}, \label{est on pt error}
	\end{align}
	 where the implicit constant is independent of $M$.
\end{pp}
\begin{proof}
	$\left\|\cdot\right\|_N$ in the following inequality can be replaced by $\|\cdot\|_{\CIN}$, we only prove one case for convenience. Recall the estimate on $\DTL m_\ell$ in \eqref{est on Dtl s m_l} and apply Lemma \ref{Estimtates for mixed derivatives} to $m_\ell,\ R_\ell,\ \varphi_\ell,\ \nabla\xi_{I},\ (\nabla\xi_{I})^{-1}$, and $\xi_{I}$. Then, we could get for $s=1,2;r\geqslant 1;k\geqslant 0;N\geqslant 0,$
	\begin{equation}\notag
	\begin{aligned}
	\left\|\DTL\nabla^r\DTL^km_{\ell}\right\|_{N}&\lesssim_{N,r}\left\|\DTL^{k+1} m_\ell\right\|_{N+r}+\sum_{N_0+N_1=N+r-1}\left\|\DTL^km_\ell\right\|_{N_0+1}\left\|m_\ell/n\right\|_{N_1+1}\\
	&\lesssim_{n, p, N, r, k, M}  \ell^{1-N-r-k}(\lambda_{q}\delta_{q}^{\frac{1}{2}})^2,\\
	\left\|\DTL^2\nabla^{r}\DTL^k m_\ell\right\|_{N}
	&\lesssim_{N,r}\left\|\DTL^{k+2} m_\ell\right\|_{N+r}+
	\sum_{N_0+N_1=N+r-1}\left\| \DTL^{k+1} m_\ell\right\|_{N_0+1}\left\|m_\ell/n\right\|_{N_1+1}\\
	&\quad+\sum_{N_0+N_1=N+r-1}\left\|\DTL^km_\ell\right\|_{N_0+1}\left\|\DTL (m_\ell/n)\right\|_{N_1+1}\\
	&\quad+\sum_{N_0+N_1+N_2=N+r-1}\left\|\DTL^km_\ell\right\|_{N_0+1}\left\|m_\ell/n\right\|_{N_1+1}\left\|m_\ell/n\right\|_{N_2+1}\\
	&\lesssim_{n, p, N, r, k, M}\ell^{-N-r-k}(\lambda_{q}\delta_{q}^{\frac{1}{2}})^2.
	\end{aligned}
	\end{equation}
Similarly, we have 
	\begin{align*}
	&\left\|\DTL^s\nabla^r\DTL^{k}R_\ell\right\|_{N}+\delta_{q+1}^{-\frac{1}{2}}\left\|\DTL^s\nabla^r\DTL^k\varphi_\ell\right\|_{N}\lesssim_{n, p, N, r, k, M} \ell_t^{-s-k}\ell^{-N-r}\lambda_{q}^{-3 \gamma} \delta_{q+1},\\
	&\left\|\DTL^s\nabla^r\nabla\xi_{I}\right\|_{\CIN}+\left\|\DTL^s\nabla^r(\nabla\xi_{I})^{-1}\right\|_{\CIN}\lesssim_{n, p, N, r, M} \ell^{-N-r}(\lambda_{q}\delta_{q}^{\frac{1}{2}})^s,\\
	&\left\|\DTL^s\nabla^r\xi_{I}\right\|_{\CIN}\lesssim_{n, p, N, r, M} \ell^{1-N-r}(\lambda_{q}\delta_{q}^{\frac{1}{2}})^s.
	\end{align*}
Moreover, from \eqref{Dtl nabla xi}, it is easy to achieve
	\begin{align*}
	\left\|\DTL^r\nabla\xi_{I}\right\|_{\CIN}
	&\leqslant\left\|\DTL^{r-1}(\nabla\xi_{I}\nabla(m_\ell/n))\right\|_{C^{0}\left(\cI_{u} ; C^{N}\left(\R^{3}\right)\right)}\\
	&\lesssim\sum_{N_0+N_1=N}\sum_{r_0+r_1=r-1}\left\|\DTL^{r_0} \nabla\xi_{I}\right\|_{C^{0}\left(\cI_{u} ; C^{N_0}\left(\R^{3}\right)\right)}\left\|\DTL^{r_1}\nabla (m_\ell/n)\right\|_{N_1}\\
	&\lesssim_{n, p, N, r, M}\ell^{2-N-r}(\lambda_{q}\delta_{q}^{\frac{1}{2}})^2,\\
	\left\|\DTL^r(\nabla\xi_{I})^{-1}\right\|_{\CIN}
	&\leqslant\left\|\DTL^{r-1}(\nabla(m_\ell/n)(\nabla\xi_{I})^{-1})\right\|_{C^{0}\left(\cI_{u} ; C^{N}\left(\R^{3}\right)\right)}\\
	&\lesssim\sum_{N_0+N_1=N}\sum_{r_0+r_1=r-1}\left\|\DTL^{r_0} (\nabla\xi_{I})^{-1}\right\|_{C^{0}\left(\cI_{u} ; C^{N_0}\left(\R^{3}\right)\right)}\left\|\DTL^{r_1}\nabla (m_\ell/n)\right\|_{N_1}\\
	&\lesssim_{n, p, N, r, M}\ell^{2-N-r}(\lambda_{q}\delta_{q}^{\frac{1}{2}})^2,
	\end{align*}
	for $r\geqslant 2$. Then, we could similarly obtain for $k\geqslant 1$,
	\begin{align*}
	\left\|\DTL\nabla^r \DTL^k\nabla\xi_{I}\right\|_{\CIN}
	&\lesssim_{N,r}\left\|\DTL^{k+1} \nabla\xi_{I}\right\|_{C^{0}\left(\cI_{u} ; C^{N+r}\left(\R^{3}\right)\right)}+\sum_{N_0+N_1=N+r-1}\left\|\DTL^k \nabla\xi_{I}\right\|_{C^{0}\left(\cI_{u} ; C^{N_0+1}\left(\R^{3}\right)\right)}\left\|m_\ell/n\right\|_{N_1+1}\\
	&\lesssim_{n, p, N, r, k, M}\ell^{1-N-r-k}(\lambda_{q}\delta_{q}^{\frac{1}{2}})^2,\\
	\left\|\DTL^2\nabla^r \DTL^k\nabla\xi_{I}\right\|_{\CIN}&\lesssim_{N,r}\left\|\DTL^{k+2} \nabla\xi_{I}\right\|_{C^{0}\left(\cI_{u} ; C^{N+r}\left(\R^{3}\right)\right)}+\sum_{N_0+N_1=N+r-1}\left\|\DTL^{k+1} \nabla\xi_{I}\right\|_{C^{0}\left(\cI_{u} ; C^{N_0+1}\left(\R^{3}\right)\right)}\left\|m_\ell/n\right\|_{N_1+1}\\
	&\quad+\sum_{N_0+N_1+N_2=N+r-1}\left\|\DTL^k \nabla\xi_{I}\right\|_{C^{0}\left(\cI_{u} ; C^{N_0+1}\left(\R^{3}\right)\right)}\left\|m_\ell/n\right\|_{N_1+1}\left\|m_\ell/n\right\|_{N_2+1}\\
	&\quad+\sum_{N_0+N_1=N+r-1}\left\|\DTL^k \nabla\xi_{I}\right\|_{C^{0}\left(\cI_{u} ; C^{N_0+1}\left(\R^{3}\right)\right)}\left\|\DTL\nabla (m_\ell/n)\right\|_{N_1}\\
	&\lesssim_{n, p, N, r, k, M}\ell^{-N-r-k}(\lambda_{q}\delta_{q}^{\frac{1}{2}})^2.
	\end{align*}
The estimate on $(\nabla\xi)^{-1}$ is similar, and we have
\begin{align*}
\left\|\DTL^s\nabla^r \DTL^k\nabla\xi_{I}\right\|_{\CIN}+\left\|\DTL^s\nabla^r \DTL^k(\nabla\xi_{I})^{-1}\right\|_{\CIN}\lesssim_{n, p, N, r, k, M}\ell^{-N-r-k}(\lambda_{q}\delta_{q}^{\frac{1}{2}})^s.
\end{align*}
Until now, \eqref{est on mixed m_l 1}--\eqref{est on mixed xi 1} have been proved. Observing that
$\partial_{t}F=\DTL F-(m_\ell/n)\cdot\nabla F,$ we could get
	\begin{align}
		\left\|\DTL^s\partial_{t}F\right\|_N&\lesssim_{n,N}\left\|\DTL^{s+1}F\right\|_N+\sum_{N_0+N_1=N}\sum_{s_1+s_2=s}\left\|\DTL^{s_1}(m_\ell/n)\right\|_{N_0}\left\|D_{t,\ell}^{s_2}\nabla F\right\|_{N_1}.\label{est on Dtl^s pt F}
	\end{align}
	By applying it to $m_\ell,\ R_\ell,\ \varphi_\ell,\ \nabla\xi_{I},\ (\nabla\xi_{I})^{-1},$ and $\xi_{I}$, we could get
	\begin{equation}\label{est on Dtl pt item}
		\begin{aligned}
		&\left\|\DTL^s\partial_{t}m_\ell\right\|_N\lesssim_{n, p, N, M}\ell^{1-N-s}(\lambda_{q}\delta_{q}^{\frac{1}{2}})^2,\\ &\left\|\DTL^s\partial_{t}R_\ell\right\|_N+\delta_{q+1}^{-\frac{1}{2}}\left\|\DTL^s\partial_{t}\varphi_\ell\right\|_N\lesssim_{n, p, N, M}\ell_t^{-s}\ell^{-N}(\ell_t^{-1}+\ell^{-1})\lambda_{q}^{-3\gamma}\delta_{q+1},\\
		&\left\|\DTL^{s}\partial_{t}\nabla\xi_{I}\right\|_{\CIN}+\left\|\DTL^{s}\partial_{t}(\nabla\xi_{I})^{-1}\right\|_{\CIN} \lesssim_{n, p, N, M} \ell^{-N-1}(\lambda_{q}\delta_{q}^{\frac{1}{2}})^s,\\
		&\left\|\DTL^s\partial_{t}\xi_{I}\right\|_{\CIN}\lesssim_{n, p, N, M} \ell^{-N}(\lambda_{q}\delta_{q}^{\frac{1}{2}})^s.
	\end{aligned}
	\end{equation}
	Next, we can similarly use Lemma \ref{Estimtates for mixed derivatives} to get
	\begin{align*}
	\left\|\DTL\nabla\partial_{t}F\right\|_{N}&\lesssim_{N}\left\|\DTL\partial_{t}F\right\|_{N+1}+\sum_{N_0+N_1=N}\left\|\partial_{t}F\right\|_{N_0+1}\left\|m_\ell/n\right\|_{N_1+1},\\
	\left\|\DTL^2\nabla\partial_{t}F\right\|_{N}
	&\lesssim_{N}\left\|\DTL^{2} \partial_{t}F\right\|_{N+1}+
	\sum_{N_0+N_1=N}\left\| \DTL \partial_{t}F\right\|_{N_0+1}\left\|m_\ell/n\right\|_{N_1+1}\\
	&\quad+\sum_{N_0+N_1=N}\left\|\partial_{t}F\right\|_{N_0+1}\left\|\DTL (m_\ell/n)\right\|_{N_1+1}\\
	&\quad+\sum_{N_0+N_1+N_2=N}\left\|\partial_{t}F\right\|_{N_0+1}\left\|m_\ell/n\right\|_{N_1+1}\left\|m_\ell/n\right\|_{N_2+1}.
	\end{align*}
	Then, we apply it to $m_\ell,\ R_\ell,\ \varphi_\ell,\ \nabla\xi_{I}$, and $(\nabla\xi_{I})^{-1}$, and combine it with \eqref{est on mixed m_l 1}--\eqref{est on mixed xi 1} and \eqref{est on Dtl pt item}. It follows that
	\begin{equation}\label{est on Dtl pt nabla item}
		\begin{aligned}
		&\left\|\DTL^s\partial_{t}\nabla m_\ell\right\|_N\lesssim_{n, p, N, M}\ell^{-N-s}(\lambda_{q}\delta_{q}^{\frac{1}{2}})^2,\\
		&\left\|\DTL^s\partial_{t}\nabla R_\ell\right\|_N+\delta_{q}^{\frac{1}{2}}\left\|\DTL^s\partial_{t}\nabla \varphi_\ell\right\|_N\lesssim_{n, p, N, M}\ell_{t}^{-s}\ell^{-N-1}(\ell_t^{-1}+\ell^{-1})\lambda_{q}^{-3\gamma}\delta_{q+1},\\
		&\left\|\DTL^s\partial_{t}\nabla \nabla\xi_{I}\right\|_{\CIN}+\left\|\DTL^s\partial_{t}\nabla (\nabla\xi_{I})^{-1}\right\|_{\CIN}\lesssim_{n, p, N, M}\ell^{-N-2}(\lambda_{q}\delta_{q}^{\frac{1}{2}})^s.
	\end{aligned}
	\end{equation}
	Observing that $
	 \partial_{t}^2F=\partial_{t}\DTL F-\partial_{t}((m_\ell/n)\cdot\nabla F),$
	we can immediately have for $s\geqslant 1$,
	\begin{align*}
		\left\|\DTL^s\partial_{t}^2 F\right\|_N&\lesssim_{N}\left\|\DTL^{s}\partial_{t}\DTL F\right\|_N+\sum_{N_0+N_1=N}\sum_{s_1+s_2=s}\left\|\DTL^{s_1}\partial_{t}(m_\ell/n)\right\|_{N_0}\left\|\DTL^{s_2}\nabla F\right\|_{N_1}\\
		&\quad+\sum_{N_0+N_1=N}\sum_{s_1+s_2=s}\left\|\DTL^{s_1}(m_\ell/n)\right\|_{N_0}\left\|\DTL^{s_2}\partial_{t}\nabla F\right\|_{N_1}\\
		&\lesssim_{N}\left\|\DTL^{s+2}F\right\|_N+\sum_{N_0+N_1=N}\sum_{s_1+s_2=s}\left\|\DTL^{s_1}(m_\ell/n)\right\|_{N_0}\left\|\DTL^{s_2}\nabla\DTL F\right\|_{N_1}\\
		&\quad+\sum_{N_0+N_1=N}\sum_{s_1+s_2=s}\left\|\DTL^{s_1}\partial_{t} (m_\ell/n)\right\|_{N_0}\left\|\DTL^{s_2}\nabla F\right\|_{N_1}+\sum_{N_0+N_1=N}\sum_{s_1+s_2=s}\left\|\DTL^{s_1} (m_\ell/n)\right\|_{N_0}\left\|\DTL^{s_2}\partial_{t}\nabla F\right\|_{N_1},
	\end{align*}
	where we have used \eqref{est on Dtl^s pt F} to get the second inequality. Applying it to $m_\ell,\ R_\ell,\ \varphi_\ell,\ \nabla\xi_{I},\ (\nabla\xi_{I})^{-1},$ and $\xi_{I}$, we have
	\begin{align*}
	&\left\|\DTL^s\partial_{t}^2 m_\ell\right\|_N\lesssim_{n, p, N, M}\ell^{-N-s}(\lambda_{q}\delta_{q}^{\frac{1}{2}})^2,\\
	&\left\|\DTL^s\partial_{t}^2 R_\ell\right\|_N+\delta_{q+1}^{-\frac{1}{2}}\left\|\DTL^s\partial_{t}^2\varphi_\ell\right\|_N\lesssim_{n, p, N, M}\ell_t^{-s}\ell^{-N}(\ell_t^{-1}+\ell^{-1})^{2}\lambda_{q}^{-3\gamma}\delta_{q+1},\\
	&\left\|\DTL^{s}\partial_{t}^2\nabla\xi_{I}\right\|_{\CIN}+\left\|\DTL^{s}\partial_{t}^2(\nabla\xi_{I})^{-1}\right\|_{\CIN} \lesssim_{n, p, N, M} \ell^{-N-2}(\lambda_{q}\delta_{q}^{\frac{1}{2}})^s,\\
	&\left\|\DTL^s\partial_{t}^2\xi_{I}\right\|_{\CIN}\lesssim_{n, p, N, M} \ell^{-N-1}(\lambda_{q}\delta_{q}^{\frac{1}{2}})^s,
	\end{align*}
	where we have used \eqref{est on mixed m_l 1}--\eqref{est on mixed xi 1} and \eqref{est on Dtl pt nabla item}. Especially, when $s=0$ and $0\leqslant r\leqslant 4$, we can get
	\begin{align*}
		\left\|\partial_{t}^{r+1}F\right\|_N&\lesssim_{N}\left\|\partial_{t}^r\DTL F\right\|_N+\sum_{r_0+r_1=r}\sum_{N_0+N_1=N}\left\|\partial_{t}^{r_0}(m_\ell/n)\right\|_{N_0}\left\|\partial_{t}^{r_1}\nabla F\right\|_{N_1}\\
		&\lesssim_{N}\left\|\partial_{t}^{r-1}\DTL^2 F\right\|_N+\sum_{r_0+r_1=r-1}\sum_{N_0+N_1=N}\left\|\partial_{t}^{r_0}(m_\ell/n)\right\|_{N_0}\left\|\partial_{t}^{r_1}\nabla \DTL F\right\|_{N_1}\\
		&\quad+\sum_{r_0+r_1=r}\sum_{N_0+N_1=N}\left\|\partial_{t}^{r_0}(m_\ell/n)\right\|_{N_0}\left\|\partial_{t}^{r_1}\nabla F\right\|_{N_1}\\
		&\lesssim_{N}\left\|\DTL^{r+1} F\right\|_N+\sum_{r_0+r_1+r_2=r}\sum_{N_0+N_1=N}\left\|\partial_{t}^{r_0}(m_\ell/n)\right\|_{N_0}\left\|\partial_{t}^{r_1}\nabla \DTL^{r_2} F\right\|_{N_1}.
	\end{align*}
	By applying it to $R_q$ and $\varphi_q$, and combining it with \eqref{est on Dtl error 0}--\eqref{est on pt m_l / n N},  \eqref{relation}, and \eqref{pp of Lambda_3}, we could immediately obtain \eqref{est on pt error}. Moreover, the implicit constants can be chosen independent of $M$. We are now ready to give the estimate on $\DTL^s\partial_{t}^2\nabla F$. Observe that
	\begin{align*}
	\left\|\DTL \nabla\partial_{t}^2 F\right\|_{N}
	&\lesssim_{N}\left\|\DTL\partial_{t}^2 F\right\|_{N+1}+
	\sum_{N_0+N_1=N}\left\|\partial_{t}^2F\right\|_{N_0+1}\left\|m_\ell/n\right\|_{N_1+1},\\
	\left\|\DTL^2\nabla\partial_{t}^2 F\right\|_{N}
	&\lesssim_{N}\left\|\DTL^{2}\partial_{t}^2 F\right\|_{N+1}+
	\sum_{N_0+N_1=N}\left\| \DTL\partial_{t}^2 F\right\|_{N_0+1}\left\|m_\ell/n\right\|_{N_1+1}\nonumber\\
	&\quad+\sum_{N_0+N_1=N}\left\|\partial_{t}^2F\right\|_{N_0+1}\left\|\DTL(m_\ell/n)\right\|_{N_1+1}+\sum_{N_0+N_1+N_2=N}\left\|\partial_{t}^2F\right\|_{N_0+1}\left\|m_\ell/n\right\|_{N_1+1}\left\|m_\ell/n\right\|_{N_2+1}.
	\end{align*}
	We could apply it to $m_\ell,\ R_\ell,\ \varphi_\ell,\ \nabla\xi_{I}$, and $(\nabla\xi_{I})^{-1}$ to get
	\begin{align*}
		&\left\|\DTL^s\partial_{t}^2 \nabla m_\ell\right\|_N\lesssim_{n, p, N, M}\ell^{-N-s-1}(\lambda_{q}\delta_{q}^{\frac{1}{2}})^2,\\
		&\left\|\DTL^s\partial_{t}^2\nabla R_\ell\right\|_N+\delta_{q+1}^{-\frac{1}{2}}\left\|\DTL^s\partial_{t}^2\nabla\varphi_\ell\right\|_N\lesssim_{n, p, N, M}\ell_t^{-s}\ell^{-N-1}(\ell_t^{-1}+\ell^{-1})^{2}\lambda_{q}^{-3\gamma}\delta_{q+1},\\
		&\left\|\DTL^{s}\partial_{t}^2\nabla\nabla\xi_{I}\right\|_{\CIN}+\left\|\DTL^{s}\partial_{t}^2\nabla(\nabla\xi_{I})^{-1}\right\|_{\CIN} \lesssim_{n, p, N, M} \ell^{-N-3}(\lambda_{q}\delta_{q}^{\frac{1}{2}})^s.\qedhere
	\end{align*}
\end{proof}
\subsection{New building blocks}
In this part, we will use the Mikado flows to construct new building blocks of the solution of the Euler-Maxwell equations. For each $I=(u,\upsilon,f)$, we could choose proper $z_I$ satisfying
\begin{equation}
	z_I=z_{u,\upsilon}+\tilde{x}_f\in\R,\qquad z_{u,\upsilon}=z_{u,\upsilon^\prime},\qquad if\ \mu(\upsilon-\upsilon^\prime)\in2\pi\Z^3,\label{z periodic}
\end{equation}
so that $\psi_{f_I}(\cdot-z_I)$ and $\Psi_{f_I}(\cdot-z_I)$ are periodic functions. For convenience, we denote $\psi_I$ for $\psi_{f_I}(\cdot-z_I)$, $\Psi_I$ for $\Psi_{f_I}(\cdot-z_I)$, and $\tilde{f}_I$ for $(\nabla\xi_I)^{-1}f_I$. Moreover, we could assume that
\begin{equation}
	\supp(\psi_I)\subset B(l_{f_I,z_I},\frac{\eta}{10}):=\left\lbrace x\in\R^3:|x-y|<\frac{\eta}{10} \ \mathrm{for\ some}\ y\in l_{f_I,z_I}\right\rbrace ,\qquad \forall f\in\mathcal{F}^j,\label{supp}
\end{equation}
where $l_{f_I,z_I}:=\left\lbrace\lambda f_I+z_I:\lambda\in \R\right\rbrace+\Z^3$
and $\eta$ is a geometric constant as chosen in \cite[Proposition 3.4]{GK22}.  Then, we can guarantee that different Mikado flows  don't intersect with each other. 
\begin{pp}\cite[Proposition 3.4]{GK22}.\label{def of eta}
	There is a constant $\eta$=$\eta(\mathcal{F})$ in \eqref{supp} such that it allows a choice of the shifts
	$z_I=z_{u,\upsilon}+\tilde{x}_f$ which ensure that for each $(\mu_q,\tau_q,\lambda_{q+1})$, the condition $\supp(\theta_{I}\chi_I(\xi_{I})\psi_{I})\bigcap \supp(\theta_{J}\chi_J(\xi_{J})\psi_{J})=\emptyset$ holds for every $I\neq J$ and that \eqref{z periodic} holds for every $u,\upsilon$ and $v^\prime$.
\end{pp}
It should be noted that the proof for the proposition requires the choice of parameters satisfying the relations
\begin{equation}
	\mu_q^{-1}\ll\lambda_{q+1}\in\N,\quad \tau_q\left\|\nabla(m_\ell/n)\right\|_0\leqslant\frac{1}{10},\quad \mu_q\tau_q\left\|\nabla(m_\ell/n)\right\|_0\leqslant\frac{\eta}{10\pi\lambda_{q+1}},\label{geomoetry condition 2}
\end{equation}
where $\eta$ is a positive constant determined by $\mathcal{F}=\underset{j\in\Z^3}{\bigcap}\mathcal{F}^j$, which has finite cardinality.\par
Noticing that $\psi_I$ is a smooth function on $\T^3$ with zero-mean, we could represent it as Fourier series:
\begin{equation}
	\psi_I(x)=\sumkO \overset{\circ}{b}_{I,k}e^{ik\cdot x},\qquad \Psi_I(x)=-\sumkO \frac{\overset{\circ}{b}_{I,k}}{|k|^2}e^{ik\cdot x}.\label{def of psi_I}
\end{equation}
Since $\psi_I\in C^\infty(\T^3)$, we have
\begin{equation}\label{est on sum of cof 1}
	\sum_{k\in\Z^3}|k|^{\tilde{n}_0+2}|\overset{\circ}{b}_{I,k}|\lesssim 1,\qquad \tilde{n}_0=\left\lceil\frac{2b(2+\alpha)}{(b-1)(1-\alpha)}\right\rceil.
\end{equation} 
From the definition of Mikado flow in \eqref{def of Mikado flow}, it is easy to get $f_I\cdot\nabla\psi_I=0$ so that $\overset{\circ}{b}_{I,k}(f_I\cdot k)=0$.  Here, we introduce
\begin{equation}
\nabla\times((\nabla\xi_I)^{\top}U(\xi_I))=\text{cof}(\nabla\xi_I)^{\top}(\nabla\times U)(\xi_I)=\mathrm{det}(\nabla\xi_I)(\nabla\xi_I)^{-1}(\nabla\times U)(\xi_I),\label{potential formula}
\end{equation}
which can be proved as following
\begin{align*}
(\nabla\times((\nabla\xi_I)^{\top}U(\xi_I)))_i&=\varepsilon_{ijk}\partial_j(\partial_k(\xi_I)_rU_r)=\varepsilon_{ijk}\partial_j\partial_k(\xi_I)_rU_r+\varepsilon_{ijk}(\partial_k(\xi_I)_r\partial_sU_r\partial_j(\xi_I)_s)\\
&=\partial_sU_r\varepsilon_{ijk}\partial_k(\xi_I)_r\partial_j(\xi_I)_s=\varepsilon_{qrs}\partial_sU_r\text{cof}(\nabla\xi_I)^{\top}_{q,i}=(\nabla\times U)_q\text{cof}(\nabla\xi_I)^{\top}_{qi}\\
&=(\text{cof}(\nabla\xi_I)^{\top}(\nabla\times U)(\xi_I))_i=(\mathrm{det}(\nabla\xi_I)(\nabla\xi_I)^{-1}(\nabla\times U)(\xi_I))_i.
\end{align*}
So Mikado flow can be written in terms of potential
\begin{equation}
\nabla_{\xi_I}\times\left(\frac{ik\times f_I}{|k|^2}e^{i\lambda_{q+1}k\cdot\xi_I}\right)=-i\lambda_{q+1}\frac{ik\times f_I}{|k|^2}\times ke^{i\lambda_{q+1}k\cdot\xi_I}=\lambda_{q+1}f_Ie^{i\lambda_{q+1}k\cdot\xi_I}.\label{Mikado flow potential}
\end{equation}
Now, we would present a crucial lemma that outlines how to construct the building blocks and provide the corresponding estimates.
\begin{lm}\label{New building blocks}
For any $c<l<d$, $\tlm\gg1$, $f\in\Z^3$, and smooth functions $a(t,x),\upsilon(t,x)\in C^\infty([c,d]\times\T^3)$, there exists a non-trivial tuple $(\cm_{k}(f,a,\upsilon,\tlm),\cE_k(f,a,\upsilon,\tlm),\cB_k(f,a,\upsilon,\tlm))$ satisfying
\begin{equation}
\left\{\begin{aligned}
	&\partial_t\cE_k-\nabla\times \cB_k=\cm_k,\\
	&\partial_t\cB_k+\nabla\times \cE_k=0,\\
	&\Div\cE_k=\Div\cB_k=\Div\cm_k=0,
\end{aligned}\right. 
\end{equation}
and $\supp\cm_k,\supp\cE_k,\supp\cB_k\subset\supp a$. Moreover, the main part of $\cm_k(f,a,\upsilon,\tlm)$, denoted by $\cm_{p,k}(f,a,\upsilon,\tlm,t_0,x_0)$, has the following form:
\begin{equation}\label{def of cm_pk}
	\begin{aligned}
		\cm_{p,k}(f,a,\upsilon,\tlm,t_0,x_0)&:=a(t,x)
		\left(1-\frac{(k\cdot\upsilon(t_0 ,x_0))^2}{|k|^2}\right)(\nabla\xi)^{-1}fe^{i\tlm k\cdot\xi},
	\end{aligned}
\end{equation}
where $(t_0,x_0)\in [c,d]\times\T^3$, and $\xi$ is the backward flow for $\upsilon$,
\begin{equation}\notag
	\left\{\begin{aligned}
		&\partial_t\xi(t,x)+\upsilon(t,x)\cdot\nabla\xi(t,x)=0,\\
		&\xi(l,x)=x.
	\end{aligned}\right.
\end{equation}
Let $\|\cdot\|_N=\|\cdot\|_{C^0([c,d];C^N(\T^3))}$, if we have for $N\geqslant 0;0\leqslant r\leqslant 4;s=0,1,2;k=0,1,2;i=0,1,$
\begin{equation}\label{assume in building blocks}
\begin{aligned}
&\|\partial_t^ra\|_N\lesssim\mu^{-N-r},&&\|(\partial_t+\upsilon\cdot\nabla)^s\partial_t^k\nabla ^ia\|_N+\|(\partial_t+\upsilon\cdot\nabla)^s\nabla ^{k+i}a\|_N\lesssim\mu^{-N-k-i}\tau^{-s},\\
&\|\partial_t^{r+1}\xi\|_N\lesssim\ell_*^{-N-r},&&\|(\partial_t+\upsilon\cdot\nabla)^s\partial_t^k\xi\|_N+\|(\partial_t+\upsilon\cdot\nabla)^s\nabla ^{k}\xi\|_N\lesssim\ell_*^{1-N-k}\tau^{-s},\\
&\|\partial_t^r(\nabla\xi)\|_N\lesssim\ell_*^{-N-r},&&\|(\partial_t+\upsilon\cdot\nabla)^s\partial_t^k\nabla ^i(\nabla\xi)\|_N+\|(\partial_t+\upsilon\cdot\nabla)^s\nabla ^{k+i}(\nabla\xi)\|_N\lesssim\ell_*^{-N-k-i}\tau^{-s},\\
&\|\partial_t^r(\nabla\xi)^{-1}\|_N\lesssim\ell_*^{-N-r},&&\|(\partial_t+\upsilon\cdot\nabla)^s\partial_t^k\nabla ^i(\nabla\xi)^{-1}\|_N+\|(\partial_t+\upsilon\cdot\nabla)^s\nabla ^{k+i}(\nabla\xi)^{-1}\|_N\lesssim\ell_*^{-N-k-i}\tau^{-s},\\
&\|\partial_t^r\upsilon\|_N\lesssim\ell_*^{-N-r},&&\|\nabla\xi-\Id\|_0\lesssim(\tlm\mu)^{-1},\quad\|(\nabla\xi)^{\top}-(\nabla\xi)^{-1}\|_0\lesssim(\tlm\mu)^{-1},
\end{aligned}
\end{equation}
where $\max\{\tau^{-1},\ell_*^{-1}\}\ll\mu^{-1}\ll\tlm$. Then, we could obtain 
\begin{equation}\label{est on N+r}
	\begin{aligned}
		&\left\|\partial_t^r\cE_k(f,a,\upsilon,\tlm)\right\|_N\lesssim_{N}|f|(\tlm|k|)^{N+r-1},&&\left\|\partial_t^r\cB_k(f,a,\upsilon,\tlm)\right\|_N\lesssim_{N}|f|(\tlm|k|)^{N+r-1},\\ 
		&\left\|\partial_t^r\cm_{k}(f,a,\upsilon,\tlm)\right\|_N\lesssim_{N}|f|(\tlm|k|)^{N+r},
	\end{aligned}
\end{equation}
for any $N\geqslant0$ and $r=0,1,2$, and
\begin{equation}\label{est on N+s}
	\begin{aligned}
	&\left\|(\partial_t+\upsilon\cdot\nabla)^s\cE_k(f,a,\upsilon,\tlm)\right\|_N\lesssim_{N}|f|(\tlm|k|)^{N-1}\tau^{-s},&&\left\|(\partial_t+\upsilon\cdot\nabla)^s\cB_k(f,a,\upsilon,\tlm)\right\|_N\lesssim_{N}|f|(\tlm|k|)^{N-1}\tau^{-s},\\ &\left\|(\partial_t+\upsilon\cdot\nabla)^s\cm_{k}(f,a,\upsilon,\tlm)\right\|_N\lesssim_{N}|f|(\tlm|k|)^{N}\tau^{-s},
\end{aligned}
\end{equation}
for any $N\geqslant0$ and $s=0,1,2$. Moreover, we have
\begin{equation}\label{est on m_pk 1}
	\begin{aligned}
	\|\partial_{t}^r(\cm_{p,k}(f,a,\upsilon,\tlm,t_0,x_0)-\cm_k(f,a,\upsilon,\tlm))\|_{N}\lesssim_{N}|f|(\tlm|k|)^{N+r}\left(\left\|\frac{(\partial_{t}\xi)^2}{(\det(\nabla\xi))^2}-\upsilon^2(t_0 ,x_0)\right\|_{0}+(\tlm\mu)^{-1}\right),\end{aligned}
\end{equation} 
for any $N\geqslant0$ and $r=0,1,2$, and
\begin{equation}\label{est on m_pk 2}
	\begin{aligned}
	\left\|(\partial_t+\upsilon\cdot\nabla)^s(\cm_{p,k}(f,a,\upsilon,\tlm,t_0,x_0)-\cm_k(f,a,\upsilon,\tlm))\right\|_{N}\lesssim_{N}|f|(\tlm|k|)^{N}\tau^{-s},
	\end{aligned}
\end{equation} 
for any $N\geqslant0$ and $s=0,1,2$. 
\end{lm}
\begin{proof}
We first define magnetic vector potential $\ccA_k(f,a,\upsilon,\tlm)$ and $\cA_k(f,a,\upsilon,\tlm)$ as
\begin{align}
\ccA_k(f,a,\upsilon,\tlm)&=-\frac{1}{\tlm^3}\frac{a(t,x)}{(\mathrm{det}(\nabla\xi))^3}(\nabla\xi)^{\top}\frac{ik\times f}{|k|^4}e^{i\tlm k\cdot\xi}\notag\\
&=:\frac{1}{\tlm^3}\cta_{k}(f,a,\upsilon,\tlm)e^{i\tlm k\cdot\xi},\label{def of ol ccA_k}\\
\cA_k(f,a,\upsilon,\tlm)&=\nabla\times\ccA_k(f,a,\upsilon,\tlm)\nonumber\\
&=-\frac{1}{\tlm^2}\left(
\frac{1}{\tlm}\nabla\left(\frac{a(t,x)}{(\mathrm{det}(\nabla\xi))^3}\right)\times\left((\nabla\xi)^{\top}\frac{ik\times f}{|k|^4}\right)+\frac{a(t,x)}{(\mathrm{det}(\nabla\xi))^2}(\nabla\xi)^{-1}\frac{ f}{|k|^2}\right)e^{i\tlm k\cdot\xi}\nonumber\\
&=:\frac{1}{\tlm^2}\ca_{k}(f,a,\upsilon,\tlm)e^{i\tlm k\cdot\xi}.\label{def of cA_k}
\end{align}
Then, the electromagnetic fields $\cE_k(f,a,\upsilon,\tlm)$ and $\cB_k(f,a,\upsilon,\tlm)$ can then be represented as
\begin{align}
\cE_k(f,a,\upsilon,\tlm)&=-\partial_t \cA_k(f,a,\upsilon,\tlm)\notag\\
&=\frac{1}{\tlm^2}\partial_t\left(\frac{1}{\tlm}	\nabla\left(\frac{a(t,x)}{(\mathrm{det}(\nabla\xi))^3}\right)\times\left((\nabla\xi)^{\top}\frac{ik\times f}{|k|^4}\right)+\frac{a(t,x)}{(\mathrm{det}(\nabla\xi))^2}(\nabla\xi)^{-1}\frac{f}{|k|^2}\right)e^{i\tlm k\cdot\xi}\notag\\
&\quad+\frac{i(k\cdot\partial_t\xi)}{\tlm}	\left(\frac{1}{\tlm}\nabla\left(\frac{a(t,x)}{(\mathrm{det}(\nabla\xi))^3}\right)\times\left((\nabla\xi)^{\top}\frac{ik\times f}{|k|^4}\right)+\frac{a(t,x)}{(\mathrm{det}(\nabla\xi))^2}(\nabla\xi)^{-1}\frac{f}{|k|^2}\right)e^{i\tlm k\cdot\xi}\label{def of cE_k}\\
&=\frac{1}{\tlm}\left(\frac{i(k\cdot\partial_{t}\xi_{I})a(t,x)}{(\mathrm{det}(\nabla\xi))^2}(\nabla\xi)^{-1}\frac{f}{|k|^2}+\ce_{c,k}(f,a,\upsilon,\tlm)\right)e^{i\tlm k\cdot\xi}\notag\\
&=:\frac{1}{\tlm}\ce_k(f,a,\upsilon,\tlm)e^{i\tlm k\cdot\xi},\notag\\
\cB_k(f,a,\upsilon,\tlm)&=\nabla\times \cA_k(f,a,\upsilon,\tlm)\notag\\
&=-\frac{1}{\tlm^3}\nabla\times\left(	\nabla\left(\frac{a(t,x)}{(\mathrm{det}(\nabla\xi))^3}\right)\times\left((\nabla\xi)^{\top}\frac{ik\times f}{|k|^4}\right)\right)e^{i\tlm k\cdot\xi}\notag\\
&\quad-\frac{i}{\tlm^2}((\nabla\xi)^{\top}k)\times \left(\nabla\left(\frac{a(t,x)}{(\mathrm{det}(\nabla\xi))^3}\right)\times\left((\nabla\xi)^{\top}\frac{ik\times f}{|k|^4}\right)\right)e^{i\tlm k\cdot\xi}\notag\\
&\quad-\frac{1}{\tlm^2}\nabla\left(\frac{a(t,x)}{(\mathrm{det}(\nabla\xi))^2}\right)\times\left((\nabla\xi)^{-1}\frac{f}{|k|^2}\right)e^{i\tlm k\cdot\xi}-\frac{1}{\tlm}\frac{a(t,x)}{\mathrm{det}(\nabla\xi)}(\nabla\xi)^{-1}\frac{ik\times f}{|k|^2}e^{i\tlm k\cdot\xi}\label{def of cB_k}\\
&\quad-\frac{1}{\tlm^2}\frac{a(t,x)}{(\mathrm{det}(\nabla\xi))^2}\nabla\times\left(\left((\nabla\xi)^{-1}-(\nabla\xi)^{\top}\right)\frac{f}{|k|^2}\right)e^{i\tlm k\cdot\xi}\notag\\
&\quad-\frac{i}{\tlm}
\frac{a(t,x)}{(\mathrm{det}(\nabla\xi))^2}\left((\nabla\xi)^{\top}k\right)\times\left(\left((\nabla\xi)^{-1}-(\nabla\xi)^{\top}\right)\frac{f}{|k|^2}\right)e^{i\tlm k\cdot\xi}\notag\\
&=-\frac{1}{\tlm}\left(\frac{a(t,x)}{\mathrm{det}(\nabla\xi)}(\nabla\xi)^{-1}\frac{ik\times f}{|k|^2}+\cg_{c,k}(f,a,\upsilon,\tlm)\right)e^{i\tlm k\cdot\xi}\notag\\
&=:\frac{1}{\tlm} \cg_k(f,a,\upsilon,\tlm)e^{i\tlm k\cdot\xi},\notag
\end{align}
where $\ce_{c,k}(f,a,\upsilon,\tlm)$ and $\cg_{c,k}(f,a,\upsilon,\tlm)$ are defined as
	\begin{align}
		\ce_{c,k}(f,a,\upsilon,\tlm)&=-\frac{1}{\tlm}\partial_t\ca_{k}(f,a,\upsilon,\tlm)+\frac{i(k\cdot\partial_t\xi)}{\tlm}\nabla\left(\frac{a(t,x)}{(\mathrm{det}(\nabla\xi))^2}\right)\times\left((\nabla\xi)^{\top}\frac{ik\times f}{|k|^4}\right),\label{def of ce_ck}\\
		\cg_{c,k}(f,a,\upsilon,\tlm)&=\frac{1}{\tlm^2}\nabla\times\left(	\nabla\left(\frac{a(t,x)}{(\mathrm{det}(\nabla\xi))^3}\right)\times\left((\nabla\xi)^{\top}\frac{ik\times f}{|k|^4}\right)\right)\notag+\frac{i}{\tlm}((\nabla\xi)^{\top}k)\times\left(	\nabla\left(\frac{a(t,x)}{(\mathrm{det}(\nabla\xi))^3}\right)\times\left((\nabla\xi)^{\top}\frac{ik\times f}{|k|^4}\right)\right)\notag\\
		&\quad+\frac{1}{\tlm}\left(\nabla\left(\frac{a(t,x)}{(\mathrm{det}(\nabla\xi))^2}\right)\times\left((\nabla\xi)^{-1}\frac{f}{|k|^2}\right)\right)+\frac{1}{\tlm}\frac{a(t,x)}{(\mathrm{det}(\nabla\xi))^2}\nabla\times\left(\left((\nabla\xi)^{-1}-(\nabla\xi)^{\top}\right)\frac{f}{|k|^2}\right)\label{def of cg_ck}\\
		&\quad+\frac{ia(t,x)}{(\mathrm{det}(\nabla\xi))^2}\left((\nabla\xi)^{\top}k\right)\times\left(\left((\nabla\xi)^{-1}-(\nabla\xi)^{\top}\right)\frac{f}{|k|^2}\right).\notag
	\end{align}
Moreover, we define the  momentum $\cm_{E,k}(f,a,\upsilon,\tlm), \cm_{B,k}(f,a,\upsilon,\tlm)$ as 
\begin{align}
\cm_{E,k}(f,a,\upsilon,\tlm)&=\partial_t \cE_{k}(f,a,\upsilon,\tlm)\notag\\
&=\left(-\frac{(k\cdot\partial_{t}\xi_{I})^2a(t,x)}{|k|^2(\mathrm{det}(\nabla\xi))^2}(\nabla\xi)^{-1}f+\cte_{k}(f,a,\upsilon,\tlm)\right)e^{i\tlm k\cdot\xi},\label{def of cm_Ek}\\
\cm_{B,k}(f,a,\upsilon,\tlm)&=-\nabla\times\cB_{k}(f,a,\upsilon,\tlm)\notag\\
&=a(t,x)(\nabla\xi)^{-1}fe^{i\tlm k\cdot\xi}+\frac{1}{\tlm}\nabla\left(\frac{a(t,x)}{\mathrm{det}(\nabla\xi)}\right)\times\left((\nabla\xi)^{-1}\frac{ik\times f}{|k|^2}\right)e^{i\tlm k\cdot\xi}\notag\\
&\quad+\frac{1}{\tlm}\left(\frac{a(t,x)}{\mathrm{det}(\nabla\xi)}\nabla\times\left(\left((\nabla\xi)^{-1}-(\nabla\xi)^{\top}\right)\frac{ik\times f}{|k|^2}\right)\right)e^{i\tlm k\cdot\xi}\notag\\
&\quad+\frac{ia(t,x)}{\mathrm{det}(\nabla\xi)}((\nabla\xi)^{\top}k)\times\left(\left((\nabla\xi)^{-1}-(\nabla\xi)^{\top}\right)\frac{ik\times f}{|k|^2}\right)e^{i\tlm k\cdot\xi}\notag\\
&\quad+\left(\frac{1}{\tlm}\nabla\times \cg_{c,k}(f,a,\upsilon,\tlm)+(i(\nabla\xi)^{\top}k)\times \cg_{c,k}(f,a,\upsilon,\tlm)\right)e^{i\tlm k\cdot\xi}\notag\\
&=\left(a(t,x)(\nabla\xi)^{-1}f+\ctg_{k}(f,a,\upsilon,\tlm)\right)e^{i\tlm k\cdot\xi},\label{def of cm_Bk}
\end{align}
where $\cte_{k}(f,a,\upsilon,\tlm)$ and $\ctg_{k}(f,a,\upsilon,\tlm)$ are defined as
	\begin{align}
		\cte_{k}(f,a,\upsilon,\tlm)&=-\frac{1}{\tlm^2}\partial_{tt}\ca_{k}(f,a,\upsilon,\tlm)-\frac{2i}{\tlm}(k\cdot\partial_{t}\xi)\partial_{t}\ca_{k}(f,a,\upsilon,\tlm)-\frac{i}{\tlm}(k\cdot\partial_{tt}\xi)\ca_{k}(f,a,\upsilon,\tlm)\notag\\
		&\quad-\frac{(k\cdot\partial_{t}\xi)^2}{\tlm}\left(\nabla\left(\frac{a(t,x)}{(\mathrm{det}(\nabla\xi))^3}\right)\times\left((\nabla\xi)^{\top}\frac{ik\times f}{|k|^4}\right)\right),\label{def of cte_k}\\
		\ctg_{k}(f,a,\upsilon,\tlm)
		&=\frac{1}{\tlm}\left(\nabla\left(\frac{a(t,x)}{\mathrm{det}(\nabla\xi)}\right)\times\left((\nabla\xi)^{-1}\frac{ik\times f}{|k|^2}\right)+\frac{a(t,x)}{\mathrm{det}(\nabla\xi)}\nabla\times\left(\left((\nabla\xi)^{-1}-(\nabla\xi)^{\top}\right)\frac{ik\times f}{|k|^2}\right)\right)\notag\\
		&\quad+\frac{a(t,x)}{\mathrm{det}(\nabla\xi)}((\nabla\xi)^{\top}k)\times\left(\left((\nabla\xi)^{-1}-(\nabla\xi)^{\top}\right)\frac{ik\times f}{|k|^2}\right)\label{def of ctg_ck}\\
		&\quad+\frac{1}{\tlm}\nabla\times \cg_{c,k}(f,a,\upsilon,\tlm)+(i(\nabla\xi)^{\top}k)\times \cg_{c,k}(f,a,\upsilon,\tlm).\notag
		\end{align}
Finally, we give $\cm_{p,k}(f,a,\upsilon,\tlm,t_0,x_0)$ and $\cm_{k}(f,a,\upsilon,\tlm)$,
\begin{align}
		\cm_{k}(f,a,\upsilon,\tlm)&:=\cm_{E,k}(f,a,\upsilon,\tlm)+\cm_{B,k}(f,a,\upsilon,\tlm),\label{def of cm_k}\\
		\cm_{p,k}(f,a,\upsilon,\tlm,t_0,x_0)&:=a(t,x)
		\left(1-\frac{(k\cdot\upsilon(t_0 ,x_0))^2}{|k|^2}\right)(\nabla\xi)^{-1}fe^{i\tlm k\cdot\xi}.\label{def of  cm_pk}
\end{align}
Based on the definition mentioned above, we can use \eqref{assume in building blocks} to  get 
\begin{equation}\label{est on cof N+r 1}
	\begin{aligned}
		&\left\|\partial_t^r\cta_{k}(f,a,\upsilon,\tlm)\right\|_N\lesssim_{N}|f||k|^{-3}\mu^{-N-r},&&\left\|\partial_t^r\ca_{k}(f,a,\upsilon,\tlm)\right\|_N\lesssim_{N}|f||k|^{-2}\mu^{-N-r},\\
		&\left\|\partial_t^r\ce_{c,k}(f,a,\upsilon,\tlm)\right\|_N\lesssim_{N}|f||k|^{-1}\mu^{-N-r}(\tlm\mu)^{-1},&&\left\|\partial_t^r\cg_{c,k}(f,a,\upsilon,\tlm)\right\|_N\lesssim_{N}|f||k|^{-1}\mu^{-N-r},\\
		&\left\|\partial_t^r\ce_k(f,a,\upsilon,\tlm)\right\|_N\lesssim_{N}|f||k|^{-1}\mu^{-N-r},&&\left\|\partial_t^r\cg_k(f,a,\upsilon,\tlm)\right\|_N\lesssim_{N}|f||k|^{-1}\mu^{-N-r},\\
		&\left\|\partial_t^r\cte_k(f,a,\upsilon,\tlm)\right\|_N\lesssim_{N}|f|\mu^{-N-r}(\tlm\mu)^{-1},&&\left\|\partial_t^r\ctg_k(f,a,\upsilon,\tlm)\right\|_N\lesssim_{N}|f|\mu^{-N-r},
	\end{aligned}
\end{equation}
for any $N\geqslant 0$ and $r=0,1,2$. Especially, we have
\begin{equation}\label{est on cof N+r 2}
	\begin{aligned}
	&\left\|\ce_{c,k}(f,a,\upsilon,\tlm)\right\|_0\lesssim|f||k|^{-1}(\tlm\mu)^{-1},&&\left\|\cg_{c,k}(f,a,\upsilon,\tlm)\right\|_0\lesssim|f||k|^{-1}(\tlm\mu)^{-1},\\
	&\left\|\cte_k(f,a,\upsilon,\tlm)\right\|_0\lesssim|f|(\tlm\mu)^{-1},&&\left\|\ctg_k(f,a,\upsilon,\tlm)\right\|_0\lesssim|f|(\tlm\mu)^{-1}.
	\end{aligned}
\end{equation}
Moreover, we could obtain the following estimates on advective derivatives,
\begin{equation}\label{est on cof N+s}
	\begin{aligned}
		|k|\left\|(\partial_t+\upsilon\cdot\nabla)^s\cta_k(f,a,\upsilon,\tlm)\right\|_N+\left\|(\partial_t+\upsilon\cdot\nabla)^s\ca_k(f,a,\upsilon,\tlm)\right\|_N&\lesssim_{N}|f||k|^{-2}\mu^{-N}\tau^{-s},\\
		\left\|(\partial_t+\upsilon\cdot\nabla)^s\ce_k(f,a,\upsilon,\tlm)\right\|_N+\left\|(\partial_t+\upsilon\cdot\nabla)^s\cg_k(f,a,\upsilon,\tlm)\right\|_N&\lesssim_{N}|f||k|^{-1}\mu^{-N}\tau^{-s},\\
		\left\|(\partial_t+\upsilon\cdot\nabla)^s\cte_k(f,a,\upsilon,\tlm)\right\|_N+\left\|(\partial_t+\upsilon\cdot\nabla)^s\ctg_k(f,a,\upsilon,\tlm)\right\|_N&\lesssim_{N}|f|\mu^{-N}\tau^{-s},
	\end{aligned}
\end{equation}
for any $N\geqslant 0$ and $s=0,1,2$.
More importantly, combining the above definition and estimates, we can obtain
\begin{equation}\label{est on N+r 1}
	\begin{aligned}
		&\left\|\partial_t^r\ccA_k(f,a,\upsilon,\tlm)\right\|_N\lesssim_{N}|f|(\tlm|k|)^{N+r-3},&&\left\|\partial_t^r\cA_k(f,a,\upsilon,\tlm)\right\|_N\lesssim_{N}|f|(\tlm|k|)^{N+r-2},\\
		&\left\|\partial_t^r\cE_k(f,a,\upsilon,\tlm)\right\|_N\lesssim_{N}|f|(\tlm|k|)^{N+r-1},&&\left\|\partial_t^r\cB_k(f,a,\upsilon,\tlm)\right\|_N\lesssim_{N}|f|(\tlm|k|)^{N+r-1},\\ 
		&\left\|\partial_t^r\cm_{E,k}(f,a,\upsilon,\tlm)\right\|_N\lesssim_{N}|f|(\tlm|k|)^{N+r},&&\left\|\partial_t^r\cm_{B,k}(f,a,\upsilon,\tlm)\right\|_N\lesssim_{N}|f|(\tlm|k|)^{N+r},\\
		&\left\|\partial_t^r\cm_{k}(f,a,\upsilon,\tlm)\right\|_N\lesssim_{N}|f|(\tlm|k|)^{N+r}, &&\left\|\partial_t^r\cm_{p,k}(f,a,\upsilon,\tlm)\right\|_N\lesssim_{N}|f|(\tlm|k|)^{N+r},
	\end{aligned}
\end{equation}
for any $N\geqslant 0$ and $r=0,1,2$, and
\begin{equation}\label{est on N+s 1}
	\begin{aligned}
		\tlm|k|\left\|(\partial_t+\upsilon\cdot\nabla)^s\ccA_k(f,a,\upsilon,\tlm)\right\|_N+\left\|(\partial_t+\upsilon\cdot\nabla)^s\cA_k(f,a,\upsilon,\tlm)\right\|_N&\lesssim_{N}|f|(\tlm|k|)^{N-2}\tau^{-s},\\
		\left\|(\partial_t+\upsilon\cdot\nabla)^s\cE_k(f,a,\upsilon,\tlm)\right\|_N+\left\|(\partial_t+\upsilon\cdot\nabla)^s\cB_k(f,a,\upsilon,\tlm)\right\|_N&\lesssim_{N}|f|(\tlm|k|)^{N-1}\tau^{-s},\\
		\left\|(\partial_t+\upsilon\cdot\nabla)^s\cm_{E,k}(f,a,\upsilon,\tlm)\right\|_N+\left\|(\partial_t+\upsilon\cdot\nabla)^s\cm_{B,k}(f,a,\upsilon,\tlm)\right\|_N&\lesssim_{N}|f|(\tlm|k|)^{N}\tau^{-s},\\ \left\|(\partial_t+\upsilon\cdot\nabla)^s\cm_{k}(f,a,\upsilon,\tlm)\right\|_N+\left\|(\partial_t+\upsilon\cdot\nabla)^s\cm_{p,k}(f,a,\upsilon,\tlm)\right\|_N&\lesssim_{N}|f|(\tlm|k|)^{N}\tau^{-s},
	\end{aligned}
\end{equation}
for any $N\geqslant 0$ and $s=0,1,2$. \eqref{est on N+r}, \eqref{est on N+s}, and \eqref{est on m_pk 2} follow.
Finally, we could calculate to get \eqref{est on m_pk 1},
\begin{equation}\notag
	\begin{aligned}
		&\quad\|\partial_{t}^r(\cm_{p,k}(f,a,\upsilon,\tlm,t_0,x_0)-\cm_k(f,a,\upsilon,\tlm))\|_N\\
		&\lesssim_{N}(\tlm|k|)^{N+r}\left(|f|\left\|a\right\|_{0}\left\|\frac{(\partial_{t}\xi)^2}{(\det(\nabla\xi))^2}-\upsilon^2(t_0 ,x_0)\right\|_{0}\left\|(\nabla\xi)^{-1}\right\|_{0}+\|\cte_{k}(f,a,\upsilon,\tlm)\|_0+\|\ctg_{k}(f,a,\upsilon,\tlm)\|_0\right)\\
		&\lesssim_{N}|f|(\tlm|k|)^{N+r}\left(\left\|\frac{(\partial_{t}\xi)^2}{(\det(\nabla\xi))^2}-\upsilon^2(t_0 ,x_0)\right\|_{0}+(\tlm\mu)^{-1}\right),
	\end{aligned}
\end{equation}
for any $N\geqslant 0$ and $r=0,1,2$. If $|d-c|$ is small enough, $\left\|\frac{(\partial_{t}\xi)^2}{(\det(\nabla\xi))^2}-\upsilon^2(t_0 ,x_0)\right\|_{0}\lesssim(\tlm\mu)^{-1}$. That is why we refer to $\cm_{p,k}$ as the main component of $\cm_k$.\qedhere
\end{proof}
Instead of using the classical Mikado flow $\psi_If=\sum_{k \in \Z^3\setminus\{0\}}\cb_{I,k}fe^{ik\cdot x}$ directly as in \cite{GK22}, in order to keep the linear Maxwell equation hold during our iteration, we shall use $\cm_{p,k}$ and $\cm_{k}$ defined in Lemma \ref{New building blocks}. But as shown in its definition, when the resonance occurs, namely $\frac{k}{|k|}\cdot\partial_t\xi_{I}\approx \frac{k}{|k|}\cdot(m_\ell/n)\approx1$, a strong electromagnetic field can only lead to a weak fluid flow. To overcome this difficulty caused by resonance, we construct new scalar functions $\psi_I^*(x)\in C^\infty(\T^3)$ satisfying
\begin{equation}\label{def of psi^*}
	\begin{aligned}
		\psi_I^*(x)&=\sum_{k \in \Z^3\setminus\{0\}}\overset{\circ}{b}_{I,k}\left(1-\frac{(k\cdot(m_\ell/n)(\tau_q u,2\pi\mu_q\upsilon))^2}{|k|^2}\right)fe^{ik\cdot x}\\
		&=(\Delta\Psi_{I}(x)-(m_\ell/n)^{\top}(\tau_q u,2\pi\mu_q\upsilon)\nabla^2\Psi_I(x) (m_\ell/n)(\tau_q u,2\pi\mu_q\upsilon)).
	\end{aligned}
\end{equation}
Moreover, we can find
	\begin{align*}
	\psi_I^*(\xi_I)(\nabla\xi_I)^{-1}f&=\sum_{k \in \Z^3\setminus\{0\}}\cm_{p,k}(f,\overset{\circ}{b}_{I,k},m_\ell/n,1,\tau_q u,2\pi\mu_q\upsilon)\\
	&=\sum_{k \in \Z^3\setminus\{0\}}\overset{\circ}{b}_{I,k}\left(1-\frac{(k\cdot(m_\ell/n)(\tau_q u,2\pi\mu_q\upsilon))^2}{|k|^2}\right)(\nabla\xi_I)^{-1}fe^{ik\cdot\xi_I}\\
	&=(\Delta\Psi_{I}(\xi_I)-(m_\ell/n)^{\top}(\tau_q u,2\pi\mu_q\upsilon)\nabla^2\Psi_I(\xi_I) (m_\ell/n)(\tau_q u,2\pi\mu_q\upsilon))(\nabla\xi_I)^{-1}f.
	\end{align*}
Actually, $\psi_I^*(x)f$ is still a special kind of Mikado flow due to $\cb_{I,k}(k\cdot f_I)=0$. We could calculate
\begin{equation}\notag
	\begin{aligned}
	(\psi_I^*\tilde{f})\otimes(\psi_I^*\tilde{f})=(\psi_I^*)^2\tilde{f}\otimes \tilde{f},&&|\psi_I^*\tilde{f}|^2\psi_I^*\tilde{f}=(\psi_I^*)^3|\tilde{f}|^2\tilde{f}.
	\end{aligned}
\end{equation}
In order to use the low-frequency component of the combination of $(\psi_I^*\tilde{f})\otimes(\psi_I^*\tilde{f})$ to cancel $n^2R_\ell$ and the low-frequency component of the combination of $|\psi_I^*\tilde{f}|^2\psi_I^*\tilde{f}$ to cancel  $n^3\varphi_\ell$, we need
$$
\begin{aligned}
	\int_{\T^3}(\psi_I^*)^2\rd x=1,&&\int_{\T^3}(\psi_I^*)^3\rd x=0,&&&\forall I=(u,\upsilon,f),  f\in\mathcal{F}^{[\upsilon],R},\\&&\int_{\T^3}(\psi_I^*)^3\rd x=1, &&&\forall I=(u,\upsilon,f),f\in\mathcal{F}^{[\upsilon],\varphi}.
\end{aligned}
$$
The difficulty lies in the fact that $(m_\ell/n)(\tau_q u,2\pi\mu_q\upsilon)$ varies for different $I$, so we need to construct a special $\psi_f$ that satisfies the above condition for different $I$. Here, we introduce the following lemma to select the special $\psi_f$.
\begin{pp}\label{prop of Mikado flow scalar function}
	For each $f\in\mathcal{F}^{j,R}$, there exists smooth functions $\psi_f$ and $\Psi_f$ satisfying $\supp\Psi_f\subseteq\T^3$ and
	\begin{equation}
		\left\lbrace
		\begin{aligned}
			&\psi_f=\Delta\Psi_f,\\
			&\int_{\T^3}\left(\psi_{f}-L^{\top}\nabla^2\Psi_f L\right)\rd x=0,&& \forall L\in \R^3,\\
			&\int_{\T^3}\left(\psi_{f}-L^{\top}\nabla^2\Psi_f L\right)^2\rd x\geqslant C_1,&& \forall L\in \R^3,\\
			&\int_{\T^3}\left(\psi_{f}-L^{\top}\nabla^2\Psi_f L\right)^3\rd x=0,&& \forall L\in \R^3,
		\end{aligned}\right.
	\end{equation}
	where $C_1>0$ is a constant independent of $L$. Meanwhile, for each $f\in\mathcal{F}^{j,\varphi}$, given $L\in\R^3$, we can choose $\psi_{f,L}=\psi_{f,L}(x;L)\in C_c^\infty(\T^3)$ which depends on $L$ satisfying
	\begin{equation}
		\left\lbrace
		\begin{aligned}
			&\psi_{f,L}=\Delta\Psi_{f,L},\\
			&\int_{\T^3}\left(\psi_{f,L}-L^{\top}\nabla^2\Psi_{f,L} L\right)\rd x=0,&& \forall L\in \R^3,\\
			&\left|\int_{\T^3}\left(\psi_{f,L}-L^{\top}\nabla^2\Psi_{f,L} L\right)^3\rd x\right|\geqslant C_2,&& \forall L\in \R^3,
		\end{aligned}\right.
	\end{equation}
	where $C_2>0$ is a constant independent of $L$.
\end{pp}
\begin{proof}
	We could first choose $\Psi=\Psi(x_1,x_2)$ with $\supp\Psi\subset\T^2$ and $\psi=\Delta\Psi$ which depends only on $(x_1,x_2)$, and then define $\psi_f=\psi(x\cdot \frac{f^*}{|f^*|}, x\cdot(\frac{f}{|f|}\times\frac{f^*}{|f^*|})),\Psi_f=\Psi(x\cdot \frac{f^*}{|f^*|}, x\cdot(\frac{f}{|f|}\times\frac{f^*}{|f^*|}))$, where $\left\lbrace f,f^*,f\times f^*\right\rbrace\subseteq\Z^3$ is an orthogonal basis. Here, we use the notation $\overline{a}=(a_1,a_2)$ to represent the tuple obtained by selecting the first two components of $a$, i.e., $\overline{a}=(a_1,a_2)$ for $a=(a_1,a_2,a_3)$. Notice that for $x=Ay$, where
	$$
	A=\left(\frac{f^*}{|f^*|},\frac{f}{|f|}\times\frac{f^*}{|f^*|},\frac{f}{|f|}\right),\quad |AL|=|L|.
	$$
	We could calculate
	$$
	\int_{\T^3}\left(\psi_{f}-L^{\top}\nabla^2\Psi_f L\right)^k\rd x=\int_{\T^3}\left(\psi-\overline{L^{\top}A}\nabla^2\Psi \overline{A^{\top}L}\right)^k\rd y=2\int_{-1}^{1}\int_{-1}^{1}\left(\Delta\Psi-\overline{L^{\top}A}\nabla^2\Psi \overline{A^{\top}L}\right)^k\rd y_1\rd y_2,
	$$
	where $k=1,2,3$. Due to this observation, we just need to do the construction in the two-dimensional case.\par Let  $\tilde{L}=\overline{A^{\top}L}=(\tilde{L}_1,\tilde{L}_2)\in\R^2$ and $R^2=\tilde{L}_1^2+\tilde{L}_2^2$ for $f\in\mathcal{F}^{j,R}$, we set
	$$
	\tilde{\Psi}_R(x_1,x_2)=(\phi_*(x_1+1)-\phi_*(x_1-1))(\phi_*(x_2+1)-\phi_*(x_2-1))=:\tilde{\phi}(x_1)\tilde{\phi}(x_2),
	$$
	where
	$$
	\phi_*(x)=\left\lbrace
	\begin{aligned}
		&e^{\frac{1}{|x|^2-1}},&&|x|<1,\\
		&0,&&|x|\geqslant1.
	\end{aligned}\right.
	$$
	Then, we could get 
	$$
	\Delta\tilde{\Psi}_R-\tilde{L}^{\top}\nabla^2\tilde{\Psi}_R \tilde{L}=(1-\tilde{L}_1^2)\tilde{\phi}^{\tp}(x_1)\tilde{\phi}(x_2)+(1-\tilde{L}_2^2)\tilde{\phi}^{\tp}(x_2)\tilde{\phi}(x_1)-2\tilde{L}_1\tilde{L}_2\tilde{\phi}^{\prime}(x_1)\tilde{\phi}^{\prime}(x_2).
	$$
	Noting that $\tilde{\phi}$ is odd, we have
	\begin{align*}
		\int_{\R}\int_{\R}\left(\Delta\tilde{\Psi}_R-\tilde{L}^{\top}\nabla^2\tilde{\Psi}_R \tilde{L}\right)\rd x_1\rd x_2&=0,\quad\int_{\R}\int_{\R}\left(\Delta\tilde{\Psi}_R-\tilde{L}^{\top}\nabla^2\tilde{\Psi}_R \tilde{L}\right)^3\rd x_1\rd x_2=0.
	\end{align*}
	Meanwhile,
	\begin{align*}
		&\quad\int_{\R}\int_{\R}\left(\Delta\tilde{\Psi}_R-\tilde{L}^{\top}\nabla^2\tilde{\Psi}_R \tilde{L}\right)^2\rd x_1\rd x_2\\&=((1-\tilde{L}_1^2)^2+(1-\tilde{L}_2^2)^2)\int_{\R}(\tilde{\phi}^{\tp}(x))^2\rd x\int_{\R}(\tilde{\phi}(x))^2\rd x+\left(4(\tilde{L}_1\tilde{L}_2)^2+2(1-\tilde{L}_1^2)(1-\tilde{L}_2^2)\right)\left(\int_{\R}(\tilde{\phi}^{\prime}(x))^2\rd x\right)^2\\
		&\geqslant\left((1-\tilde{L}_1^2)^2+(1-\tilde{L}_2^2)^2\right)\left(\int_{\R}(\tilde{\phi}^{\tp}(x))^2\rd x\int_{\R}(\tilde{\phi}(x))^2\rd x-\left(\int_{\R}(\tilde{\phi}^{\prime}(x))^2\rd x\right)^2\right)+4(\tilde{L}_1\tilde{L}_2)^2\left(\int_{\R}(\tilde{\phi}^{\prime}(x))^2\rd x\right)^2\\
		&=\left((1-\tilde{L}_1^2)^2+(1-\tilde{L}_2^2)^2\right)\left(D_1-D_2\right)+4(\tilde{L}_1\tilde{L}_2)^2D_2\\
		&\geqslant2(D_1-D_2)\left(1-\frac{R^2}{2}\right)^2+4D_2\left(\frac{R^2}{2}\right)^2\\
		&\geqslant\frac{4(D_1-D_2)D_2}{(D_1+D_2)}=:\tilde{C},
	\end{align*}
	where
	$$
	\begin{aligned}
	&D_1=\int_{\R}(\phi_*^{\tp}(x+1)-\phi_*^{\tp}(x-1))^2\rd x\int_{\R}(\phi_*(x+1)-\phi_*(x-1))^2\rd x,\\ &D_2=\left(\int_{\R}(\phi_*^{\prime}(x+1)-\phi_*^{\prime}(x-1))^2\rd x\right)^2.
	\end{aligned}
	$$
	$D_1$ and $D_2$ depend only on $\phi_*$ and thus are absolute constants. We can choose $\Psi_f=\Psi(x\cdot \frac{f^*}{|f^*|}, x\cdot(\frac{f}{|f|}\times\frac{f^*}{|f^*|}))$ with $\Psi(x_1,x_2)=\frac{1}{\tilde{C}^{\frac{1}{2}}\tilde{r}}\tilde{\Psi}_R(x_1/\tilde{r},x_2/\tilde{r})$, for $f\in\cF^{j,R}$, where the parameter $\tilde{r}<\frac{\eta}{100}$ and $\eta$ has been defined in Proposition \ref{def of eta}. Then, we have
	$$
	\begin{aligned}
		\int_{\R}\int_{\R}\left(\Delta\Psi_f-\tilde{L}^{\top}\nabla^2\Psi_f \tilde{L}\right)\rd x_1\rd x_2&=\int_{\R}\int_{\R}\left(\Delta\Psi_f-\tilde{L}^{\top}\nabla^2\Psi_f \tilde{L}\right)^3\rd x_1\rd x_2=0,\\
		\int_{\R}\int_{\R}\left(\Delta\Psi_f-\tilde{L}^{\top}\nabla^2\Psi_f \tilde{L}\right)^2\rd x_1\rd x_2&\geqslant 2.
	\end{aligned}
	$$\par
	For $f\in\mathcal{F}^{j,\varphi}$, we will construct four kinds of $\psi_{f,L}$ for different $L$. First, if $\tilde{L}$ is close to $\tilde{L}_1^2+\tilde{L}_2^2=2$, let
	$$
	\tilde{\Psi}_{\varphi,1}(x_1,x_2)=(\phi_*(x_1+1/2)-\phi_*(x_1-1/2))(\phi_*(x_2+1/2)-\phi_*(x_2-1/2))=:\overline{\phi}_1(x_1)\overline{\phi}_1(x_2),
	$$
	we could achieve 
	$$
	\begin{aligned}
		&\quad\int_{\R}\int_{\R}(\Delta\tilde{\Psi}_{\varphi,1}-\tilde{L}^{\top}\nabla^2\tilde{\Psi}_{\varphi,1} \tilde{L})\rd x_1\rd x_2=0,\\
		&\quad\int_{\R}\int_{\R}(\Delta\tilde{\Psi}_{\varphi,1}-\tilde{L}^{\top}\nabla^2\tilde{\Psi}_{\varphi,1} \tilde{L})^3\rd x_1\rd x_2=-\tilde{L}_1\tilde{L}_2\left(11(\tilde{L}_1\tilde{L}_2)^2+3-3R^2\right)\left(\int_{\R}(\overline{\phi}_1^{\prime}(x))^3\rd x\right)^2.
	\end{aligned}
	$$
	Let $\tilde{L}_1=R\cos\theta,\tilde{L}_2=R\sin\theta$, we could get
	$$
	\begin{aligned}
		\left|\int_{\R}\int_{\R}(\Delta\tilde{\Psi}_{\varphi,1}-\tilde{L}^{\top}\nabla^2\tilde{\Psi}_{\varphi,1} \tilde{L})^3\rd x_1\rd x_2\right|&=\frac{R^2}{2}\left|\sin2\theta\right|\left|\frac{11}{4}R^4(\sin2\theta)^2+3-3R^2\right|\left(\int_{\R}(\overline{\phi}_1^{\prime}(x))^3\rd x\right)^2.\\
	\end{aligned}
	$$
	Notice that $|\sin2\theta|\geqslant\frac{1}{2}$, if $|\theta-\frac{1+2k}{4}\pi|\leqslant\frac{1}{6}\pi,k=0,1,2,3$. Then, we have
	$$
	\left|\int_{\R}\int_{\R}(\Delta\tilde{\Psi}_{\varphi,1}-\tilde{L}^{\top}\nabla^2\tilde{\Psi}_{\varphi,1} \tilde{L})^3\rd x_1\rd x_2\right|>C\left(\int_{\R}(\phi_*^{\prime}(x_1+1/2)-\phi_*^{\prime}(x_1-1/2))^3\rd x\right)^2>\tilde{C}_1,
	$$
	when $\frac{18}{11}\leqslant R^2 \leqslant \frac{30}{11},|\theta-\frac{1+2k}{4}\pi|\leqslant\frac{1}{6}\pi,k=0,1,2,3$.\par
	Next, we define
	$$
	\tilde{\Psi}_{\varphi,2}(x_1,x_2)=(2\phi_*(x_1)-\phi_*(x_1-2)-\phi_*(x_1+2))(2\hat{\phi}(x_2)-\hat{\phi}(x_2-2)-\hat{\phi}(x_2+2))=:\overline{\phi}_2(x_1)\overline{\phi}_3(x_2),
	$$
	where
	$$
	\hat{\phi}(x)=\frac{1}{2}\phi_*^2(x)=\left\lbrace
	\begin{aligned}
		&\frac{1}{2}e^{\frac{2}{|x|^2-1}},&&|x|<1,\\
		&0,&&|x|\geqslant1.
	\end{aligned}\right.
	$$
	We could calculate
	$$
	\begin{aligned}
		&\quad\int_{\R}\int_{\R}(\Delta\tilde{\Psi}_{\varphi,1}-\tilde{L}^{\top}\nabla^2\tilde{\Psi}_{\varphi,2} \tilde{L})\rd x_1\rd x_2=0,\\
		&\quad\int_{\R}\int_{\R}(\Delta\tilde{\Psi}_{\varphi,1}-\tilde{L}^{\top}\nabla^2\tilde{\Psi}_{\varphi,2}\tilde{L})^3\rd x_1\rd x_2\\
		&=(1-\tilde{L}_1^2)^3\int_{\R}(\overline{\phi}_3(x))^3\rd x\int_{\R}(\overline{\phi}_2^{\tp}(x))^3\rd x+(1-\tilde{L}_2^2)^3\int_{\R}(\overline{\phi}_2(x))^3\rd x\int_{\R}(\overline{\phi}_3^{\tp}(x))^3\rd x\\
		&\quad+3(1-\tilde{L}_1^2)^2(1-\tilde{L}_2^2)\int_{\R}(\overline{\phi}_2^{\tp}(x))^2\overline{\phi}_2(x) \rd x\int_{\R}(\overline{\phi}_3(x))^2\overline{\phi}_3^{\tp}(x)\rd x\\
		&\quad+3(1-\tilde{L}_1^2)(1-\tilde{L}_2^2)^2\int_{\R}(\overline{\phi}_3^{\tp}(x))^2\overline{\phi}_3(x) \rd x\int_{\R}(\overline{\phi}_2(x))^2\overline{\phi}_2^{\tp}(x)\rd x.
	\end{aligned}
	$$
	Similarly, let $\tilde{L}_1=R\cos\theta,\tilde{L}_2=R\sin\theta$, we could get, if $R^2=2$,
	\begin{align*}
		&\quad\int_{\R}\int_{\R}(\Delta\tilde{\Psi}_{\varphi,2}-\tilde{L}^{\top}\nabla^2\tilde{\Psi}_{\varphi,2}\tilde{L})^3\rd x_1\rd x_2\\
		&=(1-\tilde{L}_1^2)^3\left(\int_{\R}(\overline{\phi}_3(x))^3\rd x\int_{\R}(\overline{\phi}_2^{\tp}(x))^3\rd x-\int_{\R}(\overline{\phi}_2(x))^3\rd x\int_{\R}(\overline{\phi}_3^{\tp}(x))^3\rd x\right)\\
		&\quad+3(1-\tilde{L}_1^2)^3\left(\int_{\R}(\overline{\phi}_3^{\tp}(x))^2\overline{\phi}_3(x) \rd x\int_{\R}(\overline{\phi}_2(x))^2\overline{\phi}_2^{\tp}(x)\rd x-\int_{\R}(\overline{\phi}_2^{\tp}(x))^2\overline{\phi}_2(x) \rd x\int_{\R}(\overline{\phi}_3(x))^2\overline{\phi}_3^{\tp}(x)\rd x\right)\\
		&=(1-\tilde{L}_1^2)^3\left(\int_{\R}\int_{\R}\left(\overline{\phi}_3(x_2)\overline{\phi}^{\tp}_2(x_1)-\overline{\phi}^{\tp}_3(x_2)\overline{\phi}_2(x_1)\right)^3\rd x_1\rd x_2\right)\\
		&=:(1-\tilde{L}_1^2)^3\tilde{C}_2.
	\end{align*}
	When  $\left|\theta-\frac{k}{2}\pi\right|\leqslant\frac{\pi}{6},k=0,1,2,3$, we could obtain $0\leqslant\tilde{L}_1^2\leqslant\frac{1}{2}$ or $\frac{3}{2}\leqslant\tilde{L}_1^2\leqslant2$ and
	$$
	\left|\int_{\R}\int_{\R}(\Delta\tilde{\Psi}_{\varphi,2}-\tilde{L}^{\top}\nabla^2\tilde{\Psi}_{\varphi,2} \tilde{L})^3\rd x_1\rd x_2\right|\geqslant\frac{1}{8}|\tilde{C}_2|.
	$$
	Moreover, there exists $\varepsilon<\frac{4}{11}$ such that
	$$
	\left|\int_{\R}\int_{\R}(\Delta\tilde{\Psi}_{\varphi,2}-\tilde{L}^{\top}\nabla^2\tilde{\Psi}_{\varphi,2} \tilde{L})^3\rd x_1\rd x_2\right|\geqslant\frac{1}{16}|\tilde{C}_2|,
	$$
	when $2-\varepsilon\leqslant R^2\leqslant2+\varepsilon,\left|\theta-\frac{k}{2}\pi\right|\leqslant\frac{\pi}{6},k=0,1,2,3$.\par
	Finally, we choose two axially symmetric functions $\tilde{\Psi}_{\varphi,3}(x_1,x_2)=\overline{\phi}_4(r),\tilde{\Psi}_{\varphi,4}(x_1,x_2)=\overline{\phi}_5(r)$ which satisfy
	$$
	\begin{aligned}
		&\int_{0}^{2\pi}\int_{0}^{1}\left(\frac{d^2\overline{\phi}_4}{dr^2}+\frac{1}{r}\frac{d\overline{\phi}_4}{dr}\right)r\rd r\rd \theta=0,&&\int_{0}^{2\pi}\int_{0}^{1}\left(\frac{d^2\overline{\phi}_5}{dr^2}+\frac{1}{r}\frac{d\overline{\phi}_5}{dr}\right)r\rd r\rd \theta=0,\\
		&\int_{0}^{2\pi}\int_{0}^{1}\left(\frac{d^2\overline{\phi}_4}{dr^2}+\frac{1}{r}\frac{d\overline{\phi}_4}{dr}\right)^3r\rd r\rd \theta=0,&&\int_{0}^{2\pi}\int_{0}^{1}\left(\frac{d^2\overline{\phi}_5}{dr^2}+\frac{1}{r}\frac{d\overline{\phi}_5}{dr}\right)^3r\rd r\rd \theta=1,\\
		&\int_{0}^{2\pi}\int_{0}^{1}\left(\frac{d^2\overline{\phi}_4}{dr^2}+\frac{1}{r}\frac{d\overline{\phi}_4}{dr}\right)\left(\frac{d^2\overline{\phi}_4}{dr^2}-\frac{1}{r}\frac{d\overline{\phi}_4}{dr}\right)^2r\rd r\rd \theta\neq0,
	\end{aligned}
	$$
	and then 
	$$
	\begin{aligned}
		&\quad\int_{\R}\int_{\R}(\Delta\tilde{\Psi}_{\varphi,4}-\tilde{L}^{\top}\nabla^2\tilde{\Psi}_{\varphi,4} \tilde{L})^3\rd x_1\rd x_2\\
		&=2\pi\left(1-\frac{1}{2}R^2\right)^3+\frac{3}{4}\pi\left(1-\frac{1}{2}R^2\right)R^4\int_{0}^{1}\left(\left(\frac{d^2\overline{\phi}_5}{dr^2}\right)+\left(\frac{1}{r}\frac{d\overline{\phi}_5}{dr}\right)\right)\left(\frac{d^2\overline{\phi}_5}{dr^2}-\frac{1}{r}\frac{d\overline{\phi}_5}{dr}\right)^2rdr.
	\end{aligned}
	$$
	The right side equals $2\pi$ when $R=0$. So there exists $\varepsilon_1$ such that if $R^2\leqslant\varepsilon_1<1$,
	$$
	\begin{aligned}
	\left|\int_{\R}\int_{\R}(\Delta\tilde{\Psi}_{\varphi,4}-\tilde{L}^{\top}\nabla^2\tilde{\Psi}_{\varphi,4} \tilde{L})^3\rd x_1\rd x_2\right|\geqslant\pi.
	\end{aligned}
	$$	
	As for $\tilde{\Psi}_{\varphi,3}(x_1,x_2)$, we could obtain
	$$
	\begin{aligned}
	\int_{\R}\int_{\R}(\Delta\tilde{\Psi}_{\varphi,3}-\tilde{L}^{\top}\nabla^2\tilde{\Psi}_{\varphi,3} \tilde{L})^3\rd x_1\rd x_2=\frac{3}{4}\pi\left(1-\frac{1}{2}R^2\right)R^4\int_{0}^{1}\left(\left(\frac{d^2\overline{\phi}_4}{dr^2}\right)+\left(\frac{1}{r}\frac{d\overline{\phi}_4}{dr}\right)\right)\left(\frac{d^2\overline{\phi}_4}{dr^2}-\frac{1}{r}\frac{d\overline{\phi}_4}{dr}\right)^2rdr.
	\end{aligned}
	$$
	So if $|R^2-2|>\varepsilon$ and $R^2>\varepsilon_1$, we could immediately get
	$$
	\begin{aligned}
		&\quad\left|\int_{\R}\int_{\R}(\Delta\tilde{\Psi}_{\varphi,3}-\tilde{L}^{\top}\nabla^2\tilde{\Psi}_{\varphi,3} \tilde{L})^3\rd x_1\rd x_2\right|\geqslant\tilde{C}_3(\varepsilon,\varepsilon_1),
	\end{aligned}
	$$
	where $\tilde{C}_3(\varepsilon,\varepsilon_1)$ depends on $\varepsilon,\varepsilon_1$. So for different $\tilde{L}=(R\cos\theta,R\sin\theta)=\overline{A^{\top}L}$, we could construct 
	$$
	\Psi_{\varphi}(x_1,x_2;\tilde{L})=\left\lbrace
	\begin{aligned}
		&\frac{1}{\overline{C}\tilde{r}^{\frac{2}{3}}}\tilde{\Psi}_{\varphi,1}(x_1/\tilde{r},x_2/\tilde{r}),&&|R^2-2|\leqslant\varepsilon,|\theta-\frac{1+2k}{4}\pi|\leqslant\frac{\pi}{6},k=0,1,2,3,\\
		&\frac{1}{\overline{C}\tilde{r}^{\frac{2}{3}}}\tilde{\Psi}_{\varphi,2}(x_1/\tilde{r},x_2/\tilde{r}),&&|R^2-2|\leqslant\varepsilon,\left|\theta-\frac{k}{2}\pi\right|\leqslant\frac{\pi}{6}, k=0,1,2,3,\\
		&\frac{1}{\overline{C}\tilde{r}^{\frac{2}{3}}}\tilde{\Psi}_{\varphi,3}(x_1/\tilde{r},x_2/\tilde{r}),&&|R^2-2|>\varepsilon,R^2>\varepsilon_1,\\
		&\frac{1}{\overline{C}\tilde{r}^{\frac{2}{3}}}\tilde{\Psi}_{\varphi,4}(x_1/\tilde{r},x_2/\tilde{r}),&&R^2\leqslant\varepsilon_1,
	\end{aligned}\right.
	$$
	where $\tilde{r}<\frac{\eta}{100},\overline{C}=\frac{\min\left\lbrace\tilde{C}_1^\frac{1}{3},|\tilde{C}_2|^\frac{1}{3},\tilde{C}_3^\frac{1}{3},\pi^\frac{1}{3}\right\rbrace}{16}$. Finally, we define $\Psi_{f,L}(x;L)=\Psi_{\varphi}(x\cdot \frac{f^*}{|f^*|}, x\cdot(\frac{f}{|f|}\times\frac{f^*}{|f^*|});\overline{A^{\top}L})$ for $f\in\cF^{j,\varphi}$. Then, we have
	\begin{align*}
		&\int_{\R}\int_{\R}\left(\Delta\Psi_{f,L}-\tilde{L}^{\top}\nabla^2\Psi_{f,L} \tilde{L}\right)\rd x_1\rd x_2=0,\quad\left|\int_{\R}\int_{\R}\left(\Delta\Psi_{f,L}-\tilde{L}^{\top}\nabla^2\Psi_{f,L} \tilde{L}\right)^3\rd x_1\rd x_2\right|\geqslant 1. \qedhere
	\end{align*}
\end{proof}
To simplify the notation, we define the scalar function $\Psi_I(x)$ and $\psi_I^*(x)$ as follows:
$$
\Psi_I(x)=\left\lbrace
\begin{aligned}
	&\Psi_f(x-z_I),&& I=(u,\upsilon,f),f\in\mathcal{F}^{[\upsilon],R},\\
	&\Psi_{f,L}(x-z_I;(m_\ell/n)(\tau_q u,2\pi\mu_q\upsilon)),&& I=(u,\upsilon,f),f\in\mathcal{F}^{[\upsilon],\varphi},
\end{aligned}\right.
$$
and
$$
\psi_I^*(x)=\Delta\Psi_{I}(x)-((m_\ell/n)(\tau_q u,2\pi\mu_q\upsilon))^{\top}\nabla^2\Psi_{I}(x)(m_\ell/n)(\tau_q u,2\pi\mu_q\upsilon).\quad
$$
$U_I(x)=\psi_I(x)f$ and $U_I^*(x)=\psi_I^*(x)f$ satisfy the definition of Mikado flow, as in Definition \ref{def of Mikado flow}.  
\subsection{The definition of the perturbation}\label{The definition of the perturbation}
In this part, we use the new building blocks defined in Lemma \ref{New building blocks} to construct the perturbation $(\tm,\tE,\tB)$. We start by defining $B_{I,k}$ as follows:
\begin{equation}
	B_{I,k}=\left\lbrace
	\begin{aligned}
		&\frac{\theta_I(t)\chi_I(\xi_I)\delta_{q+1}^{-\frac{1}{2}}\gamma_I\cb_{I,k}}{\left(\int_{\T^3}\left(\psi_I^*(x)\right)^2\rd x\right)^{\frac{1}{2}}},&& I=(u,\upsilon,f),f\in\mathcal{F}^{[\upsilon],R},\\
		&\frac{\theta_I(t)\chi_I(\xi_I)\delta_{q+1}^{-\frac{1}{2}}\gamma_I\cb_{I,k}}{\left(\int_{\T^3}\left(\psi_I^*(x)\right)^3\rd x\right)^{\frac{1}{3}}},&& I=(u,\upsilon,f),f\in\mathcal{F}^{[\upsilon],\varphi}.\\
	\end{aligned}\right.\label{def of BIk}
\end{equation}
Here, $\gamma_I=\gamma_{I}(R_\ell,\varphi_\ell,n,t,x)$ are smooth weights functions that will be chosen in the next section. Using $B_{I,k}$, we can construct the magnetic vector potential $\tilde{\mathcal{A}},\tilde{A}$ as follows:
\begin{align}
		\tcA&=\delta_{q+1}^{\frac{1}{2}}\sum_u\sumkO\sum_{I:u_I=u}\ccA_{k}(f_I,B_{I,k},m_\ell/n,\lambda_{q+1})\nonumber\\
		&=\frac{\delta_{q+1}^{\frac{1}{2}}}{\lambda_{q+1}^3}\sum_u\sumkO\sum_{I:u_I=u}\cta_{k}(f_I,B_{I,k},m_\ell/n,\lambda_{q+1})e^{i\lambda_{q+1}k\cdot\xi_I},\label{def of ol tcA}\\
		\tilde{A}=\nabla\times\tilde{\mathcal{A}}&=\delta_{q+1}^{\frac{1}{2}}\sum_u\sumkO\sum_{I:u_I=u}\cA_{k}(f_I,B_{I,k},m_\ell/n,\lambda_{q+1})\nonumber\\
		&=\frac{\delta_{q+1}^{\frac{1}{2}}}{\lambda_{q+1}^2}\sum_u\sumkO\sum_{I:u_I=u}\ca_{k}(f_I,B_{I,k},m_\ell/n,\lambda_{q+1})e^{i\lambda_{q+1}k\cdot\xi_I}.\label{def of tilde A}
\end{align}
Then, the corresponding electromagnetic field $\tE_p,\tB_p$ can be defined as
\begin{align}
		\tE_p=-\partial_t \tilde{A}&=\delta_{q+1}^{\frac{1}{2}}\sum_u\sumkO\sum_{I:u_I=u}\cE_{k}(f_I,B_{I,k},m_\ell/n,\lambda_{q+1})\notag\\
		&=\frac{\delta_{q+1}^{\frac{1}{2}}}{\lambda_{q+1}}\sum_u\sumkO\sum_{I:u_I=u}\ce_{k}(f_I,B_{I,k},m_\ell/n,\lambda_{q+1})e^{i\lambda_{q+1}k\cdot\xi_I},\label{def of tE_p}\\
		\tB_p=\nabla\times \tilde{A}&=\delta_{q+1}^{\frac{1}{2}}\sum_u\sumkO\sum_{I:u_I=u}\cB_{k}(f_I,B_{I,k},m_\ell/n,\lambda_{q+1})\notag\\
		&=\frac{\delta_{q+1}^{\frac{1}{2}}}{\lambda_{q+1}}\sum_u\sumkO\sum_{I:u_I=u}\cg_{k}(f_I,B_{I,k},m_\ell/n,\lambda_{q+1})e^{i\lambda_{q+1}k\cdot\xi_I}.\label{def of tB_p}
\end{align}
Finally, we define the momentum caused by $\tE_p$ and $\tB_p$ as
	\begin{align}
		\tm_{EB}=\partial_t \tE_p-\nabla\times \tB_p&=\delta_{q+1}^{\frac{1}{2}}\sum_u\sumkO\sum_{I:u_I=u}\cm_{k}(f_I,B_{I,k},m_\ell/n,\lambda_{q+1})\nonumber\\
		&=\delta_{q+1}^{\frac{1}{2}}\sum_u\sumkO\sum_{I:u_I=u}B_{I,k}\left(1-\frac{(k\cdot\partial_{t}\xi_{I})^2}{|k|^2(\mathrm{det}(\nabla\xi_I))^2}\right)(\nabla\xi_I)^{-1}f_Ie^{i\lambda_{q+1}k\cdot\xi_I}\label{def of tm_EB}\\
		&\quad+\delta_{q+1}^{\frac{1}{2}}\sum_u\sumkO\sum_{I:u_I=u}\left(\cte_{k}(f_I,B_{I,k},m_\ell/n,\lambda_{q+1})+\ctg_{k}(f_I,B_{I,k},m_\ell/n,\lambda_{q+1})\right)e^{i\lambda_{q+1}k\cdot\xi_I},\nonumber
		\end{align}
and its main part $\tm_p$ as
		\begin{equation}\label{def of tm_p 1}
		\begin{aligned}
		\tm_p&=\delta_{q+1}^{\frac{1}{2}}\sum_u\sumkO\sum_{I:u_I=u}\cm_{p,k}(f,B_{I,k},m_\ell/n,\lambda_{q+1},\tau_q u,2\pi\mu_q\upsilon)\\
		&=\delta_{q+1}^{\frac{1}{2}}\sum_u\sumkO\sum_{I:u_I=u}B_{I,k}\left(1
		-\frac{\left(k\cdot(m_\ell/n)(\tau_q u,2\pi\mu_q\upsilon)\right)^2}{|k|^2}\right)(\nabla\xi_I)^{-1}f_Ie^{i\lambda_{q+1}k\cdot\xi_I}\\
		&=\sum_u\sum_{I:u_I=u}\theta_I(t)\chi_I(\xi_I)\gamma_{I}\tilde{f}_IM_I(\lambda_{q+1}\xi_{I}),
		\end{aligned}
		\end{equation}
where
$$
M_I(x):=\left\lbrace
\begin{aligned}
	&\frac{\psi_I^*(x)}{\left(\int_{\T^3}\left(\psi_I^*(x)\right)^2\rd x\right)^{\frac{1}{2}}},&& I=(u,\upsilon,f),f\in\mathcal{F}^{[\upsilon],R},\\
	&\frac{\psi_I^*(x)}{\left(\int_{\T^3}\left(\psi_I^*(x)\right)^3\rd x\right)^{\frac{1}{3}}},&& I=(u,\upsilon,f),f\in\mathcal{F}^{[\upsilon],\varphi}.
\end{aligned}\right.\\
$$
We could immediately obtain that
$$
\begin{aligned}
	\langle M_I \rangle=&\langle M_I^3 \rangle=0,\langle M_I^2 \rangle=1,&& \forall I=(u,\upsilon,f),f\in\mathcal{F}^{[\upsilon],R},\\
	&\langle M_I \rangle=0,\langle M_I^3 \rangle=1,&& \forall I=(u,\upsilon,f),f\in\mathcal{F}^{[\upsilon],\varphi},
\end{aligned}
$$
where $\langle u \rangle = \fint_{\T^3} u (x) \rd x$.\par
For convenience, we denote  $G_{u,k}=\sum_{I:u_I=u}\overset{\circ}{G}_{k}(f_I,B_{I,k},m_\ell/n,\lambda_{q+1})$ for the functions $G=\ta,a,e,g,\tilde{e},\tilde{g}$, and define 
\begin{align}
s_{u,k}=\sum_{I:u_I=u}B_{I,k}\left(1-\frac{(k\cdot\partial_{t}\xi_{I})^2}{|k|^2(\mathrm{det}(\nabla\xi_I))^2}\right)\tilde{f},\label{def of suk}
\end{align}
and $\tm_{EB}$ can be rewritten as,
$$
\tm_{EB}=\delta_{q+1}^{\frac{1}{2}}\sum_u\sumkO(s_{u,k}+\tilde{e}_{u,k}+\tilde{g}_{u,k})e^{i\lambda_{q+1}k\cdot\xi_I}.
$$
Moreover, let $I=(u,\upsilon,f),I^\prime=(u^\prime,\upsilon^\prime,f^\prime)$, noticing that $B_{I,k}\bigcap B_{I^\prime,k^\prime}=\emptyset$, for $|u-u^\prime|>1,\forall k,k^\prime\in \Z^3/\left\lbrace0\right\rbrace$, then we have $\text{supp}(\tilde{e}_{u,k})\bigcap\text{supp}(\tilde{e}_{u^\prime,k^\prime})=\text{supp}(\tilde{g}_{u,k})\bigcap\text{supp}(\tilde{g}_{u^\prime,k^\prime})=\text{supp}(e_{u,k})\bigcap\text{supp}(e_{u^\prime,k^\prime})=\text{supp}(g_{u,k})\bigcap\text{supp}(g_{u^\prime,k^\prime})$ $=\text{supp}(s_{u,k})\bigcap\text{supp}(s_{u^\prime,k^\prime})=\text{supp}(\ta_{u,k})\bigcap\text{supp}(\ta_{u^\prime,k^\prime})=\text{supp}(a_{u,k})\bigcap$ $\text{supp}(a_{u^\prime,k^\prime})=\emptyset,$ for $|u-u^\prime|>1,\forall k,k^\prime\in \Z^3/\left\lbrace0\right\rbrace $.\par
Next, we introduce a time correction term $\tm_t$. From the momentum equation, we can obtain
$$
\begin{aligned}
	\quad\partial_tm_q+\Div\left(\frac{m_q\otimes m_q}{n}\right)+\nabla p(n)+nE_q+m_q\times B_q=\Div(n(R_q-c_q\Id)),
\end{aligned}$$
where $c_q=\sum_{j=q+1}^{\infty}\delta_j$. We first calculate the left-hand side of the equation, if we apply a perturbation $m_{EB}$ to $m_q$ and integrate it with respect to $x$. This yields:
\begin{align}
		&\quad\int_{\T^3}\left(\partial_t(m_q+\tm_{EB})+\Div\left(\frac{(m_q+\tm_{EB})\otimes (m_q+\tm_{EB})}{n}\right)+\nabla p(n)+n(E_{q}+\tE_p)+(m_q+\tm_{EB})\times (B_{q}+\tB_p)\right)\rd x\notag\\
		&=\int_{\T^3}\left(n \tE_p+\tm_{EB}\times B_{q}+m_{q}\times \tB+\tm_{EB}\times \tB \right)\rd x.\label{momentum q+1}
\end{align}
It's important to note that the resulting value may not be zero.  As a result, the left-hand side cannot be represented as the divergence of some matrix. To address this, we introduce a time correction term $\tm_t$ that is a smooth function of time $t$ and satisfies 
\begin{equation}
	\left\lbrace
	\begin{aligned}
		&\partial_t\tE_t-\nabla\times \tB_t=\tm_t,\\
		&\Div\tE_t=0,\\
		&\partial_t\tB_t+\nabla\times \tE_t=0,\\
		&\Div\tB_t=0,\\
	\end{aligned}
	\right.\label{C Maxwell}
\end{equation}
where $\tE_t(t)=\int_{0}^{t}\tm_t(\tau)\rd \tau$, $B_t=0$. So we can choose proper $\tE_t$ to make sure that
\begin{equation}\label{Eq of tE_t}
	\partial_{tt}\tE_t+\left(\int_{\T^3}n\rd x\right)\tE_t=-\int_{\T^3}\left(n \tE_p+\tm_{EB}\times B_{q}+m_{q}\times \tB+\tm_{EB}\times \tB \right)\rd x.
\end{equation}
Especially, in order to solve this ODE, we choose $\tE_t(0)=0$ and $ \partial_t\tE_t(0)=0.$
So if we choose the perturbation as 
$\tm=\tm_{EB}+\tm_t$, $\tE=\tE_p+\tE_t$, and $\tB=\tB_p$,  both sides of \eqref{momentum q+1} equal zero. Then, we can obtain
$$\partial_tm_{q+1}+\Div\left(\frac{m_{q+1}\otimes m_{q+1}}{n}\right)+\nabla p(n)+nE_{q+1}+m_{q+1}\times B_{q+1}=\Div(n(R_{q+1}-c_{q+1}\Id)),$$
For convenience, we further denote $\tm_c=\tm-\tm_{p}$.
\subsection{Choice of the weights}
In this section, we will give the weights functions $\gamma_I$ for different $I$.  
\subsubsection{Energy weights}
For $I\in\mathscr{I}_\varphi$, we want to choose proper weights $\gamma_I$ so that  $n^3\varphi_\ell$ can be canceled by the low frequency part of $\frac{1}{2}|\tm_p|^2\tm_p$. After a direct calculation,
\begin{equation}
	\begin{aligned}
		|\tm_p|^2\tm_p&=\sum_{I\in\mathscr{I}}\theta_I^3(t)\chi_I^3(\xi_I)\gamma_I^3M_I^3(\lambda_{q+1}\xi_I)|\tilde{f}_I|^2\tilde{f}_I\\
		&=\underbrace{\sum_{I\in\mathscr{I}}\theta_I^3(t)\chi_I^3(\xi_I)\gamma_I^3\big<M_I^3\big>|\tilde{f}_I|^2\tilde{f}_I}_{=:(|\tm_p|^2\tm_p)_L}+\underbrace{\sum_{I\in\mathscr{I}}\theta_I^3(t)\chi_I^3(\xi_I)\gamma_I^3(M_I^3(\lambda_{q+1}\xi_I)-\big<M_I^3\big>)|\tilde{f}_I|^2\tilde{f}_I}_{=:(|\tm_p|^2\tm_p)_H}.
	\end{aligned}\label{m_p^3}
\end{equation}
And we want to choose proper $\gamma_I$ to make $(|m_p|^2m_p)_L$ be $-2n^3\varphi_\ell$, namely,
\begin{equation}\label{Eq of Gamma_I phi}
	(|\tm_p|^2\tm_p)_L=\sum_{u,\upsilon}\theta_u^6\left(\frac{t}{\tau_q}\right)\chi_\upsilon^6\left(\frac{\xi_u}{\mu_q}\right)\sum_{I\in\mathscr{I}_{u,\upsilon,\varphi}}\gamma_I^3|\tilde{f}_I|^2\tilde{f}_I=-2n^3\varphi_\ell.
\end{equation}
From \eqref{est on Dtl error 0} and \eqref{est on nabla xi and xi^-1}, we have 
$$
\begin{aligned}
	|\tilde{f}_I|=|(\nabla\xi_I)^{-1}f_I|\geqslant\frac{3}{4},\quad\left\|2\lambda_{q+1}^{3\gamma}\delta_{q+1}^{-\frac{3}{2}}n^3\nabla\xi_I\varphi_\ell\right\|_{C_x^0}\leqslant3C_1(n),
\end{aligned}
$$
for sufficiently large $\lambda_0$.
Recalling Lemma \ref{Geometric Lemma II}, we can apply it with $N_0=3C_1(n)$ to the set $\mathcal{F}^{[\upsilon],\varphi}$
and set
\begin{equation}\label{def of Gamma_I phi}
\Gamma_I(t,x)=\Gamma_{f_I}^{\frac{1}{3}}(-2\lambda_{q}^{3\gamma}\delta_{q+1}^{-3/2}(\nabla\xi_I)n^3\varphi_\ell),
\end{equation}
where $\Gamma_{f_I}$ are smooth functions depends on $C_1$.  For different $[\upsilon]\in\Z_3^3$, we just need to use Lemma \ref{Geometric Lemma II} for 27 times. If we choose $\gamma_I$ for $I\in\mathscr{I}_{u,\upsilon,\varphi}$ as
$$
\gamma_I=\frac{\lambda_{q}^{-\gamma}\delta_{q+1}^{\frac{1}{2}}\Gamma_I}{|\tilde{f}_I|^{\frac{2}{3}}},
$$ we have \eqref{Eq of Gamma_I phi}.
\subsubsection{Reynolds weights}
Similarly, we could also decompose $(\tm_p\otimes \tm_p)$ into two parts,
\begin{equation}
	\begin{aligned}
		\tm_p\otimes \tm_p&=\sum_{I\in\mathscr{I}}\theta_I^2\chi_I^2(\xi_I)\gamma_I^2M_I^2(\lambda_{q+1}\xi_I)\tilde{f}_I\otimes \tilde{f}_I\\
		&=\underbrace{\sum_{I\in\mathscr{I}}\theta_I^2\chi_I^2(\xi_I)\gamma_I^2\big<M_I^2\big>\tilde{f}_I\otimes\tilde{f}_I}_{=:(\tm_p\otimes \tm_p)_L}+\underbrace{\sum_{I\in\mathscr{I}}\theta_I^2\chi_I^2(\xi_I)\gamma_I^2(M_I^2(\lambda_{q+1}\xi_I)-\big<M_I^2\big>)\tilde{f}_I\otimes\tilde{f}_I}_{=:(\tm_p\otimes \tm_p)_H}.
	\end{aligned}\label{m_p^2}
\end{equation}
We just need to give $\gamma_I$ for $I\in\mathscr{I}_R$, since we have chosen $\gamma _I$ for $I\in\mathscr{I}_\varphi$. For given $(u,\upsilon)$, we use $I(u,\upsilon)$ to represent the sets of indices $(u^\prime,\upsilon^\prime)$ such that $\max\left\lbrace |u-u^\prime|_\infty,|\upsilon-\upsilon^\prime|_\infty\right\rbrace\leqslant 1 $ (where $|w|_\infty:=\max \left\lbrace|w_1|,|w_2|,|w_3| \right\rbrace$ for any $w\in\R^3$) and rewrite
$$
\begin{aligned}
	(\tm_p\otimes \tm_p)_L&=\sum_{u,\upsilon}\theta_u^6\left(\frac{t}{\tau_q}\right)\chi_\upsilon^6\left(\frac{\xi_u}{\mu_q}\right)\sum_{I\in\mathscr{I}_{u,\upsilon,R}}\gamma_I^2\tilde{f}\otimes\tilde{f} +\sum_{J\in\mathscr{I}_\varphi}\theta_J^2\chi_J^2(\xi_I)\gamma_J^2\big<M_J^2\big>\tilde{f}_J\otimes\tilde{f}_J\\
	&=\sum_{u,\upsilon}\theta_u^6\left(\frac{t}{\tau_q}\right)\chi_\upsilon^6\left(\frac{\xi_u}{\mu_q}\right)\Big[\sum_{I\in\mathscr{I}_{u,\upsilon,R}}\gamma_I^2\tilde{f}\otimes\tilde{f}+\sum_{\substack{J\in\mathscr{I}_{u^\prime,\upsilon^\prime,\varphi}\\
	(u^\prime,\upsilon^\prime)\in I(u.\upsilon)}}\theta_J^2\chi_J^2(\xi_I)\gamma_J^2\big<M_J^2\big>\tilde{f}_J\otimes\tilde{f}_J\Big].
\end{aligned}
$$
Similarly, we choose $\gamma_I$'s such that 
\begin{equation}\label{Eq of tm_p^2}
	(\tm_p\otimes \tm_p)_L=n^2(\delta_{q+1}\Id-R_\ell),
\end{equation}
which can be rewritten as
\begin{equation}\label{Eq of Gamma_I f otimes f}
	\begin{aligned}
\sum_{I\in\mathscr{I}_{u,\upsilon,R}}\gamma_I^2f_I\otimes f_I&=\nabla\xi_I\Big[n^2(\delta_{q+1}\Id-R_\ell)-\sum_{(u^\prime,\upsilon^\prime)\in I(u.\upsilon)}\sum_{J\in\mathscr{I}_{u^\prime,\upsilon^\prime,\varphi}}\theta_J^2\chi_J^2(\xi_I)\gamma_J^2\big<M_J^2\big>\tilde{f}_J\otimes\tilde{f}_J\Big](\nabla\xi_I)^{\top}\\
&=\delta_{q+1}n^2(\Id+\delta_{q+1}^{-1}\mathcal{M}_I),
\end{aligned}
\end{equation}
where
\begin{equation}
	\begin{aligned}
		\mathcal{M}_I=\nabla\xi_I\Big[(\delta_{q+1}\Id-R_\ell)-n^{-2}\sum_{(u^\prime,\upsilon^\prime)\in I(u.\upsilon)}\sum_{J\in\mathscr{I}_{u^\prime,\upsilon^\prime,\varphi}}\theta_J^2\chi_J^2(\xi_I)\gamma_J^2\big<M_J^2\big>\tilde{f}_J\otimes\tilde{f}_J\Big](\nabla\xi_I)^{\top}-\delta_{q+1}\Id.
	\end{aligned}\label{def of M_I}
\end{equation}
So we could choose $\gamma_I$ as $\gamma_I=\delta_{q+1}^{\frac{1}{2}}n\Gamma_I$ for $I\in\mathscr{I}_{u,\upsilon,R}$, and  then
\begin{equation}\label{Eq of Gamma_I R}
	\sum_{I\in\mathscr{I}_{u^\prime,\upsilon^\prime,R}}\Gamma_I^2f_I\otimes f_I=\Id+\delta_{q+1}^{-1}\mathcal{M}_I.
\end{equation}
Recall that $\left\|R_\ell\right\|_0\lesssim \lambda_{q}^{-3\gamma}\delta_{q+1}$ and $\left\|\gamma_J^2\right\|_0\lesssim\lambda_{q}^{-2\gamma}\delta_{q+1},J\in\mathscr{I}_{u^\prime,\upsilon^\prime,\varphi}$. We can apply Lemma \ref{Geometric Lemma I} to the set $\mathcal{F}^{[\upsilon],R}$ and set
$$
\Gamma_I=\Gamma_{f_I}(\Id+\delta_{q+1}^{-1}\mathcal{M}_I).
$$
Similarly, we also apply Lemma \ref{Geometric Lemma I} for 27 times for different $\mathcal{F}^{[\upsilon],R}$. Then, we could achieve \eqref{Eq of tm_p^2}.
\subsection{Some estimates on the perturbation}
Similarly, we represent $M_I,M_I^2,M_I^3$ as Fourier series:
\begin{equation}\label{FS of M_I^k}
	M_I(x)=\sumkO \overset{\circ}{c}_{I,k}e^{ik\cdot x},\quad M_I^2(x)=\overset{\circ}{d}_{I,0}+\sumkO \overset{\circ}{d}_{I,k}e^{ik\cdot x},\quad M_I^3(x)=\overset{\circ}{n}_{I,0}+\sumkO \overset{\circ}{n}_{I,k}e^{ik\cdot x},
\end{equation}
where ${d}_{I,0}=\langle M_I^2\rangle,{n}_{I,0}=\langle M_I^3\rangle$. Noticing that $M_I\in C^\infty(\T^3)$, we could get
\begin{equation}\label{est on sum of cof 2}
\sum_{k \in \mathbb{Z}^3}|k|^{\tilde{n}_0+2}|\overset{\circ}{c}_{I, k}|+\sum_{k \in \mathbb{Z}^3}|k|^{\tilde{n}_0+2}|\overset{\circ}{d}_{I, k}|+\sum_{k \in \mathbb{Z}^3}|k|^{\tilde{n}_0+2}|\overset{\circ}{n}_{I, k}| \lesssim 1, \quad \sum_{k \in \mathbb{Z}^3}|\overset{\circ}{d}_{I, k}|^2 \lesssim 1
\end{equation}
Moreover, $\tm_p,\tm_p\otimes\tm_p,\frac{1}{2}|\tm_p|^2\tm_p$ can be written as:
\begin{align}
&\tm_p=\sum_u\sumkO\delta_{q+1}^{\frac{1}{2}}c_{u,k}e^{i\lambda_{q+1}k\cdot\xi_I},\\
&\tm_p\otimes\tm_p=n^2(\delta_{q+1}\Id-R_\ell)+\sum_u\sumkO\delta_{q+1}d_{u,k}e^{i\lambda_{q+1}k\cdot\xi_I},\label{def of tm_p^2}\\
&\frac{1}{2}|\tm_p|^2\tm_p=-n^3\varphi_\ell+\frac{1}{2}\sum_u\sumkO\delta_{q+1}^{\frac{3}{2}}n_{u,k}e^{i\lambda_{q+1}k\cdot\xi_I},\\
&\frac{1}{2}|\tm_p|^2=-n^2\kappa_\ell+\frac{3}{2}\delta_{q+1}n^2+\frac{1}{2}\sum_u\sumkO\delta_{q+1}\tr (d_{u,k})e^{i\lambda_{q+1}k\cdot\xi_I},
\end{align}
where $\kappa_\ell=\frac{1}{2}\tr R_\ell$ and the coefficients are defined as:
\begin{equation}\begin{split}\label{def of cdn}
		&c_{u,k} = \sum_{I: u_I=u} \theta_I\chi_I(\xi_I) \delta_{q+1}^{-\frac{1}{2}}\gamma_I \overset{\circ}c_{I,k} \tilde f_I ,\\
		&d_{u,k}
		=
		\sum_{I: u_I=u} \theta_I^2\chi_I^2(\xi_I)
		\delta_{q+1}^{-1}\gamma_I^2 \overset{\circ}d_{I,k}
		\tilde f_I\otimes \tilde f_I,\\
		&n_{u,k}
		= \sum_{I: u_I=u} \theta_I^3\chi_I^3(\xi_I)
		\delta_{q+1}^{-\frac 32}\gamma_I^3 \overset{\circ}n_{I,k}
		|\tilde f_I|^2 \tilde f_I.
\end{split}\end{equation}
By the choice of $\theta_{I}$, we know if $|u-u'|> 1$,
$
	\supp_{t,x} (c_{u,k})\bigcap \supp_{t,x} (c_{u',k'})=\supp_{t,x} (d_{u,k})\bigcap \supp_{t,x} (d_{u',k'})=\supp_{t,x} (n_{u,k})\bigcap \supp_{t,x} (n_{u',k'})
	= \emptyset,
$
for any $k, k'\in \Z^3\setminus\{0\}$. Next, we denote $\left\|\cdot\right\|_N=\left\|\cdot\right\|_{(\cI_u;C^N(\T^3))}$ and introduce some estimates for these coefficients:
\begin{pp}\label{est on cof 1}
	For any $0<\alpha<\frac{1}{3}$,  let the parameters $\bar{b}_3(\alpha)$ and $\Lambda_3$ be as in the statement of Proposition \ref{est on backward flow}. For any $b\in(1,\overline{b}_3(\alpha))$, $\lambda_0\geqslant\Lambda_3$, and $I=(u,\upsilon,f)\in\mathscr{I}$, we have
	\begin{align}
		\tau_q^s\left\|\DTL ^sc_{u,k}\right\|_N&\lesssim_{n, p, N, M}\mu_q^{-N}\underset{I}{\max}|\overset{\circ}{c}_{I,k}|,\label{est on Dtl cof of cuk}\\
		\tau_q^s\left\|\DTL ^sd_{u,k}\right\|_N&\lesssim_{n, p, N, M}\mu_q^{-N}\underset{I}{\max}|\overset{\circ}{d}_{I,k}|,\label{est on Dtl cof of duk}\\
		\tau_q^s\left\|\DTL ^sn_{u,k}\right\|_N&\lesssim_{n, p, N, M}\mu_q^{-N}\underset{I}{\max}|\overset{\circ}{n}_{I,k}|,\label{est on Dtl cof of nuk}\\
		\tau_q^{s}|k|\left\|\DTL^s\ta_{u,k}\right\|_N+\tau_q^s\left\|\DTL^sa_{u,k}\right\|_N&\lesssim_{n, p, N, M}\mu_q^{-N}\underset{I}{\max}|\overset{\circ}{b}_{I,k}||k|^{-2},\label{est on Dtl cof of a}\\
		\tau_q^{s}\left\|\DTL^se_{u,k}\right\|_N+\tau_q^s\left\|\DTL^sg_{u,k}\right\|_N&\lesssim_{n, p, N, M}\mu_q^{-N}\underset{I}{\max}|\overset{\circ}{b}_{I,k}||k|^{-1},\label{est on Dtl cof of eg}\\
		\tau_q^{s}\left\|\DTL^ss_{u,k}\right\|_N+\tau_q^s\left\|\DTL ^s\tilde{e}_{u,k}\right\|_N+\tau_q^s\left\|\DTL ^s\tilde{g}_{u,k}\right\|_N&\lesssim_{n, p, N, M}\mu_q^{-N}\underset{I}{\max}|\overset{\circ}{b}_{I,k}|,\label{est on Dtl cof of seg}
	\end{align} 
where $N\geqslant0$ and $s=0,1,2$. Moreover, we have for any $N\geqslant 0$ and $r=0,1,2$, 
	\begin{align}
	\left\|\partial_{t}^rc_{u,k}\right\|_N&\lesssim_{n, p, h, N}\mu_q^{-N-r}\underset{I}{\max}|\overset{\circ}{c}_{I,k}|,\label{est on cof of cuk}\\
	\left\|\partial_{t}^rd_{u,k}\right\|_N&\lesssim_{n, p, h, N}\mu_q^{-N-r}\underset{I}{\max}|\overset{\circ}{d}_{I,k}|,\label{est on cof of duk}\\
	\left\|\partial_{t}^rn_{u,k}\right\|_N&\lesssim_{n, p, h, N}\mu_q^{-N-r}\underset{I}{\max}|\overset{\circ}{c}_{I,k}|,\label{est on cof of nuk}\\
	|k|\left\|\partial_{t}^r\ta_{u,k}\right\|_N+\left\|\partial_{t}^ra_{u,k}\right\|_N&\lesssim_{n, p, h, N}\mu_q^{-N-r}\underset{I}{\max}|\overset{\circ}{b}_{I,k}||k|^{-2},\label{est on cof of a}\\
	\left\|\partial_{t}^re_{u,k}\right\|_N+\left\|\partial_{t}^rg_{u,k}\right\|_N&\lesssim_{n, p, h, N}\mu_q^{-N-r}\underset{I}{\max}|\overset{\circ}{b}_{I,k}||k|^{-1},\label{est on cof of eg}\\
	\left\|\partial_{t}^rs_{u,k}\right\|_N+\left\|\partial_{t}^r\tilde{e}_{u,k}\right\|_N+\left\|\partial_{t}^r\tilde{g}_{u,k}\right\|_N&\lesssim_{n, p, h, N}\mu_q^{-N-r}\underset{I}{\max}|\overset{\circ}{b}_{I,k}|.\label{est on cof of seg}
	\end{align}
	Especially, we could obtain
	\begin{align}
	\left\|\tilde{e}_{u,k}\right\|_0+\left\|\tilde{g}_{u,k}\right\|_0&\lesssim_{n,p,h}(\lambda_{q+1}\mu_q)^{-1}\underset{I}{\max}|\overset{\circ}{b}_{I,k}|,\label{est on cof of tilde eg 0}
	\end{align}
	where the implicit constants are independent of $M$.
\end{pp}
	\begin{proof}
	First, we can easily obtain for $r\geqslant 0$ and $s=0,1,2,$
		\begin{align}
			&\left\|\DTL^{s} \partial_{t}^r\theta_{I}\right\|_{C^{0}(\R)}=\left\|\partial_{t}^{r+s} \theta_{I}\right\|_{C^{0}(\R)} \lesssim_{r} \tau_{q}^{-s-r}.\label{est on theta_I} 
		\end{align}
	For smooth functions $F=F(x)$ and $g=g(t,x)$, we have
	\begin{equation}
		\begin{aligned}
			&\left\|\DTL F(g)\right\|_{N} \lesssim \sum_{N_{1}+N_{2}=N}\left\|\DTL g\right\|_{N_{1}}\left\|(\nabla F)(g)\right\|_{N_{2}} \text {, }\\
			&\left\|\DTL^{2} F(g)\right\|_{N} \lesssim \sum_{N_{1}+N_{2}=N}\left\|\DTL^{2} g\right\|_{N_{1}}\left\|(\nabla F)(g)\right\|_{N_{2}}+\left\|\DTL g \otimes \DTL g\right\|_{N_{1}}\left\|(\nabla^{2} F)(g)\right\|_{N_{2}}.
		\end{aligned}
	\end{equation}	
	Moreover, we could calcualte
	\begin{align*}
	\partial_{i} (F(g))&=(\nabla F)(g)\partial_{i} g,\quad \partial_{ij}(F(g))=(\nabla^2 F)(g):(\partial_{i}g\otimes\partial_{j}g)+(\nabla F)(g)\partial_{ij}g,\\	
	\partial_{ijk} (F(g))&=(\nabla^3 F)(g):(\partial_{i}g\otimes\partial_{j}g\otimes\partial_{k}g)+(\nabla^2 F)(g):(\partial_k(\partial_{i}g\otimes\partial_{j}g)+\partial_i(\partial_{j}g\otimes\partial_{k}g)+\partial_j(\partial_{i}g\otimes\partial_{k}g))\\
	&\quad+(\nabla F)(g)\partial_{ijk}g.
	\end{align*}
	where $\partial_{i},\partial_{j},\partial_{k}$ means taking the derivative with respect to time or space. By using them, we could get the  estimates on items like $\DTL^s\partial_{t}^k\nabla^iF(g)$. For convenience, we will restrict the range of parameters to $s=0,1,2;k=0,1,2;i=0,1;r=0,1,2,3,4$ in the proof, without mentioning it further. By using $\DTL ^s\chi_{I}(\xi_I)=0 $, \eqref{est on mixed nabla xi 1}, \eqref{est on mixed xi 1}, \eqref{est on mixed nabla xi 2}, and \eqref{est on mixed xi 2}, we could obtain 
	\begin{equation}\label{est on chi_I}
		\begin{aligned}
		\left\|\partial_{t}^r\left(\chi_{I}\left(\xi_{I}\right)\right)\right\|_{N} &\lesssim_{n, p, h, N} \mu_q^{-N-r},\quad\left\|D_{t,\ell}^s\partial_{t}^k\nabla^i\left(\chi_{I}\left(\xi_{I}\right)\right)\right\|_{N}+\left\|D_{t,\ell}^s\nabla^{k+i}\left(\chi_{I}\left(\xi_{I}\right)\right)\right\|_{N}\lesssim_{n, p, N, M} \mu_q^{-N-k-i}(\lambda_{q}\delta_{q}^{\frac{1}{2}})^{s}.
	\end{aligned}
	\end{equation}
	Recall that 
		$$
		\gamma_I=\left\lbrace
		\begin{aligned}
			\frac{\lambda_q^{-\gamma}\delta_{q+1}^{\frac{1}{2}}\Gamma_{f_I}^{\frac{1}{3}}(-2\lambda_{q+1}^{3\gamma}\delta_{q+1}^{-3/2}\nabla\xi_In^3\varphi_\ell)}{|\tilde{f}_I|^{\frac{2}{3}}}, \qquad&I\in\mathscr{I}_\varphi,\\
			\delta_{q+1}^{\frac{1}{2}}n\Gamma_I^{\frac{1}{2}}(\Id+\delta_{q+1}^{-1}\mathcal{M}_I),\qquad\qquad\qquad &I\in\mathscr{I}_R.
		\end{aligned}
		\right.
		$$
		By using \eqref{est on mixed nabla xi 1}, \eqref{est on mixed xi 1}, \eqref{est on mixed nabla xi 2},  \eqref{est on mixed xi 2}, $|(\nabla \xi_{I})^{-1} f_{I}|\geqslant \frac{3}{4}$ and $\Gamma_{f_{I}}\geqslant 3$, we could immediately obtain
		\begin{align}
		\left\|\partial_{t}^r\tilde{f}_I\right\|_{N}+\left\|\partial_{t}^r|(\nabla \xi_{I})^{-1} f_{I}|^{-\frac{2}{3}}\right\|_{N}&\lesssim_{n, p, h, N} \ell^{-N-r},\notag\\
		\left\|\DTL^{s}\partial_{t}^k\nabla^i\tilde{f}_I\right\|_{N}+\left\|\DTL^{s}\nabla^{k+i}\tilde{f}_I\right\|_{N}&\lesssim_{n, p, N, M} \ell^{-N-k-i}(\lambda_{q}\delta_{q}^{\frac{1}{2}})^s\label{est on mixed tf},\\
		\left\|\DTL^{s}\partial_{t}^k\nabla^i|(\nabla \xi_{I})^{-1} f_{I}|^{-\frac{2}{3}}\right\|_{N}+\left\|\DTL^{s}\nabla^{k+i}|(\nabla \xi_{I})^{-1} f_{I}|^{-\frac{2}{3}}\right\|_{N}&\lesssim_{n, p, N, M} \ell^{-N-k-i}(\lambda_{q}\delta_{q}^{\frac{1}{2}})^s.\notag
		\end{align}
		From \eqref{est on Dtl error 0}, \eqref{est on Dtl error N}, \eqref{est on Dtl nabla xi}, \eqref{est on mixed error 2}, \eqref{est on mixed nabla xi 2}, and \eqref{est on pt error}, we could get
		\begin{equation}\label{est on mixed Gamma}
		\begin{aligned}
		\left\|\partial_{t}^r(2 \lambda_{q}^{3 \gamma} \delta_{q+1}^{-\frac{3}{2}}\nabla \xi_{I} n^{3} \varphi_{\ell})\right\|_{N} &\lesssim_{n, p, h, N} \ell^{-N}(\ell_{t}^{-1}+\ell^{-1})^r,\\
		\left\|\DTL^{s}\partial_{t}^k\nabla^i(2 \lambda_{q}^{3 \gamma} \delta_{q+1}^{-\frac{3}{2}}\nabla \xi_{I} n^{3} \varphi_{\ell})\right\|_{N}+	\left\|\DTL^{s}\nabla^{k+i}(2 \lambda_{q}^{3 \gamma} \delta_{q+1}^{-\frac{3}{2}}\nabla \xi_{I} n^{3} \varphi_{\ell})\right\|_{N}
		&\lesssim_{n, p, N, M}  \ell^{-N-i}(\ell_{t}^{-1}+\ell^{-1})^k\tau_{q}^{-s}.
		\end{aligned}
		\end{equation}
	 	Up to now, we are ready to get for $I\in\mathscr{I}_{\varphi}$, 
		\begin{equation}\label{est on gamma_I phi}
		\left\|\partial_{t}^{r} \gamma_{I}\right\|_{N}\lesssim_{n, p, h, N}\delta_{q+1}^{\frac{1}{2}}\ell^{-N}(\ell_{t}^{-1}+\ell^{-1})^r,\quad
		\left\|\DTL^{s}\partial_{t}^{k}\nabla^i \gamma_{I}\right\|_{N}+\left\|\DTL^{s}\nabla^{k+i} \gamma_{I}\right\|_{N}  \lesssim_{n, p, N, M}  \delta_{q+1}^{\frac{1}{2}} \ell^{-N-i}(\ell_{t}^{-1}+\ell^{-1})^k\tau_{q}^{-s}.
		\end{equation}
		But for $I\in\mathscr{I}_R$, we need to give the estimates on $\mathcal{M}_I$ defined in \eqref{def of M_I} first,
		$$
		\begin{aligned}
		\left\|\mathcal{M}_{I}\right\|_{N} &\lesssim_{N} \delta_{q+1}\left\|\left(\nabla \xi_{I}\right)\left(\nabla \xi_{I}\right)^{\top}-\operatorname{Id}\right\|_{N}+\sum_{N_{1}+N_{2} 
		+N_{3}=N}\left\|\nabla \xi_{I}\right\|_{N_{1}}\left\|R_{\ell}\right\|_{N_{2}}\left\|\nabla \xi_{I}\right\|_{N_{3}}\\
		&\quad+\left\|n^{-2}\right\|_{N} \sum_{N_{1}+N_{2}+N_{3}+N_4=N}\sum_{(u^\prime,\upsilon^\prime)\in I(u.\upsilon)}\sum_{J\in\mathscr{I}_{u^\prime,\upsilon^\prime,\varphi}}\left\|\nabla \xi_{I}\right\|_{N_{1}}\left\|\chi_{J}^{2}\left(\xi_{J}\right)\right\|_{N_{2}}\left\|\gamma_{J}^{2}\right\|_{N_{3}}\left\|\nabla \xi_{I}\right\|_{N_{4}} \\
		&\lesssim_{n, p, h, N} \mu_q^{-N}\delta_{q+1}.
		\end{aligned}
		$$
		where we used \eqref{est on Dtl error N}, \eqref{est on nabla xi and xi^-1}, \eqref{est on nabla xi-Id}, \eqref{est on chi_I}, and \eqref{est on gamma_I phi}. Similarly, we can also achieve
		\begin{align*}
		\left\|\partial_{t}^{r} \mathcal{M}_{I}\right\|_{N} \lesssim_{n, p, h, N}\mu_q^{-N-r}\delta_{q+1},\quad
		\left\|\DTL^{s}\partial_{t}^{k}\nabla^i \mathcal{M}_{I}\right\|_{N}+\left\|\DTL^{s}\nabla^{k+i} \mathcal{M}_{I}\right\|_{N} \lesssim_{n, p, N, M}   \mu_q^{-N-k-i}\tau_{q}^{-s}\delta_{q+1}.
		\end{align*}
		Then, for $I\in\mathscr{I}_R$, we have
		\begin{equation}\label{est on gamma_I R}
		\left\|\partial_{t}^{r} \gamma_{I}\right\|_{N}\lesssim_{n, p, h, N}	\mu_q^{-N-r}\delta_{q+1}^{\frac{1}{2}},\quad
		\left\|\DTL^{s}\partial_{t}^{k}\nabla^i\gamma_{I}\right\|_{N}+\left\|\DTL^{s}\nabla^{k+i} \gamma_{I}\right\|_{N}  \lesssim_{n, p, N, M}   \mu_q^{-N-k-i}\tau_{q}^{-s}\delta_{q+1}^{\frac{1}{2}}.
		\end{equation}
		The implicit constants in
		these inequality can be chosen to be independent of $I$ because of the finite cardinality of
		$f_I$. Moreover, the implicit constants in the first inequality of \eqref{est on chi_I}--\eqref{est on gamma_I R} can be chosen independent $M$ because of the fact that the implicit constants in \eqref{est on pt xi}, \eqref{est on pt nabla xi}, \eqref{est on pt det nabla xi}, and \eqref{est on pt error} are also independent of $M$. Based on the definition of $\tilde{c}_{u,k},\tilde{d}_{u,k},\tilde{n}_{u,k}$ in \eqref{def of cdn}, we could use \eqref{est on theta_I}, \eqref{est on chi_I}, \eqref{est on gamma_I phi}, and \eqref{est on gamma_I R} to obtain \eqref{est on Dtl cof of cuk}--\eqref{est on Dtl cof of nuk} and \eqref{est on cof of cuk}--\eqref{est on cof of nuk}. Moreover, we have
		\begin{equation}\label{est on B_Ik}
		\begin{aligned}
		\left\|\partial_{t}^{r} B_{I,k}\right\|_{N}\lesssim_{n, p, h, N}\mu_q^{-N-r}\underset{I}{\max}|\overset{\circ}{b}_{I,k}|,\quad
		\left\|\DTL^{s}\partial_{t}^{k}\nabla^iB_{I,k}\right\|_{N}+\left\|\DTL^{s}\nabla^{k+i} B_{I,k}\right\|_{N}  \lesssim_{n, p, N, M}   \mu_q^{-N-k-i}\tau_{q}^{-s}\underset{I}{\max}|\overset{\circ}{b}_{I,k}|.
		\end{aligned}
	\end{equation}
	Combining it with \eqref{est on nabla xi and xi^-1}--\eqref{est on Dtl nabla xi}, and choosing parameters in Lemma \ref{New building blocks} as $\tau=\tau_{q},\ \ell_*=\ell,\ \mu=\mu_q,\ \tlm=\lambda_{q+1},\ a=B_{I,k},\ \upsilon=m_\ell/n$, we can immediately verify that the assumption \eqref{assume in building blocks} in Lemma \ref{New building blocks} holds. Furthermore, \eqref{est on Dtl cof of a}--\eqref{est on Dtl cof of seg} and \eqref{est on cof of a}--\eqref{est on cof of tilde eg 0}  can be deduced from \eqref{est on cof N+r 1}--\eqref{est on cof N+s} in Lemma \ref{New building blocks}.
	\end{proof}
The following estimates on the perturbation are a direct consequence.
\begin{pp}\label{est on perturbation}
	For any $0<\alpha<\frac{1}{3}$,  let the parameters $\bar{b}_3(\alpha)$ and $\Lambda_3$ be as in the statement of Proposition \ref{est on backward flow}, and let $\|\cdot\|_N=\|\cdot\|_{C^0(\cI^q;C^N(\T^3))}$.  Then, for any $b\in(1,\overline{b}_3(\alpha))$, $\lambda_0\geqslant\Lambda_3$, and $I=(u,\upsilon,f)\in\mathscr{I}$, we have
	\begin{align}	
	\lambda_{q+1}\left\|\partial_{t} ^r\tcA\right\|_N+\left\|\partial_{t} ^r\tilde{A}\right\|_N&\lesssim_{n, p, h}\lambda_{q+1}^{N+r-2}\delta_{q+1}^{\frac{1}{2}},\label{est on A}\\
	\left\|\partial_{t} ^r\tE_p\right\|_N+\left\|\partial_{t} ^r\tB_p\right\|_N&\lesssim_{n, p, h}\lambda_{q+1}^{N+r-1}\delta_{q+1}^{\frac{1}{2}},\label{est on E,B}\\
	\left\|\partial_{t} ^r\tm_{EB}\right\|_N+\left\|\partial_{t} ^r\tm_p\right\|_N&\lesssim_{n, p, h}\lambda_{q+1}^{N+r}\delta_{q+1}^{\frac{1}{2}},\label{est on tm_EBp}\\	
	\left\|\partial_{t} ^r(\tm_{EB}-\tm_p)\right\|_N&\lesssim_{n, p, h}\lambda_{q+1}^{N+r}(\lambda_{q+1}\mu_q)^{-1}\delta_{q+1}^{\frac{1}{2}},\label{est on tm_EB-tm_p}\\
	\left\|\partial_{t}^s\tm\right\|_N&\lesssim_{n, p, h}\lambda_{q+1}^{N+s}\delta_{q+1}^{\frac{1}{2}},\label{est on tm}\\
	\left\|\partial_{t}^s\tm_c\right\|_N&\lesssim_{n, p, h}\lambda_{q+1}^{N+s}(\lambda_{q+1}\mu_q)^{-1}\delta_{q+1}^{\frac{1}{2}},\label{est on tm_c}\\
	\left\|\partial_{t} ^s\tE\right\|_N+\left\|\partial_{t} ^s\tB\right\|_N&\lesssim_{n, p, h}\lambda_{q+1}^{N+s-1}\delta_{q+1}^{\frac{1}{2}},\label{est on tEB}
	\end{align}
where $0\leqslant N+r\leqslant 2$. Moreover, we could have the following estimates on advective derivatives, 
\begin{align}	
	\lambda_{q+1}\left\|\DTL ^s\tcA\right\|_N+\left\|\DTL ^s\tilde{A}\right\|_N&\lesssim_{n, p, h, M}\lambda_{q+1}^{N-2}\tau_{q}^{-s}\delta_{q+1}^{\frac{1}{2}},\label{est on Dtl A}\\
	\left\|\DTL ^s\tE_p\right\|_N+\left\|\DTL ^s\tB_p\right\|_N&\lesssim_{n, p, h, M}\lambda_{q+1}^{N-1}\tau_{q}^{-s}\delta_{q+1}^{\frac{1}{2}},\label{est on Dtl E,B}\\
	\left\|\DTL ^s\tm_{EB}\right\|_N+\left\|\DTL ^s\tm_p\right\|_N&\lesssim_{n, p, h, M}\lambda_{q+1}^{N}\tau_{q}^{-s}\delta_{q+1}^{\frac{1}{2}},\label{est on Dtl tm_EBp}\\
	\left\|\DTL ^s(\tm_{EB}-\tm_p)\right\|_N&\lesssim_{n, p, h, M}\lambda_{q+1}^{N}\tau_{q}^{-s}\delta_{q+1}^{\frac{1}{2}},\label{est on Dtl tm_EB-tm_p}\\
	\left\|\DTL ^s\tm\right\|_N&\lesssim_{n, p, h, M}\lambda_{q+1}^N\tau_{q}^{-s}\delta_{q+1}^{\frac{1}{2}},\label{est on Dtl tm}\\
	\left\|\DTL ^s\tm_c\right\|_N&\lesssim_{n, p, h, M}\lambda_{q+1}^{N}\tau_{q}^{-s}\delta_{q+1}^{\frac{1}{2}},\label{est on Dtl tm_c}\\
	\left\|\DTL ^s\tE\right\|_N+\left\|\DTL ^s\tB\right\|_N&\lesssim_{n, p, h, M}\lambda_{q+1}^{N-1}\tau_{q}^{-s}\delta_{q+1}^{\frac{1}{2}},\label{est on Dtl tEB}
\end{align}
where $s=0,1,2$ and $0\leqslant N\leqslant \tilde{n}_0+1$. 
\end{pp}
\begin{proof}
Noticing that $\DTL  (e^{i\lambda_{q+1}k\cdot\xi_{I}})=0$, and combining it with \eqref{est on sum of cof 1}, \eqref{est on sum of cof 2}, and the estimates on those coefficients in Proposition \ref{est on cof 1}, we could get \eqref{est on A}--\eqref{est on tm_EBp} and \eqref{est on Dtl A}--\eqref{est on Dtl tm_EBp} from their definition. In order to get \eqref{est on tm_EB-tm_p} and \eqref{est on Dtl tm_EB-tm_p}, we first calculate
$$
\begin{aligned}
\tm_{EB}-\tm_p=\delta_{q+1}^{\frac{1}{2}}\sum_u\sumkO\sum_{I:u_I=u}
\left(\cm_{k}(f,B_{I,k},m_\ell/n,\lambda_{q+1})-\cm_{p,k}(f,B_{I,k},m_\ell/n,\lambda_{q+1},\tau_q u,2\pi\mu_q\upsilon)\right).
\end{aligned}
$$
Notice that
$$
\begin{aligned}
&\quad\left(k\cdot(m_\ell/n)(\tau_q u,2\pi\mu_q\upsilon)\right)^2-\frac{(k\cdot\partial_{t}\xi_{I})^2}{(\mathrm{det}(\nabla\xi_I))^2}\\
&=\left(k\cdot(m_\ell/n)(\tau_q u,2\pi\mu_q\upsilon)\right)^2-\left(k\cdot(m_\ell/n)(t,x)\right)^2+\left(k\cdot(m_\ell/n)(t,x)\right)^2-\left(k\cdot\partial_{t}\xi_{I}\right)^2+\left(k\cdot\partial_{t}\xi_{I}\right)^2-\frac{(k\cdot\partial_{t}\xi_{I})^2}{(\mathrm{det}(\nabla\xi_I))^2}.
\end{aligned}
$$
Then, we can use \eqref{geomoetry condition 1} to get
$$
\begin{aligned}
\left\|(m_\ell/n)(\tau_q u,2\pi\mu_q\upsilon)-(m_\ell/n)(t,x)\right\|_0&\leqslant2\mu_q\left\|\nabla (m_\ell/n)\right\|_0+2\tau_{q}\left\|\partial_{t}(m_\ell/n)\right\|_0\lesssim_n (\lambda_{q+1}\mu_q)^{-1}.
\end{aligned}
$$
By combining it with \eqref{est on pt m_l / n N}, \eqref{est on nabla xi-Id}, \eqref{est on pt xi + m_l/n}, and \eqref{est on pt xi}, we have
$$
\begin{aligned}
&\quad\left\|\left(k\cdot(m_\ell/n)(\tau_q u,2\pi\mu_q\upsilon)\right)^2-\frac{(k\cdot\partial_{t}\xi_{I})^2}{(\mathrm{det}(\nabla\xi_I))^2}\right\|_0\\
&=\left\|\left(k\cdot(m_\ell/n)(\tau_q u,2\pi\mu_q\upsilon)\right)^2-\left(k\cdot(m_\ell/n)(t,x)\right)^2\right\|_0+\left\|\left(k\cdot(m_\ell/n)(t,x)\right)^2-\left(k\cdot\partial_{t}\xi_{I}\right)^2\right\|_0+\left\|\left(k\cdot\partial_{t}\xi_{I}\right)^2-\frac{(k\cdot\partial_{t}\xi_{I})^2}{(\mathrm{det}(\nabla\xi_I))^2}\right\|_0\\
&\lesssim_{n, p, h} (\lambda_{q+1}\mu_q)^{-1}.
\end{aligned}
$$
Then, we can use \eqref{est on m_pk 1} and \eqref{est on m_pk 2} in Lemma \ref{New building blocks} to obtain \eqref{est on tm_EB-tm_p} and \eqref{est on Dtl tm_EB-tm_p}, if we choose parameters as $\tau=\tau_{q},\ \ell_*=\ell,\ \mu=\mu_q,\ \tlm=\lambda_{q+1},\ a=B_{I,k},\ \upsilon=m_\ell/n$.\par
Next, we consider the estimates on $\tm_t$ and $\tE_t$. Recall \eqref{Eq of tE_t}, by using \eqref{ode ypp}, we can get the following estimate:
\begin{equation}\notag
	\begin{aligned}
		\left\|\tE_t\right\|_{C^2(\cI^q;\R^3)}&\lesssim_{n}\left\|\int_{\T^3}\left(n \tE_p+\tm_{EB}\times B_{q}+m_{q}\times \tB+\tm_{EB}\times \tB \right)\rd x\right\|_{C^0(\cI^q;\R^3)}\\
		&\lesssim_{n}\left\|\int_{\T^3}\left(n \tE_p+m_{q}\times \tB+\tm_{EB}\times \tB \right)\rd x\right\|_{C^0(\cI^q;\R^3)}+\left\|\int_{\T^3}\left(\tm_{EB}\times B_{q}\right)\rd x\right\|_{C^0(\cI^q;\R^3)}.
	\end{aligned}
\end{equation}
First, we could use Lemma \ref{est on int operator} to  obtain
\begin{align*}
	&\quad\left\|\int_{\T^3}\left(n \tE_p +m_{q}\times \tB+\tm_{EB}\times \tB\right)\rd x\right\|_{C^0(\cI^q;\R^3)}\\
	&=\frac{\delta_{q+1}^{\frac{1}{2}}}{\lambda_{q+1}}\left\|\int_{\T^3}\sum_u\sumkO (ne_{u,k}+m_q\times g_{u,k})e^{i\lambda_{q+1}k\cdot\xi_{I}}\rd x\right\|_{C^0(\cI^q;\R^3)}+\left\|\int_{\T^3}\left(\tm_{EB}\times \tB\right)\rd x\right\|_{C^0(\cI^q;\R^3)}\\
	&\lesssim\frac{\delta_{q+1}^{\frac{1}{2}}}{\lambda_{q+1}}\sum_u\sumkO\frac{\|ne_{u,k}+m_q\times g_{u,k}\|_1+\|ne_{u,k}+m_q\times g_{u,k}\|_0\|\nabla\xi_{I}\|_1}{\lambda_{q+1}|k|}+\left\|\tm_{EB}\times \tB\right\|_{C^0(\cI^q;\R^3)}\\
	&\lesssim_{n, p, h}\frac{M\delta_{q+1}^{\frac{1}{2}}}{\lambda_{q+1}(\lambda_{q+1}\mu_q)}+\frac{\delta_{q+1}}{\lambda_{q+1}}\lesssim_{n, p, h}\frac{M\delta_{q+1}^{\frac{1}{2}}}{\lambda_{q+1}(\lambda_{q+1}\mu_q)}\lesssim_{n, p, h}\frac{\delta_{q+1}^{\frac{1}{2}}}{\lambda_{q+1}},
\end{align*}
where we have used \eqref{est on m_q}, \eqref{est on nabla xi and xi^-1}, \eqref{pp of Lambda_3}, and \eqref{est on cof of eg}. To get the estimate for the second term, we calculate 
$$
\begin{aligned}
\tm_{EB}\times B_{q}&=-\left(\nabla\times(\partial_{tt}\mcA+\tB)\right)\times B_q=\left(\nabla\left(\partial_{tt}\tcA+\tB\right)\right)^{\top}B_q-(B_q\cdot\nabla)\left(\partial_{tt}\tcA+\tB\right)\\
	&=\nabla\left(\left(\partial_{tt}\tcA+\tB\right)\cdot B_q\right)-\Div\left(\left(\partial_{tt}\tcA+\tB\right)\otimes B_q\right)-(\nabla B_q)^{\top}\left(\partial_{tt}\tcA+\tB\right),
\end{aligned}
$$
where we have used $\Div B_q=0$. Then, we have
$$
\int_{\T^3}\left(\tm_{EB}\times B_{q}\right)\rd x=-\int_{\T^3}(\nabla B_q)^{\top}\left(\partial_{tt}\tcA+\tB\right)\rd x,
$$
and we could also use Lemma \ref{est on int operator} to get
\begin{align*}
&\quad\left\|\int_{\T^3}\left(\tm_{EB}\times B_{q}\right)\rd x\right\|_{C^0(\cI^q;\R^3)}\\
&=\left\|\int_{\T^3}(\nabla B_q)^{\top}\left(\partial_{tt}\tcA+\tB\right)\rd x\right\|_{C^0(\cI^q;\R^3)}\\
&=\frac{\delta_{q+1}^{\frac{1}{2}}}{\lambda_{q+1}}\left\|\int_{\T^3}\sum_u\sumkO(\nabla B_q)^{\top}(\lambda_{q+1}^{-2}\partial_{tt}\ta_{u,k}+2i\lambda_{q+1}^{-1}(k\cdot\partial_{t}\xi_{I})\partial_{t}\ta_{u,k})e^{i\lambda_{q+1}k\cdot\xi_{I}}\rd x\right\|_{C^0(\cI^q;\R^3)}\\
&\quad+\frac{\delta_{q+1}^{\frac{1}{2}}}{\lambda_{q+1}}\left\|\int_{\T^3}\sum_u\sumkO(\nabla B_q)^{\top}((i\lambda_{q+1}^{-1}(k\cdot\partial_{tt}\xi_{I})-(k\cdot\partial_{t}\xi_{I})^2)\ta_{u,k}+g_{u,k})e^{i\lambda_{q+1}k\cdot\xi_{I}}\rd x\right\|_{C^0(\cI^q;\R^3)}\\
&\lesssim\frac{\delta_{q+1}^{\frac{1}{2}}}{\lambda_{q+1}}\sum_u\sumkO\frac{\left\|(\nabla B_q)^{\top}(\lambda_{q+1}^{-2}\partial_{tt}\ta_{u,k}+2i\lambda_{q+1}^{-1}(k\cdot\partial_{t}\xi_{I})\partial_{t}\ta_{u,k}+(i\lambda_{q+1}^{-1}(k\cdot\partial_{tt}\xi_{I})-(k\cdot\partial_{t}\xi_{I})^2)\ta_{u,k}+g_{u,k})\right\|_1}{\lambda_{q+1}|k|}\\
&\quad +\frac{\delta_{q+1}^{\frac{1}{2}}}{\lambda_{q+1}}\sum_u\sumkO\frac{\left\|(\nabla B_q)^{\top}(\lambda_{q+1}^{-2}\partial_{tt}\ta_{u,k}+2i\lambda_{q+1}^{-1}(k\cdot\partial_{t}\xi_{I})\partial_{t}\ta_{u,k}+(i\lambda_{q+1}^{-1}(k\cdot\partial_{tt}\xi_{I})-(k\cdot\partial_{t}\xi_{I})^2)\ta_{u,k}+g_{u,k})\right\|_0\|\nabla\xi_{I}\|_1}{\lambda_{q+1}|k|}\\
&\lesssim_{n, p, h}\frac{M\delta_{q+1}^{\frac{1}{2}}}{\lambda_{q+1}(\lambda_{q+1}\mu_q)}\lesssim_{n, p, h}\frac{\delta_{q+1}^{\frac{1}{2}}}{\lambda_{q+1}},
\end{align*}
where we have used \eqref{est on nabla xi and xi^-1}, \eqref{est on pt nabla xi}, \eqref{pp of Lambda_3}, \eqref{est on cof of a}, and \eqref{est on cof of eg}. Similarly, we could get
\begin{equation}\notag\label{Eq of pt tE_t}
	\partial_{tt}(\partial_{t}\tE_t)+\left(\int_{\T^3}n\rd x\right)\partial_{t}\tE_t=-\left(\int_{\T^3}\partial_{t}n\rd x\right)\tE_t-\int_{\T^3}\partial_{t}\left(n \tE_p+\tm_{EB}\times B_{q}+m_{q}\times \tB+\tm_{EB}\times \tB \right)\rd x,
\end{equation}
and  then
\begin{align*}
		\left\|\tE_t\right\|_{C^3(\cI^q;\R^3)}&\lesssim_{n}\left\|\int_{\T^3}\partial_{t}\left(n \tE_p+\tm_{EB}\times B_{q}+m_{q}\times \tB+\tm_{EB}\times \tB \right)\rd x\right\|_{C^0(\cI^q;\R^3)}+\left\|\tE_t\right\|_{C^0(\cI^q;\R^3)}\\
		&\lesssim_{n}\lambda_{q+1}\left\|\int_{\T^3}\left(n \tE_p+(\nabla B_q)^{\top}\left(\partial_{tt}\tcA+\tB\right)+m_{q}\times \tB+\tm_{EB}\times \tB \right)\rd x\right\|_{C^0([0,T];\R^3)}\\&\lesssim_{n, p, h}\frac{M\delta_{q+1}^{\frac{1}{2}}}{\lambda_{q+1}\mu_q}\lesssim_{n, p, h}
		\delta_{q+1}^{\frac{1}{2}}.
	\end{align*}
Up to now, we could obtain
\begin{align}
	&\left\|\tE_t\right\|_{C^2(\cI^q;\R^3)}
	\lesssim_{n, p, h}\frac{M\delta_{q+1}^{\frac{1}{2}}}{\lambda_{q+1}(\lambda_{q+1}\mu_q)}\lesssim_{n, p, h}\frac{\delta_{q+1}^{\frac{1}{2}}}{\lambda_{q+1}},\quad\left\|\tE_t\right\|_{C^3(\cI^q;\R^3)}
	\lesssim_{n, p, h}\frac{M\delta_{q+1}^{\frac{1}{2}}}{\lambda_{q+1}\mu_q}\lesssim_{n, p, h}\delta_{q+1}^{\frac{1}{2}}, \label{est on tE_t}\\
	&\left\|\tm_t\right\|_{C^1(\cI^q;\R^3)}
	\lesssim_{n, p, h}\frac{M\delta_{q+1}^{\frac{1}{2}}}{\lambda_{q+1}(\lambda_{q+1}\mu_q)}\lesssim_{n, p, h}\frac{\delta_{q+1}^{\frac{1}{2}}}{\lambda_{q+1}},\quad\left\|\tm_t\right\|_{C^2(\cI^q;\R^3)}
	\lesssim_{n, p, h}\frac{M\delta_{q+1}^{\frac{1}{2}}}{\lambda_{q+1}\mu_q}\lesssim_{n, p, h}\delta_{q+1}^{\frac{1}{2}}. \label{est on tm_t}
\end{align} 
Finally, we could obtain \eqref{est on tm}--\eqref{est on tEB} and \eqref{est on Dtl tm}--\eqref{est on Dtl tEB}, 
\begin{align*}
	\left\|\partial_{t}^s\tm\right\|_N&\leqslant\left\|\partial_{t} ^s\tm_{EB}\right\|_N+\left\|\partial_{t} ^s\tm_{t}\right\|_N\lesssim_{n, p, h}\lambda_{q+1}^{N+s}\delta_{q+1}^{\frac{1}{2}},\\
	\left\|\partial_{t}^s\tm_c\right\|_N&\leqslant\left\|\partial_{t} ^s(\tm_{EB}-\tm_p)\right\|_N+\left\|\partial_{t} ^s\tm_{t}\right\|_N\lesssim_{n, p, h}\lambda_{q+1}^{N+s}(\lambda_{q+1}\mu_q)^{-1}\delta_{q+1}^{\frac{1}{2}},\label{est on pt tm_c}\\
	\left\|\partial_{t} ^s\tE\right\|_N&\leqslant\left\|\partial_{t} ^s\tE_p\right\|_N+\left\|\partial_{t} ^s\tE_t\right\|_N\lesssim_{n, p, h}\lambda_{q+1}^{N+s-1}\delta_{q+1}^{\frac{1}{2}},\\
	\left\|\partial_{t} ^s\tB\right\|_N&\leqslant\left\|\partial_{t} ^s\tB_p\right\|_N\lesssim_{n, p, h}\lambda_{q+1}^{N+s-1}\delta_{q+1}^{\frac{1}{2}},
\end{align*}
for $0\leqslant N+r\leqslant 2$, and
\begin{align*}
	\left\|\DTL ^s\tm\right\|_N&\leqslant\left\|\DTL ^s\tm_{EB}\right\|_N+\left\|\DTL ^s\tm_{t}\right\|_N\lesssim_{n, p, h, M}\lambda_{q+1}^N\tau_{q}^{-s}\delta_{q+1}^{\frac{1}{2}},\\
	\left\|\DTL ^s\tm_c\right\|_N&\leqslant\left\|\DTL ^s(\tm_{EB}-\tm_p)\right\|_N+\left\|\DTL ^s\tm_{t}\right\|_N\lesssim_{n, p, h, M}\lambda_{q+1}^{N}\tau_{q}^{-s}\delta_{q+1}^{\frac{1}{2}},\\
	\left\|\DTL ^s\tE\right\|_N&\leqslant\left\|\DTL ^s\tE_p\right\|_N+\left\|\DTL ^s\tE_t\right\|_N\lesssim_{n, p, h, M}\lambda_{q+1}^{N-1}\tau_{q}^{-s}\delta_{q+1}^{\frac{1}{2}},\\
	\left\|\DTL ^s\tB\right\|_N&\leqslant\left\|\DTL ^s\tB_p\right\|_N\lesssim_{n, p, h, M}\lambda_{q+1}^{N-1}\tau_{q}^{-s}\delta_{q+1}^{\frac{1}{2}},
\end{align*}
for $s=0,1,2;0\leqslant N\leqslant \tilde{n}_0+1$.
\end{proof}
\section{Definition of the new errors}\label{Definition of the new errors}   
\subsection{New Reynolds stress}
We have constructed the perturbation $\tm$ and $m_{q+1}=m_q+\tm$.
Here, we will give new Reynolds stress $R_{q+1}$ and current $\varphi_{q+1}$ in the Euler-Maxwell-Reynolds system. We will use the inverse divergence operator $\cR$ to define the new error $R_{q+1}$ and the new current $\varphi_{q+1}$. More details about the inverse divergence operator can be found in Appendix \ref{Inverse divergence operator}. We can now define $R_{q+1}$ as follows:
$$
\begin{aligned}
\Div(nR_{q+1})&=\underbrace{ n \DTL  \frac{\tm_{EB}}{n} - \Div(m_q-m_\ell)\frac{\tm_{EB}}{ n}}_{=\nabla \cdot  ( n R_T)} 
+ \underbrace{\Div\left(\frac{\tm\otimes \tm}{ n} +  n (R_\ell - \delta_{q+1}   \Id) \right)}_{=:\nabla \cdot ( n  R_O)}+\underbrace{ \Div \left(\frac{m_q \otimes \tm_t}{n}+\frac{\tm_t \otimes m_q}{n}\right)}_{=\nabla \cdot  ( n R_t)}\\
&\quad+\underbrace{  (\tm_{EB}\cdot \nabla) \frac{m_\ell}{ n} }_{=:\nabla\cdot  ( n R_N)}+\underbrace{\Div \left(\frac{(m_q-m_\ell)\otimes \tm_{EB}}{ n} + \frac{\tm_{EB}\otimes (m_q-m_\ell)}{ n} +  n(R_q-R_\ell)\right) }_{=:\nabla\cdot  ( n R_M)}\\
&\quad+\underbrace{\cR\left(\partial_{t}\tm_t+n\tE_t+\tm_t\times B_q+n\tE_p+m_q\times\tB+\tm\times\tB-(\nabla B_q)^{\top}\left(\partial_{tt}\tcA+\tB\right)\right)}_{=:\nabla\cdot(nR_{EB1})}\\
&\quad+\underbrace{\Div\left(\left(\left(\partial_{tt}\tcA+\tB\right)\cdot B_q\right)\Id-\left(\partial_{tt}\tcA+\tB\right)\otimes B_q\right)}_{=:\nabla\cdot(nR_{EB2})},
\end{aligned}\label{new Reynolds error}
$$ 
and define $R_{EB}:=R_{EB1}+R_{EB2}$. Next, we decompose $R_O$ into
\[
\nabla\cdot (n R_O)=\underbrace{\Div\left(\frac{\tm_p\otimes \tm_p}{ n} +  n (R_\ell - \delta_{q+1}   \Id)\right)}_{=\nabla \cdot (n R_{O1})}
+ \underbrace{\Div\left(\frac{\tm_c\otimes \tm_p}{ n}+\frac{\tm_p\otimes \tm_c}{ n}+\frac{\tm_c \otimes \tm_c}{ n}\right) }_{=\nabla \cdot (n R_{O2})}.
\]
Moreover, $nR_{q+1}$ can be decomposed into two components: a trace-free part and another part that has a non-zero trace. Consider the part that has a non-zero trace:
\begin{align}
		n R_{O2} &:=  \frac{\tm_c\otimes \tm_p}{ n}+\frac{\tm_p\otimes \tm_c}{ n}+\frac{\tm_c \otimes \tm_c}{ n},\label{def of RO2}\\
		n R_{t} &:= \frac{m_q \otimes \tm_t}{n}+\frac{\tm_t \otimes m_q}{n},\label{def of Rt}\\
		n R_M &:=
		\underbrace{ n(R_q-R_\ell)}_{=n R_{M1}}+\underbrace{\frac{(m_q-m_\ell)\otimes \tm_{EB}}{ n} + \frac{\tm_{EB}\otimes (m_q-m_\ell)}{ n}}_{=n R_{M2}},\label{def of RM}\\
		nR_{EB2}&:=\left(\left(\partial_{tt}\tcA+\tB\right)\cdot B_q\right)\Id-\left(\partial_{tt}\tcA+\tB\right)\otimes B_q.
\end{align}
For the part of $R_{q+1}$ which is trace free, we will use the inverse divergence operator $\cR$  and set
\begin{align}
	n R_{O1} &:= \cR \left(\Div\left(\frac{\tm_p\otimes \tm_p}{n} +  n ( R_\ell - \delta_{q+1} \Id)\right) \right),\label{def of nR_01}\\
	n R_{N} &:= \cR \left((\tm_{EB} \cdot \nabla) (m_\ell/n)\right),\label{def of nR_N}\\
	n R_{T} &:= \cR \left( n \DTL  \frac{\tm_{EB}}{ n} - \Div(m_q-m_\ell)\frac{\tm_{EB}}{n}\right),\label{def of nR_T}\\ 
	nR_{EB1}&:=\cR\left(\partial_{t}\tm_t+n\tE_t+\tm_t\times B_q+n\tE_p+m_q\times\tB+\tm\times\tB-(\nabla B_q)^{\top}\left(\partial_{tt}\tcA+\tB\right)\right),\label{def of nR_EB1}
\end{align}
where we have used $n \DTL  \frac{\tm_{EB}}{ n} - \Div(m_q-m_\ell)\frac{\tm_{EB}}{n} = \partial_t \tm_{EB} + \Div \left(\frac{\tm_{EB} \otimes m_\ell}{n}\right)$, $(\tm_{EB} \cdot \nabla)(m_\ell/n) = \Div \left(\frac{m_\ell \otimes \tm_{EB}}{n} \right) $ and \eqref{Eq of tE_t}. So $R_{q+1}$ can be defined  as
\begin{equation}\label{split of Rq+1}
	R_{q+1} :=  R_T +  R_N +  R_{O} + R_M + R_t + R_{EB} + \frac 23 \frac{\zeta(t)}{n}\Id. 
\end{equation}
$\zeta$ which will be specified in Section \ref{New current} is a function of time which does not affect $\Div(n R_{q+1})$. Then, the trace of $R_{q+1}$ is
\[
\tr ( n R_{q+1}) = \tr ( n R_{O2} +  n R_{M} + n R_{t}+n R_{EB2}) + 2\zeta,
\]
and we define
\begin{equation}\label{askR}
	\kappa_{q+1} :=\frac{1}{2} \tr R_{q+1}= \frac{1}{2}  \tr( R_{O2} +  R_{M} + R_{t} +R_{EB2}) + \frac{\zeta}{n} .
\end{equation}
\subsection{New current}\label{New current}
Recall the energy equation
$$
\begin{aligned}
\partial_t\left(\frac{|m_q|^2}{2n}+ne(n)\right)+\Div\left(\frac{m_q}{n}\left(\frac{|m_q|^2}{2n}+ne(n)+p(n)\right)\right)+m_q\cdot E_q=nD_{t,q} k_q+\Div((R_q-c_q\Id)m_q)+\Div(n\varphi_q)+\partial_tH.
\end{aligned}
$$
We could calculate
\begin{equation}\notag
	\begin{aligned}
&\quad nD_{t,q+1}k_{q+1}+\Div(n\varphi_{q+1})\\
&=\underbrace{n D_{t,q} \left(\frac {|\tm|^2}{2n^2} + \kappa_q+ \frac{(m_q-m_\ell)\cdot \tm}{n^2} \right)
}_{=: n D_{t,q+1} \kappa_{q+1}- (\zeta_1+ \zeta_2+\zeta_3+\zeta_4+\zeta_5+\zeta_6)' + \nabla \cdot (n\varphi_T)} 
+ \underbrace{\nabla\cdot \left(\frac {|\tm|^2\tm}{2n^2} + n\varphi_\ell \right)}_{=\nabla\cdot(n\varphi_O)}
+\underbrace{\frac {\tm} {n}\cdot (\Div \UL\PL( n (R_q-c_q \Id)) + Q(m_q,m_q)) }_{=\nabla \cdot (n\varphi_{H1}) + \zeta_1'} \\
&\quad\underbrace{-\Div(R_{q+1} \tm)}_{= \nabla\cdot(n\varphi_R)}  + \underbrace{\nabla \cdot \left(
	\frac{|m_q-m_\ell|^2}{2n}
	\frac {\tm} {n}\right)+ \nabla \cdot (n(\varphi_q-\varphi_\ell)) }_{=\nabla\cdot(n\varphi_{M1})}
+\underbrace{ \frac {\tm} {n} \cdot \nabla (p(n)-p_\ell(n))
}_{=\nabla \cdot(n\varphi_{M4}) + \zeta_4'}\\
&\quad+ \underbrace{\nabla\cdot  \left(\left (\frac{\tm\otimes \tm}{n}+n R_q - \delta_{q+1}n \Id -n R_{q+1} \right) \frac{m_q-m_\ell}{n}\right) + \Div (m_q-m_\ell) \frac{\tm\cdot m_\ell}{n^2}} _{=\nabla\cdot(n\varphi_{M2})+\nabla\cdot(n\varphi_{M3})+ \zeta_2'}
\\
&\quad+ \underbrace{\left(\frac{\tm\otimes \tm}{n}+n R_q-\delta_{q+1}n\Id-n R_{q+1} + \frac{(m_q-m_\ell)\otimes \tm}{n}+ \frac{\tm\otimes (m_q-m_\ell)}{n} \right): \nabla (m_\ell/n)}_{=   \nabla\cdot(n\varphi_{H2})+  \zeta_3'}\\
&\quad+\underbrace{\tm\cdot\tE+\tm\cdot(E_{q}-E_{\ell})+ (m_q-m_\ell)\cdot\tE}_{\nabla \cdot(n\varphi_{E})+\zeta^\prime_{5}}+\underbrace{\frac{\tm\cdot m_\ell\times \tB}{n}+\frac{(m_q-m_{\ell})\cdot m_\ell\times \tB}{n}+\frac{\tm\cdot m_\ell\times (B_q-B_\ell)}{n}}_{\nabla \cdot(n\varphi_{B})+\zeta^\prime_{6}},
\end{aligned}
\end{equation}
where $D_{t,q} = \left(\partial_t + \frac{m_{q}}{ n}\cdot\nabla\right)$. The functions $ \zeta_i,i=1,2,3,4,5,6$, will be defined to invert the divergence. We define
\begin{align}
	n\varphi_O &:= 
	\underbrace{\mathcal{R}\left(\nabla\cdot \left(\frac {|\tm_p|^2\tm_p}{2n^2} +n \varphi_\ell \right)\right)}_{=: n\varphi_{O1}} 
	+ \underbrace{\frac {|\tm|^2\tm - |\tm_p|^2 \tm_p}{2n^2}  }_{=: n\varphi_{O2}}, \\
	n\varphi_R &:= -R_{q+1}\tm,\\
	n\varphi_{M1} &:=  \frac{|m_q-m_\ell|^2}{2n}\frac {\tm} {n}+ n(\varphi_q-\varphi_\ell),\\
	n\varphi_{M2} &:= \left (\frac{\tm\otimes \tm}{n}+n R_q- \delta_{q+1}n \Id -n R_{q+1}  \right) \frac{m_q-m_\ell}{n}.
\end{align}
 Recalling the definition of $\kappa_q$, $R_{O2}$, $R_M$,$R_t$ ,$R_{EB2}$ and $\kappa_{q+1} = \frac{1}{2} \tr (R_{O2} + R_M + R_{t} +R_{EB2}) + \frac{\zeta}{n}$, we can immediately get
\begin{equation}
	\begin{aligned}
	\frac {|\tm|^2}{2n^2} + \kappa_q + \frac{(m_q-m_\ell) \cdot \tm}{n^2}&= \frac {1} {2}\tr\left(  \frac{\tm_p\otimes \tm_p}{n^2} - \delta_{q+1} \Id + R_\ell \right) + \frac 32 \delta_{q+1}    
	+ \frac {1} {2}\tr( R_M + R_{O2})
	\\
	&=\frac 32 \delta_{q+1} + \kappa_{q+1} -\frac{\zeta}{n} +\frac {1} {2}\tr\left(  \frac{\tm_p\otimes \tm_p}{n^2}- \delta_{q+1} \Id + R_\ell- R_{t}-R_{EB2}\right).
	\end{aligned}
	\label{rest on Dtq}
\end{equation}
Let $\zeta = \sum_{i=0}^{6}\zeta_i$, where $\zeta_i$ will be defined in \eqref{def of zeta_0} and \eqref{def of zeta 1}--\eqref{def of zeta 6}. Noticing that $n D_{t,q} (\zeta/n) = \zeta' + \Div((m_q \zeta)/n)$, we could obtain 
\begin{align}
	\nabla \cdot   (n\varphi_T) + \zeta _0' &= 
	\Div \left(-\kappa_{q+1}  \tm+ \frac {1} {2} \tr\left(  \frac{\tm_p\otimes \tm_p}{n^2} - \delta_{q+1} \Id + R_\ell - R_{t}-R_{EB2} \right)(m_q-m_\ell) - \frac{m_q\zeta}{n}
	\right)\notag\\
	&\quad+\frac{n} {2} \DTL \tr\left(  \frac{\tm_p\otimes \tm_p}{n^2} - \delta_{q+1} \Id + R_\ell- R_{t}-R_{EB2}\right) \label{div n varphi_T +zeta0}\\
	&\quad- \frac {1} {2}\tr\left(  \frac{\tm_p\otimes \tm_p}{n^2} - \delta_{q+1} \Id + R_\ell - R_{t}-R_{EB2} \right)\Div(m_q-m_\ell).\notag
\end{align}
Then, we could decompose $ \varphi_T $ into two parts $\varphi_{T1}$ and $\varphi_{T2}$:
\begin{align}
	n\varphi_{T1} &= -\kappa_{q+1}  n+ \frac{1}{2} \tr\left(  \frac{\tm_p\otimes \tm_p}{n^2} - \delta_{q+1} \Id + R_\ell - R_{t}-R_{EB2}\right)(m_q-m_\ell) - \frac{m_q\zeta}{n},\label{def of n varphi_T1}\\
	\zeta_{0} (t) &= \int_0^t \Big\langle \frac{n}{2} \DTL \tr\left(  \frac{\tm_p\otimes \tm_p}{n^2}- \delta_{q+1} \Id + R_\ell - R_{t}-R_{EB2}\right)   \Big\rangle (s)  \rd  s\notag \\
	&\quad - \int_0^t   \Big\langle\frac{1}{2}\tr\left(  \frac{\tm_p\otimes \tm_p}{n^2} - \delta_{q+1} \Id + R_\ell - R_{t}-R_{EB2} \right)\Div(m_q-m_\ell)  \Big\rangle (s)  \rd  s, \label{def of zeta_0}\\
	n\varphi_{T2} &= \mathcal{R} \left(\frac{n}2 \DTL \tr\left(  \frac{\tm_p\otimes \tm_p}{n^2} - \delta_{q+1} \Id + R_\ell - R_{t}-R_{EB2} \right) 
	\right)\notag \\
	&\quad -\mathcal{R}\left( \frac{1}{2}\tr\left(  \frac{\tm_p\otimes \tm_p}{n^2} - \delta_{q+1} \Id + R_\ell - R_{t}-R_{EB2} \right)\Div(m_q-m_\ell) \right),\label{def of n varphi_T2}
\end{align}
where $ \zeta_0$ is defined to make the divergence equation solvable. By the definition of $\mathcal{R}$, we have
\begin{equation}\notag
	\mathcal{R} (g (t, \cdot)) := \mathcal{R} (g (t, \cdot)- f(t))
\end{equation}
for every smooth periodic time-dependent vector field $g$ and for every $f$ which depends only on time. Similarly, we can define
\begin{align}
	\zeta_1 (t) &:= \int_0^t \Big \langle \frac {\tm}{n}\cdot (\Div (\UL\PL(n(R_q-c_q \Id))) + Q(m_q,m_q)) \Big\rangle (s)  \rd  s, \label{def of zeta 1}\\
	\zeta_2 (t) &:= \int_0^t \Big\langle\Div (m_q-m_\ell) \frac{\tm\cdot m_\ell}{n^2}\Big\rangle (s)  \rd  s, \label{def of zeta 2}\\
	\zeta_{3} (t) &:= \int_0^t \Big\langle\left(\frac{\tm\otimes \tm}{n}+n R_q-\delta_{q+1}n\Id-\left(n R_{q+1} -\frac{2} {3} \zeta\Id\right)\right) : \nabla (m_\ell/n)\Big\rangle (s)  \rd  s\notag\\
	&\quad+ \int_0^t \Big\langle 
	\left(\frac{(m_q-m_\ell)\otimes \tm}{n}+ \frac{\tm\otimes (m_q-m_\ell)}{n} \right): \nabla (m_\ell/n)\Big\rangle (s)  \rd  s, \label{def of zeta 3} \\
	\zeta_4 (t) &:= \int_0^t \Big \langle \frac{\tm}{n} \cdot \nabla(p(n) - p_\ell (n)) \Big\rangle (s)  \rd  s, \label{def of zeta 4}\\
	\zeta_5 (t) &:= \int_0^t \Big \langle \tm\cdot\tE+\tm\cdot(E_{q}-E_{\ell})+ (m_q-m_\ell)\cdot\tE \Big\rangle (s)  \rd  s, \label{def of zeta 5}\\
	\zeta_6 (t) &:= \int_0^t \Big \langle\frac{\tm\cdot m_\ell\times \tB}{n}+\frac{(m_q-m_{\ell})\cdot m_\ell\times \tB}{n}+\frac{\tm\cdot m_\ell\times (B_q-B_\ell)}{n} \Big\rangle (s)  \rd  s, \label{def of zeta 6}
\end{align}
and
\begin{align}
	n\varphi_{H1} &:= \mathcal{R} \left(\frac {\tm} {n}\cdot (\Div (\UL\PL(n(R_q-c_q \Id))) + Q(m_q,m_q))\right), \label{def of n varphi_H1}\\
	n\varphi_{M3} &:= \mathcal{R} \left( 
	\Div (m_q-m_\ell) \frac{\tm \cdot m_\ell}{n^2} \right), \label{def of n varphi_M3}\\
	n\varphi_{M4} &:= \mathcal{R} \left( \frac{\tm}{n} \cdot \nabla(p(n) - p_\ell (n)) \right), \label{def of n varphi_M4}\\
	n \varphi_{H2} &:=\mathcal{R} \left(\left(\frac{\tm \otimes \tm}{n}-\delta_{q+1}n\Id+n R_q-\left(n R_{q+1}-\frac23\zeta \Id\right)\right) : \nabla (m_\ell/n) \right) \notag\\
	&\quad+ \mathcal{R} \left(\left(\frac{(m_q-m_\ell)\otimes \tm}{n}+ \frac{\tm \otimes (m_q-m_\ell)}{n} \right): \nabla (m_\ell/n)  \right) -  \frac{2m_\ell\zeta}{3n} ,\label{def of n varphi_H2} \\
	n\varphi_{E} &:=\cR\left(\tm\cdot\tE+\tm\cdot(E_{q}-E_{\ell})+ (m_q-m_\ell)\cdot\tE\right) ,\label{def of n varphi_E}\\
	n\varphi_{B} &:=\cR\left(\frac{\tm\cdot m_\ell\times \tB}{n}+\frac{(m_q-m_{\ell})\cdot m_\ell\times \tB}{n}+\frac{\tm\cdot m_\ell\times (B_q-B_\ell)}{n}\right), \label{def of n varphi_B}
\end{align}
where $\zeta_{3}$ is well-defined because $n R_{q+1} - (2\zeta)/3 \Id $ is independent of $\zeta$. 
In the next part, we will give the estimates on new Reynolds error and new current.
\section{Estimates on the Reynolds stress}\label{Estimates on the Reynolds stress}
In this section, we will give the estimates on the new Reynolds stress $R_{q+1}$ and its new advective derivative $  D_{t,q+1} R_{q+1} = \partial_t R_{q+1} + (m_{q+1}/n\cdot \nabla) R_{q+1}$. The estimates on the function $ \zeta (t)$ are akin to those for the new current, which will be detailed in the next section. For the remaining sections, we set $\left\|\cdot\right\|_N =\left\|\cdot\right\|_{C^0([0,T]+\tau_q; C^N(\T^3))}$ and fix $\tilde{n}_0 =\left\lceil\frac{2b(2+\alpha)}{(b-1)(1-\alpha)}\right\rceil$ so that
\begin{equation}\label{prop of tn_0}
\lambda_{q+1}^2({\lambda_{q+1}\mu_q})^{-(\tilde{n}_0+1)}\lesssim \delta_{q+1}^\frac{1}{2}.
\end{equation}
\begin{pp}\label{prop of Reynolds error}
	For any $0<\alpha<\frac{1}{7}$, let the parameters $\bar{b}_3(\alpha)$ and $\Lambda_3$ be as in the statement of Proposition \ref{est on backward flow}. Then, for any $1<\bar{b}(\alpha)<\bar{b}_3(\alpha)$, we can find  $\Lambda_4(\alpha,b,M,n,p,h)\geqslant\Lambda_3$ such that for any $\lambda_0\geqslant\Lambda_4$ we have the following estimates: 
\begin{align}
		\left\|R_{q+1} - \frac{2}{3}\zeta/n\Id\right\|_N &\leqslant C_{n, p, h, M}\lambda_{q+1}^N \cdot
		{\lambda_q^\frac{1}{2}}\lambda_{q+1}^{-\frac{1}{2}} \delta_q^\frac{1}{4}\delta_{q+1}^\frac{3}{4} 
		\leqslant \frac{1}{2} \lambda_{q+1}^{N-3\gamma} \delta_{q+2},\label{est on R_q+1-zeta} 
		\\
		\left\|D_{t,q+1} (R_{q+1} - \frac{2}{3}\zeta/n\Id)\right\|_{N-1} &\leqslant C_{n, p, h, M} \lambda_{q+1}^N \delta_{q+1}^\frac{1}{2}\cdot
		{\lambda_q^\frac{1}{2}}\lambda_{q+1}^{-\frac{1}{2}} \delta_q^\frac{1}{4}\delta_{q+1}^\frac{3}{4}  \leqslant \frac{1}{2}\lambda_{q+1}^{N-3\gamma}\delta_{q+1}^\frac{1}{2} \delta_{q+2},\label{est on Dtl R_q+1-zeta}
\end{align}
for $N=0,1,2$, where $C_{n, p, h, M}$ depends only upon $n$, $p$, $h$, and the $M=M(n,p,h)>1$ in Propositions \ref{Inductive proposition} and \ref{Bifurcating inductive proposition}.
\end{pp}
We will consider \eqref{split of Rq+1} and  estimate the separate terms $   R_T$, $   R_N$, $   R_{O1}$, $   R_{O2}$, $ R_M$ and $R_{EB}$. For the errors $R_{O2}$, $R_t$, $R_{M}$, and $R_{EB2}$, we use a direct estimate. For the remaining errors, we add the inverse divergence operator $\cR$ and use Corollary \ref{est on R operator}. For convenience, we restrict the range of $N$ as in Proposition \ref{prop of Reynolds error} in this section, without mentioning it further. Remark that 
\begin{align}\label{rel.par}
	\frac {1}{\lambda_{q+1}\tau_q} +\frac{\delta_{q+1}^\frac{1}{2}}{\lambda_{q+1}\mu_q} \lesssim_{n,M}  {\lambda_q^\frac{1}{2}}{\lambda_{q+1}^{-\frac{1}{2}}} \delta_q^\frac{1}{4}\delta_{q+1}^\frac{1}{4}.
\end{align}
\subsection{Transport stress error}
Recalling that
$$
n R_{T} = \cR\left(n \DTL  \frac{\tm_{EB}}{n} - \Div(m_q-m_\ell)\frac{\tm_{EB}}{ n}\right).
$$
Since $\DTL  \xi_I =0$, we have  
\begin{align*}
n \DTL  (\tm_{EB}/ n) 
&=\delta_{q+1}^\frac{1}{2} \sum_{u\in\Z} \sum_{k\in \Z^3\setminus \left\lbrace 0\right\rbrace }n \DTL  \left( n^{-1}(s_{u,k} + \tilde{e}_{u,k}+\tilde{g}_{u,k})\right) e^{i\lambda_{q+1} k\cdot \xi_I}.
\end{align*}
Since $s_{u,k},\tilde{e}_{u,k}$ and $\tilde{g}_{u,k}$ satisfy
$\supp(s_{u,k}), \supp(\tilde{e}_{u,k}),\supp(\tilde{g}_{u,k})\subset (t_u-\frac{1}{2}\tau_q, t_u + \frac 32\tau_q) \times \R^3$
, and  for any $\overline{N}\geqslant 0$,
\begin{align*}
\left\|\DTL  (n^{-1}(s_{u,k} + \tilde{e}_{u,k}+\tilde{g}_{u,k}))\right\|_{\overline{N}} + (\lambda_{q+1}\delta_{q+1}^\frac{1}{2})^{-1}\left\|\DTL^2 (n^{-1}(s_{u,k} + \tilde{e}_{u,k}+\tilde{g}_{u,k}))\right\|_{\overline{N}}\lesssim_{n, p, h, M, \overline{N}} \mu_q^{-\overline{N}}\cdot\tau_{q}^{-1}\underset{I}{\max}|\overset{\circ}{b}_{I,k}|,
\end{align*}
where we have used \eqref{est on Dtl cof of seg} and \eqref{est on cof of seg}. Next, we could calculate
\begin{align*}
\Div(m_q-m_\ell)\frac{\tm_{EB}}{ n} =  \delta_{q+1}^\frac{1}{2}\sumu \sumkO \frac{\Div(m_q-m_\ell)}{n} (s_{u,k} + \tilde{e}_{u,k}+\tilde{g}_{u,k}) e^{i\lambda_{q+1} k\cdot \xi_I}.
\end{align*}
Notice that $\Div(m_q-m_\ell)  = -\partial_t n + \partial_t n_\ell$. We can obtain
\begin{align*}
	&\left\|n^{-1}\Div(m_q-m_\ell) (s_{u,k} + \tilde{e}_{u,k}+\tilde{g}_{u,k})\right\|_{\overline{N}} \\
	&\quad+(\lambda_{q+1}\delta_{q+1}^\frac{1}{2})^{-1}\left\|D_{t,\ell }(n^{-1}\Div(m_q-m_\ell) (s_{u,k} + \tilde{e}_{u,k}+\tilde{g}_{u,k}))
	\right\|_{\overline{N}}
	\lesssim_{n, p, h, M, \overline{N}} \mu_q^{-\overline{N}}\cdot\tau_{q}^{-1}\underset{I}{\max}|\overset{\circ}{b}_{I,k}|. 
\end{align*}
Finally, we could use Corollary \ref{est on R operator} to get
\begin{align}\label{est on RT}
\left\| R_T\right\|_N
\lesssim_{n, p, h, M} \lambda_{q+1}^N\frac {\delta_{q+1}^\frac{1}{2}}{\lambda_{q+1} \tau_q} , \quad
\left\|   D_{t,q+1}  R_T\right\|_{N-1}  \lesssim_{n, p, h, M} \lambda_{q+1}^N\delta_{q+1}^\frac{1}{2} \frac {\delta_{q+1}^\frac{1}{2}}{\lambda_{q+1} \tau_q}.
\end{align}
\subsection{Nash stress error}
Recall $ n R_{N} = \mathcal{R} \left( (\tm_{EB}\cdot \nabla) \frac{m_\ell}{ n}\right)$ and observe that
\begin{align*}
(\tm_{EB}\cdot \nabla) \frac{m_\ell}{ n}
&=  \delta_{q+1}^\frac{1}{2}\sum_{u} \sum_{k\in \Z^3\setminus \{0\}} ((s_{u,k} + \tilde{e}_{u,k}+\tilde{g}_{u,k})\cdot \nabla)\frac{m_\ell}{ n} e^{i\lambda_{q+1} k\cdot \xi_I}.
\end{align*}
Since $s_{u,k},\tilde{e}_{u,k}$ and $\tilde{g}_{u,k}$ satisfy
$\supp(s_{u,k}), \supp(\tilde{e}_{u,k}),\supp(\tilde{g}_{u,k})\subset (t_u-\frac{1}{2}\tau_q, t_u + \frac 32\tau_q) \times \R^3$
, and for any $\overline{N}\geqslant 0$,
\begin{align*}
\left\|((s_{u,k} + \tilde{e}_{u,k}+\tilde{g}_{u,k})\cdot \nabla)\frac{m_\ell}{ n}\right\|_{\overline{N}}
&\lesssim_{n, p, h, M, \overline{N}} \mu_q^{-{\overline{N}}}\cdot\lambda_q\delta_{q}^\frac{1}{2}\underset{I}{\max}|\overset{\circ}{b}_{I,k}|,\\
\left\|\DTL (((s_{u,k} + \tilde{e}_{u,k}+\tilde{g}_{u,k})\cdot \nabla)\frac{m_\ell}{ n})\right\|_{\overline{N}}
&\lesssim_{n, p, h, M, \overline{N}} \lambda_{q+1}\delta_{q+1}^\frac{1}{2} \mu_q^{-{\overline{N}}}\cdot\lambda_q\delta_{q}^\frac{1}{2}\underset{I}{\max}|\overset{\circ}{b}_{I,k}|,
\end{align*}
where we have used \eqref{est on m_l}, \eqref{est on Dtl cof of seg}, and \eqref{est on cof of seg}. Then, we could apply Corollary \ref{est on R operator} and obtain
\begin{align}\label{est on RN}
\left\| R_N\right\|_N 
\lesssim_{n, p, h, M} \lambda_{q+1}^N \frac {\delta_{q+1}^\frac{1}{2}}{\lambda_{q+1}\tau_q} , \quad 
\left\|   D_{t,q+1} R_N\right\|_{N-1}  \lesssim_{n, p, h, M} \lambda_{q+1}^N\delta_{q+1}^\frac{1}{2} \frac {\delta_{q+1}^\frac{1}{2}}{\lambda_{q+1}\tau_q}.
\end{align}
\subsection{Oscillation stress error}
Notice that $ R_O =  R_{O1} +  R_{O2}$, where
\begin{align*}
n R_{O1} = \mathcal{R} \left(\Div\left(\frac{\tm_p\otimes \tm_p}{ n} +  n R_\ell - \delta_{q+1}  n \Id\right) \right),\quad n R_{O2} =  \frac{\tm_p\otimes \tm_c}{ n}+\frac{\tm_c\otimes \tm_p}{ n}+\frac{\tm_c\otimes \tm_c}{ n}.
\end{align*}
We could calculate
\begin{align*}
\Div \left( \frac{\tm_p\otimes \tm_p}{ n} +  n R_\ell - \delta_{q+1}  n \Id\right)
&= \Div \left( \sum_{u\in \Z, k\in \Z^3\setminus \{0\}} \delta_{q+1} (d_{u,k}/n) e^{ i\lambda_{q+1} k\cdot \xi_I} \right)=\sum_{u\in \Z, k\in \Z^3\setminus \{0\}} \delta_{q+1} \Div\left(d_{u,k}/n\right) e^{ i\lambda_{q+1} k\cdot \xi_I},
\end{align*}
because of $\overset{\circ}{d}_{I,k} (f_I\cdot k) =0$. Since we have
\begin{align*}
\DTL  \Div (d_{u,k}/n) = \Div\left( \DTL  (d_{u,k}/n)\right) 
- \left(\nabla (m_\ell/n)\right)_{ij}\left(\nabla (d_{u,k}/n)\right)_{ji},
\end{align*}
and we could use \eqref{est on Dtl cof of duk} to obtain 
$$\left\|\Div(d_{u,k}/n)\right\|_{\overline{N}}+(\lambda_{q+1}\delta_{q+1}^\frac{1}{2})^{-1}\left\|\DTL  \Div (d_{u,k}/n)\right\|_{\overline{N}} \lesssim_{n, p, h, M,\overline{N}}  \mu_q^{-\overline{N}}\cdot\mu_q^{-1}\underset{I}{\max}|\overset{\circ}{d}_{I,k}|,$$
for any $\overline{N}\geqslant 0$. Similarly, $\supp(d_{u,k})\subset (t_u-\frac{1}{2}\tau_q, t_u +\frac{3}{2}\tau_q)\times \R^3$, so we could apply Corollary \ref{est on R operator} to obtain 
\begin{equation}\label{est on R_O1}
	\left\|R_{O1}\right\|_N
	\lesssim_{n, p, h, M} \lambda_{q+1}^N\frac {\delta_{q+1}}{\lambda_{q+1}\mu_q}, \quad
	\left\|  D_{t,q+1} R_{O1}\right\|_{N-1}\lesssim_{n, p, h, M}\lambda_{q+1}^N\delta_{q+1}^\frac{1}{2}\frac {\delta_{q+1}}{\lambda_{q+1}\mu_q}. 
\end{equation}
To get the estimate on $R_{O2}$, we use \eqref{est on tm_EBp}, \eqref{est on tm_c}, \eqref{est on Dtl tm_EBp}, and \eqref{est on Dtl tm_c} to obtain
\begin{align*}
\left\| R_{O2}\right\|_N
&\lesssim \sum_{N_0+N_1+N_2=N}\left\|n^{-1}\right\|_{N_0}\left\|\tm_p\right\|_{N_1}\left\|\tm_c\right\|_{N_2} + \sum_{N_0+N_1+N_2=N}\left\|n^{-1}\right\|_{N_0}\left\|\tm_c\right\|_{N_1}\left\|\tm_c\right\|_{N_2}\\
&\lesssim_{n, p, h, M} \lambda_{q+1}^N\cdot\frac{\delta_{q+1}}{\lambda_{q+1}\mu_q} ,\\
\left\|   D_{t,q+1} R_{O2}\right\|_{N-1}
&\leqslant\left\|\DTL   R_{O2}\right\|_{N-1} +\left\|\left(\frac{(\tm+m_q-m_\ell)}{n}\cdot\nabla\right)  R_{O2}\right\|_{N-1}\\
&\lesssim\left\|n^{-1}\right\|_{N} \sum_{N_1+N_2=N-1} 
\left(\left\|\DTL  \tm_p\right\|_{N_1}\left\|\tm_c\right\|_{N_2} 
+\left\| \tm_p\right\|_{N_1}\left\|\DTL  \tm_c\right\|_{N_2}
+\left\|\DTL \tm_c\right\|_{N_1}\left\|\tm_c\right\|_{N_2}\right)\\
&\quad+\left\|n^{-1}\right\|_{N}\sum_{N_1+N_2=N-1} 
\left(\left\|\tm\right\|_{N_1} +\left\|m_q-m_\ell\right\|_{N_1}\right)\left\|  R_{O2}\right\|_{N_2+1}\\
&\lesssim_{n, p, h, M} \lambda_{q+1}^{N}\delta_{q+1}^\frac{1}{2} \cdot\frac{ \delta_{q+1}}{\lambda_{q+1}\mu_q}.
\end{align*}
Then, we have 
\begin{equation}\label{est on RO}
\left\| R_O\right\|_N \lesssim_{n, p, h, M} \lambda_{q+1}^N \frac{\delta_{q+1}}{\lambda_{q+1}\mu_q}, \quad
\left\|   D_{t,q+1}  R_O\right\|_{N-1}   
\lesssim_{n, p, h, M} \lambda_{q+1}^N\delta_{q+1}^{\frac{1}{2}}\frac{\delta_{q+1}}{\lambda_{q+1}\mu_q}.
\end{equation} 
\subsection{Mediation stress error}
Recalling that 
$$
n R_M =
n(R_q-R_\ell)+\frac{(m_q-m_\ell)\otimes \tm_{EB}}{ n} + \frac{\tm_{EB} \otimes (m_q-m_\ell)}{ n},
$$
and using \eqref{est on m_q-m_l}, \eqref{est on R_q-R_l}, and \eqref{est on tm}, we have
\begin{align*}
\left\| R_M\right\|_{N} 
&\lesssim\left\|n\right\|_{N}\left\|R_q-R_\ell\right\|_N+\sum_{N_0+N_1+N_2=N}\left\|n^{-2}\right\|_{N_0}\left\|m_q-m_\ell\right\|_{N_1}\left\|\tm_{EB}\right\|_{N_2}\\
&\lesssim_{n, p, h, M} \lambda_{q+1}^N \cdot ( {\lambda_q^\frac{1}{2}}\lambda_{q+1}^{-\frac{1}{2}} \delta_q^\frac{1}{4}\delta_{q+1}^\frac{3}{4}  + (\ell\lambda_q)^2\delta_q^\frac{1}{2}\delta_{q+1}^\frac{1}{2})
\lesssim_{n, p, h, M} \lambda_{q+1}^N \cdot  {\lambda_q^\frac{1}{2}}\lambda_{q+1}^{-\frac{1}{2}} \delta_q^\frac{1}{4}\delta_{q+1}^\frac{3}{4},
\end{align*}
where we have used \eqref{est on m_q-m_l}, \eqref{est on R_q-R_l}, and \eqref{est on tm_EBp}. To estimate $   D_{t,q+1}   R_M$, we use the
decomposition $D_{t,q+1}   R_M$  $= \DTL   R_M + {\left(\frac{m_q-m_\ell+\tm}{n}\right)}\cdot\nabla  R_M $  to obtain
\begin{align*}
\left\|   D_{t,q+1} (n R_M)\right\|_{N-1}
&\lesssim\left\|  { \DTL(n} (R_q-R_\ell))\right\|_{N-1} 
+ \left\|    \DTL\left( \frac{(m_q-m_\ell)\otimes \tm_{EB}}{n}\right)\right\|_{N-1}+\left\|\left(\frac{m_q-m_\ell+\tm}{n}\right)\cdot\nabla (n R_M)\right\|_{N-1}\\
&\lesssim_{n, p, h, M} \lambda_{q+1}^N \delta_{q+1}^\frac{1}{2}\cdot  {\lambda_q^\frac{1}{2}}\lambda_{q+1}^{-\frac{1}{2}} \delta_q^\frac{1}{4}\delta_{q+1}^\frac{3}{4}.
\end{align*} 
where we have used \eqref{est on m_q-m_l}, \eqref{est on Dtl m_q-m_l}, \eqref{est on Dtl R_q-R_l}, \eqref{est on tm_EBp}, \eqref{est on tm}, and \eqref{est on Dtl tm_EBp}. To summarize, we obtain 
\begin{align}\label{est on RM}
\left\| R_M\right\|_{N}
\lesssim_{n, p, h,M}  \lambda_{q+1}^N {\lambda_q^\frac{1}{2}}\lambda_{q+1}^{-\frac{1}{2}} \delta_q^\frac{1}{4}\delta_{q+1}^\frac{3}{4}, \quad
\left\|   D_{t,q+1}  R_M\right\|_{N-1}
\lesssim_{n, p, h, M} \lambda_{q+1}^N \delta_{q+1}^\frac{1}{2}  {\lambda_q^\frac{1}{2}}\lambda_{q+1}^{-\frac{1}{2}} \delta_q^\frac{1}{4}\delta_{q+1}^\frac{3}{4}.
\end{align} 
\subsection{Time-corrector stress error}
Notice that 
$$
nR_{t}= \frac{m_q \otimes \tm_t}{n}+\frac{\tm_t \otimes m_q}{n}.
$$
We could use \eqref{est on m_q}  and \eqref{est on tm_t} to obtain
$$
\begin{aligned}
\left\|R_t\right\|_N&\lesssim\sum_{N_0+N_1+N_2=N}\left\|n^{-2}\right\|_{N_0}\left\|m_q\right\|_{N_1}\left\|\tm_{t}\right\|_{N_2}\lesssim_{n, p, h, M}\lambda_{q+1}^N\cdot\frac{\delta_{q+1}^{\frac{1}{2}}}{\lambda_{q+1}(\lambda_{q+1}\mu_q)(\lambda_{q+1}\ell)^N}\lesssim_{n, p, h, M} \lambda_{q+1}^N\cdot \frac{\delta_{q+1}^{\frac{1}{2}}}{\lambda_{q+1}(\lambda_{q+1}\mu_q)},
\end{aligned}
$$
and we can similarly obtain
$$
\begin{aligned}
\left\|D_{t,q+1}R_t\right\|_{N-1}&\leqslant\left\|\DTL \left(\frac{m_q \otimes \tm_t}{n}\right)\right\|_{N-1}+\left\|\DTL \left(\frac{\tm_t \otimes m_q}{n}\right)\right\|_{N-1}+\left\|\left(\left(\frac{m_q-m_\ell+\tm}{n}\right)\cdot\nabla\right) R_t\right\|_{N-1}\\
&\lesssim\left\|n^{-1}\right\|_{N-1} \sum_{N_1+N_2=N-1} 
\left(\left\|\DTL  \tm_t\right\|_{N_1}\left\|m_q\right\|_{N_2} 
+\left\| \tm_t\right\|_{N_1}\left\|\DTL  m_q\right\|_{N_2}\right)\\
&\quad+\left\|n^{-1}\right\|_{N-1}\sum_{N_1+N_2=N-1} 
\left(\left\|\tm\right\|_{N_1} +\left\|m_q-m_\ell\right\|_{N_1}\right)\left\|R_{t}\right\|_{N_2+1}\\
&\lesssim_{n, p, h, M} \ell^{1-N}\frac{\lambda_{q}\delta_{q}\delta_{q+1}^{\frac{1}{2}}}{\lambda_{q+1}(\lambda_{q+1}\mu_q)}+\lambda_{q+1}^{N-1}(\delta_{q+1}^{\frac{1}{2}}+\ell^{2}\lambda_{q}^2\delta_{q}^{\frac{1}{2}})\frac{\delta_{q+1}^{\frac{1}{2}}}{(\lambda_{q+1}\ell)(\lambda_{q+1}\mu_q)}\\
&\lesssim_{n, p, h, M} \lambda_{q+1}^N\delta_{q+1}^{\frac{1}{2}}\cdot\frac{\delta_{q+1}^{\frac{1}{2}}}{\lambda_{q+1}(\lambda_{q+1}\mu_q)},
\end{aligned}
$$
where we have used \eqref{est on m_q}, \eqref{est on m_q-m_l}, \eqref{est on Dtl m_q},  \eqref{est on tm},   and \eqref{est on tm_t}. To summarize, we obtain 
\begin{align}\label{est on R_t}
	\left\| R_t\right\|_{N}
	\lesssim_{n, p, h, M}  \lambda_{q+1}^N\cdot\frac{\delta_{q+1}^{\frac{1}{2}}}{\lambda_{q+1}(\lambda_{q+1}\mu_q)}, \quad
	\left\|   D_{t,q+1}  R_t\right\|_{N-1}
	\lesssim_{n, p, h, M} \lambda_{q+1}^N\delta_{q+1}^{\frac{1}{2}}\cdot\frac{\delta_{q+1}^{\frac{1}{2}}}{\lambda_{q+1}(\lambda_{q+1}\mu_q)}.
\end{align} 
\subsection{Electromagnetic stress error}
Recall that 	
\begin{align*}
nR_{EB1}&=\cR\left(\partial_{t}\tm_t+n\tE_t+\tm_t\times B_q+n\tE_p+m_q\times\tB+\tm\times\tB-(\nabla B_q)^{\top}\left(\partial_{tt}\tcA+\tB\right)\right),\\
nR_{EB2}&=\left(\left(\partial_{tt}\tcA+\tB\right)\cdot B_q\right)\Id-\left(\partial_{tt}\tcA+\tB\right)\otimes B_q.
\end{align*}
For the first term, we use $\left\|\cR u\right\|_N\lesssim_N\left\|u\right\|_N$ to obtain
\begin{align*}
\left\|R_{EB1}\right\|_N&\lesssim_n\left\|\partial_{t}\tm_t+n\tE_t+\tm_t\times B_q+n\tE_p+m_q\times\tB+\tm\times\tB-(\nabla B_q)^{\top}\left(\partial_{tt}\tcA+\tB\right)\right\|_N\lesssim_{n, p, h, M}\lambda_{q+1}^{N-1}\delta_{q+1}^{\frac{1}{2}},\\
\left\|\partial_{t}R_{EB1}\right\|_N&\lesssim_n\left\|\partial_{t}\left(\partial_{t}\tm_t+n\tE_t+\tm_t\times B_q+n\tE_p+m_q\times\tB+\tm\times\tB-(\nabla B_q)^{\top}\left(\partial_{tt}\tcA+\tB\right)\right)\right\|_N\lesssim_{n, p, h, M}\lambda_{q+1}^{N}\delta_{q+1}^{\frac{1}{2}}.
\end{align*}
For the second term, we could get
\begin{align*}
	\left\|R_{EB2}\right\|_N&\leqslant\left\|n^{-1}\right\|_N\sum_{N_0+N_1=N}\left\|B_q\right\|_{N_0}\left\|\partial_{tt}\tcA+\tB\right\|_{N_1}\lesssim_{n, p, h, M}\lambda_{q+1}^{N-1}\delta_{q+1}^{\frac{1}{2}}.
\end{align*}
Next, we could get the estimate on the advective derivative,
$$
\begin{aligned}
	\left\|D_{t,q+1}R_{EB1}\right\|_{N-1}&\lesssim\left\|\partial_{t}R_{EB1}\right\|_{N-1}+\left\|n^{-1}\right\|_{N-1}\sum_{N_0+N_1=N-1}\left(\left\|m_q\right\|_{N_0}+\left\|\tm\right\|_{N_0}\right)\left\|R_{EB1}\right\|_{N_1+1}\lesssim_{n, p, h, M} \lambda_{q+1}^{N-1}\delta_{q+1}^{\frac{1}{2}},\\
	\left\|D_{t,q+1}R_{EB2}\right\|_{N-1}&\leqslant\left\|\DTL \left(B_q\otimes\left(\partial_{tt}\tcA+\tB\right)\right)\right\|_{N-1}+\left\|\DTL\left( \left(\partial_{tt}\tcA+\tB\right)\cdot B_q\right)\right\|_{N-1}+\left\|\left(\left(\frac{m_q-m_\ell+\tm}{n}\cdot\nabla\right) R_{EB2}\right)\right\|_{N-1}\\
	&\lesssim\left\|n^{-1}\right\|_{N}\sum_{N_1+N_2=N-1} 
	(\left\|\DTL  B_q\right\|_{N_1}\left\|\partial_{tt}\tcA+\tB\right\|_{N_2} 
	+\left\|B_q\right\|_{N_1}\left\|\DTL (\partial_{tt}\tcA+\tB)\right\|_{N_2})\\
	&\quad+\left\|n^{-1}\right\|_{N-1}\sum_{N_1+N_2=N-1} 
	(\left\|\tm\right\|_{N_1} +\left\|m_q-m_\ell\right\|_{N_1})\left\|R_{EB2}\right\|_{N_2+1}\\
	&\lesssim_{n, p, h, M} \lambda_{q+1}^{N-1}\frac{\delta_{q+1}^{\frac{1}{2}}}{\lambda_{q+1}\tau_{q}}+\lambda_{q+1}^{N-1}(\delta_{q+1}^{\frac{1}{2}}+\ell^{2}\lambda_{q}^2\delta_{q}^{\frac{1}{2}})\delta_{q+1}^{\frac{1}{2}}\lesssim_{n, p, h, M}\lambda_{q+1}^{N}\delta_{q+1}^{\frac{1}{2}}\cdot\frac{\delta_{q+1}^{\frac{1}{2}}}{\lambda_{q+1}}.
\end{aligned}
$$
To summarize, we obtain 
\begin{align}\label{est on R_EB}
	\left\| R_{EB}\right\|_{N}
	\lesssim_{n, p, h, M}   \lambda_{q+1}^{N}\cdot\frac{\delta_{q+1}^{\frac{1}{2}}}{\lambda_{q+1}}, \quad
	\left\|   D_{t,q+1}  R_{EB}\right\|_{N-1}
	\lesssim_{n, p, h, M} \lambda_{q+1}^{N}\delta_{q+1}^{\frac{1}{2}}\cdot\frac{1}{\lambda_{q+1}}.
\end{align} 
Finally, Proposition \ref{prop of Reynolds error} follows from
\eqref{est on RT}--\eqref{est on R_EB}.

\section{Estimates on the new current error}\label{Estimates on the new current error} 
In this section, we derive the estimates for the new current $\varphi_{q+1}$ and the remaining part of the Reynolds stress, $\frac{2}{3}\zeta/n\Id$. These estimates are summarized in the following proposition.

\begin{pp}\label{prop of current}
For any $0<\alpha<\frac{1}{7}$, let the parameters $\bar{b}_3(\alpha)$ and $\Lambda_4$ be as in the statement of Proposition \ref{prop of Reynolds error}. There exists  $1<\bar{b}_0(\alpha)<\bar{b}_3(\alpha)$ such that for any $1<b<\bar{b}_0(\alpha)$, we could find 
$\Lambda_0 = \Lambda_0 (\alpha, b, M, n, p, h)\geqslant\Lambda_4,$ satisfying that if $\lambda_0\geqslant\Lambda_0$, we have the following estimates for $N=0,1,2$:
\begin{align}
	\left\|\varphi_{q+1}\right\|_N&\leqslant \lambda_{q+1}^{N-3\gamma} \delta_{q+2}^\frac{3}{2},\label{est on varphi}\\
	\left\|  D_{t,q+1} \varphi_{q+1}\right\|_{N-1} 
	&\leqslant \lambda_{q+1}^{N-3\gamma}\delta_{q+1}^\frac{1}{2} \delta_{q+2}^\frac{3}{2},\label{est on Dtl varphi}\\
	\left\|\zeta\right\|_0 +\left\|\zeta'\right\|_0&\leqslant {\frac{\varepsilon_0^2}{20\underline{M}}}   \lambda_{q+1}^{-3\gamma} \delta_{q+2}^\frac{3}{2}
	\label{est on zeta},
\end{align}
where $\underline M$ defined as in \eqref{est on m_q}.
\end{pp}
The method used to prove \eqref{est on zeta} is similar to the one used to prove \eqref{est on varphi} and \eqref{est on Dtl varphi}, we first assume that \eqref{est on zeta} holds and finally prove it. In fact, \eqref{est on zeta} only appears in the estimation of $\varphi_{T1}$ and $\varphi_{H2}$, which does not involve circular reasoning. We single out the following fact that will be used repeatedly: given $\alpha<\frac{1}{7}$, there exists $\bar{b}(\alpha)>1$ such that for any $1<b<\bar{b}(\alpha)$ and a constant $\tilde{C}_{n, p, h, M}$ depending only on $n$, ${p}$, $h$ and $M$, we can find $\lambda_0=\lambda_0(\alpha, b,M,n,p,h)$ such that
\[
C_{n, p, h, M}\left[ \frac {\delta_{q+1}}{\lambda_{q+1}\tau_q} +\frac{\delta_{q+1}^\frac{3}{2}}{\lambda_{q+1}\mu_q} + \frac {\lambda_q^\frac{1}{2}}{\lambda_{q+1}^\frac{1}{2}} \delta_q^\frac{1}{4}\delta_{q+1}^\frac54\right]
\leqslant  \lambda_{q+1}^{-3\gamma} \delta_{q+2}^\frac{3}{2},
\] 
for any $\lambda \geqslant \lambda_0$. This is possible because $\alpha<\frac17$.

\subsection{High frequency current error}
Observe that the defintion of $  \varphi_{H1}$:
\begin{equation}
n\varphi_{H1} = \mathcal{R} \left(\frac {\tm}{n}\cdot (\Div(\UL\PL(n(R_q-c_q \Id))) + Q(m_q,m_q))\right).
\end{equation}
We thus can apply Corollary \ref{est on R operator} to
\begin{align*}
&\quad \frac {\tm_{EB}}{n}\cdot(\Div (\UL\PL(n(R_q-c_q \Id))) + Q(m_q,m_q))\\
&= \delta_{q+1}^\frac{1}{2}\sum_{u\in \Z, k\in \Z^3\setminus \{0\}} n^{-1} (\Div( \UL\PL(n(R_q-c_q \Id))) + Q(m_q,m_q)) (s_{u,k} + \tilde{e}_{u,k}+\tilde{g}_{u,k})e^{i\lambda_{q+1}k\cdot \xi_I}.
\end{align*}
We first use \eqref{est on commutator m_l / n 2} to obtain
\begin{align*}
	\left\|\DTL  \Div (\UL\PL (n R_q))\right\|_{N-1}&\leqslant\left\|\Div (\UL\PL (\DTL (n R_q)))\right\|_{N-1}
	+\left\|\Div [m_\ell/n \cdot\nabla, \UL\PL] (n R_q)\right\|_{N-1}\\
	&\quad+\left\|(\nabla (m_\ell/n))_{ki} \partial_k \UL\PL (n R_q)_{ij}\right\|_{N-1}\\
	&\lesssim_{n} \ell^{1-N} \left(\left\|\DTL  (n R_q)\right\|_{1}
	+ \lambda_{q}\delta_{q}^{\frac{1}{2}}\left\|n R_q\right\|_{1}\right)+\sum_{N_1+N_2=N-1}\left\|m_\ell\right\|_{N_1{+1}} \ell^{-N_2}\left\|n R_q\right\|_1
	\\
	&\lesssim_{n, M} \ell^{1-N} \lambda_q^{2-3\gamma}\delta_{q}^{\frac{1}{2}} \delta_{q+1}.
\end{align*}
In a similar way, we also get $\left\|\DTL  (\Div \UL\PL (n c_q \Id))\right\|_{N-1}\lesssim_{n, M} \ell^{1-N} \lambda_q^{2-3\gamma}\delta_{q}^{\frac{1}{2}} \delta_{q+1}$. Combining it with \eqref{est on error}, \eqref{est on Qmm}, \eqref{est on Dtl Qmm},
\eqref{est on Dtl cof of seg}, and \eqref{est on cof of seg},  we have
\begin{align*}
&\quad\left\|n^{-1} (\Div( \UL\PL(n(R_q-c_q \Id))) + Q(m_q,m_q))(s_{u,k} + \tilde{e}_{u,k}+\tilde{g}_{u,k})\right\|_{\overline{N}}\\
&\quad+(\lambda_{q+1}\delta_{q+1}^{\frac{1}{2}})^{-1}\left\|\DTL (n^{-1} (\Div( \UL\PL(n(R_q-c_q \Id))) + Q(m_q,m_q))(s_{u,k} + \tilde{e}_{u,k}+\tilde{g}_{u,k}))\right\|_{\overline{N}}\\
&\lesssim_{n, p, h, M, \overline{N}}\mu_q^{-\overline{N}}\cdot(\lambda_q^{1-3\gamma}\delta_{q+1} + \ell\lambda_q^2\delta_q)\underset{I}{\max}|\overset{\circ}{b}_{I,k}|.
\end{align*}
We could use Corollary \ref{est on R operator} to obtain
\begin{align*}
\left\|\cR\left(\frac {\tm_{EB}}{n}\cdot(\Div (\UL\PL(n(R_q-c_q \Id))) + Q(m_q,m_q))\right)\right\|_N
&\lesssim_{n, p, h, M} \lambda_{q+1}^{N-1} \lambda_q^{1-3\gamma} \delta_{q+1}^\frac{3}{2},\\
\left\|D_{t,q+1}\cR\left(\frac {\tm_{EB}}{n}\cdot(\Div (\UL\PL(n(R_q-c_q \Id))) + Q(m_q,m_q))\right)\right\|_{N-1}& \lesssim_{n, p, h, M}\lambda_{q+1}^{N-1}\delta_{q+1}^\frac{1}{2} \lambda_q^{1-3\gamma}\delta_{q+1}^\frac{3}{2}.
\end{align*}
Moreover, we have
\begin{align*}
\left\|\cR\left(\frac {\tm_{t}}{n}\cdot(\Div (\UL\PL(n(R_q-c_q \Id))) + Q(m_q,m_q))\right)\right\|_N&\lesssim\left\|\frac {\tm_{t}}{n}\cdot(\Div \UL\PL(n(R_q-c_q \Id)) + Q(m_q,m_q))\right\|_N\\
&\lesssim_{n, p, h, M}\ell^{-N}\lambda_{q+1}^{-1} (\lambda_{q+1}\mu_q)^{-1}\delta_{q+1}^\frac{1}{2} (\lambda_q^{1-3\gamma}\delta_{q+1} + \ell\lambda_q^2\delta_q) \\
&\lesssim_{n, p, h, M} \lambda_{q+1}^{N-1}(\lambda_{q+1}\ell)^{-N}(\lambda_{q+1}\mu_q)^{-1} \lambda_q^{1-3\gamma} \delta_{q+1}^\frac{3}{2},
\end{align*}
and we can use  \eqref{est on Qmm} to get
\begin{align*}	
&\quad\left\|D_{t,q+1}\cR\left(\frac {\tm_{t}}{n}\cdot(\Div (\UL\PL(n(R_q-c_q \Id))) + Q(m_q,m_q))\right)\right\|_{N-1}\\
&=\left\|\cR\left(\partial_{t}\left(\frac {\tm_{t}}{n}\cdot(\Div (\UL\PL(n(R_q-c_q \Id))) +Q(m_q,m_q))\right)\right)\right\|_{N-1}\\
&\quad+\left\|\left(\frac{m_q+\tm}{n}\cdot\nabla\right)\cR\left(\frac {\tm_{t}}{n}\cdot(\Div (\UL\PL(n(R_q-c_q \Id))) + Q(m_q,m_q))\right)\right\|_{N-1}\\
&\lesssim_{n, p, h, M}\ell^{-N}\lambda_{q+1}^{-1} (\lambda_{q+1}\mu_q)^{-1}\delta_{q+1}^\frac{1}{2} (\lambda_q^{1-3\gamma}\delta_{q+1} + \ell\lambda_q^2\delta_q)+\lambda_{q+1}^{N-1}(\lambda_{q+1}\ell)^{-1}(\lambda_{q+1}\mu_q)^{-1}  \lambda_q^{1-3\gamma}\delta_{q+1}^{\frac{3}{2}}  \\
&\lesssim_{n, p, h, M}\lambda_{q+1}^{N-1}(\lambda_{q+1}\ell)^{-1}(\lambda_{q+1}\mu_q)^{-1}\lambda_q^{1-3\gamma}\delta_{q+1}^{\frac{3}{2}}.
\end{align*}
To summarize, we obtain 
\begin{align}\label{est on varphi_H1}
	\left\| \varphi_{H1}\right\|_{N}
	\lesssim_{n, p, h, M}  \lambda_{q+1}^N\cdot \frac{\lambda_q^{1-3\gamma}}{\lambda_{q+1}} \delta_{q+1}^\frac{3}{2}, \quad
	\left\|   D_{t,q+1}  \varphi_{H1}\right\|_{N-1}
	\lesssim_{n, p, h, M} \lambda_{q+1}^N\delta_{q+1}^{\frac{1}{2}} \cdot\frac{\lambda_q^{1-3\gamma}}{\lambda_{q+1}} \delta_{q+1}.
\end{align} 
Notice that we could use the definition of $R_{q+1}$ to rewrite $n\varphi_{H2}$ as
\begin{equation}\label{dec of n varphi_H2} 
\begin{aligned}
	n \varphi_{H2}&=\mathcal{R} \left(\left(\frac{\tm \otimes \tm}{n}+n R_q-\delta_{q+1}n\Id-n R_{q+1}+\frac23\zeta \Id\right) : \nabla (m_\ell/n)
	\right) \\ 
	&\quad+ \mathcal{R} \left(\left(\frac{(m_q-m_\ell)\otimes \tm}{n}+ \frac{\tm\otimes(m_q-m_\ell)}{n} \right): \nabla (m_\ell/n) 
	\right) -  \frac{2m_\ell\zeta}{3n} \\ 
	&=\cR\left(\left(\left(\frac{\tm_p\otimes \tm_p}{n}-\delta_{q+1}n \Id+n R_\ell \right)-  n R_{O1} - n R_t - n R_T -n R_N - n R_{EB}\right):\nabla (m_\ell/n)\right)-  \frac{2m_\ell\zeta}{3n}.
\end{aligned}
\end{equation}
Assuming \eqref{est on zeta}, we could obtain $\left\|(m_\ell \zeta)/n^2\right\|_N
\leqslant\left\|\zeta\right\|_0\left\|m_\ell/n^2\right\|_N
\leqslant \frac 1{20}\lambda_{q+1}^{N-3\gamma} \delta_{q+2}^\frac{3}{2}$ for $N=0,1,2,$ and
\begin{align*}
\left\|D_{t,q+1} ((m_\ell \zeta)/n^2) \right\|_{N-1}
&\leqslant \left(\left\|n^{-2}D_{t,q+1} m_\ell \right\|_{N-1}
+\left\| D_{t,q+1}( n^{-2})m_\ell\right\|_{N-1} \right)\left\|\zeta\right\|_0
+\left\|\zeta'\right\|_0\left\|m_\ell/n^2\right\|_{N-1}\leqslant \frac 1{20} \lambda_{q+1}^{N-3\gamma}\delta_{q+1}^\frac{1}{2} \delta_{q+2}^\frac{3}{2},
\end{align*}
for $N=1,2$. Here, we used
\begin{equation}\label{est on Dtq+1 m_l}
	\begin{aligned}
	\left\|D_{t,q+1} m_\ell\right\|_{N-1}
	&\leqslant\left\|\DTL m_\ell\right\|_{N-1}+\left\|
	((m_q-m_\ell)/n\cdot\nabla) m_\ell\right\|_{N-1}+\left\|(\tm/n\cdot\nabla) m_\ell\right\|_{N-1}\\
	&\lesssim_{n, p, h, M} \lambda_{q+1}^{N-1}\delta_{q+1}^{\frac{1}{2}}\lambda_q\delta_q^{\frac{1}{2}} + \lambda_q^N\delta_{q} + \ell^{3-N}\lambda_{q}^3\delta_q\lesssim_{n, p, h, M} \lambda_{q+1}^{N-1}\delta_{q+1}^{\frac{1}{2}}\lambda_q\delta_q^{\frac{1}{2}},
	\end{aligned}
\end{equation}
obtained from \eqref{est on m_l},  \eqref{est on m_q-m_l}, \eqref{est on pt Dtl m_l N}, and \eqref{est on tm}. In a similar way, we could get for $N=1,2,$
\begin{equation}\label{est on Dtq+1 n}
\begin{aligned}
	&\left\|D_{t,q+1}n\right\|_0 
	\lesssim \left\|\partial_t n\right\|_0 +\left\|m_{q+1}/n\right\|_0\left\|\nabla n\right\|_0
	\lesssim_{n, p, h, M} 1,\\
	&\left\|D_{t,q+1}n\right\|_N
	\lesssim \left\|\partial_t n\right\|_N + \sum_{N_1+N_2=N}\left\|m_{q+1}/n\right\|_{N_1}\left\|\nabla n\right\|_{N_2}
	\lesssim_{n, p, h, M} \lambda_{q+1}^{N}\delta_{q+1}^\frac{1}{2}.  
\end{aligned}
\end{equation}
Notice that
\begin{align*}
\left(\frac{\tm_p\otimes \tm_p}{n}-\delta_{q+1}n \Id+n R_\ell \right):\nabla (m_\ell/n)=\delta_{q+1}\sum_u\sumkO(d_{u,k}:\nabla (m_\ell/n))e^{i\lambda_{q+1}k\cdot\xi_I}.
\end{align*}
So we could use \eqref{est on pt Dtl m_l N} and \eqref{est on cof of duk} to get
$$
\left\|d_{u,k}:\nabla (m_\ell/n)\right\|_{\overline{N}}+(\lambda_{q+1}\delta_{q+1}^\frac{1}{2})^{-1}\left\|\DTL \left(d_{u,k}:\nabla (m_\ell/n)\right)\right\|_{\overline{N}} \lesssim_{n, p, h, M,\overline{N}}  \mu_q^{-\overline{N}}\lambda_{q}\delta_{q}^{\frac{1}{2}} \underset{I}{\max}|\overset{\circ}{d}_{I,k}|,
$$
for any $\overline{N}\geqslant 0$.
Using $\supp(d_{u,k}) \subset (t_u -\frac{1}{2}\tau_q, t_u + \frac 32\tau_q) \times \R^3$ and applying Corollary \ref{est on R operator} , we have
\begin{align*}
\left\|\cR\left(\left(\frac{\tm_p\otimes \tm_p}{n}-\delta_{q+1}n \Id+n R_\ell \right):\nabla (m_\ell/n)\right)\right\|_N
&\lesssim_{n, p, h, M} \lambda_{q+1}^N \lambda_q \delta_q^\frac{1}{2} \frac {\delta_{q+1}}{\lambda_{q+1}},\\
\left\| D_{t,q+1} \cR\left(\left(\frac{\tm_p\otimes \tm_p}{n}-\delta_{q+1}n \Id+n R_\ell \right):\nabla (m_\ell/n)\right)\right\|_{N-1}
& \lesssim_{n, p, h, M} \lambda_{q+1}^N\delta_{q+1}^\frac{1}{2} \lambda_q \delta_q^\frac{1}{2} \frac {\delta_{q+1}}{\lambda_{q+1}}.
\end{align*}
To estimate the remaining term, we denote $ R_\tri$ as either $ R_{O1}$, $ R_T$, $ R_N$ or $ R_{EB1}$ which can be written as $ R_\tri  = \cR G_\tri$ and satisfy
\begin{align}\label{est on G_tri} 
\left\|G_\tri\right\|_N \lesssim_{n, p, h, M} \lambda_{q+1}^N \left(\frac{\delta_{q+1}}{\mu_q} + \frac{\delta_{q+1}^\frac{1}{2}}{\tau_q} \right), 
\quad
\left\|\DTL  G_\tri \right\|_{N-1} \lesssim_{n, p, h, M} \lambda_{q+1}^N \delta_{q+1}^\frac{1}{2}
\left(\frac{\delta_{q+1}}{\mu_q} + \frac{\delta_{q+1}^\frac{1}{2}}{\tau_q} \right).
\end{align}
Furthermore, such $G_\tri$ has the form $\sum_{u,k}g_{u,k}^\tri e^{i\lambda_{q+1}k\cdot\xi_I}$ and can be decomposed into two parts
\begin{equation}\label{dec of G_tri}
G_\tri = P_{\GL} G_\tri+P_{\LL} G_\tri,    
\end{equation}
where $P_{\LL}  G_\tri$ satisfies
\begin{equation}\label{est on P G_tri}
	\begin{aligned}
	&\left\|P_{\LL} G_\tri\right\|_0 \lesssim_{n, p, h, M} \lambda_{q+1}^{-2}\delta_{q+1}^\frac{1}{2} \left(\frac{\delta_{q+1}}{\mu_q} + \frac{\delta_{q+1}^\frac{1}{2}}{\tau_q} \right),\quad\left\|\DTL P_{\LL} G_\tri\right\|_0 \lesssim_{n, p, h, M} \lambda_{q+1}^{-1}\delta_{q+1} \left(\frac{\delta_{q+1}}{\mu_q} + \frac{\delta_{q+1}^\frac{1}{2}}{\tau_q} \right), 
\end{aligned}
\end{equation}
obtained from \eqref{est on epsilon} and \eqref{est on Dtl epsilon}. Here, we could give the following decomposition 
\begin{equation}\label{dec of m_l / n}
	\begin{aligned}
		\nabla (m_\ell/n)= \nabla \left( m_\ell P_{\lesssim \frac1{64}\lambda_{q+1}} n^{-1}\right) + \nabla \left( m_\ell P_{\gtrsim \frac1{64}\lambda_{q+1}} n^{-1}\right)=: \nabla \left(m_\ell / n\right)_1 + \nabla \left(m_\ell / n\right)_2.
	\end{aligned}
\end{equation}
$\nabla \left(m_\ell / n\right)_1$ has the frequency localized to ${\leqslant} \frac1{32} \lambda_{q+1}$, since $m_\ell$ had frequency localized to $\ell^{-1}$ and $\ell^{-1} \leqslant \frac 1{128}\lambda_{q+1}$ for sufficiently large $\lambda_0$. Thus, $\cR P_{\GL} G_\tri:\nabla \left(m_\ell / n\right)_1$ has the frequency localized to $\gtrsim \lambda_{q+1}$, we have
\begin{equation}\label{est on high 1}
\begin{aligned}
	&\quad\left\|\cR(\cR P_{\GL} G_\tri:\nabla \left(m_\ell / n\right)_1)\right\|_N\\
	&\lesssim \frac 1{\lambda_{q+1}}\left\|\cR  P_{\GL} G_\tri:\nabla \left(m_\ell / n\right)_1\right\|_N\lesssim \frac 1{\lambda_{q+1}^2} \sum_{N_1+N_2=N}\left\| G_\tri\right\|_{N_1}\left\|\nabla (m_\ell/n)\right\|_{N_2}\lesssim_{n, p, h, M} \lambda_{q+1}^{N-2} \left(\frac{\delta_{q+1}}{\mu_q} + \frac{\delta_{q+1}^\frac{1}{2}}{\tau_q} \right) \lambda_q\delta_q^\frac{1}{2}.
\end{aligned}
\end{equation}
On the other hand, $\cR P_{\LL}  G_\tri:\nabla \left(m_\ell / n\right)_1$ has the frequency localized to $\lesssim \lambda_{q+1}^{-1}$, so that
\begin{equation}\label{est on low 1}
	\begin{aligned} 
	\left\|\cR(\cR P_{\LL}  G_\tri:\nabla \left(m_\ell / n\right)_1)\right\|_N
	&\lesssim \lambda_{q+1}^N\left\|\cR(\cR P_{\LL}  G_\tri:\nabla \left(m_\ell / n\right)_1)\right\|_0\lesssim \lambda_{q+1}^N\left\|P_{\LL}  G_\tri\right\|_0\left\|\nabla (m_\ell/n)\right\|_0.
	\end{aligned}	
	\end{equation}
Noting that $n$ is smooth in space-time and $n\geqslant\varepsilon_0$, we could use Bernstein's inequality to get
\begin{equation}\label{est on PG n}
\left\|P_{\gtrsim \lambda} n^{-1}\right\|_{N'}+\left\|\partial_t P_{\gtrsim \lambda} n^{-1}\right\|_{N'} 
\lesssim \lambda^{-2}(\left\|\nabla^2 n^{-1}\right\|_{N'}+\left\|\nabla^2\partial_t  n^{-1}\right\|_{N'})
\lesssim_{n} \lambda^{-2}, 
\end{equation}
for any $\lambda \geqslant 1$ and for any $N'=0,1,2,3.$ Then, $\left\|(m_\ell/n)_2\right\|_{N'} \lesssim_{n} \frac{1}{\lambda_{q+1}^2}\left\|m_\ell\right\|_{N'}$ for any $N'=0,1,2,3$, and we have
\begin{equation}\label{est on low 2}\begin{aligned} 
	\left\|\cR(\cR G_\tri:\nabla \left(m_\ell / n\right)_2)\right\|_N 
	&\lesssim \sum_{N_1+N_2=N}\left\| G_\tri\right\|_{N_1}\left\|\nabla \left(m_\ell / n\right)_2\right\|_{N_2}\lesssim \frac1{\lambda_{q+1}^{2}} \sum_{N_1+N_2=N}\left\| G_\tri\right\|_{N_1}\left\| m_\ell\right\|_{N_2+1}.
\end{aligned}\end{equation}
Up to now, we could use \eqref{est on G_tri}, \eqref{est on P G_tri} and \eqref{est on high 1}--\eqref{est on low 2} to obtain
\begin{align*}
&\quad\left\|\cR( R_{O1}:\nabla (m_\ell/n))\right\|_N+\left\|\cR( R_N:\nabla (m_\ell/n))\right\|_N+\left\|\cR( R_T:\nabla (m_\ell/n))\right\|_N+\left\|\cR( R_{EB1}:\nabla (m_\ell/n))\right\|_N\\
&\lesssim_{n, p, h, M}\lambda_{q+1}^N
\left(\frac {\delta_{q+1}^\frac{1}{2}}{\lambda_{q+1}\tau_{q}}+ \frac {\delta_{q+1}}{\lambda_{q+1}\mu_q}\right) \frac{\lambda_q\delta_q^\frac{1}{2}}{\lambda_{q+1}}.
\end{align*}
Next, let us consider their advective derivatives. Noticing that we have the following decomposition
\begin{align*}
D_{t,q+1} \cR( R_{\tri}:\nabla (m_\ell/n)) 
= \DTL \cR( R_{\tri}:\nabla (m_\ell/n)) + \left(\frac{\tm+ (m_q-m_\ell)}{n}\cdot \nabla \right)\cR( R_{\tri}:\nabla (m_\ell/n)). 
\end{align*}
It is easy to estimate the second term
\begin{align*}
\left\|\left(\frac{\tm+ (m_q-m_\ell)}{n}\cdot \nabla\right) \cR( R_{\triangle}:(m_\ell/n))\right\|_{N-1}
&\lesssim_{n, p, h, M} \lambda_{q+1}^N\delta_{q+1}^\frac{1}{2}
\left(\frac {\delta_{q+1}^\frac{1}{2}}{\lambda_{q+1}\tau_{q}}+ \frac {\delta_{q+1}}{\lambda_{q+1}\mu_q}\right) \frac{\lambda_q\delta_q^\frac{1}{2}}{\lambda_{q+1}}.
\end{align*}
As for the first term, we have
\begin{align*}
\left\|\DTL \cR( R_{\triangle}:\nabla \left(m_\ell / n\right)_1)\right\|_{N-1}
&\leqslant 
\left\|\DTL \cR(\cR P_{\GL} G_\tri :\nabla \left(m_\ell / n\right)_1)\right\|_{N-1}+\left\|\DTL \cR( \cR P_{\LL}  G_\tri:\nabla \left(m_\ell / n\right)_1)\right\|_{N-1}.
\end{align*}
To estimate the first term on the right hand, we could consider the following decomposition 
\begin{align*}
\DTL \cR P_{\GL} H
&=\cR P_{\GL} \DTL   H 
+\cR \left[(m_\ell/n)\cdot\nabla, P_{\GL}\right]  H 
+ \left[(m_\ell/n)\cdot\nabla, \cR\right] P_{\GL} H,
\end{align*}
for any smooth vector-valued function $H$ and Littlewood-Paley operator $P_{\gtrsim\lambda_{q+1}}$ projecting to the frequency $\gtrsim \lambda_{q+1}$. First, we could use Lemma \ref{lem of commutator 0} when $\ell^{-1}$ is replaced by $C\lambda_{q+1}$ and \eqref{est on PG n} to get
\begin{align*}
\left\|[(m_\ell/n)_1 \cdot\nabla,P_{\GL}] H\right\|_{N-1}
&{\lesssim \lambda_{q+1}^{N-2}\left\|\nabla(m_\ell/n)_1\right\|_0\left\|\nabla H\right\|_0}
\lesssim \lambda_{q+1}^{N-2}\left\|m_q\right\|_1\left\|\nabla H\right\|_0,\\
\left\|[(m_\ell/n)_2\cdot\nabla,P_{\GL}] H\right\|_{N-1}
&{\lesssim \lambda_{q+1}^{N-2}\left\|\nabla(m_\ell/n)_2\right\|_{1}\left\|\nabla H\right\|_0\lesssim \lambda_{q+1}^{N-4}}\left\|m_q\right\|_1\left\|\nabla H\right\|_0.
\end{align*}
From Lemma \ref{est on R commutator}, we could know
\begin{align}
	\left\|\left[{m_\ell}{\PL n^{-1}}\cdot\nabla, \cR\right]P_{\GL} H\right\|_{N-1} 
	\lesssim \sum_{N_1+N_2=N-1} \ell\left\|\nabla ({m_\ell}{\PL n^{-1}})\right\|_{N_1}\left\|H\right\|_{N_2}.\label{est on R commutator P}
\end{align}
By using $\left\|\mathcal{R}f\right\|_0\lesssim\left\|f\right\|_0,\left\|\PG n^{-1}\right\|_{N'}\lesssim_{n} \ell^3\left\|n^{-1}\right\|_{N'+3}$, we could obtain
\begin{align*}
	\left\|\left[{m_\ell}{\PG n^{-1}}\cdot\nabla, \cR\right]P_{\GL} H\right\|_{N-1} 
	&\lesssim \sum_{N_1+N_2+N_3=N-1}
	\left\|m_{\ell}\right\|_{N_1}
	\left\|\PG n^{-1}\right\|_{N_2}\left\|\nabla P_{\GL}H\right\|_{N_3}\\
	&\lesssim_{n} \ell^3 \sum_{N_1+N_3=N-1}
	\left\|m_{\ell}\right\|_{N_1}
	\left\|\nabla P_{\GL}H\right\|_{N_3}.
\end{align*}
Since $P_{\GL} \DTL   H$ and $ [(m_\ell/n)_1\cdot\nabla, P_{\GL}]  H$ have frequencies localized to $\gtrsim \lambda_{q+1}$, it follows that
\begin{equation}\label{est on Dtl R PG H}
	\begin{aligned}
	\left\|\DTL \cR P_{\GL} H\right\|_{N-1}
	&\lesssim\left\|\cR P_{\GL} \DTL   H \right\|_{N-1} 
	+\left\|\cR [(m_\ell/n)_1\cdot\nabla, P_{\GL}]  H \right\|_{N-1}
	\\
	&\quad+\left\|\cR [(m_\ell/n)_2\cdot\nabla, P_{\GL}]  H \right\|_{N-1}+\left\| [m_\ell/n\cdot\nabla, \cR] P_{\gtrsim\lambda_{q+1}} H\right\|_{N-1}\\
	&\lesssim 
	\frac 1{\lambda_{q+1}}\left\| P_{\GL} \DTL   H \right\|_{N-1} 
	+  \frac 1{\lambda_{q+1}}\left\| [(m_\ell/n)_1\cdot\nabla, P_{\GL}]  H \right\|_{N-1}\\
	&\quad+\left\| [(m_\ell/n)_2\cdot\nabla, P_{\GL}]  H \right\|_{N-1}
	+\left\| [{m_\ell}/{n}\cdot\nabla,\cR]P_{\GL} H\right\|_{N-1}\\
	&\lesssim \frac 1{\lambda_{q+1}}\left\| \DTL   H \right\|_{N-1} 
	+ \lambda_{q+1}^{N-3}\left\|m_q\right\|_1 \left\|\nabla H\right\|_0
	+\sum_{N_1+N_2=N-1} \ell\left\|\nabla (m_\ell/n)\right\|_{N_1}\left\|H\right\|_{N_2}.
\end{aligned}\end{equation}
Up to now, we only have to apply it to $H= \cR P_{\GL} G_\tri:\nabla \left(m_\ell / n\right)_1$. For such $H$, we have $H= P_{\geqslant \frac18\lambda_{q+1}}H$ for sufficiently large $\lambda_0$, and it follows that
\begin{equation}\label{est on Dtl R High 1}\begin{aligned}
	&\quad\left\|\DTL \cR(\cR  P_{\GL}  G_\tri:\nabla \left(m_\ell / n\right)_1)\right\|_{N-1}\\
	&\lesssim \frac 1{\lambda_{q+1}}\left\|\DTL (\cR  P_{\GL}  G_\tri: \nabla \left(m_\ell / n\right)_1)\right\|_{N-1} 
	+ \lambda_{q+1}^{N-3}\left\| m_q\right\|_1 
	\left\| \cR P_{\GL} G_\tri:\nabla \left(m_\ell / n\right)_1\right\|_1\\
	&\quad+\sum_{N_1+N_2=N-1} \ell\left\|\nabla (m_\ell/n)\right\|_{N_1}\left\|\cR P_{\GL} G_\tri:\nabla \left(m_\ell / n\right)_1\right\|_{N_2}\\
	&\lesssim_{n, p, h, M} \lambda_{q+1}^{N-2} \delta_{q+1}^\frac{1}{2}
	\left(\frac{\delta_{q+1}}{\mu_q} + \frac{\delta_{q+1}^\frac{1}{2}}{\tau_q} \right) \lambda_q\delta_q^\frac{1}{2}.
\end{aligned}\end{equation}
Similarly, the second inequality follows from \eqref{est on Dtl R PG H} with $H= G_\tri$,
\begin{equation}\label{est on Dtl High 1}
\begin{aligned}
	&\quad\left\|\DTL (\cR P_{\GL} G_\tri: \nabla \left(m_\ell / n\right)_1)\right\|_{N-1} \\
	&\lesssim \sum_{N_1+N_2=N-1}
	\left\|\DTL \cR  P_{\GL}  G_\tri\right\|_{N_1}\left\|\nabla (m_\ell / n)_1\right\|_{N_2}
	+\left\|\cR  P_{\GL}  G_\tri\right\|_{N_1}\left\| \DTL \nabla (m_\ell / n)_1\right\|_{N_2}\\
	&\lesssim \sum_{N_1+N_2=N-1}
	\left(\lambda_{q+1}^{-1}\left\|\DTL G_\tri\right\|_{N_1}+\lambda_{q+1}^{N_1-3}\left\|m_q\right\|_1\left\|\nabla G_\tri\right\|_0 \right)\left\|\nabla (m_\ell / n)_1\right\|_{N_2}\\
	&\quad+\sum_{\substack{N_{11}+N_{12} = N_1\\ N_1+N_2=N-1}} \ell\left\|\nabla m_\ell\right\|_{N_{11}}\left\|G_\tri\right\|_{N_{12}}\left\|\nabla (m_\ell / n)_1\right\|_{N_2}
    +\sum_{N_1+N_2=N-1}\lambda_{q+1}^{-1}\left\| G_\tri\right\|_{N_1}\left\| \DTL \nabla (m_\ell / n)_1\right\|_{N_2}\\
	&\lesssim_{n, p, h, M} \lambda_{q+1}^{N-1} \delta_{q+1}^\frac{1}{2}
\left(\frac{\delta_{q+1}}{\mu_q} + \frac{\delta_{q+1}^\frac{1}{2}}{\tau_q} \right) \lambda_q\delta_q^\frac{1}{2},
\end{aligned}
\end{equation}
where we used 
\begin{equation}\label{est on Dtl m_l / n}
\begin{aligned}
	\left\|\DTL  \nabla (m_\ell/n)_1\right\|_{N_2}
	&\leqslant
	\left\|P_{\lesssim \frac{1}{64}\lambda_{q+1}}n^{-1}\DTL  \nabla m_\ell\right\|_{N_2}
	+\left\|\DTL  m_\ell \otimes\nabla P_{\lesssim \frac{1}{64}\lambda_{q+1}}n^{-1}\right\|_{N_2}\\
	&\quad+\left\| \nabla m_\ell \DTL P_{\lesssim \frac{1}{64}\lambda_{q+1}}n^{-1}\right\|_{N_2}+\left\|m_\ell\otimes \DTL {\nabla}P_{\lesssim \frac{1}{64}\lambda_{q+1}}n^{-1}\right\|_{N_2}\\
	&\lesssim_{n, p, h, M} \lambda_{q+1}^{ 
			N_2+1}\delta_{q+1}^\frac{1}{2} \lambda_q\delta_q^\frac{1}{2},
\end{aligned}
\end{equation}
which follows from \eqref{est on m_l}, \eqref{est on pt Dtl m_l N}, and \eqref{est on mixed m_l 1}.
As for the remaining term $\left\|\DTL \cR( \cR P_{\LL}  G_\tri:\nabla \left(m_\ell / n\right)_1))\right\|_{N-1}$, 
we set $
\DTL^{L} := \partial_t + (m_\ell/n)_1 \cdot \nabla$
and calculate
\begin{align*}
&\quad\left\|\DTL^{L} \cR(\cR P_{\LL} G_\tri:\nabla \left(m_\ell / n\right)_1))\right\|_0
\\
&\lesssim\left\| \cR \DTL^{L}(\cR P_{\LL}  G_\tri:\nabla \left(m_\ell / n\right)_1))\right\|_0
+\left\|[(m_\ell/n)_1\cdot\nabla ,\cR] (\cR P_{\LL}  G_\tri:\nabla \left(m_\ell / n\right)_1))\right\|_0\\
&\lesssim \left\| \DTL^{L}\cR P_{\LL}  G_\tri\right\|_0\left\|\nabla (m_\ell / n)_1\right\|_0 +\left\|  P_{\LL} G_\tri\right\|_0\left\| \DTL^{L}\nabla (m_\ell / n)_1\right\|_0+\left\|(m_\ell/n)_1\right\|_0\left\|\nabla(\cR P_{\LL}  G_\tri:\nabla (m_\ell/n)_1)\right\|_0 \\
&\lesssim \left(\left\| \DTL^{L}P_{\LL}  G_\tri\right\|_0
+\left\|[(m_\ell/n)_1\cdot \nabla, \cR] P_{\LL}  G_\tri\right\|_0\right)\left\|\nabla (m_\ell/n)_1\right\|_0\\
&\quad+\left\|  P_{\LL}  G_\tri\right\|_0\left\| \DTL^L \nabla (m_\ell/n)_1\right\|_0
+\lambda_{q+1}\left\|P_{\LL}  G_\tri\right\|_0\left\|\nabla (m_\ell/n)_1\right\|_0\\
&\lesssim
\lambda_q\delta_q^\frac{1}{2}\left(\left\| \DTL^{L}P_{\LL}  G_\tri\right\|_0
+\lambda_{q+1}\left\|P_{\LL}  G_\tri\right\|_0\right)\\
&\lesssim_{n, p, h, M} \lambda_{q+1}^{-1}\delta_{q+1}^\frac{1}{2} \left(\frac {\delta_{q+1}^\frac{1}{2}}{\tau_{q}}+ \frac {\delta_{q+1}}{\mu_q}\right) \lambda_q\delta_q^\frac{1}{2},
\end{align*}
where we have used $\left\|\cR f\right\|_0 \lesssim\left\|f\right\|_0$. 
Since the frequency of $\DTL^{L} \cR(\cR P_{\LL}  G_\tri:\nabla \left(m_\ell / n\right)_1))$ is  localized to $\lesssim \lambda_{q+1}$, we have
\begin{align*}
&\left\|\DTL^{L} \cR(\cR P_{\LL}  G_\tri:\nabla \left(m_\ell / n\right)_1))\right\|_{N-1}
\lesssim \lambda_{q+1}^{N-1}\left\|\DTL^{L} \cR(\cR P_{\LL}  G_\tri:\nabla \left(m_\ell / n\right)_1))\right\|_0.
\end{align*}
Up to now, we could obtain
\begin{align*}
&\quad\left\|\DTL  \cR(\cR P_{\LL}  G_\tri:\nabla \left(m_\ell / n\right)_1))\right\|_{N-1}\\
&\leqslant\left\|\DTL^{L} \cR(\cR P_{\LL}  G_\tri:\nabla \left(m_\ell / n\right)_1))\right\|_{N-1} +\left\|((m_\ell/n)_2 \cdot \nabla) \cR(\cR P_{\LL}  G_\tri:\nabla \left(m_\ell / n\right)_1))\right\|_{N-1}\\
&\lesssim \lambda_{q+1}^{N-1}\left\|\DTL^{L} \cR(\cR P_{\LL}  G_\tri:\nabla \left(m_\ell / n\right)_1))\right\|_0 + \lambda_{q+1}^{-2}\left\|\cR(\cR P_{\LL}  G_\tri:\nabla \left(m_\ell / n\right)_1))\right\|_{N}\\
&\lesssim_{n, p, h, M} \lambda_{q+1}^N\delta_{q+1}^\frac{1}{2}
\left(\frac {\delta_{q+1}^\frac{1}{2}}{\lambda_{q+1}\tau_{q}}+ \frac {\delta_{q+1}}{\lambda_{q+1}\mu_q}\right) \frac{\lambda_q\delta_q^\frac{1}{2}}{\lambda_{q+1}}.
\end{align*}
Combining with \eqref{est on Dtl R High 1}, we could obtain
\begin{align*}
\left\|\DTL \cR( R_\triangle:\nabla \left(m_\ell / n\right)_1))\right\|_{N-1}
\lesssim_{n, p, h, M} \lambda_{q+1}^N\delta_{q+1}^\frac{1}{2}
\left(\frac {\delta_{q+1}^\frac{1}{2}}{\lambda_{q+1}\tau_{q}}+ \frac {\delta_{q+1}}{\lambda_{q+1}\mu_q}\right) \frac{\lambda_q\delta_q^\frac{1}{2}}{\lambda_{q+1}}.
\end{align*}
Next, we consider $\DTL \cR( R_\triangle:\nabla \left(m_\ell / n\right)_2)$. Observing that
$
\DTL \cR H
=\cR \DTL   H  
+\left [(m_\ell/n)\cdot\nabla, \cR\right] H
$
and setting $H =  R_\triangle:\nabla \left(m_\ell / n\right)_2$, we have
\begin{align*}
\left\|\DTL \cR H\right\|_{N-1}
&=\left\|\cR \DTL  ( R_\triangle:\nabla \left(m_\ell / n\right)_2)\right\|_{N-1} 
+\left\|[(m_\ell/n)\cdot\nabla, \cR] ( R_\triangle:\nabla \left(m_\ell / n\right)_2)\right\|_{N-1}\\
&\lesssim\left\| \DTL   ( R_\triangle:\nabla \left(m_\ell / n\right)_2) \right\|_{N-1} 
+\left\|[(m_\ell/n)\cdot\nabla, \cR] ( R_\triangle:\nabla \left(m_\ell / n\right)_2)\right\|_{N-1}
\end{align*}
As for the first term, we use \eqref{est on PG n} and $\left\|\DTL \nabla \left(m_\ell / n\right)_2\right\|_{N_2} \lesssim \delta_{q+1}^\frac{1}{2} \lambda_q \delta_q^\frac{1}{2}$ to get
\begin{align*}
\left\|\DTL ( R_\triangle:\nabla \left(m_\ell / n\right)_2))\right\|_{N-1} &\lesssim \sum_{N_1 + N_2= N-1}\left\|\DTL  R_\triangle\right\|_{N_1}\left\|\nabla \left(m_\ell / n\right)_2\right\|_{N_2} +\left\| R_\triangle\right\|_{N_1}\left\|\DTL \nabla \left(m_\ell / n\right)_2\right\|_{N_2}\\
&\lesssim_{n, p, h, M} \lambda_{q+1}^{N-2} \delta_{q+1} \tau_q^{-1} \lambda_q \delta_q^{\frac{1}{2}}.
\end{align*}   
The second term can be estimated as
\begin{align*}
&\quad\left\|\left[(m_\ell/n)\cdot\nabla, \cR\right]( R_\triangle:\nabla \left(m_\ell / n\right)_2))\right\|_{N-1}\\
&\lesssim \sum_{N_1 + N_2 + N_3 = N-1}\left\|(m_\ell/n)\right\|_{N_1}\left\| R_\triangle\right\|_{N_2+1}\left\|\nabla \left(m_\ell / n\right)_2\right\|_{N_3}  +\left\|(m_\ell/n)\right\|_{N_1}\left\| R_\triangle\right\|_{N_2}\left\|\nabla \left(m_\ell / n\right)_2\right\|_{N_3+1}\lesssim_{n, p, h, M}\lambda_{q+1}^N \delta_{q+1} \frac{\delta_q^{\frac{1}{2}}}{\lambda_{q+1}^2 \tau_q}.
\end{align*}
Now, we could obtain
\begin{align*}
\left\|   D_{t,q+1} \cR( R_\tri:\nabla (m_\ell/n))\right\|_{N-1}
\lesssim_{n, p, h, M}
\lambda_{q+1}^N\delta_{q+1}^\frac{1}{2}
\left(\frac {\delta_{q+1}^\frac{1}{2}}{\lambda_{q+1}\tau_{q}}+ \frac {\delta_{q+1}}{\lambda_{q+1}\mu_q}\right) \frac{\lambda_q\delta_q^\frac{1}{2}}{\lambda_{q+1}},
\end{align*}
and the estimates on $  \varphi_{H2}$ follow, 
\begin{align*}
\left\|  \varphi_{H2}\right\|_N 
\lesssim_{n, p, h, M} \lambda_{q+1}^N \lambda_q \delta_q^\frac{1}{2} \frac {\delta_{q+1}}{\lambda_{q+1}},  \quad
\left\|   D_{t,q+1}  \varphi_{H2}\right\|_{N-1} 
\lesssim_{n, p, h, M}\lambda_{q+1}^N\delta_{q+1}^\frac{1}{2} \lambda_q \delta_q^\frac{1}{2} \frac {\delta_{q+1}}{\lambda_{q+1}}.
\end{align*}
To summarize, we get
\begin{align}
\left\|  \varphi_{H}\right\|_N 
\leqslant \frac {1}{10}\lambda_{q+1}^{N-3\gamma} \delta_{q+2}^\frac{3}{2} ,  \quad
\left\|   D_{t,q+1}  \varphi_{H}\right\|_{N-1} 
\leqslant \frac {1}{10}\lambda_{q+1}^{N-3\gamma}\delta_{q+1}^\frac{1}{2} \delta_{q+2}^\frac{3}{2},\label{est on varphi_H}
\end{align}
for sufficiently small $b-1>0$ and large $\lambda_0$. 
\subsection{Transport current error}
Recalling the definition of $\varphi_T,  \varphi_{T1}$ and $\varphi_{T2}$, we could first see
\begin{align*}
	\frac{\tm_p\otimes \tm_p}{n^2} - \delta_{q+1} \Id + R_\ell=\delta_{q+1}\sumu \sum_{k\in \Z^3\setminus \{0\}}   (d_{u,k}/n^2) e^{ i\lambda_{q+1} k\cdot \xi_I},
\end{align*}
and then
\begin{align}
\left\|\frac{\tm_p\otimes \tm_p}{n^2} - \delta_{q+1} \Id + R_\ell\right\|_N
&\lesssim \delta_{q+1}\sumu\sum_{k\in \Z^3\setminus \{0\}}\left\|  (d_{u,k}/n^2) e^{ i\lambda_{q+1} k\cdot \xi_I}\right\|_{N} 
\lesssim_{n, p, h, M} \lambda_{q+1}^N \delta_{q+1}, \label{est on G_O1}\\	
\left\|   D_{t,q+1} \left(\frac{\tm_p\otimes \tm_p}{n^2} - \delta_{q+1} \Id + R_\ell\right)\right\|_{N-1}
&\lesssim \delta_{q+1} \sumu\sum_{k\in \Z^3\setminus \{0\}} \left\|  \left(\DTL (d_{u,k}/n^2)\right) e^{ i\lambda_{q+1} k\cdot \xi_I}\right\|_{N-1}\nonumber\\
&\quad+\left\|(\tm+(m_q-m_\ell)\cdot \nabla )\left(\frac{\tm_p\otimes \tm_p}{n^2} - \delta_{q+1} \Id + R_\ell\right)\right\|_{N-1}\nonumber\\ 
&\lesssim_{n, p, h, M} \lambda_{q+1}^N\delta_{q+1}^\frac{1}{2} \cdot \delta_{q+1}, \label{est on Dtl G_O1}
\end{align}
where we use \eqref{est on m_q-m_l}, \eqref{est on cof of duk}, \eqref{est on tm}, and \eqref{est on Dtl tm}. Next, recalling the definition of $n\varphi_{T1}$
\begin{align*}
n\varphi_{T1} &= -\kappa_{q+1}  \tm+ \frac{1}{2} \tr\left(  \frac{\tm_p\otimes \tm_p}{n^2} - \delta_{q+1} \Id + R_\ell-R_t-R_{EB2} \right)(m_q-m_\ell) - \frac{m_q\zeta}{n},
\end{align*}
and using \eqref{est on m_q}, \eqref{est on m_q-m_l}, \eqref{est on R_q+1-zeta}, \eqref{est on R_O1}, \eqref{est on R_t}, \eqref{est on R_EB},  \eqref{est on zeta}, \eqref{est on G_O1}, and $\kappa_{q+1} = \frac{1}{2} \tr(R_{q+1})$, we have
\begin{align*}
\left\| n \varphi_{T1}\right\|_N 
&\lesssim  \sum_{N_1+N_2=N}\left\|R_{q+1} -{\textstyle{\frac 23}}  \zeta {/n}\Id +{\textstyle{\frac 23}}  \zeta{/n}\Id\right\|_{N_1}\left\|\tm\right\|_{N_2}\\
&\quad+  \sum_{N_1+N_2=N}\left\| \frac{\tm_p\otimes \tm_p}{n^2} - \delta_{q+1} \Id + R_\ell-R_t-R_{EB}\right\|_{N_1}\left\|m_q-m_\ell\right\|_{N_2}
+\left\|m_q/n\right\|_{N}\left\|\zeta\right\|_0\\
&\lesssim _{n, p, h, M} \lambda_{q+1}^N\left(\lambda_q^\frac{1}{2}\lambda_{q+1}^{-\frac{1}{2}}\delta_q^\frac{1}{4}\delta_{q+1}^\frac{5}{4} +\ell^2\lambda_q^2 \delta_q^\frac{1}{2}\left( \delta_{q+1}+\frac{\delta_{q+1}^{\frac{1}{2}}}{\lambda_{q+1}}+\frac{\delta_{q+1}^{\frac{1}{2}}}{\lambda_{q+1}(\lambda_{q+1}\mu_{q})}\right)\right)
\lesssim_{n, p, h, M} \lambda_{q+1}^N 
\lambda_q^\frac{1}{2}\lambda_{q+1}^{-\frac{1}{2}}\delta_q^\frac{1}{4}\delta_{q+1}^\frac54.
\end{align*}
For the advective derivative, we use \eqref{est on Dtq+1 n} and \eqref{est on Dtl G_O1} to get
\begin{align*}
\left\|   D_{t,q+1} \varphi_{T1}\right\|_{N-1}
&\lesssim
\sum_{N_1+N_2=N-1}\left\|   D_{t,q+1} (R_{q+1} -{\textstyle{\frac 23} \frac{\zeta}{n}} \Id + {\textstyle{\frac 23} \frac{\zeta}{n}} \Id)\right\|_{N_1}\left\| \tm\right\|_{N_2} +\left\|R_{q+1}-{\textstyle{\frac 23}\frac{\zeta}{n} \Id + {\textstyle{\frac 23}} \frac{\zeta}{n} \Id} \right\|_{N_1}\left\|    D_{t,q+1} \tm\right\|_{N_2} \\
&\quad+ \sum_{N_1+N_2=N-1} \left\|   D_{t,q+1} \left(\frac{\tm_p\otimes \tm_p}{n^2} -\delta_{q+1}I+ R_\ell\right)\right\|_{N_1}\left\|m_q-m_\ell\right\|_{N_2}\\
&\quad+ \sum_{N_1+N_2=N-1}  \left\|\frac{\tm_p\otimes \tm_p}{n^2} -\delta_{q+1}I+ R_\ell\right\|_{N_1}\left\|   D_{t,q+1} (m_q-m_\ell)\right\|_{N_2}
+\left\|D_{t,q+1}(m_\ell\zeta /n^2)\right\|_N
\\
&\lesssim_{n, p, h, M} \lambda_{q+1}^N \delta_{q+1}^\frac{1}{2}
\lambda_q^\frac{1}{2}\lambda_{q+1}^{-\frac{1}{2}}\delta_q^\frac{1}{4}\delta_{q+1}^\frac54.
\end{align*}
As for $  \varphi_{T2}$, we could calculate
\begin{align*}
n\varphi_{T2} &= \frac{1}{2}\mathcal{R} \left(n \DTL \tr\left(  \frac{\tm_p\otimes \tm_p}{n^2} - \delta_{q+1} \Id + R_\ell  \right) 
-\tr\left(  \frac{\tm_p\otimes \tm_p}{n^2} - \delta_{q+1} \Id + R_\ell \right)\Div( m_q - m_\ell)\right)\\
&=\frac{1}{2}\cR\left(\sum_{u} \sum_{k\in \Z^3\setminus \{0\}}  \delta_{q+1} \tr(n \DTL  d_{u,k}+(\partial_t n - \partial_t n_\ell) d_{u,k}) e^{ i\lambda_{q+1} k\cdot \xi_I}\right),
\end{align*}
and use \eqref{est on cof of duk} to obtain for $\overline{N}\geqslant 0$
$$\begin{aligned}
\left\|\tr(n \DTL  d_{m,k}+(\partial_t n - \partial_t n_\ell) d_{m,k})\right\|_{\overline{N}}+(\lambda_{q+1}\delta_{q+1}^\frac{1}{2})^{-1}\left\|\DTL  \tr(n \DTL  d_{m,k}+(\partial_t n - \partial_t n_\ell) d_{m,k})\right\|_{\overline{N}} \lesssim_{n, p, h, M,\overline{N}}  \mu_q^{-\overline{N}}\tau_{q}^{-1} \underset{I}{\max}|\overset{\circ}{d}_{I,k}|.
\end{aligned}$$
Then, we could use $\supp(d_{u,k}) \subset (t_u -\frac{1}{2}\tau_q, t_u + \frac 32\tau_q) \times \R^3$ and apply Corollary \ref{est on R operator} to $n\varphi_{T2}$ to get
\begin{align*}
\left\|  \varphi_{T2}\right\|_N
\lesssim_{n, p, h, M}  \lambda_{q+1}^N
\frac {\delta_{q+1}}{\lambda_{q+1} \tau_q},\quad
\left\|   D_{t,q+1}  \varphi_{T2}\right\|_{N-1} \lesssim_{n, p, h, M}  \lambda_{q+1}^N\delta_{q+1}^\frac{1}{2}
\frac {\delta_{q+1}}{\lambda_{q+1} \tau_q}. 
\end{align*}
To summarize, we have
\begin{align}
\left\|  \varphi_{T}\right\|_N\leqslant \frac {1}{10}  \lambda_{q+1}^{N-3\gamma}\delta_{q+2}^\frac{3}{2},\quad
\left\|   D_{t,q+1}  \varphi_{T}\right\|_{N-1}\leqslant \frac {1}{10}  \lambda_{q+1}^{N-3\gamma}\delta_{q+1}^\frac{1}{2}\delta_{q+2}^\frac{3}{2},\label{est on varphi_T}
\end{align}
for sufficiently small $b-1>0$ and large $\lambda_0$. 
\subsection{Oscillation current error} 
Note that 
\begin{align*}
n\varphi_{O1}=\cR\left(\Div \left(\frac {|\tm_p|^2\tm_p}{2n^2} +n \varphi_\ell \right)\right)
=\delta_{q+1}^\frac 32\cR\left(\sumu \sum_{k\in \Z^3\setminus \{0\}}   \Div\left(n_{u,k}/n^2\right)e^{ i\lambda_{q+1} k\cdot \xi_I}\right),
\end{align*}
because of $\overset{\circ}{n}_{I,k} (f_I\cdot k) =0$. We use \eqref{est on cof of nuk} to get
\begin{align*}
\left\|\DTL  \Div (n_{u,k}/n^2)\right\|_{\overline{N}}
\lesssim\left\|\DTL  (n_{u,k}/n^2)\right\|_{\overline{N}+1} +\left\|\nabla (m_\ell/n)^{\top}:\nabla (n_{u,k}/n^2)\right\|_{\overline{N} }
\lesssim_{n, p, h, M} \mu_q^{-\overline{N}-1}\lambda_{q+1}\delta_{q+1}^\frac{1}{2} \underset{I}{\max}|\overset{\circ}{d}_{I,k}|,
\end{align*}
and then
\begin{align*}
	\left\|\Div (n_{u,k}/n^2)\right\|_{\overline{N}}+(\lambda_{q+1}\delta_{q+1}^\frac{1}{2})^{-1}\left\|\DTL  \Div (n_{u,k}/n^2)\right\|_{\overline{N}}
	\lesssim_{n, p, h, M} \mu_q^{-\overline{N}-1} \underset{I}{\max}|\overset{\circ}{d}_{I,k}|.
\end{align*}
Therefore, using $\supp(d_{u,k}) \subset (t_u -\frac{1}{2}\tau_q, t_u + \frac 32\tau_q) \times \R^3$, and  Corollary \ref{est on R operator} with \eqref{est on cof of nuk}, we have
\begin{align*}
\left\| n \varphi_{O1}\right\|_N
\lesssim_{n, p, h, M} \lambda_{q+1}^N\frac {\delta_{q+1}^\frac 32}{\lambda_{q+1}\mu_q}, \quad
\left\|   D_{t,q+1} (n \varphi_{O1})\right\|_{N-1} 
\lesssim_{n, p, h, M} \lambda_{q+1}^N \delta_{q+1}^\frac{1}{2} \frac {\delta_{q+1}^\frac 32}{\lambda_{q+1}\mu_q} .
\end{align*}
Next, recall that $ n \varphi_{O2} = \frac {|\tm|^2\tm - |\tm_p|^2 \tm_p}{2n^2}$. Then, we could use \eqref{est on tm_EBp}, \eqref{est on tm_c}, \eqref{est on Dtl tm_EBp}, and \eqref{est on Dtl tm_c} to get
\begin{align*}
\left\| n \varphi_{O2}\right\|_N
&\lesssim 
\left\|\frac{(\tm_p\cdot \tm_c) \tm}{n^2}\right\|_N 
+\left\|\frac{|\tm_c|^2 \tm}{2n^2}\right\|_N +\left\|\frac{|{\tm_p}|^2 \tm_c}{2n^2}\right\|_N
\lesssim_{n, p, h, M} \lambda_{q+1}^N \cdot \frac {\delta_{q+1}^\frac{3}{2}}{\lambda_{q+1}\mu_q},\\
\left\|   D_{t,q+1} (n \varphi_{O2})\right\|_{N-1}
&\lesssim\left\|   D_{t,q+1} \left(\frac{|{\tm_p}|^2 \tm_c}{2n^2}\right)\right\|_{N-1}
+\left\|   D_{t,q+1} \left(\frac{(\tm_p\cdot \tm_c)\tm}{n^2} +\frac{ |\tm_c|^2\tm}{2n^2}\right)\right\|_{N-1}\lesssim_{n, p, h, M} \lambda_{q+1}^N\delta_{q+1}^\frac{1}{2} \cdot \frac {\delta_{q+1}^\frac{3}{2}}{\lambda_{q+1}\mu_q}.
\end{align*}
Therefore, combining the estimates, we get
\begin{align}
\left\|  \varphi_{O}\right\|_N\leqslant \frac 1{10} \lambda_{q+1}^{N-3\gamma} \delta_{q+2}^\frac{3}{2}, \quad 
\left\|   D_{t,q+1}  \varphi_{O}\right\|_{N-1}\leqslant \frac 1{10} \lambda_{q+1}^{N-3\gamma} \delta_{q+1}^\frac{1}{2}\delta_{q+2}^\frac{3}{2},\label{est on varphi_O}
\end{align}
for sufficiently small $b-1>0$ and large $\lambda_0$. 
\subsection{Reynolds current error} 
Recall that $ n \varphi_R = (R_{q+1}-\frac23(\zeta/n) \Id) \tm + \frac23 (\zeta/n)  \tm$. Similar to the estimate for $\kappa_{q+1} \tm$ in $  \varphi_{T1}$, we have
\begin{align*}
\left\|  \varphi_R\right\|_N 
&\lesssim_{n, p, h, M} \lambda_{q+1}^N 
\lambda_q^\frac{1}{2}\lambda_{q+1}^{-\frac{1}{2}} \delta_q^\frac{1}{4}\delta_{q+1}^\frac54
\leqslant \frac {1}{10} \lambda_{q+1}^{N-3\gamma} \delta_{q+2}^\frac{3}{2},\\
\left\|   D_{t,q+1}   \varphi_R\right\|_{N-1}
&\lesssim_{n, p, h, M} \lambda_{q+1}^N\delta_{q+1}^\frac{1}{2} \lambda_q^\frac{1}{2}\lambda_{q+1}^{-\frac{1}{2}} \delta_q^\frac{1}{4}\delta_{q+1}^\frac54
\leqslant \frac {1}{10} \lambda_{q+1}^{N-3\gamma}\delta_{q+1}^\frac{1}{2} \delta_{q+2}^\frac{3}{2},
\end{align*}
for sufficiently small $b-1>0$ and large $\lambda_0$. 
\subsection{Mediation current error} 
Recall that $n \varphi_M = n \varphi_{M1} + n \varphi_{M2} + n \varphi_{M3} + n \varphi_{M4}$ where
\begin{align*}
n\varphi_{M1} &=   \frac{|m_q-m_\ell|^2}{2n}\frac {\tm}{n}+ n(\varphi_q-\varphi_\ell),\\
n\varphi_{M2} &= \left (\frac{\tm \otimes \tm}{n}+n R_q -n R_{q+1} - \delta_{q+1} n \Id \right)\frac{m_q-m_\ell}{n},\\ 
n\varphi_{M3} &= \mathcal{R} \left( 
\Div (m_q-m_\ell) \frac{\tm\cdot m_\ell}{n^2} \right),\\
n\varphi_{M4} &= \mathcal{R} \left( \frac{\tm}{n} \cdot \nabla(p(n) - p_\ell (n))  \right).
\end{align*}
For $\varphi_{M1}$, we use \eqref{est on m_q-m_l}, \eqref{est on phi_q-phi_l}, and \eqref{est on Dtl tm} to get 
\begin{align*}
\left\|  \varphi_{M1}\right\|_N \leqslant \frac 1{20} \lambda_{q+1}^{N-3\gamma} \delta_{q+2}^\frac{3}{2}, \quad 
\left\|   D_{t,q+1}    \varphi_{M1}\right\|_{N-1} \leqslant \frac 1{20} \lambda_{q+1}^{N-3\gamma}\delta_{q+1}^\frac{1}{2} \delta_{q+2}^\frac{3}{2}.
\end{align*}
for sufficiently small $b-1>0$ and large $\lambda_0$. For $\varphi_{M2}$, we estimate $\varphi_{M2}$ in a similar way as $  \varphi_{T1}$ and $\varphi_{H2}$,
\begin{align*}
\left\|\left (\frac{\tm \otimes \tm}{n}+n R_q -n R_{q+1} - \delta_{q+1} n \Id \right) \frac{m_q-m_\ell}{n}\right\|_{N}
&\lesssim_{n, p, h, M} \lambda_{q+1}^N\lambda_q^\frac{1}{2}\lambda_{q+1}^{-\frac{1}{2}} \delta_q^\frac{1}{4}\delta_{q+1}^\frac54,\\
\left\|   D_{t,q+1} \left[\left (\frac{\tm \otimes \tm}{n}+n R_q -n R_{q+1} - \delta_{q+1} n \Id \right) \frac{m_q-m_\ell}{n}\right]\right\|_{N-1}
&\lesssim_{n, p, h, M} \lambda_{q+1}^N\delta_{q+1}^\frac{1}{2}\lambda_q^\frac{1}{2}\lambda_{q+1}^{-\frac{1}{2}} \delta_q^\frac{1}{4}\delta_{q+1}^\frac54.
\end{align*}
For $\varphi_{M3}$ and $\varphi_{M4}$, notice that 
\begin{align*}
&\varphi_{M3}=\delta_{q+1}^\frac{1}{2}\cR\left(\sum_{u} \sum_{k\in \Z^3\setminus \{0\}}  (\partial_tn_\ell-\partial_t n) \frac{m_\ell}{n^2} \cdot (s_{u,k} + \tilde{e}_{u,k}+\tilde{g}_{u,k})) e^{ i\lambda_{q+1} k\cdot \xi_I}\right)+\cR\left((\partial_tn_\ell-\partial_t n)\frac{\tm_{t}\cdot m_\ell }{n^2}\right),\\
&\varphi_{M4}=\delta_{q+1}^\frac{1}{2}\cR\left(\sum_{u} \sum_{k\in \Z^3\setminus \{0\}}  \nabla(p(n) - p_\ell (n)) \frac{1}{n} \cdot (s_{u,k} + \tilde{e}_{u,k}+\tilde{g}_{u,k})) e^{ i\lambda_{q+1} k\cdot \xi_I}\right)+\cR\left(\frac{\tm_t}{n} \cdot \nabla(p(n) - p_\ell (n))  \right).
\end{align*}
Therefore, since $\supp(b_{u,k}) \subset (t_u -\frac{1}{2}\tau_q, t_u + \frac 32\tau_q) \times \R^3$ and
\begin{align*}
&\left\|(\partial_tn_\ell-\partial_t n) \frac{m_\ell}{n^2} \cdot (s_{u,k} + \tilde{e}_{u,k}+\tilde{g}_{u,k}))\right\|_{\overline{N}}\\
&\quad+(\lambda_{q+1}\delta_{q+1}^\frac{1}{2})^{-1}\left\|\DTL ( (\partial_tn_\ell-\partial_t n) \frac{m_\ell}{n^2} \cdot (s_{u,k} + \tilde{e}_{u,k}+\tilde{g}_{u,k})))\right\|_{\overline{N}}
\lesssim_{n, p, h, M} \mu_q^{-\overline{N}}\ell^2\lambda_{q}\delta_{q}^{\frac{1}{2}} \underset{I}{\max}|\overset{\circ}{c}_{I,k}|,\\
&\left\| \nabla(p(n) - p_\ell (n)) \frac{1}{n} \cdot (s_{u,k} + \tilde{e}_{u,k}+\tilde{g}_{u,k}))\right\|_{\overline{N}}\\
&\quad+(\lambda_{q+1}\delta_{q+1}^\frac{1}{2})^{-1}\left\|\DTL   (\nabla(p(n) - p_\ell (n)) \frac{1}{n} \cdot (s_{u,k} + \tilde{e}_{u,k}+\tilde{g}_{u,k})))\right\|_{\overline{N}}
\lesssim_{n, p, h, M} \mu_q^{-\overline{N}} \ell^2\underset{I}{\max}|\overset{\circ}{c}_{I,k}|,
\end{align*}
it follows from Corollary \ref{est on R operator} that
\begin{align*}
\left\|\cR\left(\sum_{u} \sum_{k\in \Z^3\setminus \{0\}} \delta_{q+1}^\frac{1}{2} (\partial_tn_\ell-\partial_t n) \frac{m_\ell}{n^2} \cdot (s_{u,k} + \tilde{e}_{u,k}+\tilde{g}_{u,k})) e^{ i\lambda_{q+1} k\cdot \xi_I}\right)\right\|_N&\lesssim_{n, p, h, M}\lambda_{q+1}^N\lambda_{q}^{\frac{1}{2}}\lambda_{q+1}^{-\frac{1}{2}} \delta_q^\frac{1}{4}\delta_{q+1}^\frac54,\\
\left\|D_{t,q+1}\cR\left(\sum_{u} \sum_{k\in \Z^3\setminus \{0\}} \delta_{q+1}^\frac{1}{2} (\partial_tn_\ell-\partial_t n) \frac{m_\ell}{n^2} \cdot (s_{u,k} + \tilde{e}_{u,k}+\tilde{g}_{u,k})) e^{ i\lambda_{q+1} k\cdot \xi_I}\right)\right\|_N&\lesssim_{n, p, h, M}\lambda_{q+1}^N\delta_{q+1}^{\frac{1}{2}}\lambda_{q}^{\frac{1}{2}}\lambda_{q+1}^{-\frac{1}{2}} \delta_q^\frac{1}{4}\delta_{q+1}^\frac54,\\
\left\|\cR\left(\sum_{u} \sum_{k\in \Z^3\setminus \{0\}} \delta_{q+1}^\frac{1}{2} \nabla(p(n) - p_\ell (n)) \frac{1}{n} \cdot (s_{u,k} + \tilde{e}_{u,k}+\tilde{g}_{u,k})) e^{ i\lambda_{q+1} k\cdot \xi_I}\right)\right\|_N&\lesssim_{n, p, h, M}\lambda_{q+1}^N\lambda_{q}^{\frac{1}{2}}\lambda_{q+1}^{-\frac{1}{2}} \delta_q^\frac{1}{4}\delta_{q+1}^\frac54,\\
\left\|D_{t,q+1}\cR\left(\sum_{u} \sum_{k\in \Z^3\setminus \{0\}} \delta_{q+1}^\frac{1}{2} \nabla(p(n) - p_\ell (n)) \frac{1}{n} \cdot (s_{u,k} + \tilde{e}_{u,k}+\tilde{g}_{u,k})) e^{ i\lambda_{q+1} k\cdot \xi_I}\right)\right\|_N&\lesssim_{n, p, h, M}\lambda_{q+1}^N\delta_{q+1}^{\frac{1}{2}}\lambda_{q}^{\frac{1}{2}}\lambda_{q+1}^{-\frac{1}{2}} \delta_q^\frac{1}{4}\delta_{q+1}^\frac54.
\end{align*}
And it follows from \eqref{est on m_l}, \eqref{est on pt Dtl m_l N} and \eqref{est on tm_t} that
\begin{align*}
	\left\|\cR\left((\partial_tn_\ell-\partial_t n)\frac{\tm_{t}\cdot m_\ell }{n^2}\right)\right\|_N&\lesssim_{n, p, h, M}\lambda_{q+1}^N\lambda_{q}^{\frac{1}{2}}\lambda_{q+1}^{-\frac{1}{2}} \delta_q^\frac{1}{4}\delta_{q+1}^\frac54,\\
	\left\|D_{t,q+1}\left((\partial_tn_\ell-\partial_t n)\frac{\tm_{t}\cdot m_\ell }{n^2}\right)\right\|_N&\lesssim_{n, p, h, M}\lambda_{q+1}^N\delta_{q+1}^{\frac{1}{2}}\lambda_{q}^{\frac{1}{2}}\lambda_{q+1}^{-\frac{1}{2}} \delta_q^\frac{1}{4}\delta_{q+1}^\frac54,\\
	\left\|\cR\left(\frac{\tm_t}{n} \cdot \nabla(p(n) - p_\ell (n))  \right)\right\|_N&\lesssim_{n, p, h, M}\lambda_{q+1}^N\lambda_{q}^{\frac{1}{2}}\lambda_{q+1}^{-\frac{1}{2}} \delta_q^\frac{1}{4}\delta_{q+1}^\frac54,\\
	\left\|D_{t,q+1}\cR\left(\frac{\tm_t}{n} \cdot \nabla(p(n) - p_\ell (n))  \right)\right\|_N&\lesssim_{n, p, h, M}\lambda_{q+1}^N\delta_{q+1}^{\frac{1}{2}}\lambda_{q}^{\frac{1}{2}}\lambda_{q+1}^{-\frac{1}{2}} \delta_q^\frac{1}{4}\delta_{q+1}^\frac54.
\end{align*}
To summarize, we get
\begin{align*}
\left\|  \varphi_{M}\right\|_N \leqslant \frac {1}{10} \lambda_{q+1}^{N-3\gamma} \delta_{q+2}^\frac{3}{2}, \quad 
\left\|   D_{t,q+1}  \varphi_{M}\right\|_{N-1} \leqslant \frac {1}{10} \lambda_{q+1}^{N-3\gamma}\delta_{q+1}^\frac{1}{2} \delta_{q+2}^\frac{3}{2}.
\end{align*}
for sufficiently small $b-1>0$ and large $\lambda_0$.
\subsection{Electromagnetic current error}
Recall that
\begin{align*}
n\varphi_{E}
&:=\cR\left(\tm\cdot\tE+\tm\cdot(E_{q}-E_{\ell})+ (m_q-m_\ell)\cdot\tE\right) ,\\
n\varphi_{B}	&:=\cR\left(\frac{\tm\cdot m_\ell\times \tB}{n}+\frac{(m_q-m_{\ell})\cdot m_\ell\times \tB}{n}+\frac{\tm\cdot m_\ell\times (B_q-B_\ell)}{n}\right). 
\end{align*}
By using $\left\|\cR f\right\|_0 \lesssim\left\|f\right\|_0$, we have for $r\leqslant 1$,
\begin{align}
\|\partial_{t}^r(n\varphi_{E})\|_N&\lesssim\sum_{r_1+r_2=r}\sum_{N_1+N_2=N}\|\partial_{t}^{r_1}\tm\|_{N_1}\|\partial_{t}^{r_2}\tE\|_{N_2}+\sum_{r_1+r_2=r}\sum_{N_1+N_2=N}\|\partial_{t}^{r_1}\tm\|_{N_1}\|\partial_{t}^{r_2}(E_q-E_\ell)\|_{N_2}\notag\\
&\quad+\sum_{r_1+r_2=r}\sum_{N_1+N_2=N}\|\partial_{t}^{r_1}(m_q-m_\ell)\|_{N_1}\|\partial_{t}^{r_2}\tE\|_{N_2}\notag\\
&\lesssim_{n, p, h, M}\lambda_{q+1}^{N+r}(\lambda_{q+1}^{-1}\delta_{q+1}+\ell^2\lambda_{q}\delta_{q}^{\frac{1}{2}}\delta_{q+1}^{\frac{1}{2}}+\ell^2\lambda_{q}^2\delta_{q}^{\frac{1}{2}}\lambda_{q+1}^{-1}\delta_{q+1}^{\frac{1}{2}})\notag\\
&\lesssim_{n, p, h, M}\lambda_{q+1}^{N+r}\lambda_{q}^{\frac{1}{2}}\lambda_{q+1}^{-\frac{1}{2}} \delta_q^\frac{1}{4}\delta_{q+1}^\frac54,\label{est on varphi_E}\\
\|D_{t,q+1}(n\varphi_{E})\|_{N-1}&\lesssim\|\partial_{t}(n\varphi_{E})\|_{N-1}+\sum_{N_1+N_2=N-1}(\|\tm\|_{N_1}+\|m_q\|_{N_1})\|n\varphi_{E}\|_{N_2+1}\notag\\
&\lesssim_{n, p, h, M}\lambda_{q+1}^{N}(\lambda_{q+1}^{-1}\delta_{q+1}+\ell^2\lambda_{q}\delta_{q}^{\frac{1}{2}}\delta_{q+1}^{\frac{1}{2}}+\ell^2\lambda_{q}^2\delta_{q}^{\frac{1}{2}}\lambda_{q+1}^{-1}\delta_{q+1}^{\frac{1}{2}})(1+\delta_{q+1}^{\frac{1}{2}})\notag\\
&\lesssim_{n, p, h, M}\lambda_{q+1}^{N}\delta_{q+1}^{\frac{1}{2}}\lambda_{q}^{\frac{1}{2}}\lambda_{q+1}^{-\frac{1}{2}} \delta_q^\frac{1}{4}\delta_{q+1}^\frac54,\label{est on Dtl varphi_E}\\
\|\partial_{t}^r(n\varphi_{B})\|_N&\lesssim\|\partial_{t}n^{-1}\|_N\sum_{r_0+r_1+r_2=r}\sum_{N_0+N_1+N_2=N}\|\partial_{t}^{r_0}\tm\|_{N_0}\|\partial_{t}^{r_1}m_\ell\|_{N_1}\|\partial_{t}^{r_2}\tB\|_{N_2}\notag\\
&\quad+\|\partial_{t}n^{-1}\|_N\sum_{r_0+r_1+r_2=r}\sum_{N_0+N_1+N_2=N}\|\partial_{t}^{r_0}(m_q-m_\ell)\|_{N_0}\|\partial_{t}^{r_1}m_\ell\|_{N_1}\|\partial_{t}^{r_2}\tB\|_{N_2}\notag\\
&\quad+\|\partial_{t}n^{-1}\|_N\sum_{r_0+r_1+r_2=r}\sum_{N_0+N_1+N_2=N}\|\partial_{t}^{r_0}\tm\|_{N_0}\|\partial_{t}^{r_1}m_\ell\|_{N_1}\|\partial_{t}^{r_2}(B_q-B_\ell)\|_{N_2}\notag\\
&\lesssim_{n, p, h, M}\lambda_{q+1}^{N+r}(\lambda_{q+1}^{-1}\delta_{q+1}+\ell^2\lambda_{q}\delta_{q}^{\frac{1}{2}}\delta_{q+1}^{\frac{1}{2}}+\ell^2\lambda_{q}^2\delta_{q}^{\frac{1}{2}}\lambda_{q+1}^{-1}\delta_{q+1}^{\frac{1}{2}})\notag\\
&\lesssim_{n, p, h, M}\lambda_{q+1}^{N+r}\lambda_{q}^{\frac{1}{2}}\lambda_{q+1}^{-\frac{1}{2}} \delta_q^\frac{1}{4}\delta_{q+1}^\frac54,\label{est on varphi_B}\\
\|D_{t,q+1}(n\varphi_{B})\|_{N-1}&\lesssim\|\partial_{t}(n\varphi_{B})\|_{N-1}+\sum_{N_1+N_2=N-1}(\|\tm\|_{N_1}+\|m_q\|_{N_1})\|n\varphi_{B}\|_{N_2+1}\\
&\lesssim_{n, p, h, M}\lambda_{q+1}^{N}(\lambda_{q+1}^{-1}\delta_{q+1}+\ell^2\lambda_{q}\delta_{q}^{\frac{1}{2}}\delta_{q+1}^{\frac{1}{2}}+\ell^2\lambda_{q}^2\delta_{q}^{\frac{1}{2}}\lambda_{q+1}^{-1}\delta_{q+1}^{\frac{1}{2}})(1+\delta_{q+1}^{\frac{1}{2}})\notag\\
&\lesssim_{n, p, h, M}\lambda_{q+1}^{N}\delta_{q+1}^{\frac{1}{2}}\lambda_{q}^{\frac{1}{2}}\lambda_{q+1}^{-\frac{1}{2}} \delta_q^\frac{1}{4}\delta_{q+1}^\frac54,\label{est on Dtl varphi_B}
\end{align}
where we used \eqref{est on m_q-m_l}--\eqref{est on Dtl B_q-B_l}, \eqref{est on pt Dtl m_l N}, \eqref{est on tm}, \eqref{est on tEB}, \eqref{est on Dtl tm}, and \eqref{est on Dtl tEB}.
To summarize, we get
\begin{align*}
	\left\|  \varphi_{E}\right\|_N &\leqslant \frac {1}{10} \lambda_{q+1}^{N-3\gamma} \delta_{q+2}^\frac{3}{2}, \quad 
	\left\|   D_{t,q+1}  \varphi_{E}\right\|_{N-1} \leqslant \frac {1}{10} \lambda_{q+1}^{N-3\gamma}\delta_{q+1}^\frac{1}{2} \delta_{q+2}^\frac{3}{2},\\
	\left\|  \varphi_{B}\right\|_N &\leqslant \frac {1}{10} \lambda_{q+1}^{N-3\gamma} \delta_{q+2}^\frac{3}{2}, \quad 
	\left\|   D_{t,q+1}  \varphi_{B}\right\|_{N-1} \leqslant \frac {1}{10} \lambda_{q+1}^{N-3\gamma}\delta_{q+1}^\frac{1}{2} \delta_{q+2}^\frac{3}{2}.
\end{align*}
for sufficiently small $b-1>0$ and large $\lambda_0$.
\subsection{Estimates on $\zeta$}\label{Estimates on zeta}
In this section, we prove
\begin{align}\label{est on zeta'}
	\left\|\zeta'\right\|_0
	\leqslant \frac{\varepsilon_0^2}{50\underline{M}(1+T+\tau_0)} \lambda_{q+1}^{-3\gamma} \delta_{q+2}^\frac{3}{2},
\end{align}
which implies \eqref{est on zeta} by integration in time.
\subsubsection{Estimates on $ \zeta_1$ and $ \zeta_3$} 
From \eqref{dec of n varphi_H2}, $\zeta_3'$ can be written as
\begin{align*}\label{dec of zeta 3}
	\zeta_3'= \underbrace{\fint_{\T^3} \left(\frac{\tm_p\otimes \tm_p}{n} - \delta_{q+1}n\Id + n R_\ell\right) : \nabla (m_\ell/n) \rd x }_{=: \zeta_{31}'} 
	\underbrace{- \int_{\T^3}
		n (R_{O1}+R_t+R_T+R_N+ R_{EB}): \nabla (m_\ell/n)\rd x }_{ \zeta_{32}'}.
\end{align*}
  In a similar way as the estimates on $\tm_t$ in Proposition \ref{est on perturbation}, we could use  \eqref{def of tm_p^2}, \eqref{est on tm_t}, and Lemma \ref{est on int operator} to obtain $$
  \left\|\zeta_1'\right\|_0+\left\|\zeta_{31}'\right\|_0 \leqslant \frac{\varepsilon_0^2}{300\underline{M}(1+T+\tau_0)} \lambda_{q+1}^{-3\gamma} \delta_{q+2}^\frac{3}{2}.$$
  As for $\zeta_{32}'$, recalling the method we have used to get estimates on $n\varphi_{H2}$, we have
  \begin{align*}
  	\left\langle n R_\tri : \nabla \frac {m_\ell}{n}\right\rangle
  	&=
  	\left\langle\mathcal{R}P_{\lesssim \lambda_{q+1}}G_\tri : \nabla \frac {m_\ell}{n}\right\rangle
  	+\left\langle \mathcal{R}P_{\GL}G_\tri : \nabla ( {m_\ell}\PL{n}^{-1})\right\rangle+\left\langle \mathcal{R}P_{\GL}G_\tri : \nabla  ({m_\ell}P_{\geqslant \ell^{-1}}{n}^{-1})\right\rangle.
  \end{align*}
  where $n R_\tri$ represents either $n R_{O1}$, $n R_T$, $n R_N$ or $n R_{EB1}$ and can be written as $\mathcal{R}G_\tri$. Since the second term has frequency localized to $\gtrsim \lambda_{q+1}$, it has zero-mean. The magnitude of the first term can be estimated {by ${\varepsilon_0^2}/(600\underline M(1+T+\tau_0)) \lambda_{q+1}^{-3\gamma}\delta_{q+2}^\frac{3}{2}$} as in \eqref{est on low 1} and \eqref{est on low 2}. The estimate on the last term follows from
\begin{align*}
&\quad|\langle \mathcal{R}P_{\GL}G_\tri : \nabla  ({m_\ell}P_{\geqslant \ell^{-1}}{n}^{-1})\rangle |\\
&\lesssim 
\left\|\mathcal{R}P_{\GL}G_\tri\right\|_0\left\|\nabla  ({m_\ell}P_{\geqslant \ell^{-1}}{n}^{-1})\right\|_0\lesssim_{n, p, h, M} \lambda_{q+1}^{-1}\left\|G_\tri\right\|_0 \ell^2\lambda_q\delta_q^\frac{1}{2}\leqslant \frac{\varepsilon_0^2}{600\underline{M}(1+T+\tau_0)} \lambda_{q+1}^{-3\gamma}\delta_{q+2}^\frac{3}{2},
\end{align*}
where we use \eqref{est on PG n} in the second inequality and \eqref{est on G_tri} in the last one. Up to now, we have proved 
\begin{align*}
	\left\|\zeta_1'\right\|_0+\left\|\zeta_3'\right\|_0
	\leqslant \frac{\varepsilon_0^2}{150\underline{M}(1+T+\tau_0)} \lambda_{q+1}^{-3\gamma} \delta_{q+2}^\frac{3}{2}.
\end{align*}
	
\subsubsection{Estimates on $ \zeta_0$, $ \zeta_2$,  and $ \zeta_4$}
	We first decompose $\zeta_0'$ into $\zeta_{01}'$ and $\zeta_{02}'$:
	\begin{align*}
		(2\pi)^3 \zeta_{01}' &= \int_{\T^3} \frac{n}2 \DTL \tr\left(  \frac{\tm_p\otimes \tm_p}{n^2}- \delta_{q+1} \Id + R_\ell \right)   \rd x= \int_{\T^3} \sum_{u} \sum_{k\in \Z^3\setminus \{0\}} \frac{ \delta_{q+1}  }2n\tr(\DTL d_{u,k}) e^{ i\lambda_{q+1} k\cdot \xi_I} \rd x, \\
		(2\pi)^3\zeta_{02}'&= - \int_{\T^3}  \frac{1}{2}\tr\left(  \frac{\tm_p\otimes \tm_p}{n^2} - \delta_{q+1} \Id + R_\ell \right)\Div(m_q-m_\ell) \rd x\\
		&=-\int_{\T^3} \sum_{u} \sum_{k\in \Z^3\setminus \{0\}} \frac{ \delta_{q+1} }2(\partial_t n_\ell-\partial_t n)\tr(d_{u,k}) e^{ i\lambda_{q+1} k\cdot \xi_I} \rd x.
	\end{align*}
	we can also use Lemma \ref{est on int operator} to get
	$$
	\left\| \zeta_0'\right\|_0 \leqslant 
	\frac{\varepsilon_0^2}{300\underline{M}(1+T+\tau_0)} \lambda_{q+1}^{-3\gamma} \delta_{q+2}^\frac{3}{2}.
	$$
	In a similar way as $\varphi_{M3}$ and $\varphi_{M4}$, we write $\zeta_2'$ and $\zeta_4'$ as
	\begin{align*}
		(2\pi)^3  \zeta_{2}' 
		&=\int_{\T^3} \Div (m_q-m_\ell) \frac{\tm\cdot m_\ell}{n^2} \rd x\\
		&=\int_{\T^3} \sum_{u} \sum_{k\in \Z^3\setminus \{0\}} \delta_{q+1}^\frac{1}{2}(\partial_{t}n_\ell-\partial_{t}n) \frac{m_\ell}{n^2} \cdot (s_{u,k} + \tilde{e}_{u,k}+\tilde{g}_{u,k}) e^{ i\lambda_{q+1} k\cdot \xi_I} \rd x+\int_{\T^3} (\partial_{t}n_\ell-\partial_{t}n) \frac{m_\ell\cdot\tm_{t}}{n^2} \rd x,
	\end{align*}
	and
	\begin{align*}
		(2\pi)^3  \zeta_{4}' 
		&=
		\int_{\T^3} \frac{\tm}{n} \cdot \nabla(p(n) - p_\ell (n)) \rd x\\
		&= \int_{\T^3} \sum_{u} \sum_{k\in \Z^3\setminus \{0\}} \delta_{q+1}^\frac{1}{2} \frac{1}{n} \nabla(p(n) - p_\ell (n)) \cdot (s_{u,k} + \tilde{e}_{u,k}+\tilde{g}_{u,k})) e^{ i\lambda_{q+1} k\cdot \xi_I} \rd x+\int_{\T^3}\frac{1}{n} \nabla(p(n) - p_\ell (n)) \cdot \tm_{t} \rd x.
	\end{align*}
	Therefore, as before, we apply Lemma \ref{est on int operator} to obtain
	$$
	\left\| \zeta_2'\right\|_0+\left\| \zeta_4'\right\|_0 \leqslant \frac{\varepsilon_0^2}{150\underline{M}(1+T+\tau_0)} \lambda_{q+1}^{-3\gamma} \delta_{q+2}^\frac{3}{2}. 
	$$
	\subsubsection{Estimates on $ \zeta_5$ and $ \zeta_6$}
	In a similar way as $\varphi_E$ and $\varphi_{B}$, it follows from \eqref{est on varphi_E}--\eqref{est on Dtl varphi_B} that
	$$
	\left\| \zeta_5'\right\|_0+\left\| \zeta_6'\right\|_0 \leqslant \frac{\varepsilon_0^2}{150\underline{M}(1+T+\tau_0)} \lambda_{q+1}^{-3\gamma} \delta_{q+2}^\frac{3}{2}. 
	$$
	Until now, we have proved \eqref{est on zeta'}.
\section{Proof of the inductive propositions}\label{Proof of the Inductive Proposition} 
\subsection{Proof of Proposition \ref{Inductive proposition}}
For any $0<\alpha<\frac{1}{7}$, let the parameters $\bar{b}_0(\alpha)$ and $\Lambda_0$ be as in the statement of Proposition \ref{prop of current}. For any $1<b<\bar{b}_0(\alpha)$ and $\lambda_0\geqslant\Lambda_0$, given a Maxwell-Euler-Maxwell-Reynolds flow $(m_q.E_q,B_q,R_q,\varphi_q)$ defined on $\cI^{q-1}\times\T^3$, we have constructed a perturbation $\tm=\tm_{EB}+\tm_{t}$ that can be applied to $m_q$. This results in a new Reynolds stress $R_{q+1}$ and a new current $\varphi_{q+1}$, which satisfy the estimates in Proposition \ref{prop of Reynolds error} and Proposition \ref{prop of current}. We now need to confirm whether $(m_{q+1}.E_{q+1},B_{q+1},R_{q+1},\varphi_{q+1})$ satisfy \eqref{est on m_q}--\eqref{est on current} at the $q+1$ step.
First, we denote the maximum implicit constant in \eqref{est on tm} and \eqref{est on tEB}, which depends on $n$, $p$, $h$, by $M_0$ and set $M=\max\{3M_0,\underline{M}(n,p,h)\}$. If we set $\|\cdot\|_N=\|\cdot\|_{C^0(\cI^{q};C^N(\T^3))}$, we have
\begin{equation}\label{Induction Proposition}
\begin{aligned}
	\|m_{q+1}-m_q\|_0+\lambda_{q+1}^{-1}\|m_{q+1}-m_q\|_1+\lambda_{q+1}^{-1}\|\partial_{t}(m_{q+1}-m_q)\|_0&\leqslant3M_0\delta_{q+1}^{\frac{1}{2}}\leqslant M\delta_{q+1}^{\frac{1}{2}},\\
	\|E_{q+1}-E_q\|_0+\lambda_{q+1}^{-1}\|E_{q+1}-E_q\|_1+\lambda_{q+1}^{-1}\|\partial_{t}(E_{q+1}-E_q)\|_0&\leqslant3M_0\lambda_{q+1}^{-1}\delta_{q+1}^{\frac{1}{2}}\leqslant M\lambda_{q+1}^{-1}\delta_{q+1}^{\frac{1}{2}},\\
	\|B_{q+1}-B_q\|_0+\lambda_{q+1}^{-1}\|B_{q+1}-B_q\|_1+\lambda_{q+1}^{-1}\|\partial_{t}(B_{q+1}-B_q)\|_0&\leqslant3M_0\lambda_{q+1}^{-1}\delta_{q+1}^{\frac{1}{2}}\leqslant M\lambda_{q+1}^{-1}\delta_{q+1}^{\frac{1}{2}}.
\end{aligned}
\end{equation}
Moreover, we could get 
\begin{align*}
\|m_{q+1}\|_0&\leqslant\|m_{q}\|_0+\|\tm\|_0\leqslant \underline{M}-\delta_{q}^{\frac{1}{2}}+M_0\delta_{q+1}^{\frac{1}{2}}\leqslant \underline{M}-\delta_{q+1}^{\frac{1}{2}},\\
\|\partial_{t}^rm_{q+1}\|_N&\leqslant\|\partial_{t}^rm_{q}\|_N+\|\partial_{t}^r\tm\|_N\leqslant M\lambda_{q}^{N+r}\delta_{q}^{\frac{1}{2}}+\frac{1}{2}M\lambda_{q+1}^{N+r}\delta_{q+1}^{\frac{1}{2}}\leqslant M\lambda_{q+1}^{N+r}\delta_{q+1}^{\frac{1}{2}},
\end{align*}
for $1\leqslant N+r\leqslant 2$, and 
\begin{align*}
\|\partial_{t}^r\tE_{q+1}\|_N&\leqslant\|\partial_{t}^rE_{q}\|_N+\|\partial_{t}^r\tE\|_N\leqslant \underline{M}-\delta_{q}^{\frac{1}{2}}+M_0\delta_{q+1}^{\frac{1}{2}}\leqslant \underline{M}-\delta_{q+1}^{\frac{1}{2}},&& 0\leqslant N+r\leqslant 1,\\
\|\partial_{t}^r\tB_{q+1}\|_N&\leqslant\|\partial_{t}^rB_{q}\|_N+\|\partial_{t}^r\tB\|_N\leqslant \underline{M}-\delta_{q}^{\frac{1}{2}}+M_0\delta_{q+1}^{\frac{1}{2}}\leqslant \underline{M}-\delta_{q+1}^{\frac{1}{2}},&& 0\leqslant N+r\leqslant 1,\\
\|\partial_{t}^r\tE_{q+1}\|_N&\leqslant\|\partial_{t}^rE_{q}\|_N+\|\partial_{t}^r\tE\|_N\leqslant M\lambda_{q}\delta_{q}^{\frac{1}{2}}+\frac{1}{2}M\lambda_{q+1}\delta_{q+1}^{\frac{1}{2}}\leqslant M\lambda_{q+1}\delta_{q+1}^{\frac{1}{2}},&&  N+r=2,\\
\|\partial_{t}^r\tB_{q+1}\|_N&\leqslant\|\partial_{t}^rB_{q}\|_N+\|\partial_{t}^r\tB\|_N\leqslant M\lambda_{q}\delta_{q}^{\frac{1}{2}}+\frac{1}{2}M\lambda_{q+1}\delta_{q+1}^{\frac{1}{2}}\leqslant M\lambda_{q+1}\delta_{q+1}^{\frac{1}{2}},&& N+r=2.
\end{align*}
The estimates on $R_{q+1}$ and $\varphi_{q+1}$ have been proved in in Proposition \ref{prop of Reynolds error} and Proposition \ref{prop of current}. 
\subsection{Proof of Proposition \ref{Bifurcating inductive proposition}}
Similar as in \cite{DK22,GK22},
we consider a given time interval $\cal I \subset (0,T)$ with $|\cal I| \geqslant 3\tau_q$ on which we can always find $u_0$ such that $\supp(\theta_{u_0}(\tau_q^{-1}\cdot))\subset\cal I$. If $I = (u_0, v, f)\in\mathscr I_R$, we replace $\gamma_{I}$ in $\tm_{EB}$ by $ \tilde{\Gamma}_{I} = -\Gamma_{I}$ which will make $\tilde{\gamma}_{I} =-\gamma_{I}$. We denote the new perturbation by $(\tm_{new},\tE_{new},\tB_{new})$. As for the other tuple, we don't change $\gamma_{I}$. Note that $\tilde{\Gamma}_{I}$ still solves \eqref{Eq of Gamma_I R} and $\tilde{\gamma}_{I}$ satisfies \eqref{Eq of Gamma_I f otimes f} , and the replacement does not change the estimates on $\tilde{\Gamma}_{I}$. We could know the estimates on the new perturbation $\tm_{new}$ is the same as $\tm$. Up to now, we could construct the new corrected dissipative Euler-Maxwell-Reynolds flow $( \overline{m}_{q+1},\overline{E}_{q+1},\overline{B}_{q+1}, c_{q+1},  \overline{R}_{q+1},  \overline{\varphi}_{q+1})$ satisfies \eqref{est on m_q}--\eqref{est on current} at $q+1$ step where $\overline{m}_{q+1}=m_q+\tm_{new}$, $\overline{E}_{q+1}=E_q+\tE_{new}$,  and $\overline{B}_{q+1}=B_q+\tB_{new}$. By the construction, the correction $ \tm_{new}$ differs from $\tm$ on the support of $\theta_{u_0}(\tau_q^{-1}\cdot)$. Therefore, $\supp_t(\tm_{q+1} -  m_{q+1})
=\supp_t(\tm_{new} -  \tm)\subset \cal I. $
Recalling \eqref{Eq of Gamma_I f otimes f} and \eqref{Eq of Gamma_I R}, and setting $\|\cdot\|_N=\|\cdot\|_{C([0,T];C^N(\T^3))}$, we have 
\begin{align*}
	\sum_{I\in \mathscr I_{u,v,R}}\gamma_I^2 |\tilde{f}_I|^2
	&=\tr\left(
	(\nabla \xi_I)^{-1}  \sum_{f\in \mathscr I_{u,v,R}} \gamma_I^2 f_I\otimes f_I [(\nabla \xi_I)^{-1}]^\top
	\right)= \tr (n^2(\delta_{q+1}\Id - R_\ell - \tilde{\mathcal{M}}_{I})),
\end{align*} 
where 
$$
\tilde{\mathcal{M}}_I := \sum_{(u',\upsilon')\in I(u,v)}\sum_{I'\in \mathscr I_{u',\upsilon',\varphi}}\theta_{I'}^2\chi_{I'}^2(\xi_{I'})\gamma_{I'}^2\fint_{\T^3}M_{I'}^2\rd x ((\nabla \xi_{I'})^{-1} f') \otimes  (\nabla \xi_{I'}^{-1} f').
$$
for $I'=(u',\upsilon',f')$. Moreover, we also have $\|\tilde{\mathcal{M}}_I\|_0 \lesssim \lambda_q^{-2\gamma} \delta_{q+1}$. And then
\begin{align*}
	|\tm_{p,new} -  \tm_p |^2
	&=  \sum_{I\in \mathscr I_R: u_I = u_0}
	4\theta_I^2(t) \chi_I^2(\xi_I) \gamma_I^2 |\tilde{f}_I|^2 (1+(M_I^2(\lambda_{q+1} \xi_I)-1))\\
	&= \sum_{I\in \mathscr I_R: u_I = u_0}
	4\theta_I^2(t)\chi_I^2(\xi_I)  (n^2(3\delta_{q+1} - \tr(R_\ell) - \tr(\tilde{\mathcal{M}}_I))\\
	&\quad+
	\sum_{k\in \Z^3\setminus\{0\}}\sum_{I\in \mathscr I_R: u_I = u_0}
	4\theta_I^2(t)\chi_I^2(\xi_I) \gamma_I^2 |\tilde{f}_I|^2\tr(
	\overset{\circ}{d}_{I,k}) e^{i\lambda_{q+1}k\cdot \xi_I}\\
	&= 4\theta_{u_0}^6(\tau_q^{-1}t) (n^2(3\delta_{q+1}- (\tr R_\ell) - \tr (\tilde{\mathcal{M}}) )+ \sum_{k\in \Z^3\setminus\{0\}} 4\delta_{q+1}\tr( \overset{\circ}{d}_{u_0,k,R})e^{i\lambda_{q+1}k\cdot \xi_I},
\end{align*}
where 
$$
\tr( \overset{\circ}{d}_{u_0,k,R}) = \sum_{I\in \mathscr I_R: u_I = u_0}
\theta_I^2(t) \chi_I^2(\xi_I) \delta_{q+1}^{-1}\gamma_I^2\overset{\circ}{d}_{I,k} |(\nabla \xi_I)^{-1} f_I|^2.
$$
Since we have $\|\tr( \overset{\circ}{d}_{u_0,k,R})\|_N \lesssim \mu_q^{-N} |\overset{\circ}{d}_{I,k}|$ for $N=0,1,2$, and then
\begin{align*}
\|\tm_{p,new} -  \tm_p\|_{C^0([0,T]; L^2(\T^3))}^2
&\geqslant 12\delta_{q+1}\|n\|_{C^0({[t_{u_0}+\frac18\tau_q, t_{u_0}+\frac78\tau_q]}; L^2(\T^3))}^2 - 4(2\pi)^3(\|n^2 R_\ell\|_0 + \|\tr(\tilde{\mathcal{M}})\|_0) \\
&\quad- {\sup_{t\in [0,T]}}\sum_{k\in \Z^3\setminus\{0\}}4 \left|\delta_{q+1} \int\tr( \overset{\circ}{d}_{u_0,k,R})e^{i\lambda_{q+1}k\cdot \xi_I} \rd x
	\right|\\
&\geqslant 12\delta_{q+1}\|n\|^2_{C^0([t_{u_0}+\frac18\tau_q, t_{u_0}+\frac78\tau_q]; L^2(\T^3))} - c_{n} \delta_{q+1}(\lambda_q^{-3\gamma} + \lambda_q^{-2\gamma} + (\lambda_{q+1} \mu_q)^{-2})\\
&\geqslant 4\delta_{q+1}\varepsilon_0^2,
\end{align*}
for sufficiently large $\lambda_0$. To get the second inequality, we use Lemma \ref{est on int operator} and have
\begin{align*}
	\sum_{k\in \Z^3\setminus\{0\}}\left| \int \tr(\overset{\circ}{d}_{u_0,k,R}) e^{\lambda_{q+1} k\cdot \xi_I} \rd x\right|
	&\leqslant \sum_{k \in \Z^3\setminus\{0\}}\frac{\|\tr( \overset{\circ}{d}_{u_0,k,R})\|_2 + \|\tr( \overset{\circ}{d}_{u_0,k,R})\|_0\|{\nabla} \xi_I\|_{C^0([t_{u_0}-\frac12\tau_q, t_{u_0}+\frac32\tau_q];C^2( \T^3))}}{\lambda_{q+1}^2 |k|^2}
	\\
	&  \leqslant (\lambda_{q+1}\mu_q)^{-2} \sum_{k \in \Z^3\setminus\{0\}}\frac{|\overset{\circ}{d}_{u_0,k,R}|}{|k|^2}
	\leqslant (\lambda_{q+1}\mu_q)^{-2} 
	\left(\sum_{k \in \Z^3\setminus\{0\}} |\overset{\circ}{d}_{u_0,k,R}|^2\right)^\frac12
	\left(\sum_{k \in \Z^3\setminus\{0\}} \frac{1}{|k|^4}\right)^\frac12.
\end{align*}
Therefore, we obtain
\begin{align*}
	\|\overline{m}_{q+1} -  m_{q+1} \|_{C^0([0,T]; L^2(\T^3))}
	&=\|\tm_{new}-\tm\|_{C^0([0,T]; L^2(\T^3))}\\
	&\geqslant \|\tm_{p,new} -  \tm_p\|_{C^0([0,T]; L^2(\T^3))}
	-(2\pi)^\frac32(\|\tm_{c,new} \|_0 + \|\tm_c\|_0)\\
	&\geqslant  2\delta_{q+1}^\frac12 \varepsilon_0 -\frac{(2\pi)^\frac32 2M_0}{\lambda_{q+1}\mu_q} \delta_{q+1}^\frac12 \geqslant \delta_{q+1}^\frac12 \varepsilon_0.
\end{align*}
for sufficiently large $\lambda_0$. In a similar way, we could get
\begin{align*}
\|\overline{E}_{q+1} -  E_{q+1} \|_{C^0([0,T]; L^2(\T^3))}\geqslant\lambda_{q+1}^{-1}\delta_{q+1}^\frac12 \varepsilon_0,\\
\|\overline{B}_{q+1} -  B_{q+1} \|_{C^0([0,T]; L^2(\T^3))}\geqslant\lambda_{q+1}^{-1}\delta_{q+1}^\frac12 \varepsilon_0.
\end{align*}
 Assume we  are given a dissipative Euler-Maxwell-Reynolds flow $( \overline{m}_q,\overline{E}_q,\overline{B}_q, c_q,  \overline{R}_q,   \overline{\varphi}_q)$ satisfies \eqref{est on m_q}--\eqref{est on current} and
$$
\supp_t(\overline{m}_q- m_q,\overline{E}_q- E_q,\overline{B}_q- B_q, \overline{R}_q- R_q,  \overline{\varphi}_q- \varphi_q)\subset \cal{J} 
$$
for some time interval $\cal{J}$. Then, we could construct the regularized flow $ R_\ell$ and $ \varphi_\ell$. Notice that they differ only in $\cal J + \ell_t\subset\cal J + (\lambda_q\delta_q^\frac12)^{-1}$. As a result, $\tm_{new}$ differs from $ \tm$ at $q+1$  step in $\cal J + (\lambda_q\delta_q^\frac12)^{-1}$. So we could obtain that the corrected dissipative Euler-Maxwell-Reynolds flows $(\overline{m}_{q+1},\overline{E}_{q+1},\overline{B}_{q+1},  c_{q+1},  \overline{R}_{q+1},  \overline{\varphi}_{q+1})$ and $( m_{q+1},E_{q+1},B_{q+1},  c_{q+1},  R_{q+1},  \varphi_{q+1})$ satisfying 
$$
\supp_t(\overline{m}_{q+1}- m_{q+1},\overline{E}_{q+1}- E_{q+1},\overline{B}_{q+1}- B_{q+1}, \overline{R}_{q+1}- R_{q+1},  \overline{\varphi}_{q+1}- \varphi_{q+1})\subset \cal{J} + (\lambda_q\delta_q^\frac12)^{-1}.
$$ 
\section{Construction of a starting tuple }\label{Construction of a Starting Tuple}
We need to construct a tuple that satisfies \eqref{est on m_q}--\eqref{est on current}. This tuple will serve as the starting point for the subsequent proof.  Here we give two important lemmas. In this section, we denote $\left\|\cdot\right\|_N=\left\|\cdot\right\|_{C^0(\cI^{-1};C^N(\T^3))}$ and the Mikado flow $U_f(x)=\psi_f(x)f=\sum_{k\in \Z^3\setminus \{0\}}\cb_{f,k}fe^{ik\cdot x}$.
\begin{lm}\label{Construction of a Starting Tuple stationary}
Given a stationary density $n(t,x)=n_0(x)\in C^{\infty}(\T^3)$ which satisfies $n \geqslant \varepsilon_0$ for some positive constant $\varepsilon_0$, and $ \int_{\mathbb{T}^3} n(x) d x=\int_{\mathbb{T}^3} h(x) \rd x$  for all $t\in\cI^{-1}$, we could construct a starting tuple $\left(m_{0}, E_{0}, B_{0}, R_{0}, \varphi_{0}\right)$ which satisfies \eqref{est on m_q}--\eqref{est on current}.
\end{lm}
\begin{proof}
Noting that the starting tuple $\left(m_{0}, E_{0}, B_{0}, R_{0}, \varphi_{0}\right)$ satisfies \eqref{Euler-Maxwell-Reynolds System}, in order to make $R_0$ small enough to satisfy \eqref{est on error} for $q=0$, we will use low-frequency component of $\Div(m_0\otimes m_0)$ to cancel $\nabla (p(n)+c_0n)$, where $c_0=\sum_{q=1}^\infty\delta_{q+1}$.
Let $e_i$ be the standard unit vectors whose $i$th component is 1, we first use building blocks in Lemma \ref{New building blocks} to define $(\overline{m}_0,\overline{E}_0,\overline{B}_0)$ and $\overline{m}_{0,p}$ , which is the main part of $\overline{m}_0$, as 
\begin{align}
	\overline{m}_0&=\sum_{i=1}^{3}\sum_{k\in \Z^3\setminus \{0\}}m_k(e_i,n^{\frac{1}{2}}(\overline{C}_{n}-p(n)-c_0n)^{\frac{1}{2}}\cb_{e_i,k}e^{ik\cdot(-\overline{x}_i)},0,\tlm),\label{def of om_p}\\
	\overline{E}_0&=\sum_{i=1}^{3}\sum_{k\in \Z^3\setminus \{0\}}E_k(e_i,n^{\frac{1}{2}}(\overline{C}_{n}-p(n)-c_0n)^{\frac{1}{2}}\cb_{e_i,k}e^{ik\cdot(-\overline{x}_i)},0,\tlm)=0,\label{def of oE_p}\\
	\overline{B}_0&=\sum_{i=1}^{3}\sum_{k\in \Z^3\setminus \{0\}}B_k(e_i,n^{\frac{1}{2}}(\overline{C}_{n}-p(n)-c_0n)^{\frac{1}{2}}\cb_{e_i,k}e^{ik\cdot(-\overline{x}_i)},0,\tlm),\label{def of oB_p}\\
	\overline{m}_{0,p}&=\sum_{i=1}^{3}\sum_{k\in \Z^3\setminus \{0\}}m_{p,k}(e_i,n^{\frac{1}{2}}(\overline{C}_{n}-p(n)-c_0n)^{\frac{1}{2}}\cb_{e_i,k}e^{ik\cdot(-\overline{x}_i)},0,\tlm,0,0)\notag\\
	&=\sum_{i=1}^{3}\sum_{k\in \Z^3\setminus \{0\}}n^{\frac{1}{2}}(\overline{C}_{n}-p(n)-c_0n)^{\frac{1}{2}}\cb_{e_i,k}e_ie^{ik\cdot(\tlm x-\overline{x}_i)}\notag\\
	&=\sum_{i=1}^{3}n^{\frac{1}{2}}(\overline{C}_{n}-p(n)-c_0n)^{\frac{1}{2}}\psi_{e_i}(\tlm x-\overline{x_i})e_i.\label{def of om_pk}
\end{align}
where $\overline{C}_{n}=2(c_0\left\|n\right\|_0+\left\|p(n)\right\|_0)$, and  $\psi_{e_i}=\psi_{e_i}(\cdot-\overline{x_i})$ which satisfy
$$
\int_{\T^3}\psi_{e_i}\rd x=\int_{\T^3}\psi_{e_i}^3\rd x=0,\qquad\int_{\T^3}\psi_{e_i}^2\rd x=1.
$$
$\overline{x_i}$ will be chosen to ensure pairwise disjoint $\supp(\psi_i)$. We can immediately obtain $\partial_{t}\overline{m}_0=0$ and 
$$
\frac{\overline{m}_{0,p}\otimes \overline{m}_{0,p}}{n}=(\overline{C_{n}-p(n)-c_0n})\Id+(\overline{C_{n}-p(n)-c_0n})\sum_{i=1}^{3}(\psi_{e_i}^2(\tilde{\lambda}x)-1) e_i\otimes e_i.
$$
We could calculate
\begin{align*}
	\partial_{t}\overline{m}_{0,p}+\Div\left(\frac{\overline{m}_{0,p}\otimes \overline{m}_{0,p}}{n}\right)+\nabla (p(n)+c_0n)&=\Div\left((\overline{C_{n}-p(n)-c_0n})\sum_{i=1}^{3}(\psi_{e_i}^2(\tilde{\lambda}x)-1) e_i\otimes e_i\right)\\
	&=-\sum_{i=1}^{3}\partial_i(p(n)+c_0n)(\psi_{e_i}^2(\tilde{\lambda}x)-1)e_i.
\end{align*} 
And we choose a constant which satisfies 
\begin{equation}\label{pp of um 1}
\underline{M}^{\frac{1}{8}}(n,p,h)>\max\{20,\varepsilon_0^{-4},480(1+c_0)\sum_{N=1}^{3}\left(\left\|h\right\|_{N}+\left\|n\right\|_N+\left\|p(n)\right\|_N\right)\}.
\end{equation}
In order to ensure $\Div E_0=h(x)-n_0(x)$, we add another correction $(m_c,E_c,B_c)$ which satisfy
$$
\begin{aligned}
	E_c&=\cR(h(x)-n(x)),\\
	B_c&=-\int_{-\tau_{-1}}^{t}\nabla\times\cR(h(x)-n(x))\rd \tau=0,\\
	m_c&=\partial_{t}E_c-\nabla \times B_c=\int_{-\tau_{-1}}^{t}\nabla\times\nabla\times\cR(h(x)-n(x))\rd \tau=0.
\end{aligned}
$$
Notice that $nE_c$ will disrupt the momentum balance in the integral sense. Furthermore, there will be new items with low-frequency in the left side of momentum equation, which would make $R_0$ not small enough. To address the two conditions, we will add two corrections. Similar to the construction of $\tm_t$ in \eqref{Eq of tE_t} from Section \ref{The definition of the perturbation}, we first add a time correction $(m_{t},E_{t},B_t=0)$ which only depends on $t$ and satisfies 
$$
\begin{aligned}
	\quad\int_{\T^3}\left(\partial_{tt}E_{t}+nE_{t}+nE_c\right)\rd x&=\partial_{tt}E_{t}+\int_{\T^3}n(x)\rd xE_{t}+\int_{\T^3}nE_c\rd x=0.
\end{aligned}
$$
We could choose
$$
E_{t}=-\frac{\int_{\T^3}nE_c\rd x}{\int_{\T^3}n\rd x},\quad m_t=0.
$$
For the second case, we first rewrite the items with low-frequency as
$$
nR_c:=\cR(h(E_c+E_{t}))+\left(\frac{|E_c+E_t|^2}{2}\Id-(E_c+E_t)\otimes (E_c+E_t)\right).
$$ 
Then, we could use $\|\cR u\|_N\leqslant C_R\|u\|_N$ and choose $\underline{M}(n,p,h) $ satisfying
\begin{equation}\label{pp of um 2}
\underline{M}^{\frac{1}{8}}(n,p,h)> C_R, 
\end{equation}
to obtain for $N=0,1$,
$$
\begin{aligned}
	\left\|nR_c\right\|_N\leqslant C\|h-n\|_N\leqslant \frac{1}{120}\underline{M}^{\frac{1}{4}}(n,p,h),\quad\left\|R_c\right\|_N\leqslant\left\|n^{-1}(nR_c)\right\|_N\leqslant\frac{1}{240}\underline{M}^{\frac{3}{8}}(n,p,h),
\end{aligned}
$$
Then, we apply Lemma \ref{Geometric Lemma I} to $\cF=\left\lbrace f_i\right\rbrace_{i=1}^6 =\left\lbrace(1,\pm1,0),(1.0,\pm1),(0,1,\pm1)\right\rbrace$. There exists $\Gamma_{i}=\Gamma_{f_i}(x)$ which satisfy
$$
\Id-\left(\frac{1}{64}\underline{M}^{\frac{1}{2}}(n,p,h)\right)^{-1}nR_c=\sum_{i=1}^{6}\Gamma_i^2f_i\otimes f_i.
$$
Notice that $\Gamma_i$ can be determined by $n$ and $h$, then $\underline{M}$ can be chosen to satisfy  
\begin{equation}\label{pp of um 3}
\underline{M}^{\frac{1}{8}}(n,p,h)\geqslant\sum_{i=1}^{6}\left\|\Gamma_{i}\right\|_1.
\end{equation}
Similarly to $(\overline{m}_{p},\overline{E}_p,\overline{B}_p)$, we could define $(\overline{m}_R,\overline{E}_R,\overline{B}_R)$ as
\begin{align}
	\overline{m}_R&=\sum_{i=1}^{6}\sum_{k\in \Z^3\setminus \{0\}}m_k(f_i,\frac{1}{8}\underline{M}^{\frac{1}{4}}(n,p,h)n^{\frac{1}{2}}\Gamma_i\cb_{f_i,k}e^{ik\cdot(-y_i)},0,\tlm),\label{def of om_R}\\
	\overline{E}_R&=\sum_{i=1}^{6}\sum_{k\in \Z^3\setminus \{0\}}E_k(f_i,\frac{1}{8}\underline{M}^{\frac{1}{4}}(n,p,h)n^{\frac{1}{2}}\Gamma_i\cb_{f_i,k}e^{ik\cdot(-y_i)},0,\tlm)=0,\label{def of oE_R}\\
	\overline{B}_R&=\sum_{i=1}^{6}\sum_{k\in \Z^3\setminus \{0\}}B_k(e_i,\frac{1}{8}\underline{M}^{\frac{1}{4}}(n,p,h)n^{\frac{1}{2}}\Gamma_i\cb_{f_i,k}e^{ik\cdot(-y_i)},0,\tlm),\label{def of oB_R}\\
	\overline{m}_{R,p}&=\sum_{i=1}^{6}\sum_{k\in \Z^3\setminus \{0\}}m_{p,k}(f_i,\frac{1}{8}\underline{M}^{\frac{1}{4}}(n,p,h)n^{\frac{1}{2}}\Gamma_i\cb_{f_i,k}e^{ik\cdot(-y_i)},0,\tlm,0,0)\notag\\
	&=\sum_{i=1}^{6}\sum_{k\in \Z^3\setminus \{0\}}\frac{1}{8}\underline{M}^{\frac{1}{4}}(n,p,h)n^{\frac{1}{2}}\Gamma_i\cb_{f_i,k}f_ie^{ik\cdot(\tlm x-y_i)}\notag\\
	&=\sum_{i=1}^{6}\frac{1}{8}\underline{M}^{\frac{1}{4}}(n,p,h)n^{\frac{1}{2}}\Gamma_i\psi_{f_i}(\tlm x-y)e_i.\label{def of om_Rp}
\end{align}
where $\psi_{f_i}(\cdot)=\psi_{f_i}(\cdot-y_i)$, which satisfy
$$
\int_{\T^3}\psi_{f_i}\rd x=\int_{\T^3}\psi_{f_i}^3\rd x=0,\qquad\int_{\T^3}\psi_{f_i}^2\rd x=1.
$$
$y_i$ will be chosen to ensure disjoint $\supp(\psi_i)$ and $\supp(\psi_{f_i})$. Immediately, we could get
\begin{align*}
	\Div\left(\frac{\overline{m}_{R,p}\otimes \overline{m}_{R,p}}{n}+nR_c-\frac{1}{64}\underline{M}^{\frac{1}{2}}(n,p,h)\Id\right)&=\Div\left(\sum_{i=1}^{6}\frac{1}{64}\underline{M}^{\frac{1}{2}}(n,p,h)\Gamma_i^2(\psi_{f_i}^2(\tilde{\lambda}x)-1)f_i\otimes f_i\right)\\
	&=\sum_{i=1}^{6}\frac{1}{64}\underline{M}^{\frac{1}{2}}(n,p,h)(f_i\cdot\nabla)(\Gamma_i^2)(\psi_{f_i}^2(\tilde{\lambda}x)-1)f_i.
\end{align*}
Up to now, we could construct the starting tuple as $(m_0=\overline{m}_{0}+\overline{m}_R,E_0=\overline{E}_c+\overline{E}_{t},B_0=\overline{B}_0+\overline{B}_R)$, and calculate
$$
\begin{aligned}
	\int_{\T^3}\left(\partial_tm_0+\Div\left(\frac{m_0\otimes m_0}{n}\right)+\nabla (p(n)+c_0n)+nE_0+m_0\times B_0\right)\rd x=0.
\end{aligned}
$$ 
Then, we have
$$
\begin{aligned}
\Div(nR_0)&=\Div\left(\frac{\overline{m}_0\otimes \overline{m}_0}{n}\right)+\nabla(p(n)+c_0n)+\Div\left(\frac{\overline{m}_R\otimes \overline{m}_R}{n}+\frac{|E_0|^2}{2}\Id-E_0\otimes E_0-\frac{1}{64}\underline{M}^{\frac{1}{2}}(n,p,h)\Id\right)\\
	&\quad+h(E_c+E_{t})+\Div\left(\frac{m_0\otimes m_0}{n}-\frac{\overline{m}_0\otimes \overline{m}_0}{n}-\frac{\overline{m}_R\otimes \overline{m}_R}{n}\right)+\Div\left(\frac{|B_0|^2}{2}\Id-B_0\otimes B_0\right).
\end{aligned}
$$
So $nR_0$ and $\kappa_0$ can be chosen as
\begin{align*}
	nR_0&=\cR\left(\Div\left(\frac{\overline{m}_{0,p}\otimes \overline{m}_{0,p}}{n}\right)+\nabla(p(n)+c_0n)\right)\\
	&\quad+\cR\left(\Div\left(\frac{\overline{m}_{R,p}\otimes \overline{m}_{R,p}}{n}+\frac{|E_0|^2}{2}\Id-E_0\otimes E_0+\cR(h(E_c+E_{t}))-\frac{1}{64}\underline{M}^{\frac{1}{2}}(n,p,h)\Id\right)\right)\\
	&\quad+\frac{m_0\otimes m_0}{n}-\frac{\overline{m}_{0,p}\otimes \overline{m}_{0,p}}{n}-\frac{\overline{m}_{R,p}\otimes \overline{m}_{R,p}}{n}+\frac{|B_0|^2}{2}\Id-B_0\otimes B_0-\frac{2}{3}H(t)\Id\\
	&=-\cR\left(\sum_{i=1}^{3}\partial_i(p(n)+c_0n)(\psi_i^2(\tilde{\lambda}x)-1)e_i\right)+\cR\left(\sum_{i=1}^{6}\frac{1}{64}\underline{M}^{\frac{1}{2}}(n,p,h)(f_i\cdot\nabla)(\Gamma_i^2)(\psi_{f_i}^2(\tilde{\lambda}x)-1)f_i\right)\\
	&\quad+\frac{m_0\otimes m_0}{n}-\frac{\overline{m}_{0,p}\otimes \overline{m}_{0,p}}{n}-\frac{\overline{m}_{R,p}\otimes \overline{m}_{R,p}}{n}+\frac{|B_0|^2}{2}\Id-B_0\otimes B_0-\frac{2}{3}H(t)\Id\\
	\kappa_0&:=\frac{1}{2}\tr R_0=\frac{|m_0|^2}{n}-\frac{|\overline{m}_{0,p}|^2}{n}-\frac{|\overline{m}_{R,p}|^2}{n}+\frac{|B_0|^2}{2n}-\frac{H(t)}{n}.
\end{align*}
where $H(t)$ is a global energy loss
depending only on time.
Finally, we will give the start current $\varphi_0$ so that
$$
\begin{aligned}
	\Div(n\varphi_0)&=\partial_{t}\left(\frac{|m_0|^2}{2n}+\frac{|E_0|^2+|B_0|^2}{2}+ne(n)\right)+\Div\left(\frac{|m_0|^2m_0}{2n^2}+m_0P^\prime(n)+E_0\times B_0\right)\\
	&\quad-n(\partial_{t}+\frac{m_0}{n}\cdot\nabla)\kappa_0-\Div(R_0m_0)-H^\prime\\
	&=\Div\left(\frac{|m_0|^2m_0}{2n^2}+m_0P^\prime(n)+E_0\times B_0-m_0\kappa_0-R_0m_0\right),
\end{aligned}
$$
where we use $n(\partial_{t}+\frac{m_0}{n}\cdot\nabla)\kappa_0=\partial_{t}(n\kappa_0)+\Div(m_0\kappa_0)$. Next, we can choose $n\varphi_0$ as
\begin{align*}
	n\varphi_0&=\cR\left(\Div\left(\frac{|\overline{m}_{0,p}|^2\overline{m}_{0,p}}{2n^2}\right)+\Div\left(\frac{|\overline{m}_{R,p}|^2\overline{m}_{R,p}}{2n^2}\right)\right)+\cR\left(\Div\left(m_0P^\prime(n)\right)\right)\\
	&\quad+\frac{|m_0|^2m_0}{2n^2}-\frac{|\overline{m}_{0,p}|^2\overline{m}_{0,p}}{2n^2}-\frac{|\overline{m}_{R,p}|^2\overline{m}_{R,p}}{2n^2}+E_0\times B_0-m_0\kappa_0-R_0m_0\\
	&=\cR\left(\sum_{i=1}^{3}\partial_i\left(\frac{(C_{n}-p(n)-c_0n)^{\frac{3}{2}}}{2n^{\frac{1}{2}}}\right)\psi_i^3(\tilde{\lambda}x)+\sum_{i=1}^{6}(f_i\cdot\nabla)\left(\frac{\underline{M}^{\frac{3}{4}}(n,p,h)}{1024n^{\frac{1}{2}}}\Gamma_i\right)\psi_{f_i}^3(\tilde{\lambda}x)\right)\\
	&\quad\frac{|m_0|^2m_0}{2n^2}-\frac{|\overline{m}_{0,p}|^2\overline{m}_{0,p}}{2n^2}-\frac{|\overline{m}_{R,p}|^2\overline{m}_{R,p}}{2n^2}+\cR\left(m_0\nabla
	P^\prime(n)\right)+E_0\times B_0-m_0\kappa_0-R_0m_0.
\end{align*}
Up to now, we have constructed a starting tuple $\left(m_0,E_0,B_0,R_0,\varphi_0\right)$ which solve \eqref{Euler-Maxwell-Reynolds System}. We choose the parameters in Lemma \ref{New building blocks} as $\tlm=\tlm$, $\mu=1$, $\tau=1$, $\ell=1$. So there exists a constant  $M_*(n,p,h)>1$ such that
$$
\begin{aligned}
	&\left\|m_0\right\|_0\leqslant\left\|m_p\right\|_0+\left\|m_R\right\|_0\leqslant\frac{1}{32} M_*^{\frac{1}{2}}(n,p,h)(2+6\tilde{\lambda}^{-1}+6\tilde{\lambda}^{-2}+2\tilde{\lambda}^{-3})\leqslant M_*(n,p,h)-\delta_{0}^\frac{1}{2},\\
	&\left\|E_0\right\|_1\leqslant\left\|E_c\right\|_1+\left\|E_t\right\|_1\leqslant\frac{1}{48} M_*^{\frac{1}{4}}(n,p,h)\leqslant M_*(n,p,h)-\delta_{0}^\frac{1}{2},\\	&\left\|B_0\right\|_1\leqslant\left\|B_p\right\|_1+\left\|B_R\right\|_1\leqslant\frac{1}{64} M_*^{\frac{1}{2}}(n,p,h)(2+6\tilde{\lambda}^{-1}+6\tilde{\lambda}^{-2}+2\tilde{\lambda}^{-3})\leqslant M_*(n,p,h)-\delta_{0}^\frac{1}{2},
\end{aligned}
$$
and we could get for $N=1,2$,
$$
\begin{aligned}
	&\left\|m_0\right\|_N\leqslant\left\|m_p\right\|_N+\left\|m_R\right\|_N\lesssim_{n,p,h}\tilde{\lambda}^{N},\quad\left\|E_0\right\|_N\leqslant\left\|E_c\right\|_N+\left\|E_t\right\|_N\lesssim_{n,h} 1,\quad\left\|B_0\right\|_N\leqslant\left\|B_p\right\|_N+\left\|B_R\right\|_N\lesssim_{n,p,h} \tilde{\lambda}^{N-1}.
\end{aligned}
$$
Combining it with \eqref{pp of um 1}--\eqref{pp of um 3}, we could choose $\underline{M}(n,p,h)$ satisfying
\begin{equation}\label{pp of um}
\underline{M}(n,p,h)>\max\left\{M_*(n,p,h),\sum_{i=1}^{6}\left\|\Gamma_{i}\right\|_1,C_R,20,\varepsilon_0^{-4},480(1+c_0)\sum_{N=1}^{3}\left(\left\|h\right\|_{N}+\left\|n\right\|_N+\left\|p(n)\right\|_N\right)\right\}.
\end{equation}
Moreover, we could obtain estimates on $R_0$ and $\varphi_0$ from their definition:
$$
\begin{aligned}
	&\left\|R_0\right\|_N\lesssim_{n,p,h}\tilde{\lambda}^{N-1}+\left\|H\right\|_0,\quad\left\|\varphi_0\right\|_N\lesssim_{n,p,h}\tilde{\lambda}^{N-1}+\left\|H\right\|_0\tilde{\lambda}^{N},&& N=0,1,2,\\
	&\left\|D_{t,q} R_0\right\|_{N} \leqslant\left\|\partial_t R_0\right\|_{N}+\left\|(m_0/n)\cdot \nabla R_0\right\|_{N} \lesssim_{n,p,h}\left\|H^{\prime}\right\|_0+\tilde{\lambda}^{N}\left(1+\left\|H\right\|_0\right),&& N=0,1, \\
	&\left\|D_{t, q} \varphi_0\right\|_{N} \leqslant\left\|\partial_t \varphi_0\right\|_{N}+\left\|(m_0/n) \cdot \nabla \varphi_0\right\|_{N} \lesssim_{n,p,h}\left\|H^{\prime}\right\|_0 \tilde{\lambda}^{N}+\left\|H\right\|_0 \tilde{\lambda}^{N+1}+\tilde{\lambda}^{N},&& N=0,1.
\end{aligned}
$$
Let $C(n,p,h)$ be the maximum of all implicit constants in the above inequalities. For $\tilde{b}(\alpha)$ sufficiently close to 1 and sufficiently
large $\lambda_0$, we can choose proper $\tilde{\lambda}$ to satisfy
$$
2 C(n,p,h) \lambda_0^{3 \gamma} \delta_1^{-\frac{3}{2}} \leqslant \tilde{\lambda} \leqslant(2 C(n,p,h))^{-1} \lambda_0 \delta_0^{\frac{1}{2}} .
$$
Moreover, we set the energy loss $H\equiv0$ or $H=H(t)$ which satisfies
$$
4 C(n,p,h)\left\|H\right\|_0 \leqslant \lambda_0^{-3 \gamma} \delta_1^{\frac{3}{2}}, \quad \text { and } \quad 4 C(n,p,h)\left\|H^{\prime}\right\|_0 \leqslant \lambda_0^{1-3 \gamma} \delta_0^{\frac{1}{2}} \delta_1^{\frac{3}{2}}.
$$
Up to now, we have constructed the starting tuple $(m_0,E_0,B_0,c_0,R_0,\varphi_0)$ for $H\equiv0$ and $H^\prime<0$.
\end{proof}
Based on the construction in Lemma \ref{Construction of a Starting Tuple stationary}. For the case of time-dependent density with small derivative, we can still construct a Starting tuple. Here we give a lemma.
\begin{lm}\label{Construction of a Starting Tuple time dependent}
	By perturbing the density $n_0\in C^{\infty}(\T^3)$ in Lemma \ref{Construction of a Starting Tuple stationary} slightly over time $t$, we could construct a starting tuple $\left(m_{0}, E_{0}, B_{0}, R_{0}, \varphi_{0}\right)$ which satisfies \eqref{est on m_q}--\eqref{est on current} with $n(t,x)\in C^{\infty}(\cI^{-1}\times\T^3)$ depending on time and satisfying $n(0,x)=n_0(x)$, and $ \int_{\mathbb{T}^3} n(t,x) \rd x=\int_{\mathbb{T}^3} h(x) \rd x$ for all $t$.
\end{lm}
\begin{proof}
We write a time-dependent density $
n(t, x):=n_0(x)+\varepsilon \hat{n}(t, x)$, where $\hat{n}(t, x)$ can be any smooth function which satisfies $\int_{\T^3}\partial_{t}\hat{n}(t, x)\rd x=0$ on $\cI^{-1}\times\T^3$ and
$$
\begin{aligned}
	\left\|\partial_{t}^r\hat{n}\right\|_{C^0\left(\cI^{-1} ; C^N\left(\mathbb{T}^3\right)\right)} \leqslant\left\|n_0\right\|_{C^N\left(\mathbb{T}^3\right)}, \forall N\in\left[0, \tilde{n}_0+1\right],r=0,1,2,3,
\end{aligned}
$$
where $\varepsilon\in(0,\frac{1}{2})$ will be chosen later. By using Lemma \ref{Construction of a Starting Tuple stationary}, we could construct a starting tuple $(m_0,E_0,B_0,c_0,R_0,\varphi_0)$ for stationary density $n_0$. $\tilde{\lambda}$, $\tilde{b}(\alpha)$, $\underline{M}(n_0,p,h)$, and $C(n_0,p,h)$ is same as the ones in the stationary condition. Since $\int_{\T^3}\partial_{t}\hat{n}(t, x)\rd x=0$, we can set
$$
\begin{aligned}
	&\tm_0:=m_0+\hat{m}_0+\hat{m}_t=m_0-\varepsilon\cR(\partial_{t}\hat{n})+\partial_{t}\hat{E}_t,\\
	&\tE_0:=E_0+\hat{E}_0+\hat{E}_t=E_0-\varepsilon\cR(\hat{n})+\hat{E}_t,\\
	&\tB_0:=B_0+\hat{B}_0=B_0-\int_{-\tau_{-1}}^{t}\nabla\times \cR(\hat{n})(\tau,x)\rd \tau=B_0.\\
\end{aligned}
$$
where $\hat{E}_t$ satisfies 
$$
\begin{aligned}
	&\partial_{tt}\hat{E}_t+\int_{\T^3}n_0\rd x\hat{E}_t+\int_{\T^3}\left(\partial_{t}(\hat{m}_0+\hat{E}_0\times B_0)+h\hat{E}_0\right)\rd x=0,\\
	&\hat{E}_t(0)=-\frac{\int_{\T^3}\left(\partial_{t}(\hat{m}_0+\hat{E}_0\times B_0)+h\hat{E}_0\right)\rd x}{\int_{\T^3}n_0\rd x},\quad\hat{E}_t(0)=0.
\end{aligned}
$$
Then, we have
$$
\begin{aligned}
	&\left\|\partial_t^2\hat{E}_t\right\|_0\leqslant C(n_0)\left\|\partial_{t}(\hat{m}_0+\hat{E}_0\times B_0)+h\hat{E}_0\right\|_0\leqslant \frac{1}{48} \underline{M}\left(n_0,p,h\right) \varepsilon(2+2\tilde{\lambda}^{-1}+4\tilde{\lambda}^{-2}+2\tilde{\lambda}^{-3}),\\
	&\left\|\partial_t^3\hat{E}_t\right\|_0\leqslant C(n_0)\left\|\partial_{t}(\partial_{t}(\hat{m}_0+\hat{E}_0\times B_0)+h\hat{E}_0)\right\|_0\leqslant \frac{1}{48} \underline{M}\left(n_0,p,h\right) \varepsilon(2+2\tilde{\lambda}^{-1}+4\tilde{\lambda}^{-2}+2\tilde{\lambda}^{-3}).
\end{aligned}
$$
So we could choose sufficient small $\varepsilon$ such that
\begin{align*}
	\left\|\tm_0\right\|_0 &\leqslant\left\|m_0\right\|_0+\left\|\hat{m}_0\right\|_0+\left\|\hat{m}_t\right\|_0\\ &\leqslant\frac{1}{48} \underline{M}\left(n_0,p,h\right)(1+3\tilde{\lambda}^{-1}+3\tilde{\lambda}^{-2}+\tilde{\lambda}^{-3}+\varepsilon(2+2\tilde{\lambda}^{-1}+4\tilde{\lambda}^{-2}+2\tilde{\lambda}^{-3}))+\varepsilon\left\|\partial_t\hat{n}\right\|_0\leqslant \underline{M}\left(n,p,h\right)-\delta_0^{\frac{1}{2}}, \\
	\left\|\tE_0\right\|_1 &\leqslant\left\|E_0\right\|_1+\left\|\hat{E}_0\right\|_1+\left\|\hat{E}_t\right\|_1\leqslant \frac{1}{48} \underline{M}\left(n_0,p,h\right)(2+\varepsilon(2+2\tilde{\lambda}^{-1}+4\tilde{\lambda}^{-2}+2\tilde{\lambda}^{-3}))+\varepsilon\left\|\hat{n}\right\|_0 \leqslant \underline{M}\left(n,p,h\right)-\delta_0^{\frac{1}{2}}, \\
	\left\|\partial_{t}\tE_0\right\|_0&\leqslant\left\|\partial_{t}\hat{E}_0\right\|_0+\left\|\partial_{t}\hat{E}_t\right\|_0\leqslant\frac{1}{48} \underline{M}\left(n_0,p,h\right) \varepsilon(2+2\tilde{\lambda}^{-1}+4\tilde{\lambda}^{-2}+2\tilde{\lambda}^{-3})+\varepsilon\left\|\partial_{t}\hat{n}\right\|_0\leqslant \underline{M}\left(n,p,h\right)-\delta_0^{\frac{1}{2}},\\
	\left\|\tB_0\right\|_0 &\leqslant\left\|B_0\right\|_0\leqslant \underline{M}\left(n,p,h\right)-\delta_0^{\frac{1}{2}},
\end{align*}
for some constants $\underline{M}(n,p,h)$ and
\begin{align*}
	\left\|\partial_{t}^r\tm_0\right\|_N&\leqslant\left\|\partial_{t}^r\hat{m}_0\right\|_N+\left\|\partial_{t}^r\hat{m}_t\right\|_N\leqslant C(n_0,p,h)\tlm^{N+r}+\varepsilon\left\|\partial_{t}^{r+1}\hat{n}\right\|_N\leqslant M\lambda_0^{N+r}\delta_{0}^{\frac{1}{2}},&& 1\leqslant N+r\leqslant3,\\
	\left\|\partial_{t}^r\tE_0\right\|_N&\leqslant\left\|\partial_{t}^r\hat{E}_0\right\|_N+\left\|\partial_{t}^r\hat{E}_t\right\|_N\leqslant C(n_0,p,h)\tlm^{N+r-1}+\varepsilon\left\|\partial_{t}^r\hat{n}\right\|_N\leqslant M\lambda_0^{N+r-1}\delta_{0}^{\frac{1}{2}},&& 2\leqslant N+r\leqslant3,\\
	\left\|\partial_{t}^r\tB_0\right\|_N&\leqslant\left\|\partial_{t}^rB_0\right\|_N \leqslant C(n_0,p,h)\tlm^{N+r-1}\leqslant M\lambda_0^{N+r-1}\delta_{0}^{\frac{1}{2}},&& 2\leqslant N+r\leqslant3.
\end{align*}
for $M(n,p,h)$ defined in Section \ref{Proof of the Inductive Proposition}. Next, we could calculate 
$$
\begin{aligned}
	\Div(n\tilde{R}_0)&=\Div(n_0R_0)+\Div(\varepsilon c_0\hat{n}\Id)+\partial_{t}(\hat{m}_0+(\hat{E}_0+\hat{E}_t)\times B_0)+h(\hat{E}_0+\hat{E}_t)+\nabla (p(n)-p(n_0))\\
	&\quad+\Div\left(\frac{\tm_0\otimes\tm_0}{n}-\frac{m_0\otimes m_0}{n_0}+\frac{|\tE_0|^2}{2}\Id-\frac{|E_0|^2}{2}\Id-\tE_0\otimes\tE_0+E_0\otimes E_0\right)
\end{aligned}
$$
and choose the new Reynolds error  $n\tilde{R}_0$ and $\tilde{\kappa}_0$ as
$$
\begin{aligned}
	n\tilde{R}_0&=n_0R_0+\varepsilon c_0\hat{n}\Id+\cR\left(\partial_{t}(\hat{m}_0+(\hat{E}_0+\hat{E}_t)\times B_0)+h(\hat{E}_0+\hat{E}_t)+ \nabla(p(n)-p(n_0))\right)\\
	&\quad+\cR\left(\Div\left(\frac{\tm_0\otimes\tm_0}{n}-\frac{m_0\otimes m_0}{n}+\frac{|\tE_0|^2}{2}\Id-\frac{|E_0|^2}{2}\Id-\tE_0\otimes\tE_0+E_0\otimes E_0\right)\right)+\frac{2}{3}\zeta(t)\Id,\\
	\tilde{\kappa}_0&=\frac{1}{2}\tr\tilde{R}_0=\frac{1}{2}\tr\left(R_0+\frac{n_0-n}{n}R_0\right)+\frac{3}{2}\varepsilon c_0\frac{\hat{n}}{n}+\frac{\zeta(t)}{n},
\end{aligned}
$$
where $\zeta(t)=\int_{\T^3}\left(\frac{|\tm_0|^2-|m_0|^2}{2n}+\frac{|\tE_0|^2+|\tB_0|^2-|E_0|^2-|B_0|^2}{2}+n_0\kappa_0-n\tilde{\kappa}_0+ne(n)-n_0e(n_0)\right)\rd x$. Notice that $n\tilde{\kappa}_0-\zeta(t)$ doesn't depend on $\zeta(t)$.
Similarly, we know
$$
\begin{aligned}
	\Div(n\tilde{\varphi}_0)
	&=\Div(n_0\varphi_0)+\Div\left(\frac{|\tm_0|^2\tm_0}{2n^2}-\frac{|m_0|^2m_0}{2n_0^2}\right)+\Div\left(\tm_0P^\prime(n)-m_0P^\prime(n_0)\right)+\Div\left(\tE_0\times B_0-E_0\times B_0\right)\\
	&\quad+\partial_{t}\left(\frac{|\tm_0|^2-|m_0|^2}{2n}+\frac{|\tE_0|^2+|\tB_0|^2-|E_0|^2-|B_0|^2}{2}+n_0\kappa_0-n\tilde{\kappa}_0+ne(n)-n_0e(n_0)+\zeta(t)\right)\\
	&\quad+\Div\left(m_0\kappa_0-\tm_0\tilde{\kappa}_0\right)+\Div\left(R_0m_0-\tilde{R}_0\tm_0\right)
\end{aligned}
$$
and we can choose
$$
\begin{aligned}
	n\tilde{\varphi}_0&=n_0\varphi_0+\cR\left(\partial_{t}\left(\frac{|\tm_0|^2}{2n}+\frac{|\tE_0|^2+|\tB_0|^2}{2}+n_0\kappa_0-n\tilde{\kappa}_0+ne(n)-n_0e(n_0)\right)\right)\\
	&\quad+\frac{|\tm_0|^2\tm_0}{2n^2}-\frac{|m_0|^2m_0}{2n_0^2}+\tm_0P^\prime(n)-m_0P^\prime(n_0)+\tE_0\times B_0-E_0\times B_0\\
	&\quad+m_0\kappa_0-\tm_0\tilde{\kappa}_0+R_0m_0-\tilde{R}_0\tm_0.
\end{aligned}
$$
Until now, we have constructed a starting tuple $(\tm_0,\tE_0,\tB_0,c_0,\tilde{R}_0,\tilde{\varphi}_0)$ for time-dependent density $n(t,x)$. Notice that we can choose $\varepsilon$ small enough such that $|\zeta|$, $|\partial_t\zeta|$, $|\partial_{tt}\zeta|$, $|\tilde{R}_0-R_0|$, $|\tilde{\varphi}_0-\varphi_0|$, $|(\partial_{t}+\tm_0/n\cdot\nabla)\tilde{R}_0-(\partial_{t}+m_0/n\cdot\nabla)R_0|$ and $|(\partial_{t}+\tm_0/n\cdot\nabla)\tilde{\varphi}_0-(\partial_{t}+m_0/n\cdot\nabla)\varphi_0|$ are small enough such that $(\tm_0,\tE_0,\tB_0,c_0,\tilde{R}_0,\tilde{\varphi}_0)$ still satisfies \eqref{est on m_q}--\eqref{est on current}. 
\end{proof}
\section{Proof of the theorems }\label{Proof of the Theorems}
\subsection{Proof of Theorem \ref{thm 1}}
For convenience, we assume that $T \geqslant 20$ in this argument. We fix $\beta < \frac{1}{7}$ and $\alpha \in (\beta, \frac{1}{7})$, and set $n(t, \cdot) = n_0$ for all $t \in \mathbb{R}$. We choose $b$ and $\lambda_0$ based on Proposition \ref{Inductive proposition}. Then, we could use Lemma \ref{Construction of a Starting Tuple stationary} to construct an initial approximate solution $(m_0,E_0,B_0, c_0, R_0, \varphi_0)$ with $H\equiv 0$ so that it solves \eqref{Euler-Maxwell-Reynolds System} on $\cI^{-1}\times\T^3$ and satisfies \eqref{est on m_q}--\eqref{est on current}. 
	We could apply Proposition \ref{Inductive proposition} iteratively to produce a sequence of approximate solutions $(m_q,E_q,B_q,c_q,R_q, \varphi_q)$, which solves \eqref{Euler-Maxwell-Reynolds System} with $H\equiv0$, and satisfies \eqref{est on m_q}--\eqref{est on current} and \eqref{Proposition of induction}. 
	First, we prove that $m_q,E_q,B_q$ is Cauchy. For any $q\leqslant q'$, we have
	\begin{align*}
		\|m_{q'} - m_q\|_{C^0([0, T];C^\beta(\T^3))}
		&\leqslant \sum_{l=1}^{q'-q}\|m_{q+l} - m_{q+l-1}\|_{C^0([0, T];C^\beta(\T^3))}\leqslant \sum_{l=1}^{q'-q}
		\|m_{q+l} - m_{q+l-1}\|_{0}^{1-\beta}
		\|m_{q+l} - m_{q+l-1}\|_{1}^{\beta}= \sum_{l=1}^{q'-q}
		\lambda_{q+1}^{\beta-\alpha},\\
		\|E_{q'} - E_q\|_{C^0([0, T];C^{1,\beta}(\T^3))}
		&\leqslant \sum_{l=1}^{q'-q}\|E_{q+l} - E_{q+l-1}\|_{C^0([0, T];C^{1,\beta}(\T^3))}\lesssim \sum_{l=1}^{q'-q}
		\|E_{q+l} - E_{q+l-1}\|_{1}^{1-\beta}
		\|E_{q+l} - E_{q+l-1}\|_{2}^{\beta}\lesssim \sum_{l=1}^{q'-q}
		\lambda_{q+1}^{\beta-\alpha},\\
		\|B_{q'} - B_q\|_{C^0([0, T];C^{1,\beta}(\T^3))}
		&\leqslant \sum_{l=1}^{q'-q}\|B_{q+l} - B_{q+l-1}\|_{C^0([0, T];C^{1,\beta}(\T^3))}\lesssim \sum_{l=1}^{q'-q}
		\|E_{q+l} - B_{q+l-1}\|_{1}^{1-\beta}
		\|E_{q+l} - B_{q+l-1}\|_{2}^{\beta}\lesssim \sum_{l=1}^{q'-q}
		\lambda_{q+1}^{\beta-\alpha},
	\end{align*}
	which will converge to $0$, when $q$ goes to infinity. So we could get
	$(m_q,E_q,B_q)$ obtain a limit
	$$
	(m,E,B)\in C^0([0,T],C^\beta(\T^3))\times C^0([0,T],C^{1,\beta}(\T^3))\times C^0([0,T],C^{1,\beta}(\T^3)).
	$$
	Similarly, the time regularity follows from  \eqref{Proposition of induction} that
	$$
	(m,E,B)\in C^\beta([0,T],C^0(\T^3))\times C^{1,\beta}([0,T],C^0(\T^3))\times C^{1,\beta}([0,T],C^0(\T^3)).
	$$
	Hence,
	$$
	(m,E,B)\in C^{\beta'}([0,T]\times\T^3)\times C^{1,\beta'}([0,T]\times\T^3)\times C^{1,\beta'}([0,T]\times\T^3).
	$$
	for $\beta'<\beta<\alpha<\frac{1}{3}$. Moreover, $(c_q,R_q,\varphi_q)$ converges to $0$ in $C^0([0,T]\times\T^3)\times C^0([0,T]\times\T^3)\times C^0([0,T]\times\T^3)$. Now, we are ready to construct two distinct tuples by using Proposition \ref{Bifurcating inductive proposition}. Fix $\bar q\in \N\cup\{0\}$ satisfying $b^{\bar  q}\geqslant \bar q$. At the $\bar q$th step, we can produce two distinct tuples $(m_q,E_q,B_q, c_q, R_q,\varphi_q)$ and $(\overline{m}_q,\overline{E}_q,\overline{B}_q, c_q, \overline{R}_q, \overline{\varphi}_q)$ which satisfy Proposition \ref{Bifurcating inductive proposition}, and we have 
	\begin{align*}
		&\|\overline{m}_{\bar q} - m_{\bar q}\|_{C^0([0, T];L^2(\T^3))} \geqslant  {\varepsilon_0} \delta_q^\frac 12, && \supp_t (m_q-\overline{m}_q) \subset \cal I, \\
		&\|\overline{E}_{\bar q} - E_{\bar q}\|_{C^0([0, T];L^2(\T^3))} \geqslant  {\varepsilon_0} \lambda_{q+1}^{-1}\delta_q^\frac 12, && \supp_t (E_q-\overline{E}_q) \subset \cal I, \\
		&\|\overline{B}_{\bar q} - B_{\bar q}\|_{C^0([0, T];L^2(\T^3))} \geqslant  {\varepsilon_0}\lambda_{q+1}^{-1} \delta_q^\frac 12, && \supp_t (B_q-\overline{B}_q) \subset \cal I, 
	\end{align*}
	with $\cal I  = (10, 10+3\tau_{\bar q-1})$.
	Next, we apply Proposition \ref{Inductive proposition} iteratively to build a new sequence $(\overline{m}_q,\overline{E}_q,\overline{B}_q, c_q, \overline{R}_q, \overline{\varphi}_q)$  which satisfy \eqref{est on m_q}--\eqref{est on current} and \eqref{Proposition of induction}. In a similar way, this new sequence converges to a solution $(n, \overline{m},\overline{E},\overline{B})$ to the compressible Euler-Maxwell system and
	$$
	(\overline{m},\overline{E},\overline{B})\in C^{\beta'}([0,T]\times\T^3)\times C^{1,\beta'}([0,T]\times\T^3)\times C^{1,\beta'}([0,T]\times\T^3).
	$$ 
	Moreover, $\overline{m}_{ {q}}$ shares initial data with $m_{{q}}$  for all $q$, because for any $q\geqslant \bar q$,
	$$
	\supp_t(m_{q} -\overline{m}_{q}) \subset \cal I + \sum_{q=\bar q}^\infty (\lambda_{q}\delta_{q}^\frac12)^{-1} \subset[9,T],
	$$ 
	and then two solutions $\overline{m}_q$ and $m_q$ have the same initial data.  However, the new solution $\overline{m}$ differs from $m$ because
	\begin{align*}
		\|m - \overline{m}\|_{C^0([0, T];L^2(\T^3))}
		&\geqslant \|m_{\bar q} -\overline{m}_{\bar q}\|_{C^0([0,T];L^2(\T^3))}
		-\sum_{q=\bar q}^\infty \|m_{q+1}-m_q - (\overline{m}_{q+1} - \overline{m}_q)\|_{C^0([0,T];L^2(\T^3))} \\
		&\geqslant \|m_{\bar q} - \overline{m}_{\bar q}\|_{C^0([0,T];L^2(\T^3))}
		-(2\pi)^\frac32 \sum_{q=\bar q}^\infty (\|m_{q+1}-m_q\|_0 + \|\overline{m}_{q+1}-\overline{m}_q\|_0)\\
		&\geqslant  {\varepsilon_0} \delta_{\bar q}^\frac 12  - 2(2\pi)^\frac32M \sum_{q=\bar q}^\infty\delta_{q+1}^\frac12>0,\\
		\|E - \overline{E}\|_{C^0([0, T];L^2(\T^3))}
		&\geqslant \|E_{\bar q} -\overline{E}_{\bar q}\|_{C^0([0,T];L^2(\T^3))}
		-\sum_{q=\bar q}^\infty \|E_{q+1}-E_q - (\overline{E}_{q+1} - \overline{E}_q)\|_{C^0([0,T];L^2(\T^3))} \\
		&\geqslant \|E_{\bar q} - \overline{E}_{\bar q}\|_{C^0([0,T];L^2(\T^3))}
		-(2\pi)^\frac32 \sum_{q=\bar q}^\infty (\|E_{q+1}-E_q\|_0 + \|\overline{E}_{q+1}-\overline{E}_q\|_0)\\
		&\geqslant  {\varepsilon_0} \lambda_{q+1}^{-1}\delta_{\bar q}^\frac 12  - 2(2\pi)^\frac32M\lambda_{q+1}^{-1} \sum_{q=\bar q}^\infty\delta_{q+1}^\frac12>0,\\
		\|B - \overline{B}\|_{C^0([0, T];L^2(\T^3))}
		&\geqslant \|B_{\bar q} -\overline{B}_{\bar q}\|_{C^0([0,T];L^2(\T^3))}
		-\sum_{q=\bar q}^\infty \|B_{q+1}-B_q - (\overline{B}_{q+1} - \overline{B}_q)\|_{C^0([0,T];L^2(\T^3))} \\
		&\geqslant \|B_{\bar q} - \overline{B}_{\bar q}\|_{C^0([0,T];L^2(\T^3))}
		-(2\pi)^\frac32 \sum_{q=\bar q}^\infty (\|B_{q+1}-B_q\|_0 + \|\overline{B}_{q+1}-\overline{B}_q\|_0)\\
		&\geqslant  {\varepsilon_0} \lambda_{q+1}^{-1}\delta_{\bar q}^\frac 12  - 2(2\pi)^\frac32M\lambda_{q+1}^{-1} \sum_{q=\bar q}^\infty\delta_{q+1}^\frac12>0,
	\end{align*} 
	if we choose $\lambda_0$ large enough.
	By changing the choice of time interval $\cal I$ and the choice of $\bar{q}$, we can generate infinitely many solutions with $H\equiv 0$ in a similar way. 
	
	\subsection{Proof of Theorem \ref{thm 2}}
	Let $n(t, \cdot) = n_0$ for all $t\in \R$ and fix $\beta' < \beta < \frac{1}{7}$, we could use similar way as in the proof for theorem \ref{thm 1} to construct  tuples $(m_q,E_q,B_q, c_q, R_q,\varphi_q)$  which satisfy \eqref{est on m_q}--\eqref{est on current} , \eqref{Proposition of induction} and $H(0)=0,H'<0$. Moreover, $(m_q,E_q,B_q)$ obtain a limit
	$$
	(m,E,B)\in C^{\beta'}([0,T]\times\T^3)\times C^{1,\beta'}([0,T]\times\T^3)\times C^{1,\beta'}([0,T]\times\T^3),
	$$
	and $(c_q,R_q,\varphi_q)$ converges to $0$ in $C^0([0,T]\times\T^3)\times C^0([0,T]\times\T^3)\times C^0([0,T]\times\T^3)$. Consider the energy equation, we have
	\begin{align*}
		\partial_t \left( \frac{|m|^2}{2 n} +  n e( n)\right) +
		\Div\left(\frac{m}{ n} \left(
		\frac{|m|^2}{2 n} +  n e( n) + p( n)\right)+E\times B\right) = H',
	\end{align*}
	in the distributional distribution. Since $H'<0$ for all $t\in [0,T]$, the constructed solution satisfies the entropy inequality \eqref{entropy inequality 1} strictly. 
\appendix
\section{H\"{o}lder spaces}
	In this section, we intrduce the notations we would use for H{\"o}lder spaces. For some time interval $\cI\subset\R$, we denote the supremum norm as  $\|f\|_0=\left\|f\right\|_{(\cI;C^0(\T^3))}:=\underset{(t,x)\in\cI\times\T^3}{\sup}|f(t,x)|$. For $N\in\N$, a multi-index $k=(k_1,k_2,k_3)\in\N^3$ and $\alpha\in(0,1]$, we denote the H\"{o}lder seminorms as
	$$
	\begin{aligned}
	&[f]_{N}=\underset{|k|=N}{\max}\left\|D^k f\right\|_{0},&&
	[f]_{N+\alpha}=\underset{|k|=N}{\max}\underset{x\neq y,t}{\sup}\frac{|D^k f(t,x)-D^k f(t,x)|}{|x-y|^\alpha},
	\end{aligned}
	$$
	where $D^k$ are spatial derivatives. Then, we can denote the H\"{o}lder norms:
    $$
	\left\|f\right\|_{N}=\sum_{j=0}^{m}[f]_{j},\qquad\left\|f\right\|_{N+\alpha}=\left\|f\right\|_{N}+[f]_{N+\alpha}.
	$$
    If $f\in C^N(\T^3)$, with a little abuse of notations, we will use the same notations as before. We give the following classical lemma without proof.
    \begin{lm}
    Assuming $f$ is sufficiently smooth, we have
    \begin{equation}
    	[f]_{s}\leqslant C(\varepsilon^{r-s}[f]_{r}+\varepsilon^{-s}\left\|f\right\|_{0}), 
    \end{equation} 	 
where $r\geqslant s\geqslant 0;\varepsilon\geqslant 0$, and
\begin{equation}
	[fg]_{r}\leqslant C([f]_{r}\left\|g\right\|_{0}+\left\|f\right\|_{0}[g]_{r}),
\end{equation}
for any $1\geqslant r\geqslant 0$. Moreover, by setting $\varepsilon=\left\|f\right\|_{0}^{\frac{1}{r}}[f]_{r}^{-\frac{1}{r}}$, we could achieve
\begin{equation}
	[f]_{s}\leqslant C\left\|f\right\|_{0}^{1-\frac{s}{r}}[f]_{r}^{\frac{s}{r}}.
\end{equation}
    \end{lm}
	\section{Inverse divergence operator}\label{Inverse divergence operator}
	In this part, we introduce the inverse divergence operators which is originally defined in \cite{DS13}.
    \begin{df}[Leray projection]
	Let $\upsilon\in C^\infty(\T^3;\R^3)$ be a smooth vector field. Let 
	\begin{equation}
		\cQ\upsilon:=\nabla\psi+\fint_{\T^3}\upsilon\label{Qv},
	\end{equation}
	where $\psi\in C^\infty(\T^3)$ is the solution of
	$$\Delta\psi=\Div\ \upsilon,$$
	with $\fint_{\T^3}\psi=0$. Furthermore, let $\mathcal{P}=\cI-\cQ$ be the Leray projection onto divergence-free fields with zero average.
    \end{df}
    \begin{df}[Inverse divergence]\label{def of R}
	Let $\upsilon\in C^\infty(\T^3;\R^3)$ be a smooth vector field. We can define $\cR\upsilon$ to be the matrix-valued periodic function
	\begin{equation}
		\cR\upsilon:=\frac{1}{4}(\nabla\mathcal{P}u+(\nabla\mathcal{P}u)^{\top})+\frac{3}{4}(\nabla u+(\nabla u)^{\top})-\frac{1}{2}(\Div u)\Id,\label{def of R operator}
	\end{equation}
	where $u\in C^\infty(\T^3;\R^3)$ is the solution of
	$$\Delta u=\upsilon-\fint_{\T^3}\upsilon,$$
	with $\fint_{\T^3}u=0$.\label{inverse operator 1}
   \end{df}
   \begin{lm}\label{choose of gamma}
	For any $\upsilon\in C^\infty(\T^3;\R^3)$ we have\\
	$(1)\cR\upsilon(x)$ is a symmetric trace-free matrix for each $x\in\T^3$,\\
	$(2)\Div\cR\upsilon=\upsilon-\fint_{\T^3}\upsilon$.\label{prop of R operator}
   \end{lm}
   \begin{proof}[Proof of lemma \ref{choose of gamma}]
	It is obvious that $\cR\upsilon$ is symmetric. Since $\cP\upsilon$ is divergence-free, we have
	\begin{align*}
	&\tr(\cR\upsilon)=\frac{3}{4}(2\Div u)-\frac{3}{2}\Div u=0,\\ &\Div(\cR\upsilon)=\frac{1}{4}\Delta(\cP u)+\frac{3}{4}(\nabla\Div u+\Delta u)-\frac{1}{2}\nabla\Div u=\Delta u=\upsilon-\fint_{\T^3}\upsilon.\qedhere
	\end{align*}
    \end{proof}
Combining it with the result that any $\upsilon\in C^0(\T^3;\R^3)$ with $\upsilon=\Div R(x)$ for some matrix function $R(x)$ always satisfies $\fint_{\T^3}\upsilon=0$, we can conclude that $\fint_{\T^3}\upsilon=0$ if and only if there exists a trace-free and symmetric matrix function $R(x)$ on $\T^3$ such that $\upsilon=\Div R$.
\begin{pp}
	For any $\alpha\in(0,1)$ and any $N\in\N$, there exists a constant $C(\alpha,N)$ with the following properties.
	For the operators $\cQ,\cP,\cR$ defined above, we have
	\begin{equation}
		\begin{aligned}
			&\left\|\cQ\upsilon\right\|_{N+\alpha}\leqslant C(N,\alpha)\left\|\upsilon\right\|_{N+\alpha},\\
			&\left\|\mathcal{P}\upsilon\right\|_{N+\alpha}\leqslant C(N,\alpha)\left\|\upsilon\right\|_{N+\alpha},\\
			&\left\|\cR\upsilon\right\|_{N+1+\alpha}\leqslant C(N,\alpha)\left\|\upsilon\right\|_{N+\alpha},\\
			&\left\|\cR(\Div A)\right\|_{N+\alpha}\leqslant C(N,\alpha)\left\|A\right\|_{N+\alpha},\\
			&\left\|\cR\cQ(\Div A)\right\|_{N+\alpha}\leqslant C(N,\alpha)\left\|A\right\|_{N+\alpha}.
		\end{aligned}\label{eest}
	\end{equation}
\end{pp}
\begin{proof}
By the standard Schauder estimates,  for any $\phi,\psi:\T^3\rightarrow\R$ with
\begin{align*}
	\left\{\begin{array} { l } 
		{ \Delta \phi = f , } \\
		{ \fint_{\T^3}\phi=0, }
	\end{array} \quad \text { and } \quad \left\{\begin{array}{l}
		\Delta \psi=\operatorname{div} F, \\
		\fint_{\T^3}\psi=0,
	\end{array}\right.\right.
\end{align*}
	we have $
	\left\|\phi\right\|_{N+2+\alpha}\leqslant C(N,\alpha)\left\|f\right\|_{N+\alpha},\left\|\psi\right\|_{N+1+\alpha}\leqslant C(N,\alpha)\left\|F\right\|_{N+\alpha},$ which yields \eqref{eest}
\end{proof} 
For $f\in C^\infty(\T^3,\R)$, we can also define the corresponding inverse divergence operator. For simplicity, with a little abuse of notations, we will use the same symbol $\cR$ as Definition \ref{def of R}.
\begin{df}
Let $f\in C^\infty(\T^3,\R)$ be a smooth vector field with $\fint_{\T^3}f=0$. We can define $\cR f$, with the property $\Div\cR f=f$, to be the vector-valued periodic function
	\begin{equation}
			\cR f:=u,
	\end{equation}
where  $u\in C^\infty(\T^3;\R^3)$ is the solution of 
$$\left\lbrace	\begin{aligned}
	&\Delta u=\nabla f,\\
	&\fint_{\T^3}u=0.
\end{aligned}\right.$$\label{inverse operator 2}
\end{df}
Similarly, we could get the estimates, that for any $\alpha\in(0,1)$ and any $N\in\N$, there exists a constant $C(N,\alpha)$ such that
	\begin{equation}
	\begin{aligned}
		&\left\|\cR f\right\|_{N+1+\alpha}\leqslant C(N,\alpha)\left\|f\right\|_{N+\alpha},\\
	\end{aligned}\label{eest2}
\end{equation}
for any $f\in C^\infty(\T^3,\R)$.
    \section{Estimates for transport equations}
In this section, we will recall some results on transport equations. The proof for the following estimates can be found in \cite{BDIS15}. 
\begin{lm}\cite[Proposition D.1]{BDIS15}.
	If f is the solution of the transport equation:
	\begin{equation}
		\left\lbrace 
		\begin{aligned}
		& \partial_tf+\upsilon\cdot\nabla f=g,\\
		&f|_{t_0}=f_0,
		\end{aligned}\right.
	\end{equation}
where $\upsilon=\upsilon(t,x)$ is a given smooth vector field. We have the following estimates
\begin{equation}
	\left\|f(t)\right\|_0\leqslant\left\|f_0\right\|_0+\int_{t_0}^{t}\left\|g(\tau,\cdot)\right\|_0\rd \tau,\label{transport equation 0}
\end{equation}
\begin{equation}
	[f(t)]_1\leqslant[f_0]_1e^{(t-t_0)[\upsilon]_1}+\int_{t_0}^{t}e^{(t-\tau)[\upsilon]_1}[g(\tau,\cdot)]_1\rd \tau,\label{transport equation 1}
\end{equation}
Moreover ,we can achieve that there exists a constant $C_N$
\begin{equation}
	\begin{aligned}
\left[f(t)\right]_{N}&\leqslant([f_0]_{N}+C_N(t-t_0)[\upsilon]_{N}[f_0]_1)e^{C_N(t-t_0)[\upsilon]_1}\\
&\quad+\int_{t_0}^{t}e^{C_N(t-t_0)[\upsilon]_1}([g(\tau,\cdot)]_{N}+C_N(t-\tau)[\upsilon]_{N}[g(\tau,\cdot)]_1)\rd \tau.
	\end{aligned}\label{transport equation N}
\end{equation}
for any $N\geqslant2$. Define $\Phi(t,\cdot)$ to be the inverse of the flux $X$ of $\upsilon$ starting at $t_0$ as identity (i.e.,$\frac{d}{dt}X=\upsilon(t,X)$ and $X(t_0,x)=x$). Under the same assumption as above,
\begin{equation}
	\left\|\nabla\Phi(t,\cdot)-\Id\right\|_0\leqslant e^{(t-t_0)[\upsilon]_1}-1,\label{DPhi-Id}
\end{equation}
\begin{equation}
	[\Phi(t,\cdot)]_N\leqslant C(t-t_0)[\upsilon]_Ne^{(t-t_0)[\upsilon]_1}.\label{PhiN}
\end{equation}
\end{lm}
	\section{Some technical lemmas}
	In this section, we introduce some lemmas given in \cite{BDSV19,DK22,GK22}. The proof for the following two lemmas can be found in \cite[Appendix]{BDSV19}.
	\begin{lm}\cite[Proposition A.1]{BDSV19}
	Suppose $F:\Omega\rightarrow\R$ and $\Psi:\R^n\rightarrow\Omega$ are smooth functions for some $\Omega\subset\R^m$. Then, for each $N\in\Z_+$, we have 
	\begin{equation}
		\begin{aligned}
			&\left\|\nabla^N(F\circ\Psi)\right\|_0\lesssim\left\|\nabla F\right\|_0\left\|\nabla\Psi\right\|_{N-1}+\left\|\nabla F\right\|_{N-1}\left\|\Psi\right\|_0^{N-1}\left\|\Psi\right\|_N,\\
			&\left\|\nabla^N(F\circ\Psi)\right\|_0\lesssim\left\|\nabla F\right\|_0\left\|\nabla\Psi\right\|_{N-1}+\left\|\nabla F\right\|_{N-1}\left\|\nabla\Psi\right\|_0^{N},
		\end{aligned}\label{est on Holider norm}
	\end{equation}
where the implicit constants in the inequalities depends only on $n, M,$ and $N$.\label{$Holider norm of composition}
	\end{lm}
\begin{lm}\cite[Proposition C.2]{BDSV19}.\label{est on int operator}
	Let $N\geqslant1$. Suppose that $a\in C^\infty(\T^3)$ and $\xi\in C^{\infty}(\T^3;\R^3)$ satisfies
	\begin{equation}
		\frac{1}{C}\leqslant|\nabla\xi|\leqslant C
	\end{equation}
		for some constant $C>1$. Then, we have 
		\begin{equation}\label{est on Rae^ikxi}
			\Bigg|\int_{\T^3}a(x)e^{ik\cdot\xi}\rd x\Bigg|\lesssim\frac{\left\|a\right\|_N+\left\|a\right\|_0\left\|\nabla\xi\right\|_N}{|k|^N},
		\end{equation}
	and for the operator $\cR$ defined in Definition \ref{def of R operator}, we have
	\begin{equation}
		\left\|\cR\left(a(x) e^{i k \cdot \xi}\right)\right\|_{\alpha} \lesssim \frac{\left\|a\right\|_{0}}{|k|^{1-\alpha}}+\frac{\left\|a\right\|_{N+\alpha}+\left\|a\right\|_{0}\left\|\xi\right\|_{N+\alpha}}{|k|^{N-\alpha}},
	\end{equation}
	where the implicit constants in the inequality is depending on $C$ and $N$, but independent of $k$.
\end{lm}
To obtain the commutator estimate, we present the following lemmas, which generalize Lemmas A.3, A.4 and A.6 from \cite{DK22,GK22}. In the following lemmas, we denote $\left\|\cdot\right\|_N=\left\|\cdot\right\|_{C^0([c,d];C^N(\T^3))}$ and  use the notation $\cI_\ell$ to represent the interval $\cI_\ell=[c-\ell,d+\ell]$.
\begin{lm}
Let $f$ and $g$ be in $C^\infty(\cI_\ell\times\T^3)$. Then, for each $N,r\geqslant0$, the following holds,
	\begin{align}
	\left\|\PL f \PL g-\PL(fg)\right\|_N&\lesssim_N\ell^{2-N}\left\|f\right\|_1\left\|g\right\|_1,\label{est on commutator 0}\\
	\left\|\partial_{t}^r(\UL f \UL g-\UL(fg))\right\|_0&\lesssim_r\ell^{2-r}\left\|\partial_{t}f\right\|_{C^0(\cI_\ell;C^0(\T^3))}\left\|\partial_tg\right\|_{C^0(\cI_\ell;C^0(\T^3))}.\label{est on time commutator 0}
	\end{align} 
If we set $f_\ell=\UL\PL f,g_\ell=\UL\PL g$ and $(fg)_\ell=\UL\PL (fg)$ . Then, for each $N,r\geqslant 0$, the following holds,
\begin{equation}\label{est on mollification commutator 0}
	\begin{aligned}
		\left\|\partial_{t}^r(f_\ell g_\ell-(fg)_\ell)\right\|_{N}&\lesssim_{N,r} \ell^{2-N-r}\left\|\partial_{t}f\right\|_{C^0(\cI_\ell;C^0(\T^3))}\left\|\partial_{t}g\right\|_{C^0(\cI_\ell;C^0(\T^3))}\\
		&\quad+\ell^{2-N-r}\left\|f\right\|_{C^0(\cI_\ell;C^1(\T^3))}\left\|g\right\|_{C^0(\cI_\ell;C^1(\T^3))}.
	\end{aligned}
\end{equation} 
\end{lm}
\begin{proof}
	Since the expression that we need to estimate is localized in frequency, by
	Bernstein's inequality it suffices to prove the case $N = 0,r = 0$. Recall the definition of $\PL$ and $\UL$, we could calculate
	\begin{align*}
		(\PL f\PL g-\PL(fg))(t,x)&=-\int_{\R^3}(f(t,x)-f(t,x-y))(g(t,x)-g(t,x-y))\check{\phi_{\ell}}(y)\rd y\\
		&\quad+(f-\PL f)(g-\PL g),
	\end{align*}
	\eqref{est on commutator 0} follows from 
	$$
	\begin{aligned}
		&|f(t,x)-f(t,x-y)|\leqslant|y|\left\|f\right\|_1, \quad|g(t,x)-g(t,x-y)|\leqslant|y|\left\|g\right\|_1,
	\end{aligned}
	$$
	and \eqref{est on space mollification function} with $k=2$.
	Similarly, we could calculate
	\begin{align*}
		(\UL f\UL g-\UL(fg))(t,x)&=-\int_{\R}(f(t,x)-f(t-\tau,x))(g(t,x)-g(t-\tau,x))\check{\phi}^t_\ell(\tau)\rd\tau\\
		&\quad+(f-\UL f)(g-\UL g),
	\end{align*}
	\eqref{est on time commutator 0} follows from
	$$
	\begin{aligned}
		&|f(t,x)-f(t-\tau,x)|\leqslant|\tau|\left\|\partial_{t}f\right\|_{C^0(\cI_\ell;C^0(\T^3))}, \quad|g(t,x)-g(t-\tau,x)|\leqslant|\tau|\left\|\partial_{t}g\right\|_{C^0(\cI_\ell;C^0(\T^3))},
	\end{aligned}
	$$
	and \eqref{est on time mollification function} with $k=2$. Finally, we use \eqref{est on commutator 0} and \eqref{est on time commutator 0} to obtain
	\begin{align*}
		\left\|\partial_{t}^r(f_\ell g_\ell-(fg)_\ell)\right\|_N
		&\leqslant\left\|\partial_{t}^r(f_\ell g_\ell-\UL(\PL f\PL g))\right\|_N+\left\|\partial_{t}^r(\UL(\PL f\PL g)-\UL\PL(fg))\right\|_N\\
		&\lesssim_{N,r}\ell^{-N}\left\|\partial_{t}^r(f_\ell g_\ell-\UL(\PL f\PL g))\right\|_0+\ell^{-r}\left\|\UL(\PL f\PL g)-\UL\PL(fg)\right\|_N\\
		&\lesssim_{N,r} \ell^{2-N-r}\left\|\partial_{t}f\right\|_{C^0(\cI_\ell;C^0(\T^3))}\left\|\partial_{t}g\right\|_{C^0(\cI_\ell;C^0(\T^3))}+\ell^{2-N-r}\left\|f\right\|_{C^0(\cI_\ell;C^1(\T^3))}\left\|g\right\|_{C^0(\cI_\ell;C^1(\T^3))}.\qedhere
	\end{align*}
\end{proof}
\begin{lm}\label{lem of commutator 0}
	Let $f$ and $g$ be in $C^\infty(\cI_{\ell_1}\times\T^3) $ and set $f_{\ell_1,\ell_2}=\ULO\PLT f,g_{\ell_1,\ell_2}=\ULO\PLT g$ and $(fg)_{\ell_1,\ell_2}=\ULO\PLT (fg)$. Then, for each $N,r\geqslant0$, the following holds,
	\begin{align}
	\left\|[g,\ULO\PLT]f\right\|_0&\lesssim\left\|f\right\|_{C^0(\cI_{\ell_1};C^0(\T^3))}(\ell_1\left\|\partial_{t}g\right\|_{C^0(\cI_{\ell_1};C^0(\T^3))}+\ell_2\left\| g\right\|_{C^0(\cI_{\ell_1};C^1(\T^3))}),\label{est on two mollification commutator 1.1}\\
	\left\|\partial_{t}^r[g,\ULO\PLT]f\right\|_N
	&\lesssim_{N,r}\ell_2^{-N}\ell_1^{-r}(\ell_1+\ell_2)\left\|f\right\|_{C^0(\cI_{\ell_1};C^0(\T^3))}\left\|\partial_{t}^{\max\left\lbrace1,r \right\rbrace}g\right\|_{C^0(\cI_{\ell_1};C^{\max\left\lbrace1,N \right\rbrace} (\T^3))}.\label{est on two mollification commutator 1.2}
	\end{align}
In particular, for any smooth function $\upsilon,F\in C^\infty(\cI_{\ell_1}\times\T^3)$ and for each $N+r\geqslant1$, we have
\begin{align}
\left\|[\upsilon\cdot\nabla,\ULO\PLT]F\right\|_0&\lesssim\left\|\nabla F\right\|_{C^0(\cI_{\ell_1};C^0(\T^3))}(\ell_1\left\|\partial_{t} \upsilon\right\|_{C^0(\cI_{\ell_1};C^0(\T^3))}+\ell_2\left\|\upsilon\right\|_{C^0(\cI_{\ell_1};C^1(\T^3))}),\label{est on two mollification commutator 2.1}\\
\left\|\partial_{t}^r[\upsilon\cdot\nabla,\ULO\PLT]F\right\|_N
&\lesssim_{N,r}\ell_2^{-N}\ell_1^{-r}(\ell_1+\ell_2)\left\|\nabla F\right\|_{C^0(\cI_{\ell_1};C^0(\T^3))}\left\|\partial_{t}^{\max\left\lbrace1,r \right\rbrace}\upsilon\right\|_{C^0(\cI_{\ell_1};C^{\max\left\lbrace1,N \right\rbrace}(\T^3))}.\label{est on two mollification commutator 2.2}
\end{align}
Moreover, when $\upsilon$ has the space frequency localized to $\ell_2^{-1}$ and time frequency localized to $\ell_1^{-1}$, namely $\upsilon=\ULO\PLT\upsilon$,  using the Bernstein's inequality, we can get 
\begin{equation}
 \left\|\partial_{t}^r[\upsilon\cdot\nabla,\ULO\PLT]F\right\|_{N}\lesssim_{N,r}\ell_2^{-N}\ell_1^{-r}(\ell_1+\ell_2)\left\|\nabla F\right\|_{C^0(\cI_{\ell_1};C^0(\T^3))}(\left\|\partial_{t}\upsilon\right\|_{C^0(\cI_{\ell_1};C^0(\T^3))}+\left\| \upsilon\right\|_{C^0(\cI_{\ell_1};C^1(\T^3))}).\label{est on two mollification commutator 2.3}
\end{equation}
\end{lm}
	\begin{proof}
	We begin by  calculating
	\begin{align*}
	(f_{\ell_1,\ell_2} g-(fg)_{\ell_1,\ell_2})(t,x)&=\int_{\R}\int_{\R^3}f(\tau,y)(g(t,x)-g(\tau,y))\check{\phi_{\ell_2}}(x-y)\check{\phi}_{\ell_1}^t(t-\tau)\rd y\rd  \tau.
	\end{align*}	
	Then, \eqref{est on two mollification commutator 1.1} follows from
	$$
	\begin{aligned}
		&\quad|g(t,x)-g(\tau,y)|\leqslant|t-\tau|\left\|\partial_{t}g\right\|_{C^0(\cI_{\ell_1};C^0(\T^3))}+|x-y|\left\|g\right\|_{C^0(\cI_{\ell_1};C^1(\T^3))},\\
		&\int_{\R^3}|x-y| |\check{\phi_{\ell_2}}(x-y)| \rd  y \lesssim \ell_2,\qquad\int_{\R}|t-\tau|| \check{\phi}^t_{\ell_1}(t-\tau)| \rd  \tau \lesssim \ell_1.
	\end{aligned}
	$$ 
	Moreover, we could obtain
	\begin{equation}\notag
	\begin{aligned}
	&\quad|\partial_{t}^r\nabla^N(f_{\ell_1,\ell_2} g-(fg)_{\ell_1,\ell_2})(t,x)|\\
	&\leqslant\int_{\R}\int_{\R^3}|f(\tau,y)||g(t,x)-g(\tau,y)||\nabla^N\check{\phi_{\ell_2}}(x-y)\partial_{t}^r\check{\phi}^t_{\ell_1}(t-\tau)|\rd y \rd  \tau\\
	&\quad+\sum_{r_0+r_1=r}\sum_{N_1+N_2=N}C\int_{\R}\int_{\R^3}|f(\tau,y)||\partial_{t}^{r_0}\nabla^{N_1}g(t,x)||\nabla^{N_2}\check{\phi_{\ell_2}}(x-y)\partial_{t}^{r_1}\check{\phi}^t_{\ell_1}(t-\tau)|\rd y \rd  \tau,
	\end{aligned}
	\end{equation}
	for some constants $C=C_{N,r}>0$. We could immediately get \eqref{est on two mollification commutator 1.2},
	\begin{align*}
		|\partial_{t}^r\nabla^N(f_{\ell_1,\ell_2} g-(fg)_{\ell_1,\ell_2})(t,x)|&\lesssim_{N,r}\ell_2^{-N}\ell_1^{-r}\left\|f\right\|_{C^0(\cI_{\ell_1};C^0(\T^3))}(\ell_2\left\| g\right\|_{C^0(\cI_{\ell_1};C^1(\T^3))}+\ell_1\left\|\partial_{t}g\right\|_{C^0(\cI_{\ell_1};C^0(\T^3))})\\
		&\quad+\ell_2^{-N}\ell_1^{-r}(\ell_1+\ell_2)\left\|f\right\|_{C^0(\cI_{\ell_1};C^0(\T^3))}\left\|\partial_{t}^rg\right\|_{C^0(\cI_{\ell_1};C^N(\T^3))}\\
		&\lesssim_{N,r}\ell_2^{-N}\ell_1^{-r}(\ell_1+\ell_2)\left\|f\right\|_{C^0(\cI_{\ell_1};C^0(\T^3))}\left\|\partial_{t}^{\max\left\lbrace1,r \right\rbrace}g\right\|_{C^0(\cI_{\ell_1};C^{\max\left\lbrace1,N \right\rbrace} (\T^3))},
	\end{align*}
 	for $N+r\geqslant 1$. Finally, \eqref{est on two mollification commutator 2.1} and \eqref{est on two mollification commutator 2.2} can be proved, if we apply \eqref{est on two mollification commutator 1.1} and \eqref{est on two mollification commutator 1.2} to $g=\upsilon_i, f=\partial_iF$. Moreover, when $\upsilon$ has the space frequency localized to $\ell_2^{-1}$ and time frequency localized to $\ell_1^{-1}$, we have
 	\begin{align*}
 	\left\|\partial_{t}^{r}\upsilon\right\|_{C^0(\cI_{\ell_1};C^{N}(\T^3))}\lesssim_{N,r}\ell_2^{-N}\ell_1^{1-r}\left\|\partial_{t}\upsilon\right\|_{C^0(\cI_{\ell_1};C^{0}(\T^3))},\ \left\|\partial_{t}^{r}\upsilon\right\|_{C^0(\cI_{\ell_1};C^{N}(\T^3))}\lesssim_{N,r}\ell_2^{1-N}\ell_1^{-r}\left\|\upsilon\right\|_{C^0(\cI_{\ell_1};C^{1}(\T^3))},
 	\end{align*}
 	and
 	\begin{align*}
 		&\quad|\partial_{t}^r\nabla^N((\partial_iF)_{\ell_1,\ell_2} \upsilon_i-(\upsilon_i\partial_iF)_{\ell_1,\ell_2})(t,x)|\\
 		&\lesssim_{N,r}\ell_2^{-N}\ell_1^{-r}\left\|\nabla F\right\|_{C^0(\cI_{\ell_1};C^0(\T^3))}(\ell_2\left\| \upsilon\right\|_{C^0(\cI_{\ell_1};C^1(\T^3))}+\ell_1\left\|\partial_{t}\upsilon\right\|_{C^0(\cI_{\ell_1};C^0(\T^3))})\\
 		&\quad+\sum_{r_0+r_1=r}\sum_{N_1+N_2=N}\ell_2^{-N_2}\ell_1^{-r_1}\left\|\nabla F\right\|_{C^0(\cI_{\ell_1};C^0(\T^3))}\left\|\partial_{t}^{r_0}\upsilon\right\|_{C^0(\cI_{\ell_1};C^{N_1}(\T^3))}\\
 		&\lesssim_{N,r}\ell_2^{-N}\ell_1^{-r}\left\|\nabla F\right\|_{C^0(\cI_{\ell_1};C^0(\T^3))}(\ell_2\left\| \upsilon\right\|_{C^0(\cI_{\ell_1};C^1(\T^3))}+\ell_1\left\|\partial_{t}\upsilon\right\|_{C^0(\cI_{\ell_1};C^0(\T^3))})\\
 		&\quad+\sum_{r_0+r_1=r}\sum_{N_1+N_2=N}\ell_2^{-N}\ell_1^{-r}\left\|\nabla F\right\|_{C^0(\cI_{\ell_1};C^0(\T^3))}(\ell_2\left\| \upsilon\right\|_{C^0(\cI_{\ell_1};C^1(\T^3))}+\ell_1\left\|\partial_{t}\upsilon\right\|_{C^0(\cI_{\ell_1};C^0(\T^3))})\\
 		&\lesssim_{N,r}\ell^{-N}\ell_1^{-r}(\ell_1+\ell_2)\left\|\nabla F\right\|_{C^0(\cI_{\ell_1};C^0(\T^3))}(\left\| \upsilon\right\|_{C^0(\cI_{\ell_1};C^1(\T^3))}+\left\|\partial_{t}\upsilon\right\|_{C^0(\cI_{\ell_1};C^0(\T^3))}).\qedhere
 	\end{align*}
\end{proof}
\begin{lm}
	For a fixed $\overline{N}\in\N$, if $\upsilon$ and $g$ satisfy
	$$\left\|\partial_{t}^r\upsilon\right\|_{C^0(\cI_\ell;C^N(\T^3))}\lesssim_{\overline{N}}\ell^{-N-r}C_\upsilon,\qquad\left\| \partial_{t}^{r+1}g\right\|_{C^0(\cI_\ell;C^{N}(\T^3))}\lesssim_{\overline{N}}C_g,$$
	for all integer $N+r\in [1,\overline{N}]$ and for some positive constants $C_\upsilon$ and $C_g$, then we have
	\begin{align}
	&\quad\left\|\partial_{t}^r([\upsilon\cdot\nabla,\UL\PL](fg)-([\upsilon\cdot\nabla,\UL\PL]f)g)\right\|_N\nonumber\\
	&\lesssim_{\overline{N}}\ell^{1-N-r}\left\|\nabla f\right\|_{C^0(\cI_\ell;C^0(\T^3))}C_\upsilon C_g+\ell^{-N-r}\left\|f\right\|_{C^0(\cI_\ell;C^0(\T^3))}C_\upsilon C_g,\label{est on mollification commutator 3}
	\end{align}
	for all integer $N+r\in [0,\overline{N}]$.
\end{lm} 
	\begin{proof}
	We first write
	$$
	\begin{aligned}
	&\quad{\left[\upsilon \cdot \nabla, \PL\UL\right](f g)-\left(\left[\upsilon \cdot \nabla, \PL\UL\right] f\right) g } \\
	&=\int_{\mathbb{R}^3}\int_{\mathbb{R}}(\upsilon(t,x)-\upsilon(\tau,y)) \cdot \nabla(f g)(\tau,y) \check{\phi}^t_\ell(t-\tau)\check{\phi_{\ell}}(x-y) \rd y\rd  \tau \\
	&\quad-\int_{\mathbb{R}^3}\int_{\mathbb{R}}(\upsilon(t,x)-\upsilon(\tau,y)) \cdot \nabla f(\tau,y) g(t,x) \check{\phi}^t_\ell(t-\tau)\check{\phi_{\ell}}(x-y) \rd y\rd  \tau \\
	&=\int_{\mathbb{R}^3}\int_{\mathbb{R}}(\upsilon(t,x)-\upsilon(\tau,y)) \cdot \nabla f(\tau,y)(g(\tau,y)-g(t,x)) \check{\phi}^t_\ell(t-\tau)\check{\phi_{\ell}}(x-y) \rd y\rd  \tau\\
	&\quad+\int_{\mathbb{R}^3}\int_{\mathbb{R}}(\upsilon(t,x)-\upsilon(\tau,y)) \cdot \nabla g(\tau,y) f(\tau,y) \check{\phi}^t_\ell(t-\tau)\check{\phi_{\ell}}(x-y) \rd y\rd  \tau .
	\end{aligned}
	$$
	Then, the inequality follows from
	\begin{align*}
	&|\partial_{t}^r\nabla^N(\upsilon(t,x)-\upsilon(\tau,y))|\lesssim \ell^{-N-r} C_\upsilon, \quad|g(t,x)-g(\tau,y)| \lesssim|t-\tau| C_g+|x-y|C_g,\\
	&\int|\tau|| \check{\phi}^t_\ell(\tau)| \rd  \tau \lesssim \ell,\quad\int|y|| \check{\phi_{\ell}}(y)| \rd  y \lesssim \ell.\qedhere
	\end{align*}
	\end{proof}
Finally, we introduce a lemma which has been proved in \cite{GK22}.
\begin{lm}\cite[Lemma A.7]{GK22}.\label{est on R commutator}
	For vector-valued functions $H$ and $\upsilon$ in $C^\infty(\cI_\ell\times\T^3)$, the following commutator estimate holds,
	\begin{equation}\label{est on commutator 4}
		\left\|[\PL\upsilon\cdot\nabla,\cR]P_{\GL}H\right\|_{N-1}\lesssim\sum_{N_1+N_2=N-1}\ell\left\|\nabla\upsilon\right\|_{N_1}\left\|H\right\|_{N_2}.
	\end{equation}
	for N=1,2, where  $\cR=\Delta^{-1}\nabla$, as defined in Definition \ref{def of R}, and $\lambda_{q+1},\ell$ are defined as in \eqref{def of parameter} and \eqref{def of l and l_t}.
\end{lm}
\section{The commutator of the space derivative and the material derivative}
We replace the detailed proof of the estimates on the mixed derivatives used in Section \ref{Estimates on mixed derivatives} through the commutator analysis here.
\begin{lm}\label{Estimtates for mixed derivatives}
Let $F$ be in $C^\infty([c,d]\times\T^3)$, we have the following estimates, for $r\geqslant1,k\geqslant0,$ 
\begin{equation}\label{est on Dtl nabla^r F}
\begin{aligned}
	\left\|\DTL \nabla^{r}\DTL^k F\right\|_{N}
	&\lesssim_{N,r}\left\|\DTL^{k+1} F\right\|_{N+r}+
	\sum_{N_0+N_1=N+r-1}\left\|\DTL^kF\right\|_{N_0+1}\left\|m_\ell/n\right\|_{N_1+1},\\
	\left\|\DTL^2\nabla^{r}\DTL^k F\right\|_{N}
	&\lesssim_{N,r}\left\|\DTL^{k+2} F\right\|_{N+r}+
	\sum_{N_0+N_1=N+r-1}\left\| \DTL^{k+1} F\right\|_{N_0+1}\left\|m_\ell/n\right\|_{N_1+1}\\
	&\quad+\sum_{N_0+N_1=N+r-1}\left\|\DTL^kF\right\|_{N_0+1}\left\|\DTL (m_\ell/n)\right\|_{N_1+1}\\
	&\quad+\sum_{N_0+N_1+N_2=N+r-1}\left\|\DTL^kF\right\|_{N_0+1}\left\|m_\ell/n\right\|_{N_1+1}\left\|m_\ell/n\right\|_{N_2+1},
\end{aligned}
\end{equation}
where $\left\|\cdot\right\|=\left\|\cdot\right\|_{C^0([c,d];C^N(\T^3))}$, $D_{t,\ell}=\partial_t+m_\ell/n\cdot\nabla$.
\end{lm}
\begin{proof}
We first calculate, for $\forall F\in  C^j([c,d]\times\T^3),j\in\N$,
\begin{align*}
	\DTL \nabla F&=\nabla \DTL F- \nabla F\nabla(m_\ell/n),\\
	\DTL^2\nabla F&=\nabla \DTL^2F-\nabla(\DTL F)\nabla(m_\ell/n)-\DTL (\nabla F\nabla(m_\ell/n)),
\end{align*}
and then we could get
	\begin{align}
		\left\|\DTL \nabla F\right\|_{N}&\lesssim_N\left\| \DTL F\right\|_{N+1}+
		\sum_{N_0+N_1=N}\left\|F\right\|_{N_0+1}\left\|m_\ell/n\right\|_{N_1+1},\label{est on Dtl nabla F}
	\end{align}
and 
	\begin{align}
		\left\|\DTL^2\nabla F\right\|_{N}
		&\lesssim_{N}\left\|\DTL^2F\right\|_{N+1}+
		\sum_{N_0+N_1=N}\left\| \DTL F\right\|_{N_0+1}\left\|m_\ell/n\right\|_{N_1+1}\nonumber\\
		&\quad+\sum_{N_0+N_1=N}\left\|\DTL \nabla F\right\|_{N_0}\left\|m_\ell/n\right\|_{N_1+1}+\sum_{N_0+N_1=N}\left\| F\right\|_{N_0+1}\left\|\DTL \nabla (m_\ell/n)\right\|_{N_1}\nonumber\\
		&\lesssim_{N}\left\|\DTL^2F\right\|_{N+1}+
		\sum_{N_0+N_1=N}\left\| \DTL F\right\|_{N_0+1}\left\|m_\ell/n\right\|_{N_1+1}+
		\sum_{N_0+N_1=N}\left\|  F\right\|_{N_0+1}\left\|\DTL (m_\ell/n)\right\|_{N_1+1}\nonumber\\
		&\quad+\sum_{N_0+N_1+N_2=N}\left\|F\right\|_{N_0+1}\left\|m_\ell/n\right\|_{N_1+1}\left\|m_\ell/n\right\|_{N_2+1}.\label{est on Dtl^2 nabla F}
	\end{align}
Then, we could use \eqref{est on Dtl nabla F} and \eqref{est on Dtl^2 nabla F} over and over again to achieve, for $r\geqslant 1$,
	\begin{align}
	\left\|\DTL \nabla^{r+1} F\right\|_{N}
	&\lesssim_{N}\left\|\DTL \nabla^r F\right\|_{N+1}+
	\sum_{N_0+N_1=N}\left\|\nabla^r F\right\|_{N_0+1}\left\|m_\ell/n\right\|_{N_1+1}\nonumber\\
	&\lesssim_{N}\left\|\DTL \nabla^{r-1} F\right\|_{N+2}+\sum_{N_0+N_1=N+1}\left\|\nabla^{r-1} F\right\|_{N_0+1}\left\|m_\ell/n\right\|_{N_1+1}+\sum_{N_0+N_1=N}\left\|\nabla^{r} F\right\|_{N_0+1}\left\|m_\ell/n\right\|_{N_1+1}\nonumber\nonumber\\
	&\lesssim_{N,r}\left\|\DTL F\right\|_{N+r+1}+\sum_{N_0+N_1=N+r}\left\|F\right\|_{N_0+1}\left\|m_\ell/n\right\|_{N_1+1},\label{est on Dtl nabla^r+1 F}
	\end{align}
and
\begin{align}
		\left\|\DTL^2\nabla^{r+1} F\right\|_{N}
		&\lesssim_{N}\left\|\DTL^2\nabla^r F\right\|_{N+1}+
		\sum_{N_0+N_1=N}\left\| \DTL \nabla^r F\right\|_{N_0+1}\left\|m_\ell/n\right\|_{N_1+1}\nonumber\\
		&\quad+\sum_{N_0+N_1=N}\left\| \nabla^r F\right\|_{N_0+1}\left\|\DTL (m_\ell/n)\right\|_{N_1+1}+\sum_{N_0+N_1+N_2=N}\left\|\nabla^r F\right\|_{N_0+1}\left\|m_\ell/n\right\|_{N_1+1}\left\|m_\ell/n\right\|_{N_2+1}\nonumber\\
		&\lesssim_{N}\left\|\DTL^2\nabla^r F\right\|_{N+1}+
		\sum_{N_0+N_1=N}\left\| \DTL F\right\|_{N_0+r+1}\left\|m_\ell/n\right\|_{N_1+1}\nonumber\\
		&\quad+\sum_{N_0+N_1=N}\left\| F\right\|_{N_0+r+1}\left\|\DTL (m_\ell/n)\right\|_{N_1+1}+\sum_{N_0+N_1+N_2=N+r}\left\|F\right\|_{N_0+1}\left\|m_\ell/n\right\|_{N_1+1}\left\|m_\ell/n\right\|_{N_2+1}\nonumber\\
		&\lesssim_{N}\left\|\DTL^2\nabla^{r-1} F\right\|_{N+2}+
		\sum_{N_0+N_1=N+1}\left\| \DTL F\right\|_{N_0+r}\left\|m_\ell/n\right\|_{N_1+1}\nonumber\\
		&\quad+\sum_{N_0+N_1=N+1}\left\| F\right\|_{N_0+r}\left\|\DTL (m_\ell/n)\right\|_{N_1+1}+\sum_{N_0+N_1+N_2=N+r}\left\|F\right\|_{N_0+1}\left\|m_\ell/n\right\|_{N_1+1}\left\|m_\ell/n\right\|_{N_2+1}\nonumber\\
		&\lesssim_{N,r}\left\|\DTL^2 F\right\|_{N+r+1}+
		\sum_{N_0+N_1=N+r}\left\| \DTL F\right\|_{N_0+1}\left\|m_\ell/n\right\|_{N_1+1}\nonumber\\
		&\quad+\sum_{N_0+N_1=N+r}\left\| F\right\|_{N_0+1}\left\|\DTL (m_\ell/n)\right\|_{N_1+1}+\sum_{N_0+N_1+N_2=N+r}\left\|F\right\|_{N_0+1}\left\|m_\ell/n\right\|_{N_1+1}\left\|m_\ell/n\right\|_{N_2+1},\label{est on Dtl^2 nabla^r+1 F}
	\end{align}
which leads directly to \eqref{est on Dtl nabla^r F}.
\end{proof}
\section{A microlocal lemma}
As in \cite{DK22,GK22}, the following microlocal lemma is important in our estimates for the error term, where we will use the notation
\begin{equation}
\mathscr{F}[f](k)=\fint_{\T^3}f(x)e^{-ix\cdot k}\rd x,\qquad f(x)=\sum_{k\in\Z^3}\mathscr{F}[f](k)e^{ik\cdot x}.
\end{equation}
\begin{lm}[Microlocal Lemma]\cite[Lemma 7.1]{GK22}\label{Microlocal Lemma}
	Let T be a Fourier multiplier defined on $C^\infty(\T^3)$ by 
	$$\mathscr{F}[Th](k)=\mathfrak{m}(k)\mathscr{F}[h](k),\qquad \forall k\in\Z^3 $$
	for some $\mathfrak{m}$ which has an extension in $S(\R^3)$ (which for convenience we keep denoting by $\mathfrak{m}$). Then, for any $n_0\in\N,\lambda>0$, and any scalar functions $a$ and $\xi$ in $C^\infty(\T^3)$, $T(ae^{i\lambda\xi})$ can be decomposed into
	$$T(ae^{i\lambda\xi})=\Big[a\mathfrak{m}(\lambda\nabla\xi)+\sum_{k=1}^{2n_0}C_k^\lambda(\xi,a):(\nabla^k\mathfrak{m})(\lambda\nabla\xi)+\varepsilon_{n_0}(\xi,a)\Big]e^{i\lambda\xi},  $$
	for some tensor valued coefficient $C_k^\lambda(\xi,a)$ and a remainder $\varepsilon_{n_0}(\xi,a)$ which is specified in the following formula:
	\begin{equation}
		\begin{aligned}
		&\varepsilon_{n_0}(\xi,a)(x)=\sum_{n_1+n_2=n_0}\frac{(-1)^{n_1}c_{n_1,n_2}}{n_0!}\\
		&\quad\cdot\int_{0}^{1}\int_{\R^3}\check{\mathfrak{m}}(y)e^{-i\lambda\nabla\xi(x)\cdot y}((y\cdot\nabla)^{n_1}a)(x-ry)e^{i\lambda Z[\xi](r)}\beta_{n_2}[\xi](r)(1-r)^{n_0}\rd ydr,
		\end{aligned}
	\end{equation}
where $c_{n_1,n_2}$ is a constant depending only on $n_1$ and $n_2$, and the function $\beta_n[\xi]$ is 
$$
\begin{aligned}
	&\alpha_n[\xi](r)=B_n(i\lambda Z^\prime(r),i\lambda^{\prime\prime},\cdots,i\lambda Z^{(n)}(r)),\\
	&Z[\xi](r)=Z[\xi]_{x,y}(r)=r\int_{0}^{1}(1-s)(y\cdot\nabla)^2\xi(x-rsy) \rd  s,
\end{aligned}
$$
with $B_n$ denoting the nth complete exponential Bell polynomial
\begin{equation}
	B_n(x_1,\cdot,x_n)=\sum_{k=1}^{n}B_{n,k}(x_1,x_2,\cdots,x_{n-k+1}),
\end{equation}
where 
$$
B_{n,k}(x_1,x_2,\cdots,x_{n-k+1})=\sum\frac{n!}{j_1!j_2!\cdots j_{n-k+1}!}\left(\frac{x_1}{1!}\right)^{j_1}\left(\frac{x_2}{2!}\right)^{j_2}\cdots\left(\frac{x_{n-k+1}}{(n-k+1)!}\right)^{j_{n-k+1}},
$$
and the summation is taken over $\{j_k\}\subset\N\cup\left\lbrace0\right\rbrace $ satisfying
\begin{equation}
j_1+j_2+\cdots+j_{n-k+1}=k,\qquad j_1+2j_2+3j_3+\cdots+(n-k+1)j_{n-k+1}=n.
\end{equation}
\end{lm}
This lemma leads to the following consequence on the anti-divergence operator $\cR$  as in \cite{DK22} and \cite{GK22}
\begin{co}\cite[Corollary 7.2]{GK22}\label{est on R operator}
Let $N=0,1,2$ and $F=\sum_{k\in\Z^3\setminus\left\lbrace 0\right\rbrace }\sum_{u\in\Z}a_{u,k}e^{i\lambda_{q+1}k\cdot\xi_u}$. Assume that a function $a_{u,k}$ fulfills the following requirements.\\
(i)The support of $a_{u,k}$ satisfies $\supp(a_{u,k})\subset(t_u-\frac{1}{2}\tau_q,t_u+\frac{3}{2}\tau_q)\times\R^3$. In particular, for $u$ and $u^\prime$ neither same nor adjacent, we have
\begin{equation}
\supp(a_{u,k})\bigcap	\supp(a_{u^\prime,k^\prime})=\emptyset,\qquad\forall k,k^\prime\in\Z^3\setminus\left\lbrace 0\right\rbrace. 
\end{equation}
(ii)For any $0\leqslant j\leqslant n_0+1$ and $(u,k)\in\Z\times\Z^3$,
\begin{equation}
\left\|a_{u,k}\right\|_j+(\lambda_{q+1}\delta_{q+1}^{\frac{1}{2}})^{-1}\left\|\DTL a_{u,k}\right\|_j\lesssim_j\mu_q^{-j}|\overset{\circ}{a}_k|\leqslant a_F,\quad \sum_k|k|^{n_0+2}|\overset{\circ}{a}_k|\leqslant a_F,
\end{equation}
for some $\overset{\circ}{a}_k$ and $a_F>0$, where $n_0=\left\lceil\frac{2b(2+\alpha)}{(b-1)(1-\alpha)}\right\rceil$ and $\left\|\cdot\right\|_j=\left\|\cdot\right\|_{C(\cI;C^j(\T^3))}$ on some time interval $\cI\subset \R$.\\
Then, for any $b\in(1,3)$, we can find $\Lambda_0=\Lambda_0(b,n)$ such that for any $\lambda_0\geqslant\Lambda_0$, $\cR F$ satisfies the following inequalities:
\begin{equation}
\left\|\cR F\right\|_N\lesssim\lambda_{q+1}^{N-1}a_F,\qquad\left\|D_{t,q+1}\cR F\right\|_{N-1}\lesssim\lambda_{q+1}^{N-1}\delta_{q+1}^{\frac{1}{2}}a_F,
\end{equation}
upon setting $D_{t,q+1}=\partial_t+\frac{m_{q+1}}{n}\cdot\nabla$.
\end{co}
\begin{proof}[Sketch of the proof.]
	The proof is relying on the decomposition   
	\begin{align}\label{dec of F}
	F = \cP_{\GL} \left(\sum_{u,k} a_{u,k} e^{i\lambda_{q+1} k\cdot \xi_u} \right) - \sum_{u,k} \varepsilon_{n_0}^{\lambda_{q+1}}(k\cdot \xi_u, a_{u,k}) e^{i\lambda_{q+1} k\cdot \xi_u},
	\end{align}
	where $\cP_{\GL}$ is defined by
	\[
	\cP_{\GL} = \sum_{2^j \geqslant \frac 38\lambda_{q+1}} P_{2^j}
	\]
	and
	\[
	\varepsilon_{n_0}^{\lambda_{q+1}} (k\cdot \xi_u, a_{u,k})
	= \sum_{2^j\geqslant \frac 38\lambda_{q+1}} \varepsilon_{n_0, j} (k\cdot \xi_u, a_{u,k}).
	\]
	The remainder $ \varepsilon_{n_0, j}(\xi,a)$ is obtained by applying Lemma \eqref{Microlocal Lemma} to $P_{2^j}$ and $n_0= \left\lceil\frac{2b(2+\alpha)}{(b-1)(1-\alpha)}\right\rceil$. In particular, the remainder part of $F$ has frequency localization 
	\begin{align}\label{dec of PL F}
		\cP_{\LL} F := F - \cP_{\GL}F
		=- \sum_{k,u}\varepsilon_{n_0}^{\lambda_{q+1}}(k\cdot \xi_u, a_{u,k}) e^{i\lambda_{q+1} k\cdot \xi_u},
	\end{align}
	and satisfies
	\begin{align}
		\left\|\sum_{u,k}\varepsilon_{n_0}^{\lambda_{q+1}}(k\cdot \xi_u, a_{u,k})\right\|_0 
		&\lesssim_{n_0} (\lambda_{q+1}\mu_q)^{-(n_0+1)} a_F\lesssim \lambda_{q+1}^{-2}a_F,\label{est on epsilon}
		\\
		\left\|\sum_{u,k}D_{t,\ell} \varepsilon_{n_0}^{\lambda_{q+1}}(k\cdot \xi_u, a_{u,k})\right\|_0 
		&\lesssim_{n_0}   \lambda_{q+1}\delta_{q+1}^\frac12 (\lambda_{q+1}\mu_q)^{-(n_0+1)} a_F 
		\lesssim \lambda_{q+1}\delta_{q+1}^\frac12\cdot\lambda_{q+1}^{-2} a_F. \label{est on Dtl epsilon}
	\end{align}
	Using this, one can easily obtain $\left\|\cR F\right\|_{N} \lesssim \lambda_{q+1}^{N-1} a_F$. To estimate the material derivative of $\mathcal{R}F$, we use the following decomposition,
	\begin{align*}
		D_{t, q+1} \mathcal{R}F
		&= \mathcal{R} D_{t,\ell} F
		+ \left[\frac{m_\ell} {n} \cdot \nabla, \mathcal{R} \right]F
		+ \left(\frac{\tm+ m_q-m_\ell}{n} \right)\cdot \nabla \mathcal{R}F.
	\end{align*}
	The first and the last terms on the right hand side can be estimated as in \cite[Corollary 8.2]{DK22}. There is a little difference in the estimation on the last term, because we add a time correction. From \eqref{est on tm_t}, we know the time correction is too small to affect the results. To estimate the second term, we further decompose it into 
	\begin{align*}
		\left[m_\ell P_{\leqslant \ell^{-1}}n^{-1} \cdot \nabla, \mathcal{R} \right]F
		+ \left[m_\ell P_{> \ell^{-1}}n^{-1} \cdot \nabla, \mathcal{R} \right]F.   
	\end{align*}
	Since $m_{\ell} P_{\leqslant \ell^{-1}}n^{-1} = P_{\lesssim \ell^{-1}}(m_{\ell} P_{\leqslant \ell^{-1}}n^{-1})$, we can estimate it as in \cite[Corollary 8.2]{DK22}. Therefore, it suffices to estimate the remaining term;
	\begin{align*}
		\left\|\left[m_\ell P_{> \ell^{-1}}n^{-1} \cdot \nabla, \mathcal{R} \right]F\right\|_{N-1}
		&\lesssim \sum_{N_1+N_2=N-1}
		\left\|m_\ell P_{> \ell^{-1}}n^{-1}\right\|_{N_1}\left\|\nabla F\right\|_{N_2}\\
		&\lesssim \lambda_{q+1}^N \ell^2 a_F
		\lesssim \lambda_{q+1}^{N-1}\delta_{q+1}^\frac12 a_F,
	\end{align*}
	where we used $\left\|m_\ell P_{>\ell^{-1}}n^{-1}\right\|_{N_1}\lesssim_{n,N_1}\ell^{2}$ and the choice of $b<3$ and sufficiently large $\Lambda_0$.	
\end{proof}
\section{Estimate for nonautonomous linear differential systems }
Here, we give an estimate for nonautonomous linear differential systems which is given in \cite{SM85}. Consider a linear system of ordinary differential equations
\begin{equation}
	\frac{\rd y}{dt}=A(t)y+g(t),\qquad y\in\R^m .\label{ode}
\end{equation}
We could get estimates on the solution to (\ref{ode}).
\begin{align*}
	\frac{d}{dt}\left\|y\right\|_\infty\leqslant\left\|A(t)\right\|_\infty\left\|y\right\|_\infty+\left\|g\right\|_\infty.
\end{align*}
where $\left\|a\right\|_\infty=\underset{i,j}{\sup}|a_{ij}|$. If there exists $K$ such that $\left\|A(t)\right\|_\infty\leqslant K$,
\begin{equation}
	\left\|y(t)\right\|_\infty\leqslant \left\|y(0)\right\|_\infty\exp(Kt)+\int_{0}^{t}\left\|g(\tau)\right\|_\infty\exp(K(t-\tau))\rd \tau.
\end{equation}
Especially, if we consider the second order ordinary differential equation:
\begin{equation}
	\left\lbrace 
	\begin{aligned}
	&\frac{d^2y}{dt^2}+a(t)\frac{\rd y}{dt}+b(t)y=R(t),\\ 
	&y(0)=0,\quad \frac{\rd y}{dt}(0)=0,\label{second order ode}
	\end{aligned}
\right.
\end{equation}
where $b(0)> \varepsilon_0>0$ for some constant $\varepsilon_0$, it can be transformed into a nonautonomous linear differential system:
\begin{align*}
	\left\lbrace
	\begin{aligned}
	&\frac{dU}{dt}=A(t)U+g(t), \\
	&U(0)=U_0,
	\end{aligned}
	\right.
\end{align*}
where 
$$
U=\left(\begin{array}{cc}
	y(t)  \\
	\frac{\rd y}{dt}(t) 
\end{array}\right),\quad
A(t)=\left(\begin{array}{cc}
	0 & 1 \\
	-b(t) & -a(t) 
\end{array}\right), \quad g(t)=\left(\begin{array}{ll}
	0 \\
	R(t)
\end{array}\right),\quad
U_0=\left(\begin{array}{cc}
	0   \\
	0
\end{array}\right).
$$
So we could get for $0\leqslant t\leqslant T$, there exists $C(T,K,\varepsilon_0)$ such that
\begin{align*}
	\left\|y(t)\right\|_\infty,\left\|\frac{\rd y}{dt}(t)\right\|_\infty\leqslant\left\|U_0\right\|_\infty\exp(Kt)+\int_{0}^{t}\left\|g(\tau)\right\|_\infty\exp(K(t-\tau))\rd \tau\leqslant C(T,K)\underset{t\in[0,T]}{\sup}|R(t)|.\label{ode y yp}
\end{align*}
Combining it with \eqref{second order ode} , we could obtain
\begin{equation}
\left\|\frac{d^2y}{dt^2}(t)\right\|_\infty\leqslant C(T,K)\underset{t\in[0,T]}{\sup}|R(t)|,\label{ode ypp}
\end{equation}
where $K=\underset{t\in[0,T]}{\sup}\max\left\lbrace |a(t)|,|b(t)|\right\rbrace $.              
\bibliographystyle{plain}

\begin{thebibliography}{99}
	
	\bibitem{BBV20}
	R. Beekie, T. Buckmaster, and V. Vicol.
	\newblock Weak solutions of ideal MHD which do not conserve magnetic helicity.
	\newblock {\em Ann. PDE}, 6(1):Paper No. 1, 40pp, 2020.
	
	\bibitem{BDDCGT04}
	C. Besse, P. Degond, F. Deluzet, J. Claudel, G. Gallice,  and C. Tessieras. 
	\newblock A model hierarchy for ionospheric plasma modeling.
	\newblock{\em Math. Models Methods Appl. Sci.}, 14(3):393--415, 2004.

	
	\bibitem{Buc15}
	T. Buckmaster.
	\newblock Onsager's conjecture almost everywhere in time.
	\newblock{\em Comm. Math. Phys.}, 333(3):1175--1198, 2015.

	\bibitem{BCV22}
	T. Buckmaster, M. Colombo, and V. Vicol.
	\newblock Wild solutions of the Navier-Stokes equations whose singular sets in time have Hausdorff dimension strictly less than 1.
	\newblock {\em J. Eur. Math. Soc.}, 24(9):3333-3378, 2022.

	\bibitem{BDIS15}
	T. Buckmaster, C. De Lellis,  P. Isett, and L. Sz{\'e}kelyhidi, Jr.
	\newblock Anomalous dissipation for $1/5$-H{\"o}lder Euler flows.
	\newblock{\em Ann. of Math.}, 182(1):127--172, 2015.
		
	\bibitem{BDS16}
	T. Buckmaster, C. De Lellis,  and L. Sz{\'e}kelyhidi, Jr.
	\newblock Dissipative Euler flows with Onsager-critical spatial regularity.
	\newblock{\em Comm. Pure Appl. Math.}, 69(9):1613--1670, 2016.
		
	\bibitem{BDSV19}
	T. Buckmaster, C. De Lellis, L. Sz{\'e}kelyhidi, Jr., and V. Vicol.
	\newblock Onsager's conjecture for admissible weak solutions.
	\newblock{\em Comm. Pure Appl. Math.}, 72(2):229--274, 2019.	
		
	\bibitem{BCP00}
	A. Bressan, G. Crasta, and B. Piccoli.
	\newblock Well-posedness of the Cauchy problem for n $\times$ n systems of conservation laws. 
	\newblock{\em Mem. Amer. Math. Soc.}, 146(694): viii+134pp, 2000.	
		
	\bibitem{BV19}
	T. Buckmaster and V. Vicol.
	\newblock Nonuniqueness of weak solutions to the Navier-Stokes equation.
	\newblock{\em Ann. of Math.}, 189(1):101--144, 2019.
	
	\bibitem{Chen84}
	F. Chen.
	\newblock {\em Introduction to Plasma Physics and Controlled Fusion}, vol. 1, Plenum Press, New York, 1984.

	\bibitem{CVY21}
	R. Chen, A. Vasseur, and C. Yu.
	\newblock Global ill-posedness for a dense set of initial data to the isentropic system of gas dynamics.
	\newblock {\em Adv. Math.}, 393:Paper No. 108057, 46pp, 2021.

	\bibitem{CL22}
	A. Cheskidov and X. Luo.
	\newblock Sharp nonuniqueness for the Navier-Stokes equations.
	\newblock{\em Invent. Math.}, 229(3):987--1054, 2022.
	
	\bibitem{CDK15}
	E. Chiodaroli, C. De Lellis, and O. Kreml.
	\newblock Global ill-posedness of the isentropic system of gas dynamics.
	\newblock {\em Comm. Pure Appl. Math.}, 68(7):1157-1190, 2015.

	\bibitem{CF22}
	E. Chiodaroli and E. Feireisl.
	\newblock	On the density of ''wild'' initial data for the compressible Euler system. 
	\newblock	arXiv:2208.04810, 2022.
	
	\bibitem{CKMS21}
	E. Chiodaroli, O. Kreml, V. M\'acha, and S. Schwarzacher.
	\newblock Non-uniqueness of admissible weak solutions to the compressible Euler equations with smooth initial data.
	\newblock {\em Trans. Am. Math. Soc.}, 374(4):2269-2295, 2021.

	\bibitem{CET94}
	P. Constantin, W. E,  and E. S. Titi. 
	\newblock Onsager's conjecture on the energy conservation for solutions of Euler's equation.
	\newblock{\em Comm. Math. Phys.}, 165(1):207--209, 1994.
	
	\bibitem{Daf05}
	C. M. Dafermos. 
	\newblock {\em Hyperbolic Conservation Laws in Continuum Physics}, vol. 3, Springer-Verlag, Berlin, 2010.
	
	
	\bibitem{DS17}
	S. Daneri and L. Sz{\'e}kelyhidi, Jr.
	\newblock Non-uniqueness and h-principle for H{\"o}lder-continuous weak solutions of the Euler equations.
	\newblock{\em Arch. Ration. Mech. Anal.}, 224(2):471--514, 2017.
	
	\bibitem{DK22}
	C. De Lellis and H. Kwon.
	\newblock On nonuniqueness of H{\"o}lder continuous globally dissipative Euler flows.
	\newblock{\em Anal. PDE}, 15(8):2003--2059, 2022.
	
	\bibitem{DS09}
	C. De Lellis and L. Sz{\'e}kelyhidi, Jr.
	\newblock The Euler equations as a differential inclusion.
	\newblock{\em Ann. of Math.}, 170(3):1417--1436, 2009.
	
	\bibitem{DS10}
	C. De Lellis and L. Sz{\'e}kelyhidi, Jr.
	\newblock On admissibility criteria for weak solutions of the Euler equation.
	\newblock{\em Arch. Ration. Mech. Anal.}, 195(1):225--260, 2010.
	
			
	\bibitem{DS13}
	C. De Lellis and L. Sz{\'e}kelyhidi, Jr.
	\newblock Dissipative continuous Euler flows.
	\newblock{\em Invent. Math.}, 193(2):377--407, 2013.
	
	\bibitem{DS14}
	C. De Lellis and L. Sz{\'e}kelyhidi, Jr.
	\newblock Dissipative Euler flows and Onsager's conjecture.
	\newblock{\em J. Eur. Math. Soc.}, 16(7):1467--1505, 2014.
	
	\bibitem{DIP17}
	Y. Deng, A. D. Ionescu, and B. Pausader.
	\newblock The Euler-Maxwell system for electrons: global solutions in 2D.
	\newblock{\em  Arch. Ration. Mech. Anal.}, 225(2):771--871, 2017.
	
	\bibitem{GJK21}
	S. S. Ghoshal, A. Jana, and K. Koumatos.
	\newblock On the uniqueness of solutions to hyperbolic systems of conservation laws.
	\newblock{\em  J. Differential Equations}, 291:110--153, 2021.
	
	\bibitem{GK22}
	V. Giri, and H. Kwon.
	\newblock On non-uniqueness of continuous entropy solutions to the isentropic compressible Euler equations.
	\newblock{\em Arch. Ration. Mech. Anal.}, 245(2):1213--1283, 2022.
	
	\bibitem{Gra08}
	L. Grafakos.
	\newblock {\em Classical Fourier Analysis}, Vol.3, Springer, New York, 2014.
	
	\bibitem{GIP16}
	Y. Guo, A. D. Ionescu, and B. Pausader.
	\newblock Global solutions of the Euler-Maxwell two-fluid system in 3D.
	\newblock{\em  Ann. of Math.}, 183(2):377--498, 2016.		
	
	\bibitem{IL18}
	A. D. Ionescu and V. Lie.
	\newblock Long term regularity of the one-fluid Euler-Maxwell system in 3D with vorticity.
	\newblock{\em  Adv. Math.}, 326:719--769, 2018.		
			
	\bibitem{Isett18}
	P. Isett.
	\newblock A proof of Onsager's conjecture.
	\newblock{\em Ann. of Math.}, 188(3):871--963, 2018.		
		
	\bibitem{Isett22}
	P. Isett.
	\newblock Nonuniqueness and existence of continuous, globally dissipative Euler flows.
	\newblock{\em Arch. Ration. Mech. Anal.}, 244(3):1223--1309, 2022.	
			
	\bibitem{IO16}
	P. Isett and S. Oh. 
	\newblock On nonperiodic Euler flows with H{\"o}lder regularity.
	\newblock{\em Arch. Ration. Mech. Anal.}, 221(2):725--804, 2016.		

	\bibitem{KKMM20}
	C. Klingenberg, O. Kreml, V. M\'acha, and S. Markfelder.
	\newblock Shocks make the Riemann problem for the full Euler system in multiple space dimensions ill-posed.
	\newblock {\em Nonlinearity}, 33(12):6517-6540, 2020.

	\bibitem{LQZZ22}
	Y. Li, P. Qu, Z. Zeng, and D. Zhang.
	\newblock Sharp non-uniqueness for the 3D hyperdissipative Navier-Stokes equations: above the Lions exponent.
	\newblock arXiv:2205.10260, 2022.
	
	\bibitem{LZZ22}
	Y. Li, Z. Zeng, and D. Zhang.
	\newblock  Sharp non-uniqueness of weak solutions to 3D magnetohydrodynamic equations.
	\newblock  arXiv: 2208.00624, 2022.

	\bibitem{LZZ21}
	Y. Li, Z. Zeng, and D. Zhang.
	\newblock Non-uniqueness of weak solutions to 3D magnetohydrodynamic equations.
	\newblock {\em J. Math. Pures Appl.}, 165(9):232-285, 2022.

	\bibitem{LGP19}
	C. Liu, Z. Guo, and Y. Peng.
	\newblock Global stability of large steady-states for an isentropic Euler-Maxwell system in $\mathbb{R}^3$.
	\newblock{\em Commun. Math. Sci.}, 17(7):1841--1860, 2019.
		
	\bibitem{LP17}
	C. Liu and Y. Peng.
	\newblock Stability of periodic steady-state solutions to a non-isentropic Euler-Maxwell system.
	\newblock{\em Z. Angew. Math. Phys.}, 68(5):Paper No. 105, 17pp, 2017.	
	
	\bibitem{LY99}
	T. Liu and T.Yang.
	\newblock Well-posedness theory for hyperbolic conservation laws.
	\newblock{\em Comm. Pure Appl. Math.}, 52(12):1553--1586, 1999.	
	
	
	\bibitem{LQ20}
	T. Luo and P. Qu.
	\newblock Non-uniqueness of weak solutions to 2D hypoviscous
	Navier-Stokes equations.
	\newblock {\em J. Differential Equations}, 269(4):2896--2919, 2020.
	
	\bibitem{LT20}
	T. Luo and E.S. Titi.
	\newblock Non-uniqueness of weak solutions to hyperviscous Navier-Stokes equations: on sharpness of J.-L. Lions exponent.
	\newblock {\em Calc. Var. Partial Differential Equations}, 59(3):Paper No. 92, 15pp, 2020.
	
	\bibitem{LXX16}
	T. Luo, C. Xie, and Z. Xin.
	\newblock Non-uniqueness of admissible weak solutions to compressible Euler systems with source terms.
	\newblock {\em Adv. Math.}, 291:542-583, 2016.
	
	
	\bibitem{MK18}
	S. Markfelder and C. Klingenberg.
	\newblock The Riemann problem for the multidimensional isentropic system of gas dynamics is ill-posed if it contains a shock.
	\newblock {\em Arch. Ration. Mech. Anal.}, 227(3):967-994, 2018.
	
	
	\bibitem{MY22}
	C. Miao and W. Ye.
	\newblock On the weak solutions for the MHD systems with controllable total energy and cross helicity.
	\newblock arXiv:2208.08311, 2022.
	
	\bibitem{Peng15}
	Y. Peng. 
	\newblock Stability of non-constant equilibrium solutions for Euler-Maxwell equations.
	\newblock{\em J. Math. Pures Appl.}, 103(1):39--67, 2015.
	
	\bibitem{RG69}
	H. Rishbeth and O. K. Garriott. 
	\newblock {\em Introduction to Ionospheric Physics}, Academic Press, 1969.
	
	\bibitem{Sch93}
	V. Scheffer.
	\newblock An inviscid flow with compact support in space-time. 
	\newblock{\em J. Geom. Anal}, 3(4):343--401, 1993.
	
	\bibitem{Shn97}
	A. Shnirelman.
	\newblock On the nonuniqueness of weak solution of the Euler equation.
	\newblock{\em Comm. Pure Appl. Math.}, 50(12):1261--1286, 1997.
	
	\bibitem{Shn00}
	A. Shnirelman.
	\newblock Weak solutions with decreasing energy of incompressible Euler equations.
	\newblock{\em Comm. Math. Phys.}, 210(3):541--603, 2000.
	
	\bibitem{SM85}
	G. S{\"o}derlind and R. Mattheij. 
	\newblock Stability and asymptotic estimates in nonautonomous linear differential systems.
	\newblock{\em SIAM J. Math. Anal.}, 16(1):69--92, 1985.
	
	\bibitem{Wie18}
	E. Wiedemann.
	\newblock Weak-strong uniqueness in fluid dynamics. 
	\newblock{\em London Math. Soc. Lecture Note Ser.}, 452:289--326, 2018.
		
	
	


	
\end{thebibliography}

\end{document}